\documentclass[10pt]{amsart}

\usepackage{lmodern}
\usepackage{relsize}
\usepackage[bbgreekl]{mathbbol}
\usepackage{amsfonts}
\DeclareSymbolFontAlphabet{\mathbb}{AMSb} 
\DeclareSymbolFontAlphabet{\mathbbl}{bbold}

\usepackage{amsmath}
\usepackage{bm}
\RequirePackage{etoolbox}
\usepackage[all,cmtip]{xy}
\usepackage{stmaryrd}

\usepackage{amssymb}
\usepackage{mathrsfs}
\usepackage{amsthm}
\usepackage{enumerate}
\usepackage{leftidx}
\usepackage{wasysym}
\usepackage{indentfirst}
\usepackage[OT2,T1]{fontenc}
\usepackage{epstopdf}
\usepackage {hyperref}
\hypersetup{
	colorlinks=false,       
	linkcolor=red,          
	citecolor=brown,        
}

\usepackage{color}

\newtheorem{thm}[subsection]{Theorem}

\newtheorem{cor}[subsection]{Corollary}
\newtheorem{lem}[subsection]{Lemma}

\theoremstyle{definition}
\newtheorem{defn}[subsection]{Definition}
\theoremstyle{remark}

\DeclareMathOperator{\ab}{ab}
\DeclareMathOperator{\ad}{ad}

\DeclareMathOperator{\Fr}{Fr}

\DeclareMathOperator{\Gal}{Gal}

\DeclareMathOperator{\Sh}{Sh}

\newcommand{\To}{\longrightarrow}
\newcommand{\isom}{%
	\ifbool{@display}{\overset{\sim}{\longrightarrow}}{\xrightarrow\sim}%
}

\newcommand{\ZZ}{\mathbb{Z}}
\newcommand{\QQ}{\mathbb{Q}}
\newcommand{\RR}{\mathbb{R}}

\newcommand{\FF}{\mathbb{F}}

\newcommand{\CC}{\mathbb{C}}

\newcommand{\ignore}[1]{}

\DeclareSymbolFont{cyrletters}{OT2}{wncyr}{m}{n}
\DeclareMathSymbol{\Sha}{\mathalpha}{cyrletters}{"58}

\newcommand{\fkk}{{\mathfrak k}}

\newcommand{\fkp}{{\mathfrak p}}

\newcommand{\der}{\mathrm{der}}

\def\Gal{{\rm Gal}}

\newcommand{\GL}{{\rm GL}}

\newcommand{\inv}{{\rm inv}}

\newcommand{\Qpbar}{\overline{\mathbb Q}_p}

\def\hat{\widehat}

\def\tilde{\widetilde}
\def\rho{\varrho}
\newcommand{\SL}{\mathrm{SL}}

\def\benu{\begin{enumerate}}
	\def\eenu{\end{enumerate}}
\def\beq{\begin{equation}}
\def\eeq{\end{equation}}
\def\bit{\begin{itemize}}
	\def\eit{\end{itemize}}

\newenvironment{lr}[1][]{\refstepcounter{equation}\par\medskip
	\noindent {(\theequation) #1} }{\medskip}

\newenvironment{lrnumber}[1][]{\refstepcounter{subsection}\par\medskip
	\noindent {(\thesubsection) #1} }{\medskip}

\title{Shimura varieties and gerbes}
 
 \author[R.~P.~Langlands and  M.~Rapoport]{\textit{R.~P.~Langlands} in Princeton and \textit{M.~Rapoport} in Bonn}
 
\thanks{{Last updated:  \today}}

\begin{document}

	\maketitle
	\tableofcontents
	
	\section*{Translator's note}
	
	This document presents an English translation of the paper \textit{Shimuravarietäten und Gerben} by Langlands and Rapoport, originally published in \textit{Journal für die reine und angewandte Mathematik,} 378 (1987), pages 113–220.
	
	The translation was completed by Yihang Zhu. While the translator is not fluent in German, he has prior familiarity with certain parts of the paper. The translation process was assisted by AI tools, which also helped generate the LaTeX code for the formulas. The re-typed version \cite{LRonline}
	served as the input for these tools.
	
	One of the main challenges in this project was the need to manually correct, and in some cases completely re-type, the AI-generated formulas by comparing them with the original text. This was partly due to the inherent imprecision of AI and partly because the re-typed version mentioned above contained additional typos. Further complicating the process, the notations in the original text were occasionally so intricate—and not entirely free of errors—that accurately discerning the intended formulas required a non-trivial mathematical understanding.
	
	The translator corrected obvious typos and misprints in the original text without explicit mention. However, in cases where corrections were less straightforward, he chose to remain faithful to the original text rather than making speculative modifications.

All footnotes have been added by the translator with the aim of providing clarifications, though it was not his intention to offer thorough annotations. The original page numbers are displayed in the margins. In Sections 3 and 5, the places where proofs begin and end are marked in the margins.

A recent trend in the mathematical literature surrounding this paper favors spelling ``gerbe'' as ``gerb'' in English. However, this translation adheres to the more established spelling ``gerbe''. 
 
A few words about later developments following this paper are in order. The main construction in \S 3, that of the quasimotivic Galois group $\mathscr Q$, contains a well-known mistake, and corrections have been provided by Pfau \cite{PfauThesis,Pfau96} and Reimann \cite{reimann1997zeta}. See the footnote above \eqref{3.e} for details. The ``future works of Kottwitz'' mentioned several times in the paper have appeared as \cite{Kot90} and  \cite{kottwitz1992points}. The conjecture in \S 5 is now known as the \emph{Langlands--Rapoport Conjecture}. The paper 
\cite{milne92} contains some useful discussions of this conjecture; see also \cite{milne2003gerbes}. Some  special cases of the conjecture are proved in \cite{reimann1997zeta}. 
Further refinements and generalizations of the conjecture  are given in \cite{rapoportguide}, whose emphasis is on parahoric level, and in  \cite{kisin2012modp}, where the assumption that $G_{\der}$ is simply connected is removed, and a refined version involving connected components is formulated, among other things. Moreover, in \cite{kisin2012modp}, a weak version of the Langlands--Rapoport Conjecture for Shimura varieties of abelian type of hyperspecial level is proved. This work has the sequel \cite{KSZ}, where the weak version is strengthened, from which a stable trace formula for the Frobenius--Hecke action on the cohomology is derived.  (In the case of Hodge type, the paper \cite{lee} derives a Lefschetz number formula from results in \cite{kisin2012modp} in the style of Kottwitz \cite{Kot90}, circumventing the Langlands--Rapoport Conjecture.) The paper \cite{vanHoftenLR} extends, in some cases, the results of \cite{kisin2012modp} to parahoric level.
 
The translator thanks Michael Rapoport and Sug Woo Shin for some useful suggestions. 

Corrections and suggestions are welcome—please send them to Yihang Zhu
\begin{center}
	\texttt{\href{mailto:yhzhu@tsinghua.edu.cn}{yhzhu@tsinghua.edu.cn}}
\end{center}

 	\bibliographystyle{alpha}

 \bibliography{myref}
	\newpage
	\section{Introduction}  
	\numberwithin{equation}{section}
	The\marginpar{113} problem of the extension of Hasse-Weil zeta functions and, more generally, of motivic L-functions remains a central issue in number theory. It is often divided into two problems that need to be solved separately. First, it is necessary to show that every motivic L-function is equal to an automorphic L-function, and second, to prove that every automorphic L-function can be extended. Both problems have been solved in very few cases, and then only thanks to the efforts of many mathematicians over a long period of time.
	
	After abelian varieties, Shimura varieties are probably the most accessible in arithmetic terms, and this work aims to contribute to the first problem for the motivic L-functions associated with them. Shimura varieties arise from the theory of automorphic functions and are defined by a reductive group \( G \) with additional structure. Experience suggests that the solution to the problem in a given case consists of two components: a group-theoretic description of the points of the reduced variety, and combinatorial arguments that include the fundamental lemma from [L3], in order to transform the resulting expression into a form that can be compared with the Arthur--Selberg trace formula.
	
	Here, we focus solely on the first problem and refer to future works by R. Kottwitz for the transformation.
	
	The goal of this work is to formulate a conjecture about the reduction modulo \( p \) of Shimura varieties, make it plausible by relating it to the standard conjectures of Grothendieck, and prove the facts leading to a purely group-theoretic formula for the number of points of the reduction modulo \( p \). We have to make some restrictive assumptions about the prime number \(p\), which, however, allow certain cases of bad reduction in addition to the case of good reduction. Let  \(
	S = \text{Sh}(G, h) \quad (K = K^p K_p \subset G(\mathbb A_f))
	 \)
	be a Shimura variety, defined over the reflex field \( E = E(G, h) \), and let \( \mathfrak p \) be a prime of \( E \) above \( p \), which is good. Let \( \kappa \) be the residue field of \( E_{\mathfrak p } \), and \( \kappa_m \) the degree \( m \) extension. The formula takes the following form: The number 	\( |S(\kappa_m)| \) is equal to 
	\begin{align}\label{1.1}\sum_{\{ \varepsilon\}_{\text{st}}}  \iota(\varepsilon) \cdot \text{vol}\left( G_{\varepsilon}(\QQ) Z_K \backslash G_{\varepsilon}(\mathbb A_f) \right) \cdot 
\sum_{\substack{\{\gamma \}, \{\delta \} \\ \kappa(\varepsilon; \gamma, \delta) = 1}} O_\gamma(f^p) TO_\delta(\phi_p).
	\end{align}
\marginpar{114}Here \( Z_K \) is equal to \(Z(\mathbb A_f) \cap K \); the outer sum runs over the stable conjugacy classes of \( Z(\QQ) \cap Z_K \backslash G(\QQ) \). The number \( \iota(\varepsilon) \) is a cohomological invariant:
	\[
	\iota(\varepsilon) = \left| \ker \left( H^1(\QQ, G_{\varepsilon}) \to H^1(\mathbb A, G_{\varepsilon}) \times H^1(\QQ, G_{\text{ab}}) \right) \right|.
	\]
	The inner sum runs over conjugacy classes of elements \( \gamma \in G(\mathbb A_f^p) \) and twisted conjugacy classes of elements \( \delta \in G(L) \), where \( L \) is the unramified extension of \( \mathbb{Q}_p \) with residue field \( \kappa_m \). The symbols \( O_\gamma \) and \( TO_\delta \) refer to orbit integrals and twisted orbit integrals, and the function on \( G(\mathbb A_f^p) \) is given by  
	\[
	f^p = \frac{1}{\text{vol}(K^p)} \cdot \text{char}_{K^p},
	\]  
	while \( \phi_p \) is an element defined by the Shimura variety in the Hecke algebra of \( G(L) \) with respect to \( K_p \). 
	The pairs \( \gamma, \delta \) are required to satisfy that \( \gamma \) is stably conjugate to \( \varepsilon \), the norm of \( \delta \) is stably conjugate to \( \varepsilon \), and most importantly, that the Kottwitz invariant \( \kappa(\varepsilon; \gamma, \delta) \) is defined and equal to 1. Although this formula does not appear in our work, and is reserved for a future work by Kottwitz, it can be easily derived from our results. This follows from the results of Kottwitz [K3].
	
	In the case where the Shimura variety \( S_K \) is compact, the numbers in \eqref{1.1} appear in the power series expansion of the logarithm of the Hasse-Weil Zeta function of \( S_K \) at the good prime \( p \). To compare the right hand side of \eqref{1.1} with the corresponding power series expansion of automorphic L-functions, associated with automorphic representations of \( G \), one would first need to rewrite the twisted orbital integral of \( \phi_p \) into a standard orbital integral of a function \( f_p \) on \( G(\mathbb{Q}_p) \).
	If the fundamental lemma were generally available, this could be achieved after stabilization, and with Kottwitz's theorem [K3], the function \( f_p \) could then be written explicitly. However, it is the appearance of the condition on the Kottwitz invariant that complicates this simple approach. This condition can be seen as the reason for introducing the endoscopic groups in this context.	The way out of this dilemma is to compare the Hasse-Weil Zeta function not with an automorphic L-function associated with \( G \), but with a product of automorphic L-functions associated with automorphic representations of endoscopic groups of \( G \) such as a product of the form
	\begin{align} \label{1.2}
\prod_H \prod_{\Pi} L(s, \Pi, r_H)^{m(H) } .  
	\end{align}  Here, \( \Pi \) is an L-packet, \( r_H \) is a virtual representation of the L-group \( ^L H \), and the exponent \( m \) is possibly a rational number. One might expect that the power series expansion of the logarithm of \eqref{1.2} can be written as a sum of stabilized traces in a certain sense:
\begin{align}\label{1.3}
\sum_H ST(f_H) .
\end{align} 
\marginpar{115}The functions \( f_H \) are dictated by the form of \eqref{1.2}. The further strategy is now to express the individual terms of \eqref{1.3} in terms of the stable
 trace formulas for the various \( H \), and on the other hand, to carry out the stabilization of the right hand side of \eqref{1.1}, which also leads to a sum of stable trace formulas for the endoscopic groups corresponding to appropriate functions \( f_H' \). Finally, the functions \( f_H \) and \( f_H' \) would need to be compared. Although this program is still future work, Kottwitz has succeeded in writing the elliptic part of the sum on the right hand side of \eqref{1.1} in the appropriate form. 

As a side note, it follows easily from the representation of the Hasse-Weil Zeta function in the form \eqref{1.2} that the eigenvalues of the Frobenius element are roots of polynomials whose coefficients are formed from the eigenvalues of Hecke operators. This statement is also a consequence of congruence relations, which in many cases are relatively simple to prove. However, while in the case of the group \( \GL(2) \), the representation of the Hasse-Weil Zeta function as a product of automorphic L-functions is a consequence of congruence relations, this is hardly to be expected for other groups.

For the computation of the Zeta function, it suffices to have the formula \eqref{1.1} available, instead of our conjecture, and the attempt to prove this formula in the case of the group \( G \) of symplectic similitudes leads to a relatively elementary but subtle question about principally polarized abelian varieties, which we will explicitly state at this point. Let \( (A, \Lambda) \) be an abelian variety of CM type over \( \mathbb{C} \) with a \( \mathbb{Q} \)-class of polarizations. Then \( (A, \Lambda) \) is already defined over a finite extension of \( \mathbb{Q} \), and at a given prime \( \mathfrak p \) it has good reduction \( (\bar A, \bar \Lambda) \)  over the finite field \( \mathbb{F}_q , q = p^n \). The Frobenius endomorphism of \( (\bar A, \bar \Lambda) \) over \( \mathbb{F}_q \) lifts to an endomorphism of \( (A, \Lambda) \), which, through its action on rational cohomology, defines an element \( \varepsilon \in G(\mathbb{Q}) \). On the other hand, the $l$-adic cohomology \( H^1(\bar  A \otimes _{\mathbb{F}_q} \overline {\mathbb{F}} _q  , \QQ_l ) \) for all \( l \neq p \) defines a conjugacy class \( \gamma \in G(\mathbb A_f^p) \), and crystalline cohomology (i.e., the theory of Dieudonné modules) defines a twisted conjugacy class \( \delta \in G(L) \). Kottwitz's conjecture is that the invariant \( \kappa(\varepsilon; \gamma, \delta) \) equals $1$.
	
	The focus of the present work, however, is not on proving our conjecture or its weakening, but on its formulation. The conjecture describes the points of reduction modulo \( p \) as a disjoint union of double cosets, parametrized by conjugacy classes of \emph{admissible} homomorphisms of \emph{gerbes} into the neutral gerbe associated with \( G \), 
	\[
	\phi:\mathscr  P \to \mathscr G_G.
	\]
	Here, \(\mathscr  P \) is an explicitly constructed gerbe. In order to carry out this construction, we have chosen to define the concept of a gerbe in the simplest possible way, as an extension of a Galois group by an algebraic group (the ``kernel'' of the gerbe), since we felt unable to work with the abstract notion. The construction of \( \mathscr P \) is a major concern of this work, as is the construction of the homomorphism
	\[
	\psi_{T,\mu}: \mathscr P \to \mathscr G_T,
	\] \marginpar{116}associated with an algebraic torus \( T \) defined over \( \mathbb{Q} \) and a cocharacter \( \mu \) of \( T \). If \( (T, h_T) \subset (G, h) \) defines a ``special point'' of the Shimura variety, we obtain from \( \psi_{T,\mu} \), with \( \mu = \mu_{h_T} \), by composition with the inclusion \( \mathscr G_T \subset \mathscr G_G \), a homomorphism \( \psi_{T,\mu}: \mathscr P \to \mathscr G_G \), which turns out to be admissible. In the case of good reduction, every admissible homomorphism is conjugate to a homomorphism of this form for an appropriate \( (T, h_T) \). Since there is an element \( \delta \) in the kernel of \( \mathscr P \) (more precisely, a sequence of elements \( \delta_n \), which are powers of one another), whose image in \( T \) under \( \psi_{T,\mu} \) is rational, we obtain from \( \psi_{T,\mu} \) an element \( \gamma \in G(\mathbb{Q}) \). It is easy to show that \( (\gamma, h_T) \) is a Frobenius pair in the sense of [L1], and that the mapping \( \phi \mapsto (\gamma, h) \) gives a bijection between the \emph{local} equivalence classes of admissible homomorphisms \( \phi \) and the equivalence classes of Frobenius pairs. Here, two homomorphisms of gerbes over \( \mathbb{Q} \) are said to be locally equivalent if their localizations at each place are conjugate.	This shows that the conjecture formulated here follows from the one first stated in [L1]; however, the new conjecture is more precise, insofar as the embedding of the group \( I(\mathbb A_p^f) \) into \( G(\mathbb A_p^f) \) defined in [L1] is not defined there with  sufficient precision to derive the formula  \eqref{1.1} from it.
	
	However, the new conjecture also has another advantage over the old one, which even historically is its origin. The old conjecture fails in the simplest cases of bad reduction. This is reflected, on the one hand, in the fact that not every admissible homomorphism is of the form \( \psi_{T,\mu} \), and on the other hand, in the fact that the lifting theorem of Zink [Z1] no longer holds in these cases. In this work, we have largely neglected the case of bad reduction. The second author intends to treat some cases of bad reduction in more detail in a forthcoming work, thus shedding light on the conjecture from a second perspective.
	
	The nature of the new conjecture has its philosophical background in Grothendieck's theory of motives. The existence of the category of motives over a finite field can so far only be proven using still unproven standard conjectures. On the other hand, the theory of Tannakian categories, designed by Grothendieck and developed by Saavedra Rivano in [Sa] and by Deligne and Milne in [DM], associates a gerbe (we have kept the French word because the closest translations
	seem to be already taken) to an abelian category with additional structures, particularly a tensor product.  The gerbe \( \mathscr P \) is the one that would arise from the category of motives over finite fields, if it existed.
	
	We now explain the organization of the article. In \S 2, we define a gerbe \(\mathscr  T_{\nu_1, \nu_2} \) with additional structures over a global field, which arises from a torus \( T \) and two cocharacters \( \nu_1, \nu_2 \). If these cocharacters arise from ``averaging'' a single cocharacter \( \mu \), we construct an explicit neutralization \( \mathscr T_{\nu_1, \nu_2} \cong \mathscr G_T \). The main results of this section are Theorems \ref{2.1} and \ref{2.2}. The construction uses class field theory and the theory of Tate-Nakayama and proceeds similarly to the construction of the Taniyama group. We believe it is a useful endeavor to characterize the construction via a universal property. This would replace our inconsistent cocycle verifications and perhaps provide more insight into the prevailing relations here. In \S 3, we apply these constructions to the torus \( T \), whose character group is generated by the set of Weil-q numbers.\marginpar{117} We show that this torus has such simple cohomology that the gerbe \( \mathscr P \) constructed in this way, as well as the homomorphisms \( \psi_{T,\mu}: \mathscr P \to \mathscr G_T \) associated with \( (T, \mu) \), can be uniquely characterized. Sections 2 and 3 provide the basis for the formulation of our conjecture. The conjecture about the reduction modulo \( p \) of a Shimura variety \( S(G, h) \) is formulated at the beginning of \S 5. There, we also precisely explain what conditions we place on \( G \) and the prime \( p \). The rest of \S 5 is dedicated to proving Theorems \ref{5.21} and \ref{5.25}, which allow the transition from the conjecture to the formula \eqref{1.1}. A significant portion of this section is devoted to the definition of the Kottwitz invariant. The remaining sections are of less central importance to our work. In \S 4, we show, assuming the Tate and Hodge conjectures, that the gerbe corresponding to the Tannakian category \( M_{\mathbb F} \) of motives over finite fields, constructed using standard conjectures, is isomorphic to the gerbe \( \mathscr  P \) from Section 3, together with additional structures defined, for example, by various cohomology theories. 
	
	This section is essentially a somewhat more detailed, though still sketchy, exposition of Chapter VI in [Sa]. In \S 6, we show that the identification of gerbes made in \S 4 implies the conjecture of \S 5 for Shimura varieties that solve certain moduli problems, such as for the variety corresponding to the symplectic group. This section relies on the work of T. Zink [Z1], [Z2]. Since this is a conditional result, we have kept it brief. However, we explicitly point out the proof of Theorem \ref{6.3}, where Theorem \ref{4.4} is used. It is at this point that the Kottwitz conjecture formulated earlier for abelian varieties should come into play, so that the formula \eqref{1.1} should follow from it, at least in the case of the symplectic group. We have indeed structured the proof intentionally so that it may possibly apply to the proof of formula \eqref{1.1}. In \S 7, we provide two examples: one shows that in the conjecture of \S  5, we must assume the derived group to be simply connected, and the other shows that in the case of bad reduction, many admissible homomorphisms \( \phi \) are not of the form \( \psi_{T,\mu} \).
	
	 Two questions that are almost unavoidable when one tries to prove the conjecture in general are not addressed in this work. First, if \( G \to G' \) is a central extension with \( G_{\text{der}} = G'_{\text{der}} = G_{\text{sc}} \), and if the conjecture holds for \( G' \), does it also hold for \( G \)? Secondly, it often happens that a Shimura variety \( \text{Sh}_1 \) is naturally embedded into a second \( \text{Sh}_2 \). This is especially the case for special points, for which \( \text{Sh}_1 \) is zero-dimensional. How does this embedding reflect in the description of the reduction of both varieties?
	
	We would like to thank Kottwitz for sharing his results and ideas with us, and for making his notes available. The first author presented the current results during the Issai Schur Memorial Lecture in Tel Aviv in May 1984 and the Ritt Lecture at Columbia University in February 1985, and thanks the audiences at both venues for their patience and understanding when listening to still provisional material. We also both thank the Humboldt Foundation for supporting the first author's repeated visits to Heidelberg, during which this work was carried out.
	
	In the following, let \( p \) be a fixed prime number, \( \overline{\mathbb{Q}} \) the field of algebraic numbers, which we have embedded into an algebraic closure \( \overline{\mathbb{Q}}_p \) of \( \mathbb{Q}_p \). We denote by \( \mathbb F \) the corresponding algebraic closure of \( \mathbb{F}_p \). 
	
	\section{Cohomological constructions}
\renewcommand{\theequation}{\thesection.\alph{equation}}

\marginpar{118}	
This section is devoted to some constructions in Galois cohomology. First we explain our terminology.

Let \( k \) be a field of characteristic zero, which for us will either be a global or local field. Let \( G \) be an algebraic group over an algebraic closure \( \overline{k} \). If \( G = \text{Spec}\, A \), where \( A = \Gamma(G) \), and if \( \sigma \in \text{Gal}(\overline{k}/k) \), an automorphism \( \kappa \) of \( G(\overline{k}) \) is called \( \sigma \)-linear if there exists a \( \sigma \)-linear automorphism \( \kappa' \) of the algebra \( A \) such that
\[
\kappa'(f)(\kappa(g)) = \sigma(f(g)), \quad f \in A, \, g \in G(\overline k).
\]

The simplest example arises from the action of \( \text{Gal}(\overline{k}/k) \) on \( G(\overline{k}) \) if \( G \) is defined over \( k \). A Galois gerbe over \( k \) is a group extension
\[
1 \longrightarrow G(\overline{k}) \longrightarrow \mathscr{G} \longrightarrow \text{Gal}(\overline{k}/k) \longrightarrow 1,
\]
together with a section \( \sigma \mapsto g_\sigma \) for \( \sigma \in \text{Gal}(\overline{k}/K) \), for a suitable finite extension \( K \) of \( k \), such that the \( \sigma \)-linear automorphism
\[
\kappa(\sigma): g \longmapsto g_\sigma g g_{\sigma}^{-1}, \quad g \in G(\overline{k}),
\]
arises from a \( K \)-structure on \( G \). Additionally, for each \( \sigma \in \text{Gal}(\overline{k}/k) \) and each representative \( g_\sigma \), the automorphism \( \kappa(\sigma) \) must be \( \sigma \)-linear. It is understood that the finite extension \( K \) can be replaced by a larger finite extension \( K' \), so that in reality, we are dealing with a germ of sections \( \{g_\sigma\} \) for the Krull topology on \( \text{Gal}(\overline{k}/k) \). In the notation, we often omit the germ of sections and simply denote a Galois gerbe, sometimes just called a gerbe, by \( \mathscr{G} \). We call \( G \) the \emph{kernel} of the gerbe. A homomorphism of gerbes is a homomorphism of the corresponding extensions that maps the germs of sections into each other and whose restriction to the kernel is algebraic. An element \( g \) in the kernel defines an automorphism of a gerbe by conjugation. Indeed, for a sufficiently large finite extension \( K/k \), the element \( g \) is \( K \)-rational. We have 
\[
g_\sigma g g_{\sigma}^{-1} = \sigma(g) = g,
\]
so conjugation by \( g \) preserves the germ of sections. Two homomorphisms \( \phi_1 \) and \( \phi_2 \) between two gerbes are called \emph{equivalent}, if \( \phi_2 = \text{ad}\,g \circ \phi_1 \) with \( g \) in the kernel of the second gerbe. We refer to §4 for a comparison of our terminology with that commonly used in the theory of Tannakian categories. The simplest example of a gerbe is given by an algebraic group \( G \) defined over \( k \), the semi-direct product
\[
\mathscr{G}_G = G(\overline{k}) \rtimes \text{Gal}(\overline{k}/k).
\]
Such a gerbe is called \emph{neutral}.

We now introduce some gerbes over local fields, starting with the weight gerbe \( \mathscr{W} \), which is defined over \( \mathbb{R} \). It is an extension \marginpar{119}
\[  
1 \longrightarrow \mathbb{C}^\times \longrightarrow \mathscr{W} \longrightarrow \text{Gal}(\mathbb{C}/\mathbb{R}) \longrightarrow 1.
\]
If \( \text{Gal}(\mathbb{C}/\mathbb{R}) = \{1, \iota\} \), then \( \mathscr{W} \) is generated by \( \mathbb{C}^\times \) and \( w = w(\iota) \), where
\[
w(\iota)^2 = -1 \in \mathbb{C}^\times
\]
and
\[
w z w^{-1} = \iota(z) = \overline{z}, \quad z \in \mathbb{C}^\times.
\]
Here, \( \mathbb{C}^\times \) is to be understood as \( \mathbb{G}_m(\mathbb{C}) \).

Next, we introduce the Dieudonné gerbes. The actual Dieudonné gerbe will turn out to be a direct limit of such gerbes. Let \( K \) be a finite Galois extension of \( \mathbb{Q}_p \). We define an extension as follows:
\[
1 \longrightarrow K^\times \longrightarrow \mathscr{D}_K^K \longrightarrow \text{Gal}(K/\mathbb{Q}_p) \longrightarrow 1.
\]
It is generated by \( K^\times \) and \( d_K^K(\sigma) \), where \( d_K^K(\sigma) \mapsto \sigma \), \( d_K^K(\sigma) z = \sigma(z) d_K^K(\sigma) \) for \( z \in K^\times \), and \( d_K^K(\rho) d_K^K(\sigma) = d^K_{\rho,\sigma} d^K_K (\rho \sigma) \). Here, \( d^K_{\rho, \sigma} \) is a 2-cocycle in the fundamental class of the extension \( K/\mathbb{Q}_p \). The extension is uniquely determined up to isomorphism, and by Theorem 90, this isomorphism is unique up to conjugation by an element of \( K^\times \).

We obtain the gerbe \( \mathscr{D}^K \) by pulling back and pushing forward using the diagram
$$\xymatrix{  &&& \Gal(\Qpbar/\QQ_p) \ar[d] \\ 1 \ar[r] & K^\times \ar[d] \ar[r] & \mathscr D^K_K \ar[r] & \Gal(K/\QQ_p) \ar[r] & 1  \\ & \overline {\QQ}_p^\times} $$
The group \( \overline{\mathbb{Q}}_p^\times \) is again to be understood as \( \mathbb{G}_m(\overline{\mathbb{Q}}_p) \).

We denote the representative of \( \sigma \) in \( \mathscr{D}^K \) by \( d^K_\sigma = d_\sigma \), and regard \( d^K_{\rho,\sigma} \) as a cocycle of the group \( \text{Gal}(\overline{\mathbb{Q}}_p/\mathbb{Q}_p) \). Let \( K \subset  K' \). Then
\[
(d^{K'}_{\rho,\sigma})^k = d^K_{\rho, \sigma} c_\rho \rho (c_\sigma) c_{\rho \sigma }^{-1}, \quad k = [K' : K].
\]
We can thus define a homomorphism \( \mathscr{D}^{K'} \to \mathscr{D}^K \) by sending \( z \in \overline{\mathbb{Q}}_p^\times \) to \( z^k \) and \( d^{K'}_{\rho} \) to \( c_{\rho}^{-1} d^K_{\rho} \). By Theorem 90, this homomorphism is uniquely determined up to conjugation by an element of \( \overline{\mathbb{Q}}_p^\times \).

\begin{lem} \label{2.1}
	Let\marginpar{120} \( \phi: \mathscr{D}^K \to \mathscr{G} \) be a homomorphism of gerbes. Then there exists an unramified extension \( L \) of \( \mathbb{Q}_p \), a homomorphism \( \psi: \mathscr{D}^L \to \mathscr{G} \), and an extension \( K_1 \) containing both \( K \) and \( L \), such that the following diagram commutes:
	\[
	\xymatrix{ 
		\mathscr{D}^{K_1} \ar[r] \ar[d] & \mathscr{D}^L \ar[d]^{\psi} \\
		\mathscr{D}^K \ar[r]_{ \phi} & \mathscr{G}. }
	\]
\end{lem}

Let \( L \) be the unramified extension of \( \mathbb{Q}_p \) whose degree over \( \mathbb{Q}_p \) is equal to that of \( K \), and let \( K_1 \) be the compositum of \( L \) and \( K \). Since the two cocycles \( d^K_{\rho,\sigma} \) and \( d^L_{\rho,\sigma} \) are equivalent, \( \mathscr{D}^K \) and \( \mathscr{D}^L \) are isomorphic. The isomorphism is uniquely determined up to conjugation by an element of \( \overline{\mathbb{Q}}_p^\times \). The existence of \( \psi \) is therefore clear.

It is clear that the procedure we used to define \( \mathscr{D}^K \) can be applied over any local field. Over \( \mathbb{R} \), it leads to the group \( \mathscr{D}^\mathbb{C} = \mathscr{W} \), and we actually only need it over \( \mathbb{R} \) and \( \mathbb{Q}_p \). Nevertheless, for clarity, we will introduce the global gerbes of this section more generally than strictly necessary, and for this, we need the general \( \mathscr{D}^K \)'s. The trivial gerbe
\[
1 \to \text{Gal}(\overline{k}/k) \to \text{Gal}(\overline{k}/k) \to 1
\]
is denoted by \( \text{Gal}_k \), and when \( k = \mathbb{Q}_l \), by \( \mathscr{G}_l \).

Now let \( T \) be a torus over the finite number field \( F \), which splits over the Galois extension \( L \). As usual, let
\[
X^*(T) = \text{Hom}(T, \mathbb{G}_m), \quad X_*(T) = \text{Hom}(X^*(T), \mathbb{Z}).
\]
The element corresponding to \( \mu \in X_*(T) \) in \( \text{Hom}(\mathbb{G}_m, T) \) is written as \( x \mapsto x^\mu \). We fix two places \( v_1 \) and \( v_2 \) of \( F \), which we extend to \( L \). The extended places are denoted by \( v_1' \) and \( v_2' \). Let \( \nu_1 \) and \( \nu_2 \) be two elements of \( X_*(T) \) with the following properties:
\begin{lr}\label{2.a}
	\( \nu_i \) is invariant under \( \text{Gal}(L_{v_i'}/F_{v_i}) \).
\end{lr}
\begin{lr}\label{2.b}
	We have
	\[
	\sum_{\sigma \in \text{Gal}(L/F)/\text{Gal}(L_{v_1'}/F_{v_1})} \sigma \nu_1 + \sum_{\sigma \in \text{Gal}(L/F)/\text{Gal}(L_{v_2'}/F_{v_2})} \sigma \nu_2 = 0.
	\]
\end{lr} 

According to the Tate-Nakayama theory, \( \nu_i \) corresponds to a class \( \alpha_i \) in $$ H^2(\text{Gal}(L_{v_i'}/F_{v_i}), T(L_{v_i})) .$$ Due to \eqref{2.b}, there also exists a global class in \( H^2(\text{Gal}(L/F), T(L)) \) whose local components outside \( v_1 \) and \( v_2 \) are trivial and whose components at \( v_i \) are equal to \( \alpha_i \). This global class, though not uniquely defined, plays an important role in the construction of the gerbe \( \mathscr T_{\nu_1, \nu_2} \) associated with \( T, \nu_1, \nu_2 \). It is now explicitly introduced using the Weil group.

The\marginpar{121} gerbe \( \mathscr{T} = \mathscr{T}_{\nu_1, \nu_2} \) is a gerbe over \( F \) with kernel \( T \), hence an extension
\[
1 \to T(\overline{F}) \to \mathscr{T} \to \text{Gal}(\overline{F}/F) \to 1.
\]
To define it, we need a 2-cocycle \( \{t_{\rho,\sigma}\} \) of \( \text{Gal}(\overline{F}/F) \) with values in \( T(\overline{F}) \). Then \( \mathscr{T} \) is generated by \( t_\rho \) for \( \rho \in \text{Gal}(\overline{F}/F) \) and \( T(\overline{F}) \), and
\[
t_\rho t_\sigma = t_{\rho, \sigma} t_{\rho \sigma }.
\]

The localization \( \mathscr{G}_v \) of \( \mathscr{G} \) at the place \( v \) of \( F \) is a gerbe over \( F_v \), which can be defined by the following diagram:
\[\xymatrix{ 
	&&& \text{Gal}(\overline{F}_v/F_v) \ar[d] \\
	1 \ar[r] & G(\overline{F}) \ar[r] \ar[d]  & \mathscr{G} \ar[r] & \text{Gal}(\overline{F}/F)  \\
	& G(\overline{F}_v) }
\]
A homomorphism \( \phi \) from a gerbe \( \mathscr{H} \) defined over \( F_v \) to \( \mathscr{G} \) is a homomorphism from \( \mathscr{H} \) to \( \mathscr{G}_v \).

Usually, it is better to choose a finite extension \( K \) of \( F \) such that \( \mathscr G \) can be defined over \( K \), that is, by a diagram:
\[\xymatrix{ 
	&&& \text{Gal}(\overline{F} /F) \ar[d] \\
	1 \ar[r] & G(K) \ar[r] \ar[d]  & \mathscr{G}_K \ar[r] & \text{Gal}(K/F) \ar[r] & 1   \\
	& G(\overline{K} ) }
\]
If \( v_0' \) is an extension of \( v \) to \( K \), then \( \mathscr G_v \) can be defined over \( K_{v_0'} \), and if \( K \) is sufficiently large, \(\mathscr  H \) is also defined over \( K_{v_0'} \) and \( \phi \) (is induced) by \( \phi: \mathscr H_{K_{v_0'}} \to \mathscr G_{K_{v_0'}} \).

To avoid the choice of \( v_0 '\), we can use the embedding
\[
G(K) \to \prod_{v'|v} G(K_{v'})
\]
to introduce an extension
\[
1 \to \prod_{v'|v} G(K_{v'}) \to \mathscr G^*_v \to \text{Gal}(K/F) \to 1.
\]
In\marginpar{122} the case where \( G \) is defined over \( F \) and \( g_\sigma g g_{\sigma}^{-1} = \sigma(g) \), and we are not interested in any other, the extension \( \mathscr G^*_v \) can be defined using \( \mathscr G_v \) alone.

Indeed, to each \( v' \), we assign a \( \mu = \mu_{v'}  \in  \text{Gal}(K/F) \) such that \( |\mu x|_{v'} = |x|_{v_0'} \) for \( x \in K \). Then \( \mu: x \mapsto \mu x \) extends to an isomorphism \( \mu: K_{v_0'} \to K_{v'} \), which in turn defines \( \mu: G(K_{v_0'}) \to G(K_{v'}) \). Consequently,
\[
\prod_{v'|v} G(K_{v'})
\]
is an induced object for the group \( \text{Gal}(K/F) \). The cocycle \( \{g_{\rho, \sigma}\}\) takes values in this group and indeed in its center. We can introduce a second cocycle by writing
\[
\sigma \mu = \mu' \sigma_{\mu' }, \quad \rho \mu' = \mu''\rho _{\mu''}, \quad \rho_{\mu''}, \sigma_{\mu'} \in \text{Gal}(K_{v_0'}/F_v)
\]
and setting
\[
g'_{\rho, \sigma} = \prod_{\mu''} \mu'' g_{ \rho_{\mu''} , \sigma_{\mu'}}.
\]
Here, \( \mu, \mu', \mu'' \) lie in \( \{\mu_{v'}\} \), which is a system of representatives for \( \text{Gal}(K_{v_0'}/F_v) \) in \( \text{Gal}(K/F) \), and the product extends over this set. By the Shapiro lemma applied to the center, the two cocycles \( \{g_{\rho, \sigma}\}\) and \( \{g'_{\rho, \sigma}\}\) are equivalent.

We can apply the same procedure to \(\mathscr  H \) to obtain \( \mathscr H^* \). Then \( \phi: \mathscr H \to \mathscr G_v \) also defines \( \phi^*: \mathscr H^* \to \mathscr G^*_v \). If we apply the Shapiro lemma again, this time to \( G \), we immediately see that two homomorphisms \( \phi, \phi_1 \) are equivalent if and only if \( \phi^*_1 = \text{ad}\,g \circ \phi^* \) with \( g \in \prod_{v'|v} G(K_{v'}) \).

These somewhat cumbersome considerations are presented because \( T \) is to be equipped with homomorphisms
\[
\zeta_{v_i}: \mathscr D^{L_{v_i'}} \to \mathscr T, \quad i = 1, 2, \quad \zeta_v: \text{Gal}_{F_v} \to \mathscr T, \quad v \neq v_i.
\]
The restriction of \( \zeta_{v_i} \) to the kernel is defined by \( x \mapsto x^{\nu_i} \). Let \( d^1_{\rho,\sigma} \) and \( d^2_{\rho, \sigma} \) be the cocycles of \( \text{Gal}(\overline{F}/F) \) defined by \( (d^{L_{v_1'}}_{\rho, \sigma})^{\nu_1} \) and \( (d^{L_{v_2'}}_{\rho, \sigma})^{\nu_2} \), respectively, with values in \( \prod_{v'|v_i'} T(L_{v'}) \).

According to the preceding remarks, \( \zeta_{v_i} \) and \( \zeta_v \) will be defined if we can assign to each \( \rho \in \text{Gal}(L/F) \) an element \( e_{\rho} \in T(\mathbb A_L) \) such that
\[
e_{\rho} \rho (e_\sigma) e_{\rho \sigma}^{-1} t_{\rho,\sigma}  = d^1_{\rho, \sigma} d^2_{
	\rho, \sigma}.
\]
The mapping \( \zeta^*_v \), where \( v \) can be equal to \( v_i \), is obtained by projecting \( e_{\rho} \) onto \( \prod_{v'|v} T(L_{v'}) \) to get \( e_{\rho}(v) \), and then sending \( d_{\rho} \) (if \( v = v_i \)) or \( \rho \) (if \( v \neq v_i \)) to \( e_{\rho}(v) t_{\rho} \).

To\marginpar{123} obtain the cocycle \( \{t_{\rho,\sigma}\} \) and the elements \( e_{\rho} \), we choose in the usual way [MS1] diagrams
$$ \xymatrix{ 1 \ar[r] & L^{\times }_{v_i'} \ar[r] \ar[d]  & W_{L_{v_i'}/ F_{v_i}} \ar[r]  \ar[d] & \Gal (L_{v_i'}/ F_{v_i}) \ar[r] \ar[d]  & 1 \\ 1 \ar[r] & C_L \ar[r] & W_{L/F} \ar[r] & \Gal(L/F) \ar[r] & 1    }$$
We further choose, as in [MS1], systems of representatives \( \mathfrak S_i \) of \( \text{Gal}(L/F)/\text{Gal}(L_{v_i'}/F_{v_i}) \) with \( 1 \in \mathfrak S_i \). Finally, we choose representatives \( \omega_{\sigma} \in W_{L_{v_i'}/F_{v_i}} \) of \( \sigma \in \text{Gal}(L_{v_i'}/F_{v_i}) \) and \( \omega_{\mu} \in W_{L/F} \) of \( \mu \in \mathfrak S_i \) with \( \omega_1 = 1 \) in both cases and set
\[
\omega_{\mu \sigma} = \omega_{\mu} \omega_{\sigma}, \quad \mu \in \mathfrak S_i, \quad \sigma \in \text{Gal}(L_{v_i'}/F_{v_i}).
\]

These representatives, as well as the corresponding cocycle defined by the equations
\[
\omega_{\rho} \omega_{\sigma} = A_{\rho, \sigma}(i) \omega_{\rho\sigma}, \quad \rho, \sigma \in \text{Gal}(L/F),
\]
depend on \( i \). The following conditions hold:
\begin{enumerate}
	\item \( A_{\rho, \sigma}(i) = d^{L_{v_i'}}_{\rho,\sigma}, \quad \rho,\sigma \in \text{Gal}(L_{v_i'}/F_{v_i}) \),
	\item \( A_{\mu, \sigma}(i) = 1, \quad \sigma \in \text{Gal}(L_{v_i'}/F_{v_i}), \quad \mu \in \mathfrak S_i \),
	\item \( A_{\mu
		\rho, \sigma}(i) = \mu A_{\rho,\sigma}(i) \in \mu(L_{v_i'}), \quad \rho,\sigma \in \text{Gal}(L_{v_i'}/F_{v_i}),  \quad  \mu \in \mathfrak S_i \).
\end{enumerate}

To define \( d^i_{\rho,\sigma} \), we need a system of representatives of \( \text{Gal}(L_{v_i'}/F_{v_i}) \). We take the \( \omega_{\sigma} \).

Let
\[
\eta = \sum_{\sigma \in \text{Gal}(L/\mathbb{Q})/\text{Gal}(L_{v_1'}/F_{v_1})} \sigma \nu_1 = -\sum_{\sigma \in \text{Gal}(L/\mathbb{Q})/\text{Gal}(L_{v_2'}/F_{v_2})} \sigma \nu_2.
\]
The cocycles \( \{A_{\rho,\sigma}(1)\} \) and \( \{A_{\rho,\sigma}(2)\} \) are cohomologous. Let
\[
A_{\rho,\sigma}(1) A^{-1}_{\rho,\sigma}(2) = B_{\rho} \rho (B_{\sigma})   B^{-1}_{\rho\sigma},
\]
with \( B_{\rho} \in C_L \). Then also\begin{equation} \label{2.c}
	A^{\eta}_{\rho,\sigma}(1) A^{-\eta}_{\rho,\sigma}(2) = B^{\eta}_{\rho} \rho(B_{\sigma})^{\eta} B^{-\eta}_{\rho\sigma}.
\end{equation} 

We introduce the following elements of \( C_L \otimes X_*(T) \):
\[
C_{\rho}(i) = \prod_{\mu \in \mathfrak S_i} A^{-\rho \mu \nu_i}_{\rho, \mu}(i); \quad E_{\rho} = C_{\rho}(1) C_{\rho}(2) B^{\eta}_{\rho}.
\]
A simple calculation shows that
\begin{equation} \label{2.d}
	E_\rho \rho(E_\sigma) E_{\rho \sigma}^{-1} = d^1_{\rho,\sigma} d^2_{\rho,\sigma} , \quad \rho,\sigma \in \Gal(L/F). 
\end{equation} 
This\marginpar{124} equation only makes sense when one notes that
\[
\prod_{v'|v_i} T(L_{v'}) \hookrightarrow C_L \otimes X_*(T).
\]

The coboundary of \( C_{\rho}(i) \) is
\[
\left\{ \prod_{\mu} A^{-\rho \mu \nu_i}_{\rho, \mu} \right\} \left\{ \prod_{\mu} \rho A^{- \rho \sigma \mu \nu_i}_{\sigma, \mu} A^{\rho \sigma \mu \nu_i}_{\rho\sigma , \mu} \right\}
\]
or
\[
A^{\pm \eta}_{\rho,\sigma}(i) \prod_{\mu} A^{-\rho \mu \nu_i}_{\rho, \mu}(i) A^{\rho \mu \nu_i}_{\rho, \mu \sigma_\mu}(i),
\]
with the sign equal to \( (-1)^i \). Furthermore,
\[
\prod_{\mu} A^{-\rho \mu \nu_i}_{\rho, \mu}(i) A^{\rho \mu \nu_i}_{\rho, \mu \sigma_ \mu}(i) = \prod_{\mu} A^{\rho \mu \nu_i}_{\rho \mu, \sigma_ \mu}(i) = \prod_{\mu'} \mu'  A^{\mu' \nu_i}_{\rho_{\mu'}, \sigma_ \mu}(i) = d^i_{\rho, \sigma}.
\]
Thus, \eqref{2.d} follows from \eqref{2.c}.

The elements \( d^i_{\rho, \sigma} \) already belong to \( \prod_{v'|v_i} T(L_{v'}) \subset T(\mathbb A_L) \). We lift \( E_{\rho} \) to \( e_{\rho} \in T(\mathbb A_L) \) and thus obtain from \eqref{2.d} the equations
\[
e_{\rho} \rho (e_\sigma) e^{-1}_{\rho \sigma} t_{\rho, \sigma} = d^1_{\rho, \sigma} d^2_{\rho, \sigma}
\]
with \( t_{\rho, \sigma} \in T(L) \).

It must now be checked that \( \mathscr T \) is independent of all the choices made during its construction, with the additional local homomorphisms, noting that a change of \( e_{\rho} \) by \( f \rho (f^{-1}) e_{\rho} \), \( f \in T(\mathbb A_L) \), is allowed, as the local homomorphisms are then replaced by equivalent ones.

We first enumerate these choices: 
\begin{enumerate}
	\item The liftings \( e_{\rho} \) of \( E_{\rho} \);
	\item The bounding cochain \( \{B_{\rho}\} \);
	\item The system of representatives \( \mathfrak S_i \) and the representatives \( \omega_{\mu} \);
	\item The embedding \( W_{L_{v_i'}/F_{v_i}} \to W_{L/F} \) and the representatives \( \omega_{\rho} \), \( \rho \in \text{Gal}(L_{v_i'}/F_{v_i}) \);
	\item The place \( v'_i \) dividing \( v_i \).
\end{enumerate}

(i) If we replace \( e_{\rho} \) by \( e'_{\rho} = s_{\rho} e_{\rho} \), we obtain the cocycle
\[
t'_{\rho, \sigma} = s^{-1}_{\rho} \rho (s_\sigma) ^{-1}  s_{\rho \sigma } t_{\rho, \sigma}.
\]
The homomorphism defined by \( t'_{\rho} \mapsto s^{-1}_{\rho} t_{\rho} \) transfers the primed local homomorphisms into the unprimed ones.

(ii) If we replace \( B_{\rho} \) by \( F \rho (F^{-1}) B_{\rho} \), then \( E_{\rho} \) is replaced by \( E_{\rho} F^{\eta} \rho(F^{-\eta}) \). We can thus replace \( e_{\rho} \) by \( e_{\rho} f \rho(f^{-1}) \), where \( f \) is a lifting of \( F^{\eta} \). Consequently, \( \{t_{\rho, \sigma}\} \) remains unchanged.

(iii) The\marginpar{125} representatives \( \omega_{\mu} \), \( \mu \in \mathfrak S_1 \), can be replaced by \( b_{\mu} \omega_{\mu} \). Then, in general, \( \omega_{\sigma} \) is replaced by \( b_{\sigma} \omega_{\sigma} \), where \( b_{\mu \rho} = b_{\mu} \), \( \mu \in \mathfrak S_1 \), \( \rho \in \text{Gal}(L_{v_1'}/F_{v_1}) \). We would then have
\[
B'_{\rho} = b_{\rho} B_{\rho}
\]
instead of \( B_{\rho} \) and
\[
C'_{\rho}(1) = C_{\rho}(1)  \cdot \prod_{\mu} b^{-\rho \mu \nu_1}_{\rho} \rho (b_\mu)^{-\rho \mu \nu_1} b^{\rho \mu \nu_1}_{\rho \mu} = 
C_{\rho}(1) b^{-\eta}_{\rho} F \rho (F)^{-1}
\]
with
\[
F = \prod_{\mu} b^{\mu \nu_i}_{\mu},
\]
since
\[
b^{\rho \mu \nu_i}_{\rho \mu} = b^{\mu' \nu_i}_{\mu'},
\]
when \( \rho \mu = \mu' \rho_{\mu'} \), \( \rho_{\mu'} \in \text{Gal}(L_{v_1'}/F_{v_1}) \). Therefore, we can replace \( e_{\rho} \) by \( e_{\rho} f \rho(f^{-1}) \) with \( f \) a lifting of \( F \), and \( t_{\rho, \sigma} \) remains unchanged.

The system of representatives \( \{\mu\} \) can be replaced by \( \{\mu' = \mu \sigma(\mu)\} \), where \( \sigma(\mu) \in \text{Gal}(L_{v_1'}/F_{v_1}) \). We choose \( \omega'_{\mu'} = \omega_{\mu'} = \omega_{\mu} \omega_{\sigma(\mu)} \). Then, in general,
\[
\omega'_{\mu' \sigma} = \omega_{\mu} \omega_{\sigma(\mu)} \omega_{\sigma} = \mu ( A_{\sigma(\mu), \sigma}(1)) \omega_{\mu' \sigma}, \quad \sigma \in \text{Gal}(L_{v_1'}/F_{v_1}).
\]
We set
\[
b_{\mu' \sigma} = \mu A_{\sigma(\mu), \sigma}(1).
\]
This lies in \( \mu(L_{v_i'}) \). The new cocycle is
\[
A'_{\rho, \sigma}(1) = \rho(b_{\sigma}) b_{\rho} b^{-1}_{\rho\sigma} A_{\rho, \sigma}(1).
\]

Furthermore,
\[
C'_{\rho}(1) = \prod_{\mu'} A'_{\rho, \mu'}(1)^{-\rho \mu' \nu_1} = \left\{ \prod_{\mu} A_{\rho, \mu \sigma(\mu)}(1)^{-\rho \mu \nu_1} \right\} \left\{ \prod_{\mu} \big ( b_{\rho} b^{-1}_{\rho \mu'} \rho \mu(A_{\sigma(\mu), 1}) \big )^{-\rho \mu \nu_1} \right\}.
\]
The second factor equals
\[
\prod (b_{\rho} b^{-1}_{\rho \mu'})^{-\rho \mu \nu_1},
\]
since \( A_{\sigma(\mu), 1} = 1 \). The first factor is
\[
C_{\rho}(1) \cdot \prod_{\mu} A^{-\rho \mu \nu_1}_{\rho \mu, \sigma(\mu)}.
\]
Consequently,
\[
E'_{\rho} = E_{\rho} \left\{ \prod_{\mu} A^{-\rho \mu \nu_1}_{\rho \mu, \sigma(\mu)} b^{\rho \mu \nu_1}_{\rho \mu'} \right\},
\]
if\marginpar{126} \( B'_{\rho} = b_{\rho} B_{\rho} \). The second factor lies in \( \prod_{v'|v_1} T(L_{v'}) \) and is therefore automatically lifted.

Thus, \( t'_{\rho, \sigma} \) and \( t_{\rho, \sigma} \) differ only by an element of \( \prod_{v'|v_1} T(L_{v'}) \). Since \( t_{\rho, \sigma} \) and \( t'_{\rho, \sigma} \) lie in \( T(L) \), we have \( t_{\rho, \sigma} = t'_{\rho, \sigma} \). Similarly, \( e'_{\rho} \) and \( e_{\rho} \) differ only by an element of \( \prod_{v'|v_1} T(L_{v'}) \). Consequently, the local homomorphisms outside the place \( v_1 \) are the same. To show that they are also the same at \( v_1 \), it suffices to show that the projection of
\[
\prod_{\mu} A^{-\rho \mu \nu_1}_{\rho \mu, \sigma(\mu)} b^{\rho \mu \nu_1}_{\rho \mu}
\]
onto \( T(L_{v_1'}) \) is equal to 1 when \( \rho \in \text{Gal}(L_{v_1'}/F_{v_1}) \). The only term in the product contributing to this projection is the one corresponding to \( \mu = 1 \), and it is 1, since \( \sigma(1) = 1 \) and
\[
A_{\rho \mu, 1} = 1, \quad b_{\rho} = 1.
\]

(iv) The embedding \( \phi: W_{L_{v_1'}/F_{v_1'}} \to W_{L/F} \) can only be altered by replacing \( \phi \) with \( \text{ad}\,x \circ \phi \), \( x \in C_L \). We can thus replace \( \omega_{\sigma} \) by \( x \sigma(x^{-1}) \omega_{\sigma} \) for each \( \sigma \in \text{Gal}(L/F) \). All cocycles remain unchanged.

If we replace \( \omega_{\rho} \) by \( b_{\rho} \omega_{\rho} \), \( b_{\rho} \in L^{\times}_{v_1'} \), we can keep \( \omega_{\mu} \) unchanged, so that \( A_{\rho, \sigma}(1) \) is replaced by
\[
A'_{\rho, \sigma}(1) = A_{\rho, \sigma}(1) b_{\rho} \rho(b_\sigma) b^{-1}_{\rho \sigma},
\]
where
\[
b_{\mu \rho} = \mu(b_{\rho}), \quad \rho \in \text{Gal}(L_{v_1'}/F_{v_1}).
\]
Since \( \mu(b_{\rho}) \in \mu(L_{v_1'}) \), this change has no effect on \( t_{\rho, \sigma} \) or the local homomorphisms outside \( v_1 \). That this change also does not affect the equivalence class of \( \zeta_{v_1} \) is easily checked by noting that the altered cocycle on \( \text{Gal}(L_{v_1'}/F_{v_1}) \) also appears in the definition of \( \mathscr D^{L_{v_1'}} \). The representatives \( \omega_{\sigma} \), \( \sigma \in \text{Gal}(L_{v_1'}/F_{v_1}) \), are replaced by \( \omega'_{\sigma} = b_{\sigma} \omega_{\sigma} \), and the projection of \( e_{\sigma} \) onto \( T(L_{v_1'}) \) is multiplied by \( b^{\eta}_{\sigma} \).

(v) The last arbitrariness to address is the choice of \( v'_1 \). Any other place \( v''_1 \) of \( L \) dividing \( v_1 \) is obtained by choosing \( \lambda \in \mathfrak S_1 \) and defining \( v''_1 \) by
\[
|\lambda x|_{v''_1} = |x|_{v'_1}.
\]
Then,
\[
\text{Gal}(L_{v''_1}/F_{v_1}) = \lambda \text{Gal}(L_{v'_1}/F_{v_1}) \lambda^{-1}.
\]
Accordingly, we obtain the cocycles and objects corresponding to the place \( v''_1 \) by conjugating with \( \lambda \). For example,
\[
A'_{\rho, \sigma}(1) = \lambda \left( A_{\lambda^{-1} \rho \lambda, \lambda^{-1} \sigma \lambda}(1) \right).
\]

This\marginpar{127} cocycle is cohomologous to the old one, and a simple calculation yields
\[
A'_{\rho, \sigma}(1) = A_{\rho, \sigma}(1) D_{\rho} \rho(D_\sigma) D^{-1}_{\rho\sigma }
\]
with
\[
D_{\rho} = \lambda \left( A^{-1}_{\lambda^{-1}, \rho \lambda}(1) \right) A^{-1}_{\rho, \lambda}(1) A_{\lambda, \lambda^{-1}}(1).
\]

Furthermore, if we abbreviate \( A'_{\rho,\sigma}(1) \) as \( A'_{\rho, \sigma} \), then
\[
C'_\rho (1) = \prod_{\mu} (A'_{ \rho,\lambda\mu \lambda^{-1}}) ^{-\rho \lambda \mu \nu_1} ,
\]
since \( \nu_1' = \lambda \nu_1 \). Let \( \lambda \mu = \mu' \lambda _{\mu'} \) with \( \mu' \in  \mathfrak S_1 \) and \( \lambda _{\mu'} \in \text{Gal}(L_{v_1'} / F_{v_1}) \). Then,
\[
A'_{\rho, \lambda \mu \lambda^{-1}} = A'_{\rho, \mu' \lambda _{\mu'}  \lambda^{-1}} =  \rho (A'_{\mu', \lambda_{\mu'} \lambda^{-1}})^{-1} A'_{\rho, \mu'} A'_{\rho \mu', \lambda_{\mu'}   \lambda^{-1}} ,
\]
and
\[
A'_{\rho, \mu'} = A_{\rho, \mu'} \rho ( D_{\mu'})  D_{\rho} D_{\rho \mu'}^{-1}.
\]
From this, we immediately conclude that \( E_{\rho} \) and \( E_{\rho}' \) differ by the factor
\begin{align} \label{2.e}
	\left\{ \prod_{\mu} \rho (A'_{\mu, \lambda_\mu \lambda^{-1}})^{\rho \mu \nu_1} \right\} \left\{ \prod_{\mu} A'_{\rho \mu, \lambda_\mu \lambda^{-1}} \right\}^{-\rho \mu \nu_1} \left\{ \prod_{\mu} \rho(D_\mu)^{-\rho \mu \nu_1} D_{\rho \mu}^{\rho \mu \nu_1}  \right\}		 \end{align}
where we write \( \mu \) instead of \( \mu' \).

We consider
\begin{align}\label{2.f}
	\prod_{\mu} \left( A'_{\rho \mu, \lambda_\mu \lambda^{-1}} D_{\rho \mu}^{-1} \right)^{\rho \mu \nu_1},
\end{align} 
which depends on \( \rho \) and show that it can be expressed as a product
\[
\rho(X) Y Z_{\rho},
\]
where \( X \) and \( Y \) are independent of \( \rho \), and \( Z_{\rho} \) lies in \( \prod_{v' | v_1} T(L_{v'}) \). Then, \eqref{2.e} equals
\[
\rho(Y) Y^{-1}\rho (Z_1) Z_{\rho}^{-1},
\]
from which it immediately follows that \( \{ t_{\rho, \sigma} \} \) is preserved, as well as the local homomorphisms outside \( v_1 \).

Instead of \eqref{2.f}, we consider
\begin{align}\label{2.g}
	A'_{\rho \mu, \lambda_\mu \lambda^{-1}} D_{\rho \mu}^{-1},
\end{align}
and write it as
\[\rho 
(X_{\mu}) Y_{\mu'} Z_{\mu, \rho}.
\]
Here,\marginpar{128} \( \rho \mu = \mu' \rho_{\mu'} , \rho_{\mu'} \in \text{Gal}(L_{v_1'} / F_{v_1}) \). The expression \eqref{2.g} resembles
\[
A_{\rho \mu, \lambda_ \mu \lambda^{-1}} \rho\mu (D_{\lambda_\mu \lambda^{-1}}) D_{\rho \mu \lambda_\mu \lambda^{-1}}^{-1},
\]
which we expand to obtain a product with seven terms:
\[
A_{\rho\mu, \lambda_\mu \lambda^{-1}} \rho \mu \lambda (A_{\lambda^{-1}, \lambda _\mu}^{-1}) \rho \mu (A_{\lambda_\mu \lambda^{-1}, \lambda}^{-1}) \rho \mu (A_{\lambda, \lambda^{-1}}) \lambda (A_{\lambda^{-1}, \rho \mu \lambda_\mu}) A_{\rho \mu \lambda_\mu \lambda^{-1}, \lambda} A_{\lambda, \lambda^{-1}}^{-1}.
\]

We set
\[
X_{\mu} = \mu \lambda (A_{\lambda^{-1}, \lambda _\mu}^{-1}) \mu (A_{\lambda, \lambda^{-1}}).
\]
From the entire product, we extract
\[
A_{\rho \mu, \lambda_\mu \lambda^{-1}} \rho \mu (A_{\lambda_\mu \lambda^{-1}, \lambda}^{-1})   A_{\rho \mu \lambda_\mu \lambda^{-1}, \lambda} = A_{\rho \mu, \lambda _\mu},
\]
and add it to \( Z_{\mu', \rho} \). The term \( A_{\lambda, \lambda^{-1}}^{-1} \) is added to \( Y_{\mu'} \). What remains is
\[
\lambda (A_{\lambda^{-1}, \rho \mu \lambda_\mu}) = \lambda (A_{\lambda^{-1}, \mu' \rho_{ \mu'}\lambda _\mu}) = A_{\mu', \rho_{ \mu'} \lambda_ \mu}^{-1} \lambda (A_{\lambda^{-1} \mu', \rho_{ \mu'} \lambda _\mu}) \lambda (A_{\lambda^{-1}, \mu'}).
\]
From this, \( Z_{\mu, \rho} \) gets the product
\[
A_{\mu', \rho_{\mu'} \lambda _\mu}^{-1} \lambda (A_{\lambda^{-1} \mu', \rho_{ \mu'} \lambda _\mu}) = \lambda (A_{\lambda^{-1} \mu', \rho_{\mu'} \lambda_\mu}),
\]
and \( Y_{\mu'} \) gets the term \( \lambda (A_{\lambda^{-1}, \mu'}) \).

As already emphasized, to show the equivalence of the homomorphisms at \( v_1 \), we only need to consider the homomorphisms from \( \mathscr D^{L_{v_1''}} \) to \(\mathscr  T \). The homomorphism defined by \( v_1'' \) corresponds to the cocycle \( A'_{\rho, \sigma}(1),  \rho , \sigma \in \text{Gal}(L_{v_1''} / F_{v_1}) \). The one defined by \( v_1' \) corresponds to the projection of \( d^1_{\rho, \sigma} \) onto \( T(L_{v_1''}) , \rho, \sigma \in \text{Gal}(L_{v_1''} / F_{v_1}) \). Since \( \rho \lambda = \lambda \lambda^{-1}\rho \lambda = \lambda \rho_\lambda , \rho \in \Gal(L_{v_1''}/F_{v_1})\), this projection is nothing other than \( A'_{\rho, \sigma}(1) \). Up to a coboundary, for \( \rho \in \text{Gal}(L_{v_1''} / F_{v_1}) \), the projection of \eqref{2.e} onto \( T(L_{v_1''}) \) equals that of
$
\rho(Z_1) Z_{\rho}^{-1},
$ 
which in turn equals
\[\rho(A_{\lambda, 1}) \rho \lambda (A_{1, 1}) A_{\rho \mu, 1}^{-1} \lambda (A_{1, \rho_ \lambda}^{-1}) = 1,
\]
since \( \lambda _ \lambda = 1 \) and \(\rho \lambda = \lambda \rho_\lambda \).

There is one more compatibility to verify, which is even more troublesome than the previous ones. Namely, let \( L' \supset L \) be a second Galois extension of \( F \), and let \( w_i' \) be primes of \( L' \) dividing \( v_i' \). Let
\[
\nu_i' = [L'_{w_i'} : L_{v_i'}] \nu_i = l_i \nu_i.
\]
We can introduce \( \mathscr T_{\nu_1', \nu_2'} \) along with additional local homomorphisms, and we must show that the thus equipped \(\mathscr  T_{\nu_1, \nu_2} \) and \(\mathscr  T_{\nu_1', \nu_2'} \) are canonically isomorphic.

We\marginpar{129} simplify the notations as follows:
\[
G = \text{Gal}(L/F), \quad G' = \text{Gal}(L'/F), \quad H = \text{Gal}(L'/L),
\]
\[
G_i = \text{Gal}(L_{v_i'}/F_{v_i}), \quad G'_i = \text{Gal}(L'_{w_i'}/F_{v_i}), \quad H_i = \text{Gal}(L'_{w_i'}/L_{v_i'}).
\]
The group \( H_i \) is \( H \cap G_i \), and we choose representatives \( \lambda \) of \( H_i \) in \( H \). Then we choose representatives \( \mu' \) of \( H G'_i \) in \( G' \). As representatives of \( G'_i \) in \( G' \), we take the \( \mu' \lambda \). Finally, we take the projections \( \mu \) of \( \mu' \) onto \( G \) as representatives of \( G_i \) in \( G \). In the Weil group \( W_{L'/F} \), we choose \( \omega_{\mu' \lambda} = \omega_{\mu'} \cdot \omega_{\lambda} \) with \( \omega_{\lambda} \in W_{L'/L} \). We also choose right representatives \( \nu \) of \( H_i \) in \( G'_i \) and set
\[
\omega_{\gamma \nu} = \omega_{\gamma} \omega_{\nu}, \quad \gamma \in H_i.
\]
Then we have
\[
\omega_{\lambda \cdot \gamma} \omega_{\nu} := \omega_{\lambda} \cdot \omega_{\gamma} \cdot \omega_{\nu} = \omega_{\lambda \gamma \nu},
\]
so that
\[
A_{\gamma, \nu} = 1,
\]
when \( \gamma \) is in \( H \) and \( \nu \) is contained in the right representative system.

Let
\[
D_{\rho'}(i) = \prod_{\gamma \in H} A_{\rho', \gamma}(i).
\]
We consider the cocycle
\[
A'_{\rho', \sigma'}(i) = A^l_{\rho', \sigma'}(i) D^{-1}_{\rho'}(i) \rho'( D^{-1}_{\sigma'}(i) ) D_{\rho' \sigma'}(i)
\]
with \( l = [H : 1] \). It is easy to verify that it equals
\[
\prod_{\gamma} A^{-1}_{\rho', \gamma} A_{\rho', \sigma'  \gamma}.
\]
Consequently, it depends only on the projection \( \sigma \) of \( \sigma' \) onto \( G \) and is \( 1 \) if \( \sigma' \in H \).

We first show that
\begin{equation}\label{2.h}
	A'_{\rho', \sigma'}(i) = 1,
\end{equation} 
when \( \rho' \in H \) and \( \sigma' \in H G'_i \). The element \( \sigma' \gamma \) can be written as \( \delta \nu \) with \( \delta \in H \) and \( \nu \) from the right representative system. Then we have
\[
A'_{\rho', \sigma' \gamma}(i) = \rho'(  A'_{\delta, \nu}(i) )^{-1} A'_{\rho' \delta, \nu}(i) A'_{\rho', \delta}(i) = A'_{\rho', \delta}(i).
\]
Consequently,
\[
\prod_{\gamma} A'_{\rho', \sigma' \gamma}(i) = \prod_{\delta} A'_{\rho', \delta}(i).
\]
Since \( \delta \) also runs through \( H \), \eqref{2.h} holds.

From\marginpar{130} the cocycle condition, we conclude that \( A'_{\rho', \sigma'}(i) \) depends only on the projections \( \rho, \sigma \) and is contained in \( C_L \) if \( \rho', \sigma' \) lie in \( H G'_i \). Indeed, it lies in \( L_{v_i'} = C_L \cap \prod_{w' | v_i'} L'_{w'} \), since
\[
A'_{\rho', \sigma'}(i) = \prod_{\gamma} A^{-1}_{\rho', \gamma} A_{\rho', \sigma' \gamma} = \prod_{\gamma} A^{-1}_{\rho', \gamma} A_{\rho', \gamma \nu}
\]
with \( \nu = \nu(\sigma) \). The right hand side is itself equal to
\[
\prod_{\gamma} \rho'( A^{-1}_{\gamma, \nu} ) A_{\rho' \gamma, \nu} = \prod_{\gamma} A_{\rho' \gamma, \nu},
\]
and \( \rho' \gamma \in H G'_i \).

Since the cocycles \( A^l_{\rho', \sigma'} \) and \( A_{\rho, \sigma} \) are cohomologous anyway, we can choose \( A_{\rho, \sigma} \) such that
\[
A_{\rho, \sigma} = A'_{\rho', \sigma'}, \quad \rho', \sigma' \in H G'_i.
\]
In general, there exists a cochain \( B_{\rho'}(i) \) such that
\[
A^l_{\rho', \sigma'}(i) = A_{\rho, \sigma}(i) B_{\rho'}(i)\rho'( B_{\sigma'}(i) ) B^{-1}_{\rho' \sigma'}(i).
\]
The product \( B_{\rho'}(i) D^{-1}_{\rho'}(i) \) is a cocycle on \( H G'_i \). Consequently,
\[
B_{\rho'}(i) = D_{\rho'}(i) F' \rho'(F')^{-1}, \quad F' \in C_{L'}.
\]
Since we can replace \( B_{\rho'}(i) \), \( \rho' \in G' \), by
\[
B_{\rho'}(i) F'^{-1} \rho'(F'),
\]
we may assume that
\[
B_{\rho'}(i) = D_{\rho'}(i), \quad \rho' \in H G'_i.
\]
After carefully and appropriately choosing the occurring cocycles and cochains, we will now be able to verify the desired compatibility without too much effort.

We have 
\[
\sum_{G'/G_1'} \sigma \nu'_1 =  \eta' = l \eta.
\]
If
\[
A_{\rho' , \sigma '}(1) A^{-1}_{\rho',\sigma'}(2) = B_{\rho'} \rho'(B_{\sigma'})  B^{-1}_{\rho' \sigma'},
\]
then the coboundary of \( B^l_{\rho'} \) is equal to that of \( B_{\rho} B_{\rho'}(1) B^{-1}_{\rho'}(2) \). Consequently, it follows that
\[
B^l_{\rho'} = B_{\rho} B_{\rho'}(1) B^{-1}_{\rho'}(2) F'  \rho (F')^{-1}, F' \in C_{L'}.
\]

We\marginpar{131} first compute \( C_{\rho'}(i) \), where we sometimes omit the \( i \):
\begin{align*}
	C_{\rho'}(i) & = \prod_{\mu'} \prod_{\lambda} A^{-\rho \mu' \lambda \nu_i'}_{\rho', \mu' \lambda} \\ &  =  \prod_{\mu'} \prod _{\lambda} \left\{ \rho'(  A_{\mu',\lambda}^{\rho\mu' \lambda \nu_i'}) A_{\rho'\mu', \lambda}^{-\rho \mu' \lambda\nu'_i} A_{\rho',\mu'}^{-\rho \mu' \nu'_i}\right\} \\& = \prod_{\mu',\lambda} \left\{  A_{\rho'\mu',\lambda}^{-\rho\mu'\lambda\nu_i'} A_{\rho',\mu'}^{-\rho\mu'\nu'_i} \right\}.
\end{align*} 
We expand
\[
\prod_{\mu', \lambda} A_{\rho',\mu'}^{-\rho \mu' \nu'_i}
\]
as
\[
\left\{ \prod_{\mu} A^{-\rho \mu \nu_i}_{\rho, \mu} \right\} \left\{ \prod_{\mu'} B_{\rho'}(i)^{-\rho\mu' \nu_i} \rho'(  B_{\mu'}(i) )^{-\rho \mu' \nu_i} B_{\rho' \mu'}(i)^{\rho\mu' \nu_i} \right\}
\]
and note that
\[
\prod_{\mu'} B_{\rho'}(i)^{-\rho \mu' \nu_i} = B_{\rho'}(i)^{(-1)^i \eta}	. \]

From these equations, we conclude that \( E_{\rho'} \) is the product of \( E_{\rho} \), a coboundary, and the expressions
\begin{equation}\label{2.i}
	\left\{ \prod_{\mu', \lambda} A^{-\rho \mu' \nu_i'}_{\rho' \mu', \lambda} \right\} \left\{ \prod_{\mu'} \rho  (B_{\mu'}(i))^{-\rho \mu' \nu_i} B_{\rho' \mu'}(i)^{\rho \mu' \nu)i}  \right\}, \quad i = 1, 2.
\end{equation} 
We multiply this expression by the coboundary
\[
\rho \bigg( \prod_{\mu'} B_{\mu'}(i)^{\mu' \nu_i} \bigg )  \prod_{\mu'} B_{\mu'}(i)^{-\mu' \nu_i}
\]
and obtain
\[
\left\{ \prod_{\mu', \lambda} A^{-\rho\mu' \nu_i'}_{
	\rho' \mu', \lambda} \right\} \left\{ \prod_{\mu'} B_{\rho' \mu'}(i)^{\rho \mu' \nu_i} B_{\mu'}(i)^{-\mu'\nu_i} \right\}.
\]

Let \( \rho' \mu' = \mu_1' \rho'_{\mu_1'} \),   \( \rho'_{\mu_1'} \in  HG_i' \). We have \[	B_{\rho'\mu'}(i) B_{\mu'_1}(i)^{-1} = A_{\mu'_1, \rho'_{\mu'_1}} A^{-1}_{\mu'_1, \rho'_{\mu'_1}} \mu'_1 ( B_{\rho'_{ \mu'_1}}(i) )  =\mu'_1 ( B_{\rho'_{ \mu'_1}}(i) ) 
\]
and
\[
A_{\mu'_1 \rho'_{ \mu'_1}, \lambda}(i) = \mu'_1  \big( A_{\rho'_{  \mu'_1}, \lambda}(i) \big) A_{\mu'_1, \rho_{ \mu'_1} \lambda}(i) A^{-1}_{\mu'_1, \rho_{\mu'_1}}(i) = \mu'_1 \big(  A_{\rho'_{ \mu'_1}, \lambda}(i) \big).
\]

Consequently, \eqref{2.i} is replaced by
\begin{equation}\label{2.j}
	\left\{ \prod_{\mu', \lambda} \mu' \big( A_{\rho' _{ \mu'}, \lambda}(i) \big ) ^{-\mu' \nu'_1} \right\} \left\{ \prod_{\mu'} \mu'  \big( B_{\rho'_{  \mu'}}(i) \big) ^{\mu'\nu_1} \right\}. 		 	\end{equation}
We\marginpar{132} write \( \mu' \) instead of \( \mu'_1 \) to simplify the notation. The element \( \rho'_{\mu'}\) lies in \( H G'_i \), so   
\begin{align*}
	B_{\rho'_{\mu'}} (i) & = \prod_{\gamma \in H} A_{\rho'_{\mu'}, \gamma} (i) \\ & = \prod_{\lambda} \prod_{\gamma \in H_i} A_{\rho'_{\mu'}, \lambda\gamma } (i) \\ & =\prod_{\lambda}\prod_{\gamma} A_{\rho'_{\mu'}\lambda,\gamma} (i) A_{\rho'_{\mu'}, \lambda}(i). 
\end{align*}
Since \( [H_i : 1] = l_i \) and \( \nu'_i = l_i \nu_i \), we conclude that \eqref{2.j} is equal to
\begin{equation}\label{2.k}
	\prod_{\mu',\lambda,\gamma} \mu' (A_{\rho'_{\mu'} \lambda , \gamma } ) ^{\mu' \nu_i}, 
\end{equation} 
which lies in \( \prod_{w' | v_i} T(L'_{w'}) \). This immediately yields the equation
\begin{equation}\label{2.l}
	t_{\rho',\sigma'} = t_{\rho,\sigma}, \quad \rho',\sigma' \in \Gal(L'/F),
\end{equation} 
since \( t^{-1}_{\rho, \sigma} t_{\rho', \sigma'} \) on one hand lies in \( T(L') \) and on the other hand lies in
\[
\prod_{w' | v_1} T(L'_{w'}) \prod_{w' | v_2} T(L'_{w'}).
\]

It is also clear that the local homomorphisms outside \( v_1 \) and \( v_2 \) are equivalent. We also want to show that, for example,
\[ \xymatrix{\mathscr D^{L'_{w_1'}} \ar[r] \ar[d] & \mathscr T'_{\nu_1', \nu_2'} \ar[d] \\ \mathscr D^{L_{v_1'}} \ar[r] & \mathscr T_{\nu_1,\nu_2}}
\]
is commutative. On the right is the isomorphism resulting from \eqref{2.e}. The commutativity on the kernel of \( \mathscr D^{L'_{w'_1}} \) is clear. What else needs to be considered is the projection of \eqref{2.k} onto \( T(L'_{w'_1}) \), namely
\[
\prod_{\gamma \in H_1} A^{\nu_1}_{ \rho', \gamma}(1) = G ^{\nu_1}_{\rho'}.
\]		 	Therefore, it suffices to show the following equation:
$$ A'_{\rho',\sigma'}(1) =  A^{l_1}_{\rho',\sigma'} (1) G^{-1}_{\rho'} \rho'(G_{\sigma'})^{-1} G_{\rho'\sigma'} , \quad \rho',\sigma' \in G'_1 , 	$$
where both expressions are contained in the group \( L^{\prime  \times}_{w'_1} \).

We\marginpar{133} have already seen that
\begin{align*}
	A'_{\rho',\sigma'} (1) & = \prod_{\gamma\in H} A^{-1}_{\rho',\gamma}(1) A_{\rho', \sigma'\gamma}(1) \\ & = \prod_{\gamma} A^{-1}_{\rho',\gamma} (1) A_{\rho', \gamma \nu} (1), 
\end{align*}
where $\nu$ is the representative of $\sigma'$ in the right representative system. We have also seen that the product on the right hand side is equal to $\prod_{\gamma} A_{\rho'\gamma, \nu}(1)$. We are interested in this product as an element of $L'_{w_1'}$, so  we can project it onto the corresponding coordinate and obtain
\[
\prod_{\gamma \in H_1} A_{\rho'\gamma, \nu}(1).
\]
This product, in turn, is equal to
\[
\prod_{\gamma \in H_1} A^{-1}_{\rho',\sigma'}(1) A_{\rho',\sigma' \gamma}(1)  = A^{l_1}_{ \rho', \sigma'}(1) G^{-1}_{\rho'} \rho'( G^{-1}_{\sigma'} ) G_{\rho'\sigma'}.
\]

The construction of $\mathscr T_{\nu_1, \nu_2}$ is obviously functorial in $T$, $\nu_1$, and $\nu_2$. Namely, if we have a homomorphism $\phi: T \to T'$ defined over $F$ and $\nu'_i = \phi(\nu_i)$ for $i = 1, 2$, then we immediately obtain a homomorphism $\phi$ from $\mathscr T_{\nu_1, \nu_2}$ to $\mathscr T_{\nu'_1, \nu'_2}$.

We formulate some of the results  obtained  above in a theorem.
\begin{thm}\label{2.2}
	Let $T$ be a torus defined over $F$ that splits over the Galois extension $L$ of $F$. Let $v_1$ and $v_2$ be two places of $F$, and let $v'_1$ and $v'_2$ be chosen extensions to $L$. Let two cocharacters $\nu_1$ and $\nu_2$ in $X_*(T)$ be given, satisfying conditions \eqref{2.a} and \eqref{2.b}. Then the above construction defines a gerbe $\mathscr T = \mathscr T_{\nu_1, \nu_2}$ over $F$ with kernel $T$, equipped with homomorphisms 
	\begin{align*}
		& \zeta_v: \Gal_{F_v} \To \mathscr T, \quad  v \neq v_1, v_2,\\ 
		& \zeta_{v_i}: \mathscr D^{L_{v'_i}} \To \mathscr  T, \quad   i = 1, 2,
	\end{align*}
	such that $\zeta_{v_i}$ on the kernel is given by $x \mapsto x^{
		\nu_i}$. This gerbe is uniquely determined up to an isomorphism that transforms the local homomorphisms into equivalent ones. If $L' \supset L$ is a larger Galois extension to which the places $v'_i$ are extended to $w'_i$, then the above construction yields an isomorphism
	\[
	\mathscr  T'_{\nu_1', \nu_2'} \stackrel{\sim}{\longrightarrow} \mathscr T_{\nu_1, \nu_2},
	\]
	where \marginpar{134}
	\[
	\nu'_i = [L'_{w'_i}: L_{v'_i}] \cdot \nu_i \quad \text{for} \quad i = 1, 2,
	\]
	such that the following two diagrams commute up to equivalence:
	\[ \xymatrix{ & \mathscr T'_{\nu_1',\nu_2'} \ar[dd]^{
		\sim } \\ \Gal_{F_v} \ar[ur]^{\zeta'_v} \ar[dr]^{\zeta_v}  && \text{for }  v \neq v_1, v_2, \\ & \mathscr T_{\nu_1,\nu_2} }
	\quad 
	\]
	\[
	\xymatrix{\mathscr D^{L'_{w_i'}} \ar[r] \ar[d] & \mathscr T'_{\nu_1',\nu_2'} \ar[d]^{
		\sim } \\ \mathscr D^{L_{v_i'}} \ar[r] & \mathscr T_{\nu_1,\nu_2}    .    } 
	\]
	The construction is functorial in $T$, $\nu_1$, and $\nu_2$.
\end{thm} 

One way to achieve the above situation is for $\nu_1$ and $\nu_2$ to arise from averaging the same cocharacter. Let $T$ be a torus over $F$, and let $\mu$ be an element of $X_*(T)$. We set
\[
\nu_i = (-1)^{i+1} \sum_{\sigma \in \text{Gal}(L_{v'_i}/F_{v_i})} \sigma \mu.
\]
We will introduce a canonical homomorphism $\tau_\mu$ from the corresponding gerbe $\mathscr T_{\nu_1, \nu_2}$ to the neutral gerbe $\mathscr G_T$.

Before doing so, we define in a very simple way a homomorphism $\xi_\mu$ from $\mathscr D^{L_{v'}}$ to $\mathscr G_T$. The field $L$ splits $T$, and $v'$ is a prime of $L$ dividing a prime $v$ of $F$. Let
\[
\nu = \sum_{\sigma \in \text{Gal}(L_{v'}/F_v)} \sigma \mu.
\]
The homomorphism $\xi_\mu$ is defined as follows:
\[
\xi_\mu(z) = z^\nu, \quad z \in L_{v'}^\times,
\]
\[
\xi_\mu(d_\rho) = \prod_{\sigma \in \text{Gal}(L_{v'}/F_v)} d_{\rho,\sigma}^{ \rho \sigma \mu} \rtimes \rho,
\]
where $d_{\rho,\sigma}$ is the cocycle defining $\mathscr D^{L_{v'}}$, and $d_\rho = d^{L_{v'}} _{\rho}$. It is easy to verify that $\xi_\mu$ is indeed a homomorphism.

The homomorphism $\tau_\mu$ to be introduced has the following properties:
\begin{itemize}
	\item[(i)] $\tau_\mu \circ \zeta_{v_1}$ is equivalent to $\xi_\mu$;
	\item[(ii)] $\tau_\mu \circ \zeta_{v_2}$ is equivalent to $\xi_{-\mu}$;
	\item[(iii)] $\tau_\mu \circ \zeta_v$ is the canonical neutralization of $\mathscr G_T$.
\end{itemize}
The homomorphism $\tau_\mu$ is defined by
\[
\tau_\mu: t_\rho \longmapsto s_\rho \rtimes \rho,
\]
where\marginpar{135} $s_\rho \in T(L)$ and
\[
s_\rho \rho(s_\sigma)  s_{\rho\sigma}^{-1} = t_{\rho, \sigma}.
\]
Furthermore, there exists $f \in T(\mathbb{A}_L)$ such that
\[
f e_\rho s_\rho \rho(f^{-1}) = e'_{\rho}
\]
satisfies the following properties:
\begin{itemize}
	\item[(i)] the projection of $e'_\rho$ onto $L_{v'}$ is $1$ if $v'$ divides neither $v_1$ nor $v_2$;
	\item[(ii)] the projection of $e'_\rho$ onto $T(L_{v'_i})$ is
	\[
	\prod_{\sigma \in \text{Gal}(L_{v'_i}/F_{v_i})} A_{\rho, \sigma}^{(-1)^{i+1} \rho\sigma \mu} (i).
	\]
\end{itemize}

Let $\{B_\rho\}$ be the cocycle appearing in \eqref{2.c}. We set
\[
F = \prod_{\rho\in \text{Gal}(L/F)} B_{\rho}^{-\rho \mu},
\]
\[
E_\rho(i) = \prod_{\tau \in \mathfrak S_i} \prod_{\sigma \in \text{Gal}(L_{v'_i}/F_{v_i})} A_{\rho\tau, \sigma}^{(-1)^{i+1} \rho \tau \sigma \mu} (i),
\]
so  
\[
E_\rho(i) \in \prod_{v'|v_i} T(L_{v'}).
\]
A straightforward calculation yields the equation\begin{equation}\label{2.m}
	E_\rho = E_\rho (1) E_\rho (2) F \rho (F^{-1}). 
\end{equation}

Indeed, for $\tau \in \mathfrak S_i$ and $\sigma \in \text{Gal}(L_{v'_i}/F_{v_i})$, we have
\[
A_{\rho, \tau \sigma}(i) = A_{\rho \tau, \sigma}(i) A_{\rho, \tau}(i),
\]
so  $C_\rho(i)$ is the product of
\[
\prod_{\sigma \in \text{Gal}(L/F)} A_{\rho, \sigma}^{(-1)^{i} \rho\sigma \mu}(i)
\]
and
\[
\prod_{\tau \in \mathfrak S_i} \prod_{\sigma \in \text{Gal}(L_{v'_i}/F_{v_i})} A_{\rho \tau, \sigma}^{(-1)^{i+1} \rho\tau \sigma \mu} = E_\rho(i).
\]
Consequently, 
\begin{align*}
	E_\rho & = \left \{ \prod_{\sigma} A_{\rho,\sigma} ^{-\rho\sigma\mu} (1) A_{\rho,\sigma}^{\rho\sigma \mu} (2) \right \} \left \{ B_{\rho}^{\eta} E_{\rho}(1) E_\rho (2) \right \} \\ & = \left \{ \prod_\sigma \big( B_\rho \rho(B_\sigma) B_{\rho\sigma}^{-1} \big)^{-\rho\sigma \mu}  \right \} \left \{ B_{\rho}^{\eta} E_{\rho}(1) E_\rho (2) \right \} \\ & = F^{-1} \rho(F) E_\rho(1) E_\rho (2). 
\end{align*}

The\marginpar{136} elements $E_\rho(i)$ have already been lifted to $T(\mathbb{A}_L)$. We lift $F$ to an element $f$ and obtain
\[
s_\rho  e_\rho f \rho (f^{-1}) = E_\rho (1) E_\rho(2)
\]
with $s_\rho \in T(L)$. Thus,
\[
e'_\rho = E_\rho(1) E_\rho(2),
\]
so  $e'_\rho$ clearly possesses the desired properties.

It must be checked once again whether $s_\rho$ is independent of the choices made in its construction, up to a coboundary. These choices are enumerated as (i) to (v) below. The possibility mentioned under (v) requires special attention.

(i) If we take a lift $e''_\rho = r_\rho e_\rho$ with $r_\rho\in T(L)$, then $t_{\rho, \sigma}$ is modified to $t''_{\rho, \sigma} = r_{\rho}^{-1} \rho(r_{\sigma}^{-1}) r_{\rho \sigma} t_{\rho, \sigma}$, and the isomorphism between the two resulting gerbes is defined by
\[
t''_\rho \longmapsto r_{\rho}^{-1} t_\rho.
\]
Since $s_\rho$ is replaced by $r_{\rho}^{-1} s_\rho$, $t''_\rho$ maps to $r_{\rho}^{-1} s_\rho \rtimes \rho$ in the neutral gerbe $\mathscr G_T$, which is also the image of $r_{\rho}^{-1} t_\rho$.

(ii) If we replace $B_\rho$ by $B_\rho  G \rho (G^{-1})$, then $F$ is multiplied by
\[
\prod_\rho G^{-\rho \mu} \rho(G)^{\rho \mu}.
\]
If we lift $G$ to $g$, we can choose
\[
f' = f \prod_\rho g^{-\rho \mu} \rho(g)^{\rho \mu}
\]
instead of $f$, and
\[
f' \rho(f')^{-1} = f \rho(f^{-1}) g^{-\eta} \rho (g^\eta).
\]
We have already seen that $e_\rho$ is replaced by $e_\rho g^\eta \rho (g^{-\eta})$. Thus, $\{s_\rho\}$ remains unchanged.

(iii) A similar calculation shows compatibility under changes of the representatives $\omega_\tau$ for $\tau \in \mathfrak S_i$. If we change the system of representatives $\mathfrak S_i$ itself, then $B_\rho$ is multiplied by an element of $\prod_{v'|v_i} L_{v'}^\times$. Consequently, neither $t_{\rho, \sigma}$ nor $s_\rho$ is changed.

(iv) A change in the embedding $W_{L_{v'_1}/F_{v_1}} \to W_{L/F}$ does not alter the cocycles and thus does not change $\{s_\rho\}$. A change in $\omega_\rho$ for $\rho \in \text{Gal}(L_{v'_1}/F_{v_1})$ multiplies $B_\rho$ by an element of $\prod_{v'|v_1} T(L_{v'})$ and causes no change in $\{s_\rho\}$.

(v) If\marginpar{137} we change $v'_1$ and replace it with $v''_1$,
\[
|\lambda x|_{v''_1} = |x|_{v'_1},
\]
then
\[
\sum_{\sigma \in \text{Gal}(L_{v''_1}/F_{v_1})} \sigma \mu = \lambda \sum_{\sigma \in \text{Gal}(L_{v'_1}/F_{v_1})} \sigma \lambda^{-1} \mu,
\]
and in general, this will only equal $\lambda \nu_1$ if $\lambda$ can be chosen such that $\lambda^{-1} \mu = \mu$. Therefore, we assume this condition. However, even in this case, $\tau_\mu$ is not independent of $v''_1$. Indeed, let $\tilde{A}_{\rho, \sigma}$ be liftings of $A_{\rho, \sigma}$, and let $C_{\rho,\sigma, \tau}$ be the Teichm\"uller cocycle defined by
\[
\rho(\tilde{A}_{\sigma,\tau}) \tilde{A}_{\rho\sigma ,\tau}^{-1} \tilde{A}_{\rho,\sigma\tau } \tilde{A}_{\rho, \sigma}^{-1} = C_{\rho,\sigma, \tau}.
\]
Then
\[
r_\rho = \prod_{\sigma \in \text{Gal}(L/\mathbb{Q})} C_{\rho,\sigma, \lambda}^{-\rho \sigma \mu}
\]
is a 1-cocycle of $\text{Gal}(L/\mathbb{Q})$ with values in $T(L)$, and $\tau_\mu$ is replaced by $\tau'_\mu: t_\rho \mapsto r_\rho \tau_\mu(t_\rho)$ when $v'_1$ is replaced by $v''_1$.

Indeed, $E_\rho$ is multiplied by the expression \eqref{2.e}, which is equal to
\[
\rho (Y) Y^{-1} \rho(Z_1) Z_{\rho}^{-1}.
\]
On the other hand, $F$ is multiplied by
\[
U = \prod_{\sigma \in \text{Gal}(L/F)} D_{\sigma}^{-\sigma \mu}.
\]
By definition,
\[
U = \prod_\sigma \lambda (A_{\lambda^{-1}, \sigma \lambda})^{\sigma \mu} A_{\sigma, \lambda}^{\sigma \mu} A_{\lambda, \lambda^{-1}}^{-\sigma \mu}
\]
and
\[
Y = \prod_{\tau \in \mathfrak S_1} A_{\lambda, \lambda^{-1}}^{-\tau \nu_1} \prod_{\tau \in \mathfrak S_1} \lambda (A_{\lambda^{-1}, \tau})^{\tau \nu_1}.
\]

Consequently,
\[
U Y^{-1} = \prod_\sigma \left( \lambda (A_{\lambda^{-1}, \sigma \lambda})^{\sigma \mu} A_{\sigma, \lambda}^{\sigma \mu} \right) \prod_\tau \lambda (A_{\lambda^{-1}, \tau})^{\tau \nu_1}.
\]
Due to the cocycle property, the first product on the right hand side is equal to
\[
\prod_\sigma \lambda (A_{\lambda^{-1} \sigma, \lambda} A_{\lambda^{-1}, \sigma})^{\sigma \mu} = \prod_\sigma A_{\sigma, \lambda}^{\sigma \mu} \prod_\sigma \lambda (A_{\lambda^{-1}, \sigma})^{\sigma \mu}.
\]
If\marginpar{138} $\sigma$ lies in $\tau \text{Gal}(L_{v'_1}/F_{v_1})$, then
\[
\lambda (A_{\lambda^{-1}, \sigma}) \equiv \lambda (A_{\lambda^{-1}, \tau}) \quad \text{mod} \quad \tau(L_{v'_1}).
\]
We can therefore lift $U Y^{-1}$ to the product of $
\prod_\sigma \tilde{A}_{\sigma, \lambda}^{\sigma \mu}
$
and an element of
$
\prod_{v'|v_1} T(L_{v'}).$ 	 	 Furthermore,
\[
\prod_\sigma \tilde{A}_{\sigma, \lambda}^{\sigma \mu} \prod_\sigma \rho (\tilde{A}_{\sigma, \lambda})^{- \rho \sigma \mu} = \prod_\sigma C_{\rho,\sigma, \lambda}^{- \rho \sigma \mu} = r_\rho,
\]
since
\[
\prod_\sigma \tilde{A}_{\rho, \sigma \lambda}^{ \rho \sigma \mu} \tilde{A}_{\rho,\sigma}^{-\rho\sigma \mu} = 1.
\]
Thus, $s'_\rho = r_\rho s_\rho$.

Finally, we want to show that $\tau_\mu$ is independent of $L$. Let $L' \supset L$ be a Galois extension of $F$. We reuse the earlier formulas. To avoid using the same symbol with two meanings, we write
\[
B^l_{\rho'} = B_\rho B_{\rho'}(1) B_{\rho'}^{-1}(2) G \rho(G)^{-1}.
\]
From the earlier calculations, it follows that $E_{\rho'}$ is the product of $E_\rho$ and two other factors. One of these factors is irrelevant for $s_{\rho'}$ because it lies in
\[
\prod_{v'|v_1} T(L_{v'}) \prod_{v'|v_2} T(L_{v'}).
\]
The other factor is
\[
G^\eta \rho(G)^{-\eta} \prod_{i=1}^2 \left\{ \prod_{\mu'} B_{\mu'}(i)^{\mu' \nu_i} \prod_{\mu'} \rho(B_{\mu'}(i))^{-\rho \mu' \nu_i} \right\}.
\]
On the other hand,
\[
F' = \prod_{\rho' \in \text{Gal}(L'/F)} B_{\rho'}^{-\rho' \mu}.
\]

We show that
\[
F' F^{-1} G^\eta \prod_{i=1}^2 \prod_{\mu'} B_{\mu'}(i)^{\mu' \nu_i} \equiv V  \quad \text{modulo} \quad \prod_{v'|v_1, v_2} T(L_{v'}),
\]
where $V$ is not only invariant but actually the image of an element in $T(\mathbb{A}_F)$. Thus, $s_{\rho'} = s_\rho$. 
Indeed, we show that
\[
W = \left( \prod_{\gamma \in H} B_\gamma^{-1} \right) G
\]
lies in $C_L$, and
\[
V = \prod_{\rho \in \text{Gal}(L/\mathbb{Q})} \rho (W)^{\rho \mu},
\]
which suffices, since $\mathbb A_L \to C_L$ is surjective.

Let\marginpar{139} $\rho'\in \text{Gal}(L'/F)$, and let $\mu_i'$ be its representative modulo $HG_i'$. It must be shown that
\begin{equation} \label{2.n}
	{\rho'}^{-1} \left (  \prod_{\gamma\in H} B^{-1}_{\rho'\gamma} B_{\rho} G B_{\mu'_1}(1) B^{-1}_{\mu'_2} (2)  \right )
\end{equation}
is congruent to
\begin{equation} \label{2.o}
	\prod_{\gamma \in H} B^{-1}_\gamma G,
\end{equation}
with equality if $\rho' \in H$.

First, we have
\begin{align*}
	{\rho'}^{-1} \left (  \prod_{\gamma\in H} B^{-1}_{\rho'\gamma}  \right) & = \prod_{\gamma} B_{\sigma'} ~ \cdot ~ \prod_{\gamma} B^{-1}_ \gamma \cdot A^{-1}_{\sigma', \rho'\gamma}(1) A_{\sigma', \rho' \gamma}(2) \\ & 
	= B^l_{\sigma'} \prod_{\gamma} \big( A^{-1}_{\sigma', \rho'\gamma}(1) A_{\sigma', \rho'\gamma}(2)
	\big ) \prod_{\gamma} B^{-1}_ \gamma, 	 	 
\end{align*}		
where $\sigma' = {\rho'}^{-1}$. Thus, \eqref{2.n} is equal to
\begin{equation} \label{2.p}
	B_{\sigma'}(1) B^{-1}_{\sigma'}(2) \prod_{\gamma} \big( A^{-1}_{\sigma', \rho'\gamma}(1) A_{\sigma', \rho'\gamma}(2) \big ) \sigma' \big( B_{\mu_1'}(1) B^{-1}_{\mu_2'}(2) \big)  \sigma(B_\rho) B_{\sigma}
\end{equation}
times \eqref{2.o}.

Second, we have
\[
\sigma' \left( B_{\mu_i'}(i) \right) B_{\sigma'}(i) = B_{\sigma' \mu_i'}(i) A^l_{\sigma', \mu_i'}(i) A^{-1}_{\sigma, \mu_i}(i),
\]
and
\[
B_{\sigma' \mu_i'}(i) = \prod_{\gamma \in H} A_{\sigma' \mu_i', \gamma}(i),
\]
since $\sigma' \mu_i' \in HG_i'$.

If $\rho'\in H$, then $B_\rho = B_{\sigma} = 1$, $\mu_i' = 1$, and
\[
\prod_{\gamma} A_{\sigma', \rho' \gamma}(i) = \prod_{\gamma} A_{\sigma', \gamma}(i),
\]
which yields the desired equality.

In general, if $\rho' \in \mu_i' H \nu_i$ with $\nu_i$ from the right representative system of $H_i$ in $G_i'$, then
\[
\prod_{\gamma} A_{\sigma', \rho'\gamma}(i) = \prod_{\gamma} A_{\sigma', \mu_i' \gamma \nu_i}(i)
\]
and
\[
A_{\sigma', \mu_i' \gamma \nu_i}(i) A^{-1}_{\sigma', \mu_i' \gamma}(i) = \sigma' \left( A_{\mu_i' \gamma, \nu_i}(i)^{-1} \right) A_{\sigma' \mu_i' \gamma, \nu_i}(i) \in \sigma' \mu_i' \gamma(L'_{w_i'}).
\]
Furthermore,
\[
A^{-1}_{\sigma', \mu_i' \gamma}(i) A_{\sigma' \mu_i', \gamma}(i) A_{\sigma', \mu_i'}(i) = \sigma' \left( A_{\mu_i', \gamma}(i) \right) = 1.
\]
Thus,\marginpar{140} \eqref{2.p} is equal to
\[
\sigma(B_\rho) B_{\sigma} A^{-1}_{\sigma, \mu_1}(1) A_{\sigma, \mu_2}(2).
\]
But
\[
A_{\sigma, \mu_i}(i) \equiv A_{\sigma, \rho}(i) \pmod{L'_{w_i'}},
\]
and
\[
\sigma(B_\rho) B_{\sigma} A^{-1}_{\sigma, \rho}(1) A_{\sigma, \rho}(2) = 1,
\]
since $
\rho \sigma = 1$.

We now formulate the main result in a theorem.

\begin{thm}\label{2.3}
	In the situation of Theorem \ref{2.2}, let $\mu$ be an element of $X_*(T)$ that yields the cocharacters $\nu_i$ by averaging:
	\[
	\nu_i = (-1)^{i+1} \sum_{\sigma \in \Gal(L_{v_i'}/F_{v_i})} \sigma \mu.
	\]
	Then the above construction provides a homomorphism $\tau_{\mu}$ from the gerbe $\mathscr T_{\nu_1, \nu_2}$ to the neutral gerbe $\mathscr G_T$, such that $\tau_{\mu} \circ \zeta_v$ for $v \neq v_1, v_2$ is equivalent to the canonical neutralization, and such that $\tau_{\mu} \circ \zeta_{v_1}$ is equivalent to $\xi_{\mu}$ and $\tau_{\mu} \circ \zeta_{v_2}$ is equivalent to $\xi_{-\mu}$ (the definition of $\xi_{\mu}$ appears after Theorem \ref{2.2}). The homomorphism $\tau_{\mu}$ is uniquely determined up to composition with an automorphism of $\mathscr G_T$ that is equivalent to the identity automorphism at all places. The construction is functorial in $T$ and $\mu$.  
\end{thm}

	\section{The pseudomotivic Galois group}
	In this section, we introduce two gerbes: not only the truly important one, which we call the pseudomotivic Galois group, but also a subordinate one, which exists solely to be able to formulate the conjecture in the fifth section for the most general Shimura variety.

	The pseudomotivic Galois group \( \mathscr P \) is defined as a direct limit. Let \( L \) be a finite Galois extension of \( \mathbb{Q} \) in the field \( \overline{\mathbb{Q}} \) of algebraic numbers in \( \mathbb{C} \), and let \( m \) be a natural number. We first define a gerbe  \( \mathscr P(L,m) \), whose kernel is a torus \( P(L,m) \), which we define by prescribing the Galois module of its characters. Let \( q = p^m \), where \( p \) is the fixed rational prime.

\begin{defn} The group \( X(L,m) \) of Weil numbers associated with \( L \) and \( m \) consists of those \( \pi \in L \) with the following properties:
\begin{enumerate}
	\item [(a)] There exists an integer \( \nu_1 = \nu_1(\pi) \) such that for all Archimedean places \( v \) of \( L \), we have
	\[
\Big|\prod_{\sigma \in \Gal(L_v/\RR)} \sigma\pi \Big |    = q^{\nu_1} . 
	\]
	\item [(b)] For each prime \( v \) of \( L \) above \( p \), there exists an integer \( \nu_2(v) = \nu_2(\pi, v) \in \mathbb{Z} \) such that
	\[
	|\pi|_v = 
	\Big |\prod_{\sigma \in \Gal(L_v/\mathbb{Q}_p) }\sigma \pi \Big |_p = q^{\nu_2(v)} . 
	\]  
	\item [(c)]  At\marginpar{141} all finite places outside \( p \), \( \pi \) is a unit.
 \end{enumerate} 
\end{defn}

From Dirichlet's Unit Theorem, it follows that \( X(L,m) \) is finitely generated. We divide out the finite group of the roots of unity contained in \( X(L,m) \) to obtain a torsion-free module \( X^*(L,m) \). The corresponding torus is denoted by \( P(L,m) \). The module \( X^*(L,m) \) is clearly a \( \Gal(\overline{\mathbb{Q}}/\mathbb{Q}) \)-module. Consequently, \( P(L,m) \) is defined over \( \mathbb{Q} \).

If necessary, to avoid misunderstandings, we denote by \( \chi_\pi \) the character of \( P(L,m) \) corresponding to the Weil number \( \pi \in X(L,m) \). We fix an embedding \( \overline{\mathbb{Q}} \to \overline{\mathbb{Q}}_p \) once and for all, and consequently a prime place \( v_2 \) of every finite-degree number field \( L \) contained in \( \overline{\mathbb{Q}} \). Let \( v_1 \) be the Archimedean place of \( L \) given by \( L \hookrightarrow \overline{\mathbb{Q}} \hookrightarrow \mathbb{C} \). We define the  cocharacters \( \nu_1, \nu_2 \) in \( X_*(L,m) = X_*(P(L,m)) \) as follows:
\begin{align}
\langle \nu_1, \chi_\pi \rangle = \nu_1(\pi) , \\ 
 \langle \nu_2, \chi_\pi \rangle = \nu_2(\pi, v_2)  
\end{align} 

The relation \eqref{2.a} is obvious, while \eqref{2.b} follows from the product formula. It should be noted that this \( \nu_1 \) is even invariant under the full group \( \Gal(\overline{\mathbb{Q}}/\mathbb{Q}) \).

The following functorialities are present:
\begin{lr}
For  \(L \subset L'\), we have \(\phi^* =  \phi^*_{L,L'}: X^*(L,m) \to X^*(L',m) \)  sending \(\pi\)  to itself. 
\end{lr}
\begin{lr}
	For \(m | m'\), we have \(\phi^* = \phi^*_{m,m'} : X^*(L,m) \to X^*(L,m')\) sending \(\pi\) to \(\pi^{m ' / m} \). 
\end{lr}

The contragredient maps are homomorphisms from the \( \mathbb{Q} \)-tori \( P(L',m) \) and \( P(L,m') \) into \( P(L,m) \).

\begin{lem}\label{3.2}
	\begin{enumerate}
		\item[(a)] We have 
		\[
		\phi_{m,m'}(\nu_i') = \nu_i.
		\]  
		\item[(b)]  We have 
		\[
		\phi_{L,L'}(\nu_i') = [L'_{v_i'} : L_{v_i}]\nu_i,
		\]  
		where \( v'_i \) are the places of \( L' \) defined by \( L' \hookrightarrow \overline{\mathbb{Q}} \hookrightarrow \mathbb{C} \) and \( L' \hookrightarrow \overline{\mathbb{Q}} \hookrightarrow \overline{\mathbb{Q}}_p \). 
	\end{enumerate} \end{lem}
\marginpar{\textit{Proof of Lem.~\ref{3.2}.}}
If \( m \) is replaced by \( m' \), then \( q \) is replaced by \( q^{m'/m} \). From this, the first claim follows immediately. The second claim follows from the equation  
\[
\prod_{\sigma \in \Gal(L'_{v_i'}/\mathbb{Q}_p)} \sigma\pi = \left( \prod_{\sigma \in \Gal(L_{v_i}/\mathbb{Q}_p)} \sigma\pi \right)^{[L'_{v_i'} : L_{v_i}]},
\]  
which holds for \( \pi \in L \). \marginpar{\textit{QED Lem.~\ref{3.2}.}}\\

The\marginpar{142} procedure from the previous section associates the gerbe \( \mathscr P(L,m) \) with the triple \( P(L,m), \nu_1, \nu_2 \) and the field \( L \). From Lemma \ref{3.2}(a), it follows that there is a canonical homomorphism  
\[
\phi = \phi_{m,m'} : \mathscr  P(L,m') \to \mathscr P(L,m).
\]  

In the previous section, it was shown that the gerbe defined by \( P(L,m), \nu_1, \nu_2 \), and \( L \) is isomorphic to the gerbe defined by \( P(L,m), [L'_{v_1'} : L_{v_1}]\nu_1, [L'_{v_2'} : L_{v_2}]\nu_2 \), and \( L' \). Consequently, there is also a canonical homomorphism  
\[
\phi = \phi_{L,L'} : \mathscr P(L',m) \to \mathscr P(L,m).
\]  

As preparation for the next section, where we derive a canonical isomorphism between the pseudomotivic Galois group  
\[
\mathscr P = \varprojlim_{L,m} \mathscr P(L,m)
\]  
and the actual motivic Galois group based on the standard conjectures and the Tate conjecture, we need to demonstrate certain cohomological properties of the inverse system \( P(L,m) \).  Before doing so, we introduce a second, somewhat artificial gerb, which, for lack of a better name, we call the quasimotivic Galois group. It is defined using quasi-Weil numbers.

\begin{defn}\label{3.3}
The group \( Y(L,m) \) of quasi-Weil numbers associated with \( L \) and \( m \) consists of those \( \pi \in L \) satisfying the following properties:
\begin{enumerate}
	\item[(a)]  There exists an integer \( \nu_1 = \nu_1(\pi) \) such that for all Archimedean places \( v \) of \( L \),  
\[ \Big | 
\prod_{\sigma \in \Gal(L/\mathbb{Q})} \sigma \pi \Big |^{[L_v : \mathbb{R}]} = q^{\nu_1 [L : \mathbb{Q}]}.
\]  
\item[(b)] For each place \( v \) of \( L \) above \( p \), there exists \( \nu_2(v) = \nu_2(\pi, v) \in \mathbb{Z} \) such that  
\[
|\pi|_v = \Big|\prod_{\sigma \in \Gal(L_v/\mathbb{Q}_p)} \sigma \pi \Big |_p = q^{\nu_2(v)}.
\]   
\item[(c)]  At all finite prime places \( v \) outside \( p \), \( \pi \) is a unit. 
\end{enumerate} 
\end{defn} 

Just as \( X(L,m) \), the groups \( Y(L,m) \) also form an inverse system. To obtain \( Y^*(L,m) \), one divides \( Y(L,m) \) by the (typically infinite) group of all units, not just the roots of unity. This gives rise to an inverse system of gerbes \( \mathscr Q(L,m) \), whose limit \( \mathscr  Q \) we call the quasimotivic Galois group.  

The group \( P(L,m) \) contains a sequence \( \{\delta_n\} \), where \( m | n \) and \( n \) is sufficiently large, of distinguished rational points defined by the equations  
\[
\chi_\pi(\delta_n) = \pi^{\frac{n}{m}},
\]  
where \( \frac{n}{m} \) must be divisible by the order of the torsion of \( X(L,m) \).  The rationality of \( \delta_n \) follows from the relation  
\[
\sigma(\chi_\pi(\delta_n)) = (\sigma \pi)^{\frac{n}{m}} = \chi_{\sigma \pi}(\delta_n). 
\]   

The equations  \marginpar{143}
\[
\phi_{m,m'}(\delta_n) = \delta_n, \quad \phi_{L,L'}(\delta_n) = \delta_n
\]  
are easily verified. In \( Q(L,m) \), the kernel of \( \mathscr Q(L,m) \), there is no uniquely determined \( \delta \). We can define several, and this lack of uniqueness will mostly be ignored.  If \( k \) is sufficiently large, then there exists a commutative diagram of \( \Gal(L/\mathbb{Q}) \)-modules:
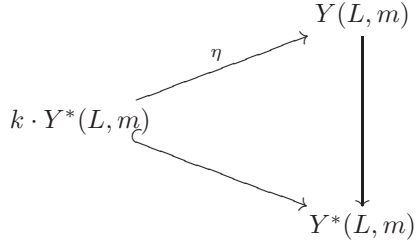
\[ \xymatrix { && Y(L,m) \ar[dd] \\ 
	k \cdot Y^*(L,m)  \ar[urr]^{\eta} \ar@{^(->}[drr]  \\ 
 && 	Y^*(L,m)   }
\]   
We define \( \delta_n \) by the equations  
\[
\chi_\pi(\delta_n) = \eta(\pi^k)^{\frac{n}{mk}}.
\]  
Thus, for every \( \pi \), the number \( \pi^{-n} \chi_\pi(\delta_n) \) is a unit.

Let \( T \) be a torus defined over \( \mathbb{Q} \), and let \( \mu \in X_*(T) \). If \( T \) splits over \( L \), and \( m \) is sufficiently large, we associate to the pair \( (T, \mu) \) a homomorphism  
\[
\psi_\mu : Q(L,m) \to T
\]  
defined over \( \mathbb{Q} \).\footnote{This part of the article contains a well-known mistake. For the current definition of $Q(L,m)$, the map $\psi_{\mu}$ cannot be constructed as claimed. The element $\pi_\lambda = \lambda(\gamma)$ defined below does not satisfy Definition \ref{3.3}(a) in general and hence is not an element of $Y(L,m)$. (If \eqref{3.e} holds, then $\pi_{\lambda}$ is indeed an element of $X(L,m) \subset Y(L,m)$ as claimed here, so we have $\psi_\mu : Q(L,m) \to T$ factoring through $P(L,m)$.) For correction, one needs to appropriately modify the definition of $Q(L,m)$, and thereby modify the definition of the gerbe $\mathscr Q$. Such modifications have been proposed by Pfau \cite{PfauThesis,Pfau96}, and later by Reimann \cite{reimann1997zeta} (see \S B.2). (References are at the end of Translator's Note.) The relationship between their versions is explained in Remark B2.13 of \textit{loc.~cit.}.} If 
\begin{align}\label{3.e}
(\sigma - 1)(\iota + 1)\mu = (\iota + 1)(\sigma - 1)\mu = 0,
\end{align} 
then \( \psi_\mu \) can be factored through \( P(L,m) \) and can be regarded as a homomorphism from \( P(L,m) \) to \( T \).  
It should be noted that the embedding  
\[
X(L,m) \hookrightarrow Y(L,m)
\]  
defines a homomorphism  
\[
Q(L,m) \to P(L,m).
\]

The homomorphism \( \psi_\mu \) is defined via a homomorphism of Galois modules \( X^*(T) \to Y^*(L,m) \), specifically via \( X^*(T) \to Y(L,m), \lambda \mapsto \pi_\lambda \), where it must hold that  
\[
\pi_{\sigma \lambda} = \sigma(\pi_\lambda).
\]  
If the equations \eqref{3.e} hold, then the following will also be true:  
\[
\sigma \big( \iota(\pi_\lambda) \cdot \pi_\lambda \big) = \pi_{\sigma \iota \lambda + \sigma \lambda} = \pi_{(\iota+1)\sigma \lambda} = \pi_{(\iota+1)\lambda} = \iota(\pi_\lambda) \cdot \pi_\lambda,
\]  
and  
\[
\iota(\pi_\lambda) \pi_\lambda = \pi_{\iota \lambda + \lambda} = \pi_{\iota ' \lambda + \lambda} = \iota'(\pi_\lambda) \pi_\lambda, \quad \iota' = \sigma \iota \sigma^{-1}.
\]  
Thus, \( \iota(\pi_\lambda) \pi_\lambda \) is a rational number, and \( \pi_\lambda \in X(L,m) \).

To
\marginpar{144} define \( \lambda \mapsto \pi_\lambda \), we choose a suitable \( \gamma \in T(\mathbb{Q}) \) and set  
\[
\pi_\lambda = \lambda(\gamma).
\]  
For the group  \( \mathscr P \), we then have \( \psi_\mu(\delta_m) = \gamma \). For the group \( \mathscr Q \),  
\( 
\lambda \big( \gamma^{-1} \psi_\mu(\delta_m) \big) \) 
is a unit for every \( \lambda \).  Thus, we can replace \( \gamma \) with \( \psi_\mu(\delta_m) \) and assume that \( \psi_\mu(\delta_m) = \gamma \).

The homomorphism \( \psi_\mu \) must satisfy the conditions\footnote{Here $v_i' = \infty, p$ for $i=1,2$.  }
\[
\psi_\mu(\nu_i) = -(-1)^i \sum_{\sigma \in \Gal(L_{v_i} / \mathbb{Q}_{v_i'})} \sigma \mu,
\]  
which, when translated to \( \gamma \), means that  
\[ 
\Big| 
\prod_{\sigma \in \Gal(L_{v_i} / \mathbb{Q}_{v_i'})} \sigma \lambda(\gamma) \Big |  = q^{-(-1)^i  \Big\langle \lambda  , ~ \sum_{\sigma \in \Gal(L_{v_i} / \mathbb{Q}_{v_i'})} \sigma \mu \Big \rangle } .
\]  
Finally, \( \lambda(\gamma) \) is a unit outside the infinite places and \( p \).  

We choose \( m \) large enough so that there exists an element \( a \) in \( L \) such that the ideal \( (a) \) is a power \( \mathfrak p^r \) of the prime ideal \( \mathfrak  p \) corresponding to the place \( v_2 \), and such that  
\[
\big | \text{Nm}_{L_{v_2} / \mathbb{Q}_{v_2}} a  \big |_p = q.
\]  
The element \( \gamma \) is then defined by the equations  
\[
\lambda(\gamma) = \prod_{\sigma \in \Gal(L / \mathbb{Q})} \sigma(a)^{\langle \lambda, \sigma \mu \rangle }.
\]  Then \( \gamma \) is rational, and 
\begin{align*}
\Big| \prod_{\sigma \in \Gal(L_{v_2} / \mathbb{Q}_{v_2})} \sigma \lambda(\gamma) \Big|_p &  = \Big| \prod_{\rho \in \Gal(L_{v_2} / \mathbb{Q}_{v_2})} \prod_{\sigma \in \Gal(L / \mathbb{Q})} \rho \sigma(a)^{\langle \lambda, \sigma \mu \rangle }   \Big|_p \\ & =  \Big| \prod_{\rho, \sigma \in \Gal(L_{v_2} / \mathbb{Q}_{v_2})} 
\rho  \sigma(a)^{\langle\lambda, \sigma \mu \rangle }  \Big|_p  \\ & = q^{\big \langle \lambda, \sum_{\sigma} \sigma \mu \big \rangle }.
\end{align*}    
According to the previous section, \( \psi_\mu \) can canonically be extended  to a homomorphism of gerbes:  
\(
\psi_\mu : \mathscr Q(L, m) \to \mathscr G_T  
\) 
or  
\(
\psi_\mu : \mathscr P(L, m) \to \mathscr G_T.
\) 
 
 From now on in this section, we will focus exclusively on the gerbes \( \mathscr P(L,m) \) and begin with a simple remark.

\begin{lem}\label{3.4}
Let \( K \) be the maximal CM subfield of \( L \). Then \( X(L,m) = X(K,m) \).
\end{lem} 

\marginpar{\textit{Proof of Lem.~\ref{3.4}.}}
We\marginpar{145} recall that the Galois extension \( K \) is a CM field if the complex conjugation \( \iota \) lies in the center of \( \Gal(K/F) \). We therefore need to show that \[  \iota(\pi) = \iota'(\pi) \] when \( \pi \in X(L,m) \), and \( \iota' = \sigma \iota \sigma^{-1}, \sigma \in \Gal(L/F) \). If \( L \) is a real field, this is clear. Otherwise, we have  \[ \iota(\pi) \pi = q^{\nu_1(\pi)} = \iota'(\pi) \pi ,\]  and consequently \( \iota'(\pi) = \iota(\pi) \). \marginpar{\textit{QED Lem.~\ref{3.4}.}}\\

We now assume without loss of generality that \( L \) is a CM field. Let \( L_0 \) be its totally real subfield.

\begin{lem}\label{3.5}
 \( H^2(\mathbb{Q}, P(L,m)) \) satisfies the Hasse principle.
\end{lem} 
\marginpar{\textit{Proof of Lem.~\ref{3.5}.}}
Clearly, it suffices to show that the homomorphism
\[
H^1\left(\Gal(L'/\QQ), P(\mathbb A_{L'})\right) \to H^1\left(\Gal(L'/\QQ), X_* \otimes C_{L'}\right)
\]
is surjective when \( P = P(L,m) \) and \( L_1 \subset L_2 \). So, with Tate-Nakayama, we need to establish the surjectivity of

\[
\bigoplus_{v \in \mathbb{Q}} H^{-1}\left(\Gal(L_v'/\mathbb{Q}_v), X_*\right) \to H^{-1}\left(\Gal(L'/\mathbb{Q}), X_*\right).
\]

There is an exact sequence of Galois modules
\[
1 \to X_*'\to X_* \to \mathbb{Z} \to 1,
\]
where \( \iota \) acts as \( -1 \) on \( X_*' \). An element from \( H^{-1}(\Gal(L'/Q), X_*) \) is represented by \( \lambda \in X^* \) with
\[
\sum_{\sigma \in \Gal(L'/Q)} \sigma \lambda = 0.
\]
Consequently, \( \lambda \) lies in \( X_*' \), and
\[
\sum_{\sigma \in \Gal(L'_{v_1}/\mathbb{R})} \sigma \lambda = 0,
\]
so that \( \lambda \) represents an element from \( H^{-1}(\Gal(L'_{v_1}/\mathbb{R}), X_*) \). Thus, the lemma is proven. \marginpar{\textit{QED Lem.~\ref{3.5}.}}\\
 
 There are still additional cohomological properties of the groups \( P(L,m) \) to address. Before that, we make some remarks about the modules \( X_*(L,m) \). The Serre group \( ^LS \) is a torus over \( \mathbb{Q} \) defined by its character module \( X^*({}^LS) \):
\[
X^*({}^LS) = \left\{ \lambda = \sum _ {\rho :L \to \overline{\mathbb{Q}} } n_{\rho}[\rho] \mid n _{\rho} + n _{\iota \rho} \equiv k(\lambda), k(\lambda) \in \mathbb{Z} \right\}.
\]
\marginpar{146}The group \( ^LS \) is equipped with a distinguished cocharacter \( \mu \):
\[
\mu :\sum _{\rho} n_\rho [\rho] \mapsto n_1,
\]
where \( 1 \) denotes the given embedding \( L \to \overline{\mathbb{Q}} \). The  cocharacter \( \mu \) defines \( \psi_\mu : P(L,m) \to {}^LS \) and \( \psi^*_\mu : X^*({}^LS) \to X^*(L,m) \).

\begin{lem}\label{3.6}
 For sufficiently large \( m \), the homomorphism \( \psi^*_\mu \) is surjective, and thus \( P(L,m) \to  {}^LS \) is injective.
\end{lem} 

To each \( \rho \in \Gal(L/\mathbb{Q}) \), a prime \( v_{\rho} \) dividing \( p \) is assigned, with the condition:
\[
|x|_{v_{\rho}} = | \rho^{-1} x | _{v_2}.
\]
If \( \sigma \in \Gal(L_{v_2}/\mathbb{Q}_{v_2}) \), then \( v _{\rho\sigma} = v _{\rho}\). The module \( X_*(L,m) \) contains the distinguished elements \( \nu_\infty = \nu_1 \) and \( \nu_{\rho}= -
\rho \nu_2 \). The following relations hold:
\[
\nu_{\rho \sigma } = \nu_{\rho}, \quad \sigma \in \Gal(L_{v_2}/\mathbb{Q}_{v_2}),
\]
and
\[
\nu_{\rho} + \nu_{\iota \rho} = \nu_\infty k, \quad k = 2[L_{v_2} : \mathbb{Q}_p][L : L_0]^{-1}.
\]
The first relation is clear. The second follows from
\[
q^{\nu_{\rho}(\pi) + \nu_{\iota \rho}(\pi) } = \Big| \prod_{\sigma \in \Gal(L_{v_{\rho}}/\mathbb{Q}_p)}\rho^{-1}\sigma(\pi \iota(\pi)) \Big|_p = |\pi\iota (\pi)|_p ^{[L_{v_2} : \mathbb{Q}_p]}. 
\]

\begin{lem}\label{3.7}
	\begin{enumerate}
		\item [(a)] If the \( p \)-dividing places of \( L_0 \) do not split in \( L \) (note that either none of them splits, or all of them split), then
		\[
		\nu_{\rho} = \frac{k}{2}\nu_\infty,
		\]
		and for sufficiently large \( m \), \( \{ \nu_\infty \} \) is a \( \mathbb{Z} \)-basis of \( X_*(L,m) \). 
		\item [(b)]   If the (\(p\)-dividing) places of \( L_0 \) split in \( L \), let \( \rho_1, \dots, \rho_r \) be a system of representatives for the double cosets \( \Gal(L_{v_1}/\RR) \backslash \Gal(L/\mathbb{Q}) / \Gal(L_{v_2}/\mathbb{Q}_p) \). Then, for sufficiently large \( m \), the set
		\[
		\{ \nu_\infty, \nu_{\rho_1}, \dots, \nu_{\rho_r} \}
		\]		
		is a \( \mathbb{Z} \)-basis of \( X_*(L,m) \).
	\end{enumerate}
\end{lem}
\marginpar{\textit{Proof of Lem.~\ref{3.6}, \ref{3.7}.}}
In case (a), \( \nu_\rho = \nu_{\iota\rho} \), and the given relation is clear. In both cases it is clear that the elements that are to form a basis generate a submodule of finite index, since from \[  \nu_\infty(\pi) = \nu_\rho(\pi) = 0  \] for each $\rho$, we conclude that \( \pi \) is a root of unity. We only need to show that in case (a), \( \nu_\infty \), and in case (b), \( \nu_\infty \) and \( \nu_{\rho_1}, \dots, \nu_{\rho_r} \), can be arbitrarily prescribed.\footnote{The authors probably mean that one can always find $\chi \in X^*(L,m)$ such that the pairings $\langle \nu_{\bullet}  , \chi \rangle$ are equal to arbitrarily prescribed integers.  } We will prove this statement together with Lemma \ref{3.6}.

The\marginpar{147} map \( X^*({}^LS) \to X^*(L,m) \) is dual to the map \( \psi_\mu : X_*(L,m) \to X_*({}^LS) \), which is defined by
\[
\psi_\mu(\nu_\infty) = \mu + \iota \mu, \quad L_0 \neq  L, \quad \psi_\mu(\nu_\infty) = \mu, \quad L_0 = L,
\]
and
\[
\psi_\mu(\nu_{\rho_i}) = \sum_{\sigma \in \Gal(L_{v_2}/\mathbb{Q}_p)} \rho_i \sigma \mu .
\]
Thus,
\[
\psi_\mu(\nu_{\rho_i}) (\sum a_{\rho}[\rho]) = \sum_{\sigma} a_{\rho_i \sigma} ,\] 
\[ \psi_\mu(\nu_\infty)  (\sum a_{\rho}[\rho]) =  \begin{cases}
	a_1 + a_{\iota}, & L_0 \neq L, \\
	a_1, & L_0 = L. 
\end{cases}
\]

Since \( a_1 \) and \( a_\iota \) can be chosen arbitrarily, it is clear that we can find a \( \chi \) in the image of \( X^*({}^LS) \) with the given \( \langle \nu_\infty, \chi \rangle \). In case (b), we can choose \( a_{\rho_1}, \dots, a_{\rho_r} \) and \( a_{\iota \rho_1} \) arbitrarily, and at the same time require that \( a_{\rho_i \sigma} = a_{\iota \rho_1 \sigma} = 0 \) if \( \sigma \neq 1 \), \( \sigma \in \Gal(L_{v_2}/\mathbb{Q}_p) \).\footnote{Here one takes $\rho_1 = 1$. One should delete the requirement that $a_{\iota \rho_1 \sigma} = 0$.   } Thus, both lemmas are proven. \marginpar{\textit{QED Lem.~\ref{3.6}, \ref{3.7}.}}

\begin{cor}
For sufficiently large \( m \) and \( m | m' \), the map \( \phi_{m,m'} : X_*(L,m') \to X_*(L,m) \) is an isomorphism. 
\end{cor}

 We can simplify the notation a bit by using the symbols \( P(L) \), sometimes \( P_L \), or \( X^*(L) \) instead of \( P(L,m) \) or \( X^*(L,m) \), with the understanding that \( m \) is sufficiently large.

The next statements concern the first cohomology group of \( P_L \). Since they are trivial in case (a) of Lemma \ref{3.7}, we can restrict ourselves to case (b) in the proofs. We consider the exact sequence of \( \text{Gal}(L/\QQ) \)-modules:  
\begin{align}\label{3.f}
   0 \to \mathbb{Z}\nu_\infty \to X_*(L) \to X_*(N) \to 0.
\end{align}
Here, \( N \) is a torus, and \( X_*(N) \) is defined by this sequence. Let \( \bar{\nu}_\rho \) denote the image of \( \nu_\rho \) in \( X_*(N) \).  

The structure of the Galois module \( X_*(N) \) can be immediately inferred from Lemma \ref{3.7}. Let \( H = \Gal(L_{v_2}/\mathbb{Q}_p) \) and \( H_0 = H \cup H \iota \). We obtain an exact sequence of \( \Gal(L/\QQ) \)-modules:  
\begin{align}\label{3.g}
0 \to \text{Ind}_{H_0}^{G} 1 \to \text{Ind}_{H}^{G} 1 \to X_*(N) \to 1.
\end{align}
\begin{lem}\label{3.9}
\( H^1(\mathbb{Q}, P_L) \) satisfies the Hasse principle.  
\end{lem} 
\marginpar{\textit{Proof of Lem.~\ref{3.9}.}}
From \eqref{3.f} and Theorem 90, it suffices to show that the Hasse principle holds for \( H^1(\mathbb{Q}, N) \). This follows from \eqref{3.g}, Theorem 90, the Shapiro Lemma, and the Hasse principle for \( H^2(H_0, K_0^\times) \), where \( K_0 \) denotes the fixed field of \( H_0 \). 
\marginpar{\textit{QED Lem.~\ref{3.9}.}}

\begin{lem}\label{3.10}
Let\marginpar{148} \( L \) be a   CM-field. Then there exists a CM-extension \( L'\) of \( L \), such that for every prime \( v \) of \( \mathbb{Q} \), the transition homomorphism  
\[
H^1(\mathbb{Q}_v, P_{L'}) \to H^1(\mathbb{Q}_v, P_L)
\]  
is zero.  
\end{lem}
\marginpar{\textit{Proof of Lem.~\ref{3.10}.}}
We use the Tate-Nakayama theory to prove that  
\[
H^{-1}(\text{Gal}(L'_{v'}/\mathbb{Q}_v), X_*(L')) \to H^{-1}(\text{Gal}(L_v/\mathbb{Q}_v), X_*(L))
\]  
is zero, where \( v \) and \( v' \) denote extensions of \( v \) in \( L \) and \( L' \), respectively. Clearly, it suffices to show that  
\[
H^{-1}(\text{Gal}(L'_{v'}/\mathbb{Q}_v), X_*(N')) \to H^{-1}(\text{Gal}(L_v/\mathbb{Q}_v), X_*(N))
\]  
is zero, where \( X_*(N') \) is defined by the sequence  
\[
0 \to \mathbb{Z}\nu'_{\infty} \to X_*(L') \to X_*(N') \to 0.
\]  
An element of \( H^{-1}(\text{Gal}(L_v/\mathbb{Q}_v), X_*(N)) \) is represented by  
\[
\nu = \sum_{\text{Gal}(L/\mathbb{Q}) / \text{Gal}(L_{v_2}/\mathbb{Q}_p)} a_\rho \bar{\nu}_\rho, \quad a_\rho \in \mathbb{Z}.
\]  
Let \( s = [L : \mathbb{Q}] \); then \( \nu \) represents the trivial class if \( s |  a_\rho\) for every \( \rho \).

An element of \( H^{-1}(\text{Gal}(L'_{v'}/\mathbb{Q}_v), X_*(N')) \) is similarly represented by  
\[\nu' = \sum_{\text{Gal}(L'/\mathbb{Q}) / \text{Gal}(L'_{v_2'}/\mathbb{Q}_p)} a_{\rho'}\bar{\nu}_{\rho'},  \]  
and the image \( \phi_{L,L'}(\nu_{\rho'}) \) is \( r \nu_\rho \), where \( \rho \) is the image of \( \rho' \) in \( \text{Gal}(L/\mathbb{Q}) \), and  
\[r = [L'_{v_2'}: L_{v_2}].\]  
The task is therefore to choose \( L' \) such that \( s | r \). This can be achieved by setting \( L'= KL \), where \( K \) is a suitably chosen abelian extension of \( \mathbb{Q} \). \marginpar{\textit{QED Lem.~\ref{3.10}.}} \\ 

Applying global Tate-Nakayama theory in the same way yields the following statement. 
\begin{lem}\label{3.11}
Let \( L \) be a CM-field with idele class group \( C_L \). Then there exists a CM-extension \( L' \) of \( L \) such that the transition homomorphism  
\[
H^1(\mathrm{Gal}(L'/\mathbb{Q}), X_*(L') \otimes C_{L'}) \to H^1(\mathrm{Gal}(L/\mathbb{Q}), X_*(L) \otimes C_L)
\]  
is zero.  
\end{lem}  

The\marginpar{149} following corollary immediately follows from Lemmas \ref{3.9} and \ref{3.10}:  
\begin{cor}
Let \( L \) be a  CM-field. Then there exists a CM-extension \( L' \) of \( L \) such that the transition homomorphism  
\[
H^1(\mathbb{Q}, P_{L'}) \to H^1(\mathbb{Q}, P_L)
\]  
is zero.  
\end{cor}

 By Lemma \ref{3.6}, there exists an embedding  
 \[
 \psi_L = \psi_\mu : \mathscr P_L \hookrightarrow \mathscr G_{^LS}.
 \]  
 From the definition of \( \psi_\mu \) and Lemma \ref{3.7}, the commutativity of the following diagram (up to equivalence) follows:  
 \begin{equation}\label{3.h}
\xymatrix{\mathscr P_{L'} \ar[r]^{\psi_{L'}} \ar[d] & \mathscr G_{L'} \ar[d]  	\\ \mathscr P_L \ar[r]^{\psi_L} & \mathscr G_L , }
 \end{equation}
 where \( \mathscr  G_L = \mathscr G_{^LS} \) and \( \mathscr G_{L'} = \mathscr G_{ ^{L'} S} \).  Furthermore, from the definitions in the previous section, the following hold:  
 \begin{align}\label{3.i}
\psi_\mu \circ \zeta_{v_1} \cong \xi_\mu, \quad \psi_\mu \circ \zeta_{v_2} \cong \xi_{-\mu}.
 \end{align}
\begin{lr}\label{3.j} If $v$ neither divides $\infty$  nor $p$, then $\psi_\mu \circ \zeta_v$ is equivalent to the canonical neutralization of $\mathscr G_L$.
\end{lr}

That the local homomorphisms are compatible with modifications of \(L\), up to equivalence, follows from the results of the previous section.  

Two local homomorphisms differ by an element in  
\[
H^1\left(\mathrm{Gal}(\overline{\mathbb{Q}}_v/\mathbb{Q}_v),  {}^LS (\overline{\mathbb{Q}}_v)\right).
\]  
The following lemma thus implies that, for a \emph{non-archimedean} place \(v\) of \(\mathbb{Q}\), any two compatible families of local homomorphisms are automatically equivalent. In particular, the condition \eqref{3.j} as well as the second part of \eqref{3.i} are automatically satisfied.  
 \begin{lem}\label{3.13}
 	 Let \( v \) be a non-archimedean place of \( \mathbb{Q} \). For every CM-field \( L \), there exists a CM-field \( L' \) that contains \( L \) such that  
 	 \[
 	 H^1\left(\mathrm{Gal}(\overline{\mathbb{Q}}_v/\mathbb{Q}_v), {}^{L'}S(\overline {\mathbb{Q}}_v)\right) = 0.
 	 \]
 \end{lem} 
 \begin{proof}
 We consider the exact sequence of algebraic tori over \( \mathbb{Q}_v \) ([De2]):  
 \[
 1 \to R_{L'_0/\mathbb{Q}} \mathbb{G}_m \to \mathbb{G}_m \times R_{L'/\mathbb{Q}} \mathbb{G}_m \to {}^{L'} S \to 1.
 \]  
 \marginpar{150}The associated long exact sequence of cohomology gives:  
 \[
 0 \to H^1(\mathbb{Q}_v, {}^{L'} S) \to \bigoplus_{v'_0} \mathrm{Br}({L'_0}_{v'_0}) \to \big(\bigoplus_{v'} \mathrm{Br}(L'_{v'}) \big) \oplus \mathrm{Br}(\mathbb{Q}_v).
 \]  
 The sums are taken over the places lying above \( v \) in the respective fields. The claim follows because we can find a Galois extension of CM type \( L' \supset L \) such that all places in \( L'_0 \) above \( v \) split in \( L' \). Specifically, we choose a totally real Galois extension \( K \) of \( L_0 \) that splits \( L \) at every place dividing \( v \), and set \( L'_0 = KL_0 \), \( L' = KL \).  \end{proof} 

  In the following section, we will encounter an inverse system of gerbes \(\mathscr  M_L \) with properties similar to those of \( \mathscr P_L \). Firstly, there exist embeddings  
  \[
  \phi_\mu : \mathscr M_L \hookrightarrow \mathscr G_L,
  \]  
  and secondly, there exist local homomorphisms \( \zeta_v \) for which the analogues of \eqref{3.h}, \eqref{3.i}, and \eqref{3.j} hold. Furthermore, the following will hold:  \begin{align}
\phi_\mu(M_L) = \psi_\mu(P_L).
  \end{align}
  We therefore identify \( M_L \) and \( P_L \).  
  
  \begin{lem}\label{3.14} Assume that for every place \( v \) of \( \mathbb{Q} \) and every \( L \), there exists an isomorphism  
  \[
  \eta_L(v) : \mathscr  M_L(v) \to \mathscr P_L(v),
  \]  
  which is the identity on the kernel \( M_L \). Then, there exists an isomorphism \( \eta_L : \mathscr M_L \to \mathscr P_L \), such that \( \eta_L(v) \) is equivalent to the localization of \( \eta_L \) for every \( v \). The family \( \{\eta_L\} \) is uniquely determined up to equivalence, and the diagrams  
  \begin{equation}\label{3.l}
  	\xymatrix{ 	\mathscr M_{L'} \ar[r] \ar[d]_{\eta_{L'}} &  \mathscr  M_L \ar[d]^{\eta_L} \\
 	\mathscr 	P_{L'} \ar[r] & \mathscr  P_L  }
  \end{equation} 
  commute up to equivalence.   
\end{lem}
  \marginpar{\textit{Proof of Lem.~\ref{3.14}.}}
  Let \( \mathscr M_L \) be defined by the cocycle \( \{m_{\rho,\sigma} \mid \rho , \sigma \in \text{Gal}(\overline{\mathbb{Q}}/\mathbb{Q})\} \), and \( \mathscr P_L \) by \( \{p_{\rho,\sigma}\} \). Then, \( \{m_{\rho,\sigma}\} \) and \( \{p_{\rho,\sigma}\} \) are locally equivalent. By the Hasse principle from Lemma \ref{3.5}, we may assume \( m_{\rho,\sigma} \equiv p_{\rho,\sigma} \).  
  
  Any two locally equivalent isomorphisms \( \eta_L \) and \( \eta_L' : \mathscr M_L \to \mathscr P_L \) differ by a class in \( H^1(\mathbb{Q}, P_L) \). Consequently, Lemma \ref{3.9} implies that two locally equivalent isomorphisms are also globally equivalent. Thus, the uniqueness of \( \eta_L \) is established.

  The  \marginpar{151} commutativity of diagram \eqref{3.l} is proved similarly. Let \( p_{\rho}' \in \mathscr P_{L'} \) be representatives of elements in the group \( \text{Gal}(\overline{\mathbb{Q}}/\mathbb{Q}) \) with  
  \[
  p_{\rho}' p_\sigma' = p_{\rho,\sigma}' p_{\rho\sigma}',
  \]  
  and let \( b_\rho p_\rho \), \( b_\rho \in P_L(\overline{\mathbb{Q}}) \), be the image of \( p_{\rho}' \). Similarly, let \( m_{\rho}' \) be representatives with  
  \[
  m_{\rho}' m_\sigma' = m_{\rho,\sigma}' m_{\rho\sigma}',
  \]  
  where \( m_{\rho,\sigma}' = p_{\rho,\sigma}' \), and let \( a_\rho m_\rho \) be the image of \( m_{\rho}' \).  
  Finally, let  
  \[
  \eta_{L'}(m_{\rho}') = c_{\rho}' p_{\rho}', \quad c_{\rho}' \in \mathscr  P_{L'},
  \]  
  and let \( c_\rho \) be the image of \( c_{\rho}' \) in \( \mathscr P_L \). The isomorphism \( \eta_L \) could have been defined through the commutativity of \eqref{3.l}, namely:  
  \begin{align}\label{3.m}
\eta_L(m_\rho) = a_{\rho}^{-1} b_\rho c_\rho p_\rho.
  \end{align}
  Thus, the commutativity follows from uniqueness.  
   
   To define \( \eta_L \), we use Lemma \ref{3.10} and choose \( L' \) such that the transition homomorphisms \( H^1(\QQ_v, P_{L'}) \to H^1(\QQ_v, P_L) \) are zero for all \( v \). A possible isomorphism \( \xi : \mathscr M_{L'} \to \mathscr P_{L'} \) is defined by \( m'_{\rho} \to p' _{\rho} \), since \( m'_{\rho, \sigma} = p'_{\rho, \sigma} \). The localization \( \xi(v) \) differs from \( \eta_{L'}(v) \) by an element \( \alpha_v \) of \( H^1(\QQ_v, P_{L'}) \). We define \( \eta_L \) by equation \eqref{3.m}, via $c'_\rho$ and the resulting $c_{\rho}$ not using \( \eta_{L'} \) (which has not been defined yet), but using \( \xi \) instead. Then \( \eta_L(v) \) and the localization of \( \eta_L \) differ by the image of \( \alpha_v \) in \( H^1(\QQ_v, P_L) \), which by assumption is zero. 
     \marginpar{\textit{QED Lem.~\ref{3.14}.}} \\ 
      
   In the next section, we will need a statement that guarantees we can identify the gerbes \( \mathscr M_L \) and \( \mathscr P_L \) as sub-gerbes of \( \mathscr G_L \).
   
   \begin{lem}\label{3.15}
With the notation from the previous lemma, assume further that the local homomorphisms at the infinite place, \( \psi_\mu \circ \zeta_{v_1} \) and \( \phi_\mu \circ \zeta_{v_1}' \), from $\mathscr W$ to $\mathscr G_L$, are equivalent to each other for all \( L \). Then the following diagram is commutative up to equivalence:
\begin{equation}
	\xymatrix{  \mathscr M_L \ar[dd]^{\eta_L} \ar[rrrd]^{\phi_{\mu}} \\ &&& \mathscr G_L \\ \mathscr P_L \ar[rrru]^{\psi_\mu} }
\end{equation}
   \end{lem} 
   \begin{proof}
   	  By assumption, for each place the local homomorphisms \( \zeta_v \) and \( \zeta_v' \) form a compatible family  for varying \( L \). Therefore, the same holds for the local homomorphisms \( \psi_\mu \circ \zeta_v \) and \( \phi_\mu \circ \zeta_v' \) for \(\mathscr G_L \). For a non-archimedean place, these local homomorphisms are, by Lemma \ref{3.13}, equivalent, and this is true for the infinite place by assumption. After identifying \( \mathscr M_L \) with \( \mathscr P_L \), the embeddings \( \phi_\mu \) and \( \psi_\mu \) differ by a class from \( H^1(\QQ, {}^LS) \), which is thus locally trivial. The statement follows from the Hasse principle for \( H^1(\QQ, {}^LS) \) ([De2], C.1). 
     \end{proof}

   We   \marginpar{152} note that in \S 2, we explicitly constructed the gerbe \( \mathscr T_{\nu_1, \nu_2} \) and the homomorphism \( \tau_\mu \), without providing a unique characterization. For the gerbe \( \mathscr P \) and the homomorphism \( \tau_\mu \), associated with a cocharacter of a torus that satisfies the Serre condition, such a characterization follows directly from Lemmas \ref{3.14} and \ref{3.15}. We leave the formulation to the reader (one can use the universal property of the Serre group).

     \section{The motivic Galois group}
   
   In this section, we introduce the motivic Galois group by defining and describing the corresponding Tannakian category. First, we recall the relationship between gerbes and Tannakian categories in a way we find understandable and show how the terminology of §2 compares with that in [DM] and [Sa].
   
   The concept of a Galois gerbe over a field \( k \) of characteristic zero, with the associated kernel \( G \), was already defined at the beginning of §2. In that section, the notion of a homomorphism between Galois gerbes was also explained. Additionally, we encountered examples, including the neutral Galois gerbe \( \mathscr G_G \) associated with a linear algebraic group \( G \) defined over \( k \). If \( V \) is a finite-dimensional vector space over the algebraic closure \( \bar{k} \) of \( k \), we can introduce a Galois gerbe \( \mathscr G_V \). This gerbe consists of all isomorphisms \( g: V \to V \) that are additive and \( \sigma \)-linear with respect to some \( \sigma = \sigma(g) \in \text{Gal}(\bar{k}/k) \). The homomorphism \( g \mapsto \sigma(g) \) defines an exact sequence:
   \[
   1 \to \text{GL}(V) \to \mathscr G_V \to \text{Gal}(\bar{k}/k) \to 1.
   \]
   The section germ is determined by the choice of a \( K \)-basis of \( V \) for a finite extension \( K \) of \( k \). A representation of a Galois gerbe \( \mathscr G \) is a homomorphism \( \mathscr G \to \mathscr G_V \). The category of all representations of a given Galois gerbe \( \mathscr G \) is equipped with an obvious \( k \)-linear structure, a tensor product, and a unit object \( \mathbf {1} \), which is the representation defined by the natural projection onto \( \text{Gal}(\bar{k}/k) \).

   Although these concepts of a gerbe and its associated Tannakian category are not the usual ones, they are closely related to those in [DM] and [Sa] for the cases of interest to us. From [DM], §3.10, it follows that any gerbe in the sense of Giraud that corresponds to a Tannakian category over \( k \) is an inverse limit of algebraic gerbes. Let \( \underline {\mathscr G} \) be an algebraic gerbe in the sense of Giraud (we underline the symbol for clarity, distinguishing it from Galois gerbes), and let \(  Q  \in \text{ob}~
   \underline {\mathscr G}_{\text{Spec}\,\bar{k}} \) be an object.  We will associate a Galois gerbe \( \mathscr G \) with the pair \( (\underline  {\mathscr G},  Q ) \) and show that this association defines an equivalence of categories. A homomorphism \( (\underline {\mathscr G},  Q ) \to (\underline {\mathscr H}, R) \) in the first category consists of a pair: a cartesian functor \( \Phi: \underline {\mathscr G} \to \underline {\mathscr H} \) and an isomorphism \[ \tau: \Phi( Q ) \cong  R \] in \( \underline {\mathscr H}_{\text{Spec}\,\bar{k}} \). \marginpar{153}In the correspondence between algebraic Tannakian categories and algebraic gerbes, when replacing Giraud gerbes with Galois gerbes, the Tannakian categories must be replaced by algebraic Tannakian categories, \emph{together} with a fiber functor over \( \bar{k} \).
   
   According to the last proposition in the appendix to [DM], there exists a finite extension \( k' \) of \( k \) and an object \(  Q _{k'} \) in \( \underline {\mathscr G}_{\text{Spec}\,\bar{k}'} \), whose inverse image in \( \underline {\mathscr G}_{\text{Spec}\,\bar{k}} \) is equipped with an isomorphism with \(  Q  \). In other words (see the appendix to [DM]), \(  Q  \) is equipped with descent data over \( k' \). Let \( ^{\sigma}  Q  \in \text{ob}\, \underline {\mathscr G}_{\text{Spec}\,\bar{k}} \) be the object obtained by pulling back via \( \sigma: \bar{k} \to \bar{k} \).  
   The descent data is nothing other than a system of isomorphisms \( \phi_\sigma:  Q  \to {}^\sigma  Q  \) for \( \sigma \in \text{Gal}(\bar{k}/k) \) that satisfies the usual cocycle condition. We define  
   \[
   \mathscr G = \{\phi:  Q  \to {}^\sigma  Q  \mid \sigma \in \text{Gal}(\bar{k}/k)\}.
   \]
   
   Because there is an obvious isomorphism  
   \[
   \sigma: \text{Isom}( Q , {}^{\rho} Q ) \to \text{Isom}( {}^\sigma  Q , {}^{\sigma \rho} Q),
   \]
   as well as a pairing  
   \[
  \text{Isom}( Q , {}^{\rho} Q )  \times \text{Isom}( {}^{\rho} Q , {}^{\rho \sigma} Q )  \to \text{Isom}( Q , {}^{\rho\sigma} Q ),
   \]
   the set \( \mathscr G \) forms a group, specifically an extension of \( \text{Gal}(\bar{k}/k) \) by \( G(\bar{k}) \), where \( G \) is by definition \(\text{Aut}(  Q ) \). 
   By assumption, \( G \) is an algebraic group over \( \bar{k} \). The descent data defines a splitting of the extension over \( \text{Gal}(\bar{k}/k') \). In this way, we obtain a Galois gerbe \(\mathscr  G \).
   
   A quasi-inverse is easy to define. Let \( \mathscr G \) be a Galois gerbe. Then the Tannakian category \( \mathrm{Rep} \, \mathscr G \) is equipped with a fiber functor over \( \bar{k} \), which assigns to a representation its underlying vector space over \( \bar{k} \). Thus, we obtain the Giraud gerbe \( \underline { \mathscr{G}} \) associated with the Tannakian category \( \text{Rep} \, \mathscr G \), along with an element \(  Q  \) in \( \text{ob}\, \underline {\mathscr G}_{\text{Spec}\,\bar{k}} \), the fiber functor. Since for any given \( \underline {\mathscr G }\), two objects in \( \text{ob}\, \underline {\mathscr G}_{\text{Spec}\,\bar{k}} \) are isomorphic, \( \mathscr {G} \) is uniquely determined up to isomorphism. The isomorphism between two possible \( \mathscr {G} \) is uniquely determined up to conjugation by an element of \( G(\bar{k}) \). Consequently, we do not distinguish too much between \(\underline {\mathscr  G} \) and \( \mathscr {G} \).
   
   The category of motives over an arbitrary field can be constructed  assuming the standard conjectures, and we will sketch the construction below. This category is a rigid abelian tensor category over \( \mathbb{Q} \) with \( \text{End}(\mathbf 1) = \mathbb{Q} \). If the field has positive characteristic, the category does not possess a fiber functor over \( \mathbb{Q} \), although étale cohomology provides a fiber functor over \( \mathbb{Q}_\ell \). This introduces unexpected difficulties because the proof of Theorem III, 3.2.2 in [Sa] is incomplete, as noted in [DM]. As a result, the category of motives is not immediately a Tannakian category in the sense of [DM]. According to the note at the end of [DM], this gap may not be serious. But it turns out that the problem is irrelevant for us here, since it is solved by the (very strong) standard conjectures anyway.
    
    In order not to interrupt the later discussion, we will carry out the necessary considerations at this point.

   Let   \marginpar{154} \( C \) be a rigid abelian tensor category over \( k \) with \( \text{End}(\mathbf 1) = k \), and let \( k' \) be an extension of \( k \) without further specific properties. As in [DM], p.~156, we introduce the subcategory of essentially constant ind-objects $C^e$ of $\mathrm{Ind} (C)$ equivalent to $C$ (see [Sa], II.2.3.4).  The functor \( i: C^e \to \text{Ind}(C)_{(k')} \) is formally defined by the same formulas as in [DM]. We have 
   \[
   \text{Hom}(i(X), i(Y)) \cong k '   \otimes_k \text{Hom}(X, Y).
   \]
   Let \( C_{k'} \) be the abelian tensor subcategory of \( \text{Ind}(C)_{(k')} \) generated by \( i(X) \) for \( X \) in \( C^e \) ([DM], 1.14). We consider \( i \) as a functor from \(C  \) to \( C_{k'} \). If \( \omega \) is a \( k' \)-valued fiber functor on \( C\), we can define \( \omega' \) as in [DM] and get the same commutative diagram.\footnote{The equation numbers (4.a)-(4.c) are missing in the original paper.  } \setcounter{equation}{3}
    \begin{equation}\label{4.d}
    	\xymatrix{ C  \ar[r] \ar[rd]_-{\omega} & C_{k'} \ar[d]^-{\omega'} \\ & \mathrm{Vec}_{k'}}
    \end{equation}
   We assume that there exists an \( \omega \) for which \( \omega' \) is exact and faithful, and we show that \(C \) is then a Tannakian category.
   
   There is little left to prove. From \eqref{4.d}, it follows that:  
   \[
   \text{Aut}^\otimes \omega  = \text{Aut}^\otimes \omega'.
   \]
   Then, by [DM], Th. 2.11, we conclude that \( \text{FIB}(C) \) ([DM], §3.6) is an affine Giraud gerbe \(\underline {\mathscr G} \). It remains to show that   
   \[
 C  \overset{\lambda}{\longrightarrow} \text{Rep}_k \underline {\mathscr G}
   \]
   is an equivalence of categories. If \( C_0 \subset  C \) is a finitely generated subcategory, there exists a functor \( \text{FIB}(C) \to \text{FIB}(C_0) \). Since any \( \Phi \in \text{Rep}_k \underline {
   	\mathscr G} \) factors through \( \text{FIB}(C_0) \) with \( C_0 \) finitely generated, we may, without loss of generality, assume that \( C \) is finitely generated ([DM], §1.14), as \( C \) is a filtered union of its finitely generated subcategories.  In this case, due to the proposition at the end of [DM], \( k' \) can be replaced by a finite Galois extension without losing the exactness and faithfulness of \( \omega' \). So let \( k' \) be finite and Galois.
   
   We can then identify \( C_{k'} \) with \( C_{(k')} \), identifying \( i(x) \) in the above sense with the \( i(x) \) of [DM], p.~156. Via the morphism \( \text{id} \to ij \), \( X \) in \( C_{k'} \) becomes a subobject of \( ij(X) \). When passing to dual objects, every \( X \) in \(C_{k'} \) becomes a quotient of an object in \( i(C) \).

   Let\marginpar{155} \( \underline {\mathscr G}'  = \text{FIB}(C_{k'}) \), \( i^* \) the induced functor \( i^*: \text{FIB}(C_{k'}) \to \text{FIB}(C ) \), and \( i^{**}: \text{Rep}_k \underline {\mathscr G} \to \text{Rep}_{k'} \underline { \mathscr  G}'  \). The diagram 
   \[ \xymatrix{  C \ar[r]^-{\lambda} \ar[d]^-i &  \text{Rep}_k \underline {\mathscr G}  \ar[d]^-{i^{**}} \\ C_{k' } \ar[r]^-{\lambda'} &   \text{Rep}_{k'} \underline {\mathscr G}' }
   \]
   commutes.

    The functor $i$ is exact and faithful, and
    \[
    \dim_k \text{Hom}(X, Y) = \dim_{k'} \text{Hom}(i(X), i(Y)).
    \]
    If $\Phi_1$ and $\Phi_2$ lie in $\text{Rep}_k \underline {\mathscr G}$, then the map
    \begin{equation}  \label{4.e}
    	\text{Hom}(\Phi_1, \Phi_2) \otimes_k k' \to \text{Hom}(\Phi_1(\omega), \Phi_2(\omega))
    \end{equation}
    is an embedding. We expose our naivety below by giving the reason for this. It follows that
    \begin{equation}  \label{4.f}
    	\dim_k \text{Hom}(\Phi_1, \Phi_2) \leq \dim_{k'} \text{Hom}(i^{**}\Phi_1, i^{**}\Phi_2).
    \end{equation}
    Since $\lambda'$ is an equivalence of categories, $\lambda$ is consequently fully faithful, and the inequality \eqref{4.f} becomes an equality.
    
    Because $i^{**}\Phi_1$ is a subobject of some $\lambda' i(X) = i^{**}\lambda(X)$, $\Phi_1$ is a subobject of $\lambda(X)$. The quotient $\Phi_2$ is a subobject of $\lambda(Y)$, and $\Phi_1$ is the kernel of a morphism $\lambda(X) \to \lambda(Y)$. Since $\lambda$ is fully faithful, $\Phi_1$ lies  in the essential image of $\lambda$.
    
    By definition ([DM], p.~153), for every automorphism $\sigma$ of $k'$ over $k$, there exist linear isomorphisms
    \[
    \alpha_i : \Phi_i(\omega) \To \Phi_i(\omega) \otimes_\sigma k', \quad i = 1, 2,
    \]
    which are compatible with $\text{Hom}(\Phi_1, \Phi_2)$. Let $\phi \in \text{Hom}(\Phi_1, \Phi_2)$. Then
    \begin{equation} \label{4.g} \begin{array}{l}
\alpha_2 \phi = (\phi \otimes 1) \alpha_1, \\ 
\alpha_2 a = (1 \otimes a) \alpha_1 = (\sigma^{-1}(a) \otimes 1) \alpha_1, \quad a \in k.
    	\end{array}
    \end{equation}
    If \eqref{4.e} were not an embedding, there would exist a minimal $r$ and
    \[
    \phi_1, \dots, \phi_r \in \text{Hom}(\Phi_1, \Phi_2), \quad a_1, \dots, a_r \in k',
    \]
    such that $\{\phi_1, \dots, \phi_r\}$ is linearly independent over $k$, while
    \[
    \sum_{i=1}^r a_i \phi_i = 0
    \]
    in\marginpar{156} $\text{Hom}(\Phi_1(\omega), \Phi_2(\omega))$. Necessarily, $r \geq 2$. From \eqref{4.g}, we deduce for every $\sigma$,
    \[
    \sum_{i=1}^r \sigma^{-1}(a_i) \phi_i = 0,
    \]
    from which we immediately obtain a contradiction.

   We now proceed to the construction of the category of motives. Let \( V_{\mathbb F} \) be the category of smooth projective varieties (not necessarily connected) over \(\mathbb F \). If \( X \in \text{ob} V_{\mathbb F} \), let \( C^*(X) \) denote the graded \( \mathbb{Q} \)-vector space of algebraic cycles (graded by codimension) modulo \emph{numerical equivalence}. We construct a new category \( CV_{\mathbb F} \) from \( V_{\mathbb F} \) with the same objects, which we denote by \( h(X) \) for clarity, and with morphisms given by algebraic correspondences of degree zero from \( X \) to \( Y \), ([Sa], A.0.3.3)
   \[
   \text{Hom}_{CV_{\mathbb F}}\big(h(X), h(Y)\big) = CV^0(X, Y).
   \]
   If \( X \) is irreducible of dimension \( n \), then \( CV^0(X, Y) = C^n(X \times Y) \). The category \( CV_{\mathbb F} \) is a \( \mathbb{Q} \)-linear category equipped with a direct sum \( \oplus \) (corresponding to disjoint union) and a tensor product \( \otimes \) (corresponding to the Cartesian product), with the unit object \( \mathbf{1} = h(\text{Spec}\, \mathbb F) \). These operations are endowed with natural associativity and commutativity laws (see [DM], §1.1). The assignment \( X \mapsto h(X) \) defines a \emph{contravariant} functor (the graph of a morphism):
   \[
   h:V_{\mathbb F}^{\text{opp}} \to CV_{\mathbb F}.
   \]
   The ``false'' category of \emph{effective} motives \( \dot{M}^+_{\mathbb F} \) is the pseudo-abelian (Karoubi) envelope of \( CV_{\mathbb F} \), obtained by formally adding images and kernels of projectors (see [DM], p.~201). In \( \dot{M}^+_{\mathbb F} \), the motive \( h(\mathbb{P}^1) \) decomposes as:
   \[
   h(\mathbb{P}^1) = \mathbf{1} \oplus L.
   \]
   Here, \( \text{Hom}(M, N) \isom \text{Hom}(M \otimes L,  N \otimes L) \) for two effective motives \( M \) and \( N \) (see [Sa], VI, 4.1.2.5). The ``false'' category of motives \( \dot{M}_{\mathbb F} \) is obtained from \( \dot{M}^+_{\mathbb F} \) by formally inverting the object \(L\). Let \(T = L^{-1} \) be the ``Tate motive'', and define \( M(n) = M \otimes T^n \). Then:
   \[
   \text{Hom}_{\dot{M}_{\mathbb F}}\big(M(m), N(n)\big) = \text{Hom}(M \otimes L^{N-m}, N \otimes L^{N-n}), \quad N \geq m, n.
   \]
   We aim to turn this constructed tensor category into a (\( \mathbb{Z} \))-graded polarized Tannakian category. To achieve this, we must assume the standard conjectures about algebraic cycles. Fix a prime \( l \neq p \) and work with the graded \( l \)-adic cohomology:
   $
   H^*_l(X) = \bigoplus_i H^i(X, \mathbb{Q}_l).$ 
   This cohomology is equipped with the cycle map 
   \[
   \gamma^i: C^i(X) \to H^{2i}_l(X)(i),
   \] 
   and the trace map 
   \[
   \text{Tr}_X: H^{2n}_l(X)(n) \to \mathbb{Q}_l, \quad n = \dim X.
   \] 
   The\marginpar{157} Künneth formula, Poincaré duality, and the hard Lefschetz theorem hold. The standard conjectures state:
   
  1. The cycle maps \( \gamma^i \) are injective.

  2. The map 
   	\[
   	l^{n-2p}: C^p(X) \to C^{n-p}(X)
   	\]
   	 defined by an ample divisor    	is bijective for \( 0 \leq 2p \leq n = \dim X \).

   3.  On the primitive algebraic cycles 
   	\[
   	C^p_{\text{pr}}(X) \underset{\text{Def}}{=} \text{Ker } l^{n-2p+1} = C^p(X) \cap H^{2p}_{\text{pr}}(X),
   	\] 
   	the symmetric bilinear form:
   	\[
   	C^p_{\text{pr}}(X) \times C^p_{\text{pr}}(X) \to \mathbb{Q}, \quad (x, y) \mapsto (-1)^p \cdot \text{Tr}_X(l^{n-2p}xy)
   	\] 
   	is positive definite for \( 0 \leq 2p \leq n \). By decomposing into primitive components and appropriate sign conventions (see [Sa], VI. A.2.2.2.3), one obtains a positive-definite symmetric bilinear form on all of \( C^p(X) \): 
   	\[
   	(x, y) \mapsto \text{Tr}_X(x \cdot *y).
   	\] 
   	
   (These conjectures are not independent of each other; see [K].)

   From these conjectures, it follows that there exist elements
   \[
   \pi^i \in C^{2n}(X \times X), \quad n = \dim X,
   \]
   which induce the decomposition into Künneth components of the diagonal in \( l \)-adic cohomology:
   \[
   \Delta_X = \sum \pi^i \in \bigoplus H^{2n-i}_l(X) \otimes H^i_l(X).
   \]
   (This last fact has actually been proven over \(\mathbb F \); see [K-M].)
   
   Thus, we obtain a decomposition of \( \text{id}_{h(X)} \) into pairwise orthogonal idempotents 
   \[
   \text{id}_{h(X)} = \pi^0 \oplus \cdots \oplus \pi^n,
   \]
   and correspondingly, a decomposition in \( \dot{M}^+_{\mathbb F} \):
   \[
   h(X) = h^0(X) \oplus \cdots \oplus h^n(X).
   \]
   This grading of objects in \( CV_{\mathbb F} \) extends in an obvious way to a grading of objects in \( \dot{M}_{\mathbb F} \), ensuring that the Künneth formula holds.
    
   To\marginpar{158} obtain the category of \emph{genuine} motives \( M_{\mathbb F} \) from the category of \emph{false} motives \( \dot{M}_{\mathbb F} \), we modify the commutativity law \( \psi \) for \emph{homogeneous} objects by setting 
   \[
   \psi = (-1)^{r \cdot s} \dot{\psi}: M \otimes N \isom  N \otimes M, \quad \deg M = r, \deg N = s,
   \]
   and extending it linearly. We aim to show that this construction yields a semisimple Tannakian category over \( \mathbb{Q} \). To this end, we first verify that the standard conjectures imply that the \( l \)-adic cohomology defines a faithful, exact functor:
 $
   H_l: ({M_{\mathbb F}})_{\mathbb{Q}_l} \to \text{Vec}_{\mathbb{Q}_l}.
$
   This is an immediate consequence of the injectivity of the canonical map
   \begin{equation}\label{4.h}
 \text{Hom}(M, N) \otimes \mathbb{Q}_l \to \text{Hom}\big(H^*_l(M), H^*_l(N)\big).
   \end{equation} 
   Let \( X \) be a variety. The bilinear form \( \text{Tr}_X(x \cdot y) \) on \( C^*(X) \) takes values in \( \mathbb{Q} \) and, due to the standard conjectures, is non-degenerate. It is the restriction of the corresponding bilinear form with values in \( \mathbb{Q}_l \) on \( l \)-adic cohomology. Consequently, the kernel of the natural map
   \[
   C^*(X) \otimes \mathbb{Q}_l \to H^*_l(X)
   \]
   is contained in the kernel of the bilinear form on \( C^*(X) \otimes \mathbb{Q}_l \), which arises from the \( \mathbb{Q}_l \)-linear extension of the bilinear form on \( C^*(X) \). It is therefore zero. The above claim is thus proven if \( M \) and \( N \) are motives associated with varieties. The claim then follows easily for effective motives and, finally, for all motives.   Next, we note that the proof of Proposition 6.5 in [DM] (see [Sa], VI.4.2.2) shows that every indecomposable motive is simple. (In \textit{loc.~cit.}, it is not used that the faithful functor employed there takes values in \( \text{Vec}_\mathbb{Q} \); a faithful \( \mathbb{Q} \)-linear functor into \( \text{Vec}_{k'} \) for any extension \( k' \) of \( \mathbb{Q} \) serves the same purpose.) We can therefore (with the same remark) apply Lemma 6.6 in [DM] and conclude that \( M_{\mathbb F} \) is a semisimple \( \mathbb{Q} \)-linear abelian tensor category. Clearly, \( \text{End}(\mathbf{1}) = \mathbb{Q} \), so it remains only to prove the rigidity of the category. This follows as in [DM], p.~204, bottom. One obtains
   \[
   h(X)^\vee = h(X)(n),
   \]
   if \( X \) is an irreducible variety of dimension \( n \). That \( M_{\mathbb F} \) is a Tannakian category now follows from the injectivity of \eqref{4.h}, which yields the faithful exact functor \( \omega' \) in \eqref{4.d}.

   The category of motives has additional structures. More precisely, \( M_{\mathbb F} \), with its \( \mathbb{Z} \)-grading \( w \) and the Tate motive \(T \), is a Tate triple over \( \mathbb{Q} \) (see [DM], §5). There exists a polarization \( \pi_M \) of this Tate triple, characterized by the property that for a variety \( X \), the polarization set \( \pi_M(h^r(X)) \) contains the form \( \text{Tr}_X(x \cdot *y) \).
    
   If we take the field \( \overline{\mathbb{Q}} \) of algebraic numbers as the base field instead of \( \mathbb F \), a category of motives \( M_{\overline{\mathbb{Q}}} \) was constructed in [DM]. In this construction, the algebraic correspondences (modulo numerical equivalence) are replaced by cohomological correspondences induced by \emph{absolute Hodge cycles}. An advantage\marginpar{159} of this approach is that no unproven conjectures need to be used. A disadvantage is that the reduction modulo \( p \) of an absolute Hodge cycle is not defined. In the following, we are only interested in the sub-Tannakian category \(CM_{\overline{\mathbb{Q}}} \) of \(M _{\overline{\mathbb{Q}}} \) generated by abelian varieties of CM type and \( \mathbb{P}^1 \) (see [DM], 1.14). For the constructions of \(  {CM}_{\overline{\mathbb{Q}}} \) and \( M_{\mathbb F} \) to be compatible, we must in the following \emph{assume the Hodge conjecture for abelian varieties of CM type}. The standard conjectures about algebraic cycles for these varieties are then, incidentally, a consequence of Hodge theory. An abelian variety of CM type over \( \overline{\mathbb{Q}} \) has good reduction modulo \( p \), as does the projective line. More precisely, such an abelian variety is already defined over a finite extension \( K \) of \( \mathbb{Q} \). After possibly enlarging \( K \), it has good reduction at the place over \( p \) distinguished by the fixed embedding of \( \overline{\mathbb{Q}} \) into \( \overline{\mathbb{Q}}_p \) (see [Se-Ta], Th.~6).
    
   Here, the specialization map on the Chow rings is a homomorphism ([Fu], 20.3.1.) 
   \[ A(X) \to A(X_{\mathbb F}). \]
   From the standard conjectures, it follows that this map factors through the \emph{numerical} equivalence classes and is compatible with the cycle maps in \( l \)-adic cohomology, i.e., the following diagram commutes: \[
   \xymatrix{ 
   	C^*(X) \ar[r] \ar[d] &  C^*(X_{\mathbb F}) \ar[d] \\
   	H^{2*}_l(X)\ar[r]^{\sim} &  H^{2*}_l(X_{\mathbb F})(*). }
   \]
   
   Indeed, if \( x \in A(X) \) is numerically equivalent to zero, then \( \text{Tr}_X(x \cdot *x) = 0 \), and thus the image of \( x \) in \( A(X_{\mathbb F}) \) is also numerically equivalent to zero. It is easy to see that this yields a tensor functor describing the reduction modulo \( p \) of these motives
   \[ \text{red}: CM_{\overline{\mathbb{Q}}} \to M_{\mathbb F}. \]
   
   More precisely, let \( L \subset \overline{\mathbb{Q}} \) be a CM field, and let \( ^L CM_{\overline{\mathbb{Q}}} \) be the sub-Tannakian category of \(  {CM}_{\overline{\mathbb{Q}}} \) generated by abelian varieties of CM type by \( L \) and by \( \mathbb{P}^1 \). For \( L \subset L' \), we obtain
   \[ ^L CM_{\overline{\mathbb{Q}}} \to {}^{L'} CM _{\overline{\mathbb{Q}}}'. \]
   Here, we regard an abelian variety \( X \) with complex multiplication by \( L \) as a direct summand of the abelian variety (up to isogeny) with complex multiplication by \( L' \), \( X \otimes_L L' \) (with the obvious definition). The rational cohomology defines a natural fiber functor of \( ^L CM_{\overline{\mathbb{Q}}} \) over \( \mathbb{Q} \):
   \[ ^L CM_{\overline{\mathbb{Q}}} \to \text{Vec}_{\mathbb{Q}}, \quad h(X) \mapsto \bigoplus H^i(X \times_{\overline{\mathbb{Q}}} \mathbb{C}, \mathbb{Q}). \]
   Deligne\marginpar{160} ([D2], (A)) has shown that the corresponding neutralized gerbe over \( \mathbb{Q} \) is the Giraud gerbe \( \underline {\mathscr G} _L =  \underline {\mathscr G} _{^LS} \) associated with the Serre group (where $ \underline {\mathscr G} _{^LS}$ corresponds to the Galois gerbe \(  {\mathscr G} _L =   {\mathscr G} _{^LS}  \)), and that the homomorphism corresponding to the above functors is the natural projection:
   \[ \mathscr G_{L'} \to \mathscr G_L. \] 
   Let \( ^L M_{\mathbb F} \) be the sub-Tannakian category of \( M_{\mathbb F} \) generated by the image of \( ^LCM_{\overline{\mathbb{Q}}} \) under the reduction functor ([DM], 1.14). Since \( ^LCM_{\overline{\mathbb{Q}}} \) is an algebraic Tannakian category, \( ^L M_{\mathbb F} \) is also an algebraic Tannakian category ([DM], 2.20). Its gerbe will be denoted by \( \mathscr M_L \). If \( L \subset L' \), then the homomorphism of gerbes \( \mathscr  M_{L'} \to \mathscr M_L \) is surjective because the functor \( ^L M_{\mathbb F} \to {}^{L'} M_{\mathbb F} \) is fully faithful, and every subobject of an image is the image of a subobject ([DM], 2.21, see Note 1 at the end of the section). We obtain commutative diagrams:
   \[\xymatrix{
 \underline {\mathscr M}_{L'} \ar[r]^{\phi_{L'}} \ar[d] & \underline {\mathscr G}  _{L'} \ar[d] \\
    	\underline {\mathscr  M}_L 
    	\ar[r]^{\phi_L} &  \underline {\mathscr  G}_L.} 
   \]
   To verify the properties of \( \mathscr M_L \) that we desire, we must \emph{assume the Tate conjecture} for the \( l \)-adic cohomology (for the fixed \( l \)) of algebraic varieties over finite fields [Sa], A.4. Let \( \text{Tate}_{\mathbb F_{p^m}} \) be the Tannakian category over \( \mathbb{Q}_l \) of continuous semi-simple \( l \)-adic representations of the Galois group \( T_m = \text{Gal}(\FF/\FF_{p^m}) \), equipped with a natural fiber functor:
   \[
   \omega_m: \text{Tate}_{\FF_{p^m}} \to \text{Vec}_{\mathbb{Q}_l}.
   \] 
   We form the direct limit (restriction):
   \[
   \text{Tate}_\FF = \varinjlim \text{Tate}_{\FF_{p^m}},
   \]
   and thus obtain a neutralized, semi-simple Tannakian category over \( \mathbb{Q}_l \). The Frobenius element in \( \text{Gal}(\FF/\FF_{p^m}) \) defines an isomorphism \( T_m \cong  \widehat{\mathbb{Z}} \). The pro-algebraic group \(\mathbb  T_m = \text{Aut}^\otimes(\omega_m) \) is the algebraic hull of \( T_m \) over \( \mathbb{Q}_l \), and we have:
   \[
   \text{Rep.cont.}_{\mathbb{Q}_l}(T_m) \isom  \text{Rep}_{\mathbb{Q}_l}(\mathbb T_m)
   \]
   (cf. [Sa], V.0.3.1). In particular, \( \mathbb T_m \) is abelian, and consequently ([DM], 2.23), since \( \text{Rep}_{\mathbb{Q}_l}(\mathbb T_m) \) is a semi-simple Tannakian category, $\mathbb  T_m $ is a projective limit of algebraic groups of multiplicative type. We want to determine the character group of \( \mathbb  T_m \). To do this, we extend the scalars from \( \mathbb{Q}_l \) to \( \overline{\mathbb{Q}}_l \), as explained at the beginning of this section.
   \[(  
   \text{Rep.cont.}_{\mathbb{Q}_l}T_m)_{\overline{\mathbb{Q}}_l} \isom  ( \text{Rep}_{\mathbb{Q}_l}\mathbb T_m)_{\overline{\mathbb{Q}}_l} = \text{Rep}_{\overline{\mathbb{Q}}_l}\mathbb T_m. 
   \]
   (For\marginpar{161} the last identification, see [DM], 3.12.) It is easy to see that the category on the left is simply the category of diagonalizable representations of \( T_m \) in a \( \overline{\mathbb{Q}}_l \)-vector space, where the eigenvalues of the Frobenius element are \( l \)-adic units. These form a semi-simple Tannakian category over \( \overline{\mathbb{Q}}_l \), in which the simple objects have rank 1 ([DM], 1.7.3). Furthermore, the simple objects are parametrized by the group \( \mathcal{O}^*_{\overline{\mathbb{Q}}_l} \) of \( l \)-adic units. We conclude ([Sa], VI.3.5.1) that 
   \[
   X^*(\mathbb T_m) \cong  \mathcal{O}^*_{\overline{\mathbb{Q}}_l}.
   \]
   The transition map \( X^*(\mathbb T_m) \to X^*(\mathbb T_{m'}) \) for \( m | m' \) sends \( \pi \) to \( \pi^{m'/m} \).

   	In the following lemma, we transition to Galois gerbes. Consequently, $\phi_L$ is only determined up to conjugation by an element of $^LS(\overline{k})$.
   	
   	\begin{lem}
   		Let $L$ be Galois. The canonical homomorphism
   		\[
   		\mathscr M_L \xrightarrow{\phi_L} \mathscr G_L
   		\]
   		is injective. The kernels $M_L$ of $\mathscr M_L$ and $P_L$ of $\mathscr P_L$ are identical as subtori of $^LS$.
   	\end{lem}
   	
   	\begin{proof}
   		From the definitions, it follows that every object in $^L M_{\mathbb F}$ is isomorphic to a subquotient of $\text{red}(X)$ for some $X \in \text{ob}\, {}^L CM_{\overline{\mathbb{Q}}}$. Therefore ([DM], 2.21), the first statement is clear (see Note 1). We have previously shown that the $l$-adic cohomology defines a faithful functor of Tannakian categories over $\mathbb{Q}_l$ 
   		\[
   		(^L M_{\mathbb F})_{\mathbb{Q}_l} \xrightarrow{H_l} \text{GradTate}_\FF.
   		\]
   		The Tate conjecture over $\FF$ is \emph{equivalent} to the statement that this functor is fully faithful ([Sa], A.4). (This conjecture is a consequence of the Tate conjecture over finite fields.) Because the Tannakian category $\text{GradTate}_\FF$ is semisimple, every subobject of an image is the image of a subobject. Consequently, the homomorphism on the character modules of the kernels is injective ([DM], 2.21, see Note 1) 
   		\[
   		X^*(M_L) \to X^*(\mathbb T) = \varinjlim X^*(\mathbb T_m).
   		\]
   		Since $M_L$ is an algebraic group, it follows from the definition of the direct limit on the right hand side that $X^*(M_L)$ is torsion-free, and thus $M_L$ is a torus.
   		
   		Consider the diagram with the obvious arrows:
   		$$ 
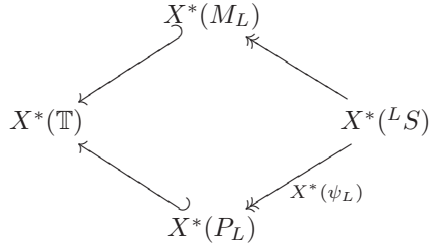
\xymatrix{ & X^*(M_L) \ar@{_(->}[dl] \\ 
   			X^*(\mathbb T)  &  & X^*(^LS) \ar@{->>}[lu] \ar@{->>}[ld]^{X^*(\psi_L)} \\
   			& X^*(P_L) \ar@{_(->}[ul] &
   		} $$
   		Here\marginpar{162}, the surjectivity of the upper diagonal arrow is already clear, while for the lower diagonal arrow, it follows from Lemma \ref{3.6}. Thus, it suffices to prove the commutativity of the diagram. Consider an element $\chi \in X^*(^LS)$ of the form
   		\[
   		\chi = \sum n_\sigma \cdot [\sigma], \quad n_\sigma + n_{\iota\sigma} = 1, \quad n_\sigma = 0 \text{ or } 1.
   		\]
   		Its image under $X^*(\psi_L)$ is (see calculation after \ref{3.7}) the Weil number $\pi_\chi = \pi \in L$ with
   		\begin{enumerate}
   			\item[(1)] $\pi \cdot \overline{\pi} = q$,
   			\item[(2)] $\left| \prod_{\sigma \in \text{Gal}(L_v/\mathbb{Q}_p)} \sigma\pi \right|_p = q^{-\sum_{\sigma \in \text{Gal}(L_v/\mathbb{Q}_p)} n_\sigma}$ for $v|p$.
   		\end{enumerate}
   		Here, $q = p^m$ for sufficiently large $m$. On the other hand, $\chi$ corresponds to an abelian variety of CM type $(L, \Phi)$ up to isogeny, where $\Phi = \{\sigma \mid n_\sigma = 1\}$ ([Mu2], p.~212). This abelian variety is defined over a finite extension of $\mathbb{Q}$ and has good reduction modulo $p$ over a finite field $\mathbb{F}_q$. Let $\pi$ be the Frobenius endomorphism of this abelian variety over $\mathbb{F}_q$, regarded as an element of $L$. This is a unit outside $\infty$ and $p$ (existence of $l$-adic cohomology) and satisfies (1) (Weil conjectures) and (2) (Shimura-Taniyama formula). Consequently, the images of $\chi$ under the two homomorphisms from $X^*(^LS)$ to $X^*(\mathbb T)$ are identical. The claim follows because such elements generate $X^*(^LS)$ ([D2], A.2).
   	\end{proof}

  	\begin{cor}
  		Let $L$ be Galois. $^L M_{\mathbb F}$ is the sub-Tannakian category of motives whose Frobenius endomorphisms, with respect to a sufficiently large finite field, have all eigenvalues in $L$.
  	\end{cor}

  		The Tannakian category $^L M_{\mathbb F}$ is equipped with the fiber functor over $\mathbb{Q}_l$ defined by the $l$-adic cohomology for each $l \neq p$:
  		\[
  		^L M_{\mathbb F} \xrightarrow{H_l} \text{Vec}_{\mathbb{Q}_l}.
  		\]
  		This defines a trivialization $\zeta_l'$ of $\mathscr M_L$ over $\mathbb{Q}_l$. The crystalline cohomology (tensored with $\mathbb{Q}$) assigns to each motive an $(F)$-isocrystal over the fraction field $k = K(\FF)$ of the Witt vectors over $\FF$ (see Note 2 at the end of this section). The gerbe associated with the Tannakian category of isocrystals over $k$ is the Dieudonné gerbe $\mathscr D$. More precisely, this Tannakian category belongs to the filtered inverse system of gerbes $\mathscr D^{L_n}$, where $L_n$ denotes the unramified extension of $\mathbb{Q}_p$ of degree $n$. This follows from [Sa], VI.3.3.2 (see also the note in the next section). The crystalline cohomology thus defines a homomorphism of gerbes 
  		\[
  		\zeta_p':\mathscr  D \to \mathscr M_L.
  		\] 
  		The\marginpar{163} locally constructed homomorphisms are obviously compatible with respect to changes in $L$. For $v = l \neq p$, the localized gerbes $\mathscr P_L(v)$ and $\mathscr M_L(v)$ are isomorphic as neutral gerbes with the same kernel. For $v = p$, the isomorphism classes of $\mathscr P_L(v)$ and $\mathscr M_L(v)$ as gerbes over $\mathbb{Q}_p$ with $M_L = P_L$ are uniquely determined by the restrictions of the local homomorphisms $\zeta_p$ and $\zeta_p'$ to the kernel. To show that these restrictions are identical, it suffices to compare the restrictions $\nu^{L_n}$ and $\nu^{\prime L_n}$ of
  	$\psi_L \circ \zeta_p: \mathscr D^{L_n} \to \mathscr G_L$ and $ \phi_L \circ \zeta_p':\mathscr  D^{L_n} \to \mathscr  G_L$ respectively. 
  		The first is given by the formula
  		\[
  		\left| \prod_{\sigma \in \text{Gal}(L_n/\mathbb{Q}_p)} \sigma \pi_\chi \right|_p = q^{- \langle \nu^{L_n} , \chi \rangle }, \quad \chi \in X^\ast({}^LS).
  		\]
  		The homomorphism $\phi_L \circ \zeta_p': \mathscr D^{L_n} \to \mathscr G_L$ defines an element $b$ in Kottwitz's set $B(^LS)$ (cf. the note to §5). The homomorphism $\nu^{\prime L_n}$ is given by the sequence of slopes of the Newton polygon of the isocrystal structure associated with $\phi_L \circ \zeta_p'$ on the representations of $^LS$ (cf.~[K4], §4.2). Kottwitz has shown ([K4], 4.3) that
  		\[
  		\langle \nu^{\prime L_n}, \chi \rangle = n \cdot \text{val}(\chi(b)),
  		\]
  		where $\text{val}$ denotes the normalized valuation. The desired equality follows because, in the case where $\chi = \chi_\Phi$ belongs to an abelian variety of CM type $(L, \Phi)$ over $\mathbb{F}_q$, its Frobenius endomorphism $\pi_\chi$ is related to the $F$ of the Dieudonné module as follows ([Dem], p.~63):
  		\[
  		\pi_\chi = b \cdot \sigma(b) \cdots \sigma^{m-1}(b).
  		\]

  	 To define the local additional structure of \( \mathscr M_L \) at the infinite place, we use the graded polarization \( \pi_M \) of the Tate triple $
  	 \left( ^LM_\FF, w, \mathbf{1}(1) \right).$
  	 Since there are motives of odd degree in \(  ^LM_\FF \), we may apply  [DM, 5.20]. The Tate triple \( (\mathbf V, w, T) \) described there is precisely the graded Tannakian category associated with the \( \mathbb{R} \)-gerbe \(\mathscr  W \). Thus, we obtain a homomorphism of gerbes, unique up to equivalence,
  	 \[
  	 \zeta_\infty': \mathscr  W \to \mathscr  M_L,
  	 \]
  	 which defines a functor on the associated Tate triples
  	 \[
  	 \left(  ^LM_\FF, w, \mathbf{1}(1) \right) \to (\mathbf V, w, T),
  	 \]
  	 characterized up to isomorphism by the property that it transforms the polarization \( \pi_M \) into the canonical polarization \( \pi_{\text{can}} \) of \( (\mathbf V, w, T) \). We want to show that the composition \( \phi_L \circ \zeta_\infty' \) is equivalent to
  	 $
  	 \xi_\mu: \mathscr W \to \mathscr G_L.$
  	 However, just like \( ^LM_\FF \), \( ^LCM_{\overline {\mathbb{Q}}} \) is naturally equipped with a graded polarization \( \pi_{CM} \), and under the reduction functor, \( \pi_{CM} \) transforms into \( \pi_M \). Thus, the composition \( \phi_L \circ \zeta_\infty' \) is the homomorphism from \( \mathscr W \) to \(\mathscr G_L \) associated with \( (^LCM_{\overline { \mathbb{Q}}}, \pi_{CM}) \) by Prop. 5.20 in [DM]. This homomorphism is easy to determine. On the kernel, it is given by the grading, and thus equals the weight homomorphism
  	 \[
 \nu = \mu + \iota\mu: \mathbb{G}_m(\mathbb{C}) \to  {}^LS(\mathbb{C}).
  	 \]
  	 
   The\marginpar{164} graded polarization \( \pi_{CM} \) on the neutral Tannakian category \(^LCM_{\overline { \mathbb{Q}}}\) is of the form \( \pi_C \) for a well-defined Hodge element \( C \in {}^LS(\mathbb{R}) \) ([DM], 4.22 and 4.25 (b)). From Hodge theory, it follows that \( C = \nu(i) \). Indeed, \( C \) induces on an \( \mathbb{R} \)-rational representation \( V \) of \( ^LS \) precisely the Weil operator [W] associated with the Hodge structure induced on \( V \), and the bilinear form defined by Weil [W, IV, Cor.~to Th.~7] lies in \( \pi_{CM}(V) \). By the definition of the homomorphism \( \phi_L \circ \zeta_\infty ' \) ([DM], 5.20), the chosen generator \( w_\iota \in \mathscr  W \) with \( w_\iota^2 = -1 \) is mapped to
   \[
   \nu(i) \rtimes \iota \in {}^LS(\mathbb{C}) \rtimes \text{Gal}(\mathbb{C}/\mathbb{R}).
   \]
   
   It is clear that this homomorphism is equivalent to \( \xi_\mu: \mathscr W \to \mathscr G_L \):
   \[
   \mu(i)(\nu(i) \rtimes \iota)\mu(i)^{-1} = \mu(-1) \rtimes \iota.
   \]
   In summary, we have shown that at every place \( v \), the gerbes \( \mathscr P_L(v) \) and \( \mathscr M_L(v) \) are isomorphic. Moreover, due to the preceding remarks, the conditions of Lemma \ref{3.15} are satisfied. We thus obtain the following theorem:\footnote{The numbering 4.3 is missing in the original text.  }
   \setcounter{subsection}{3}
   \begin{thm}\label{4.4}
   There exists a family of isomorphisms, \( \eta_L: \mathscr M_L \to \mathscr P_L \), such that \( \eta_L \circ \zeta_v' \cong  \zeta_v \), and the following diagrams commute up to equivalence:
   \[ \xymatrix{\mathscr M_{L'} \ar[rr] \ar[dd]_{\eta_{L'}} && \mathscr M_L \ar[dd]^{\eta_L} \\ \\ \mathscr P_{L'} \ar[rr] && \mathscr P_L      }  \qquad \xymatrix{\mathscr  M_L \ar[rd]^{\phi_L} \ar[dd] \\ & \mathscr G_L \\ \mathscr P_L \ar[ru]_{\psi_L} }
   \] 
   \end{thm}
   
 \textbf{Note 1.} 
   The cited reference, however, only deals with tensor functors \( \omega: C^\prime \to C \) between \emph{neutral} Tannakian categories over \( k \). To apply it here, it suffices to verify that if \( \omega \) satisfies one of the two conditions below, then the functor \( \omega_{(k_1)}: C^\prime_{(k_1)} \to C_{(k_1)} \), arising after a finite base field extension \( k_1/k \), also satisfies the same condition. The conditions are:
   \begin{enumerate}
   	\item[1.] \( \omega \) is fully faithful, and every subobject of \( \omega(X^\prime) \), \( X^\prime \in \text{ob}\, C^\prime \), is isomorphic to the image of a subobject of \( X^\prime \).
   	\item[2.] Every object in \( C \) is a subquotient of an object of the form \( \omega(X^\prime) \), \( X^\prime \in \text{ob}\, C^\prime \).
   \end{enumerate}
   
   The Tannakian category \( C_{(k_1)} \) can be identified with the category of \( k_1 \)-modules $
   \left( X, \alpha_X: k_1 \to \text{End}(X) \right) $
   in \( C \) ([DM], p.~156). Here, the functor \( i = i_{k_1/k}: C \to C_{(k_1)} \) (external tensor product with \( k_1 \)) is left adjoint to \( j = j_{k_1/k}: C_{(k_1)} \to C \), which sends \( (X, \alpha_X) \) to \( X \). We have \( k_1 \otimes_k \text{Hom}(X, Y) = \text{Hom}(i(X), i(Y)) \). 
   Now, suppose \( \omega \) satisfies condition 1. For \( (X^\prime, \alpha_{X^\prime}), (Y^\prime, \alpha_{Y^\prime}) \in \text{ob}\, C^\prime_{(k_1)} \),
   \[
   \text{Hom}\big( (X^\prime, \alpha_{X^\prime}), (Y^\prime, \alpha_{Y^\prime}) \big) = \text{Hom}\big( \left( \omega(X^\prime), \omega(\alpha_{X^\prime}) \right), \left( \omega(Y^\prime), \omega(\alpha_{Y^\prime}) \right) \big)
   \]
   as subsets of \( \text{Hom}(\omega(X^\prime), \omega(Y^\prime)) \), since \( \omega \) is fully faithful by assumption. Thus, \( \omega_{(k_1)} \) is fully faithful. 
   Similarly, let \( (Y, \alpha_Y) \) be a subobject of \( \left( \omega(X^\prime), \omega(\alpha_{X^\prime}) \right) \). By assumption, \( Y \) is isomorphic to an object of the form \( \omega(Y^\prime) \) for some subobject \( Y^\prime \) of \( X^\prime \). Due to the full faithfulness of \( \omega \), \( \alpha_Y \) then defines a \( k_1 \)-module structure \( \alpha_{Y'}: k_1 \to \text{End}(Y^\prime) \) on \( Y^\prime \), and \( (Y, \alpha_Y) \) is isomorphic to the image under \( \omega_{(k_1)} \) of the subobject \( (Y^\prime, \alpha_{Y^\prime}) \) of\marginpar{165} \( (X^\prime, \alpha_{X^\prime}) \).   
   Now, suppose \( \omega \) satisfies condition 2. Let \( X_1 \) be an object of \( C_{(k_1) }\), and let \( j(X_1) \) be a subquotient of \( \omega(X^\prime) \). The adjunction homomorphism identifies \( X_1 \) with a subobject of \( i \circ j(X_1) \), which in turn is a subquotient of \( i\left( \omega(X^\prime) \right) = \omega_{(k_1)}\left( i(X^\prime) \right) \).

   \textbf{Note 2.} According to a communication from L. Illusie, the problem of defining a cycle class in crystalline cohomology with the usual formal properties has now been solved. For \emph{rational} crystalline cohomology, such a solution was provided by H. Gillet and W. Messing (unpublished); for ``integral'' crystalline cohomology, a solution was achieved by M. Gros (Thèse de 3ème cycle, Orsay 1983).
   
   \section{A conjecture on Shimura varieties and its consequences}
   For the definition of a Shimura variety, the reader is referred to [De1]. A Shimura variety is associated with an algebraic group \( G \) defined over \( \mathbb{Q} \), a homomorphism \( h \) defined over \( \mathbb{R} \) from \( \mathbb S = \text{Res}_{\mathbb{C}/\mathbb{R}} \mathbb{G}_m \) to \( G \), and an open compact subgroup \( K \) of \( G(\mathbb A_f) \). Let \( \mathscr G = \mathscr G_G \) be the neutral gerbe \[  G(\overline{\mathbb{Q}}) \rtimes \text{Gal}(\overline{\mathbb{Q}}/\mathbb{Q}) . \] 
   
   What is given is not actually \( h \), but the conjugacy class \( X_\infty \) of \( h \) under \( G(\mathbb{R}) \). To this class, we assign an equivalence class of homomorphisms of gerbes over \( \mathbb{R} \), \[ \xi_\infty : \mathscr W \to \mathscr  G . \]
   
   The group \( \mathbb S \) is equipped with a canonical cocharacter \( \mu \), and the map \( (z, w) \mapsto z^\mu w ^{\bar \mu} \), where \(\bar \mu = \iota(\mu) \), is an isomorphism of \( \mathbb S \) with \( \mathbb{G}_m \times \mathbb{G}_m \), which is defined over \( \mathbb{C} \). We define \( \xi_\infty(z) = h(z^{\mu + \bar \mu}) \) for \( z \in \mathbb{C}^\times \), and \( \xi_\infty(w) = h((-1)^\mu) \rtimes \iota \) when \( w = w(\iota) \). 
   
   Since \( \mu + \overline{\mu} \) is central, \( \xi_\infty \) defines a homomorphism of the gerbe \( \text{Gal}(\mathbb{C}/\mathbb{R}) \) to \( 
   \mathscr G_{G_{\text{ad}}} = \mathscr G_{\text{ad}} \), and therefore a class \( a \) in \( H^1(\mathbb{R}, G_{\text{ad}}) \).
   
   \begin{lem}\label{5.1}
Let \( G' \) be the \( \mathbb{R} \)-form of \( G \) corresponding to the class \( a \). Then \( G'_{\mathrm{ad}}(\mathbb{R}) \) is compact.
   \end{lem}
   \marginpar{\textit{Proof of Lem.~\ref{5.1}.}}
   Let \( T \) be a Cartan subgroup defined over \( \mathbb{R} \) through which \( h \) factors. Then \( h(\mu) \) is a cocharacter of \( T \), which we also denote by \( \mu \). If \( \alpha \) is a root of \( T \), we can choose the root vectors \( X_\alpha, X_{-\alpha} \) such that
   \[  [X_\alpha, X_{-\alpha}] = H_\alpha, \quad \alpha(H_\alpha) = 2, \quad \iota(X_\alpha) = \pm X_{-\alpha}.    \]
   If the sign is positive (resp.~negative), the root \( \alpha \) is non-compact (resp.~compact). The sign is given by \( -(-1)^{\langle \mu, \alpha \rangle} \) (see [Sh], §4). When transitioning to \( G' \), the action of \( \iota \) is replaced by that of \( (-1)^\mu \iota = \iota' \). Therefore, for \( \iota' \), we have
   \[   \iota'(X_\alpha) = -X_{-\alpha}    \]
   for each \( \alpha \), and \( G'_{\text{ad}}(\mathbb{R}) \) is compact because each root of \( G' \) is compact.
       \marginpar{\textit{QED Lem.~\ref{5.1}.}}\\

       We       \marginpar{166} have fixed a prime \( p \) once and for all. It is assumed that \( G \) splits over an unramified extension of \( \mathbb{Q}_p \). It is further assumed that the compact subgroup that defines the Shimura variety is a product:
       \[
       K = K^p \cdot K_p
       \]
       with \( K^p \subset G(\mathbb A_f^p) \) and \( K_p \subset G(\mathbb{Q}_p) \). The time is certainly not ripe to formulate a conjecture about the reduction of Shimura varieties without imposing strong restrictions on \( K_p \), although special cases are certainly accessible ([C], [DR]). We assume one of the following two conditions. Either (a) \( K_p \) is  hyperspecial in the sense of [T], §3.1, or (b) \( G_{\text{ad}} \) is anisotropic over \( \mathbb{Q}_p \), and \( K_p \) is the maximal compact subgroup of \( G(\mathbb{Q}_p) \). The reader can likely generalize the considerations of this section a bit.  However, since this work mainly considers case (a), and in a further work to be written by the second author, case (b), and since the natural framework for the conjectures below is not at all clear to us, we have preferred to remain in these two cases.
       
       We also assume that the derived group \( G_{\text{der}} \) of \( G \) is simply connected, an assumption that does not restrict generality, and the reason for which we will explain later. Let \( \mathfrak k \) be the completion of the maximal unramified extension \( \QQ_p^{\text{un}} \) of \( \mathbb{Q}_p \). The group \( \text{Gal}(\QQ_p^{\text{un}} / \mathbb{Q}_p) \) acts on \( \mathfrak k \). The building \( \mathscr B(G, \mathfrak k) \) is defined in [T]. In the cases (a) and (b), there are facets \( w_1, \dots, w_r \) of minimal dimension and points \( x_i \in w_i \), such that the set \( \{ x_1, \dots, x_r \} \) is preserved by \( \text{Gal}(\QQ_p^{\text{un}} / \mathbb{Q}_p) \), and such that \( K_p \) is the stabilizer in \( G(\mathbb{Q}_p) \) of the center of the simplex spanned by \( \{ x_1, \dots, x_r \} \).
       
       We now introduce the concept of an admissible homomorphism \( \phi : \mathscr{Q} \to \mathscr  G \). This homomorphism can be factored through a \(\mathscr  Q_L \), and we often work with \( \mathscr Q_L \) instead of \( \mathscr  Q \). Four conditions must be satisfied. First, let \( G_{\text{ab}} \) be the quotient \( G / G_{\text{der}} \) and \( h_{\text{ab}} \) the composition,
       \[ \mathbb 
       S \to G \to G_{\text{ab}}.
       \]
       We write \( h_{\text{ab}}(z, w) = z^{\mu_{\text{ab}}} \cdot {w^{\bar \mu_{\text{ab}}}} \). Since the group \( G_{\text{ab}} \) is a torus, we can introduce the homomorphism \( \psi_{\mu_{\text{ab}}} : \mathscr{Q} \to \mathscr G_{\text{ab}} = \mathscr G_{ G_{\text{ab}} } \) associated with \( \mu_{\text{ab}} \).

\begin{lr}\label{5.a}
    	\textit{The composition  \(\phi_{\mathrm{ab}} : \mathscr Q \xrightarrow{\phi} \mathscr G \to \mathscr G_{\mathrm{ab}} \) is equivalent to \(  \psi_{\mu_{\mathrm{ab}}} \).}
     \end{lr}
 
There are several homomorphisms to \( \mathscr G \):
\begin{enumerate}
	\item[(i)] \( \zeta_{\infty} = \zeta_{v_1} : \mathscr  W \to \mathscr G \),   
	\item[(ii)] \( \zeta_p = \zeta_{v_2} : \mathscr D \to \mathscr  G \),  
	\item[(iii)] \( \zeta_l : \mathscr G_l \to \mathscr G , \quad  l \neq p \).
\end{enumerate}
\begin{lr}\label{5.b}
\textit{The composition \(\phi \circ \zeta_{\infty} \)   is equivalent to \( \xi_{\infty} \).} 
\end{lr}
\begin{lr} \label{5.c} 
\textit{The composition \(\phi \circ \zeta_l \) is equivalent to the canonical neutralization \(\xi_l\)  of  \( \mathscr G \) over \( \mathbb{Q}_l\).} \end{lr}

The\marginpar{167} condition at the place \( p \) is more complicated to clarify. Let \( w_1^0, \dots, w_r^0 \) and \( x_i^0 \in w_i^0 \) be the special facets and points that define \( K_p \), and let 
\[\mathscr 
X = \{(x_1, \dots, x_r) | \exists g \in G(\mathfrak k) \text{ such that } x_i = g \cdot x_i^0 \text{ for all } i \}.
\]
We define \( \sigma(i) \), \( 1 \leq i \leq r \), \( \sigma \in \text{Gal}(\overline{\mathbb{Q}}/\mathbb{Q}) \), by the equation  
\[ \sigma(x_i^0) = x_{\sigma(i)}^0 . \]  
By Lemma \ref{2.1}, there exists an unramified extension \( L \) of \( \mathbb{Q}_p \), such that \( \zeta_p = \phi \circ \zeta_p \) can be viewed as a homomorphism from \( \mathscr D^L \) to \( \mathscr G \), which itself is defined by a homomorphism \( \xi_p \) from the extension
\[
1 \to L^\times _1 \to \mathscr D^ L _{L_1} \to \text{Gal}(L_1/\mathbb{Q}_p) \to 1
\]
to \( G(L_1) \rtimes \text{Gal}(L_1/\mathbb{Q}_p) \), where \( L_1 \) is a finite Galois extension of \( \mathbb{Q}_p \). Let \( L_2 \) be the maximal unramified subfield of \( L_1 \) over \( L \). The cocycle defining \( \mathscr D^L _{L_1} \) is actually defined over \( \text{Gal}(L/\mathbb{Q}_p) \). Therefore, \( \text{Gal}(L_1/L) \) and, in particular, \( \text{Gal}(L_1/L_2) \), are embedded in \( \mathscr D^L _{L_1} \), and the restriction of \( \xi_p \) to \( \text{Gal}(L_1/L_2) \) is defined by a 1-cocycle of this group with values in \( G(L_1) \). According to a theorem of Steinberg (see [Se], p. 3), there exists an unramified extension \( L' \) of \( L \) such that this cocycle, when viewed as a cocycle of \( \text{Gal}(L_1L'/L_2L') \) with values in \( G(L_1L') \), becomes trivial. Consequently, we can replace \( L \) by \( L_2L' \) and \( \xi_p \) by an equivalent homomorphism \( \xi_p' \), and assume that \( L_1 = L \).
 
The homomorphism \( \xi_p' \) then maps \( \mathscr D_L ^L \) to \( G(L) \rtimes \text{Gal}(L/\mathbb{Q}_p) \). The group \( 
\mathscr D_L^ L \) is isomorphic to the Weil group \( W_{L/\mathbb{Q}_p} \), and \( W_{L/\mathbb{Q}_p} \) is naturally mapped to \( \text{Gal}(\mathbb{Q}_p^{\text{un}}/\mathbb{Q}_p) \). Let \( \xi \) be the homomorphism from \( W_{L/\mathbb{Q}_p} \) to \( G(\mathfrak k) \rtimes \text{Gal}(\mathbb{Q}_p^{\text{un}}/\mathbb{Q}_p) \) defined by the commutativity of the following diagram:
\[  \xymatrix{ & W_{L/\QQ_p}  \ar[rd]\ar[d] \ar[ld] \\  G(L) \rtimes \Gal(L/\QQ_p) &  G(L) \rtimes \Gal(\mathbb{Q}_p^{\text{un}}/\mathbb{Q}_p)   \ar[l] \ar[r] \ar[d] &  \Gal(\mathbb{Q}_p^{\text{un}}/\mathbb{Q}_p) \\ & G(\mathfrak k) \rtimes \Gal(\mathbb{Q}_p^{\text{un}}/\mathbb{Q}_p) .  }  \]
Let \( w \) be a preimage of the Frobenius element \( \sigma \) under the canonical homomorphism \( W_{L/\mathbb{Q}_p} \to \text{Gal}(\mathbb{Q}_p^{\text{un}}/\mathbb{Q}_p) \). The element \( w \) is uniquely determined up to a unit in \( L^\times \). Therefore, a second possible \( w' \) can be written as \( x \cdot \sigma(x)^{-1} \cdot w \) with \( x \in \mathfrak k^\times \). The restriction of \( \xi_p' \) to \( \mathscr D_L^ L = \mathbb G_m \) is algebraic. Let \( c = \xi_p'(x) \). Then \( \xi(w') = c \xi(w) c^{-1} \). Hence, \( F = \xi(w) \) is uniquely determined by \( \xi_p' \) alone,  up to conjugation by an element of \( \xi_p'(\mathfrak k^\times) \).

When we factor \( h \) by a torus \( T \), it defines a cocharacter \( \mu \) of \( T \). Although \( \mu \) itself depends on \( h \) and not only on its conjugacy class, the orbit of \( \mu \) under the Weyl group is well-defined and independent of \( h \) and, in an obvious sense, of \( T \). \marginpar{168}Therefore, for any Cartan subgroup \( T \) of  \( G \), we can associate an orbit \( \{ \mu \} \) in \( X^*(T) \) to the Shimura variety. On the other hand, any two points \( y \), \( z \) in \( \mathscr B(G, \mathfrak k) \) lie in an apartment, and \( z = y + a \) with \( a \in X_*(T) \otimes \mathbb{R} \). Again, the orbit of \( a \) under the Weyl group (rather than \(a\)) is uniquely determined, which we denote by \( \text{inv}(z, y) \).

The group \( G(\mathfrak k) \rtimes \text{Gal}(\mathbb{Q}_p^{\text{un}}/\mathbb{Q}_p) \) acts on \( \mathscr X\) \[ g \rtimes \sigma : (x_1, \dots, x_r) \mapsto (g \sigma(x_1), \dots, g \sigma(x_r)) \] and also acts in the same way on \( \mathscr B(G, \mathfrak k) \). Let \( X_p \) be the set of \( x = (x_1, \dots, x_r) \) in \( \mathscr X \) such that for every \( i \), \[\text{inv}(x_{\sigma(i)}, F {x_i}) = \{ \mu \} .\]

\begin{lr}\label{5.d}
\textit{The set \( X_p \) is non-empty. }
\end{lr} 
  
 We define  
 \[ X_l = \{ x \in G(\overline{\mathbb{Q}}_l) \mid \phi \circ \zeta_l = \mathrm{ad}\,x \circ \xi_l \}, \quad l \neq p. \]
 Since  
 \[ G(\mathbb{Q}_l) = \{ g \in G(\overline{\mathbb{Q}}_l) \mid \mathrm{ad}\,g \circ \xi_l = \xi_l \}, \] 
 the group \( G(\mathbb{Q}_l) \) acts on the right on \( X_l \)  simply transitively. We want to introduce a restricted product \( X^p = \prod_l' X_l \). According to [T], \S3.9, we can find a finite set \( S \) and for each \( l \notin S \), a hyperspecial point \( x_l \) in \( \mathscr B(G, \mathbb{Q}_l) \). Neither \( S \) nor the \( x_l \)'s are uniquely determined, but two possible families of \( x_l \)'s are almost everywhere equal. Therefore, for each finite extension \( L'\) of \( \mathbb{Q} \) and almost every place \( v \) of \( L' \), there is a distinguished hyperspecial subgroup \( K_v \).
 Let \( \phi \) be defined by \( q_{\rho}\mapsto g_\rho \rtimes \rho\).  Then, with the notation from the second section, \( \phi \circ \zeta_l \) is defined by \( \rho \mapsto \phi \left( e_\rho(l) \right) \cdot g_\rho \rtimes \rho \), where \(\rho \in \text{Gal}(\overline{\mathbb{Q}}_l/\mathbb{Q}_l) \). However, we can choose a finite extension \( L' \) of \( \mathbb{Q} \) such that \( \phi(e_\rho) \cdot g_\rho \) lies in \( G(\mathbb A_{L'}) \) for each \( \rho \). Therefore, \( \phi \left( e_\rho(l) \right) \cdot g_\rho \) lies in \( \prod_{v/l} K_v \) for almost all \( l \) and each \( \rho \). The cocycle \( \{ \phi \left( e_\rho(l) \right) \cdot g_\rho \} \) can then be bounded for almost every \( l \) in \( \prod_{v/l} K_v \), which provides a distinguished subset of \( X_l \) for almost every \( l \), and thus a restricted product.
 
 Let  
 \[ I_\phi = \{ g \in G(\overline{\mathbb{Q}}) \mid \mathrm{ad}\,g \circ \phi = \phi \}. \]
 The group \( I_\phi \) acts on \( X_l \), \( l \neq p \), and on \( X^p \). Let  
 \[ J_\phi = \{ g \in G(\overline{\mathbb{Q}}) \mid \mathrm{ad}\,g \circ \xi_p = \xi_p \} \] and  
 \[ J_\phi' = \{ g \in G(\mathfrak k) \mid \mathrm{ad}\,g \circ \xi_p' = \xi_p' \}. \]
 \marginpar{169}The group \( J_\phi' \) acts on \( X_p \). If \( \xi_p' = \mathrm{ad}\,h \circ \xi_p \), then \( g \mapsto hgh^{-1} \) maps the groups \( I_\phi \) and \( J_\phi \) into \( J_\phi' \). Thus, \( I_\phi \) acts on \( X_p \).  Let  
 \[  X_\phi(K) = I_\phi \backslash X_p \times X^p / K^p. \]
 Let \( \mathfrak p \) now be a prime of \( E \), the reflex field of the Shimura variety  
 \[ \Sh = \Sh_K = \Sh_K(G, h). \] 
  We assume that \( \mathfrak p \) is defined by the embeddings \( E \subset \overline{\mathbb{Q}} \subset \overline{\mathbb{Q}}_p \). The embedding \( \overline{\mathbb{Q}} \subset \overline{\mathbb{Q}}_p \) can always be chosen in this way. \( E \) is unramified at the place \( p \). Let \( r = [E_p : \mathbb{Q}_p] \) and \( \Phi(x_1, \ldots, x_r) = F^r(x_1, \ldots, x_r) \). Then, \( \Phi \) acts on \( X_p \) and on \( X_\phi \).

 Let \( \mathcal O_{\mathfrak p} \) be the ring of integers of \( E_{\mathfrak p} \). We are looking for a model of \( \text{Sh} \) over \( \mathcal O_{\mathfrak p} \) with at least the following properties, and we conjecture that it exists if \( K^p \) is sufficiently small.
 
\begin{lr}\label{5.e}
 	\textit{ Let \( \kappa \) be the residue field of \( \mathcal O_{\mathfrak p} \) and \( \bar{\kappa} \) its algebraic closure. Then, \( \mathrm{Sh}(\bar {\kappa}) \) as a set with the action of the Frobenius element is isomorphic to the union over equivalence classes of admissible \( \phi \) of the sets \( X_\phi(K) \) with the action of \( \Phi \).}
\end{lr}
 \begin{lr}\label{5.f}\textit{ 
If \( \mathrm{Sh} \) is proper over \( E \), then the model is proper over \( \mathcal O_{\mathfrak p} \).}
 \end{lr}
\begin{lr}\label{5.g}\textit{
If \( K_p \) is hyperspecial} [T], \textit{the model over \( \mathcal O_{\mathfrak p} \) is smooth. }
\end{lr} 
 
 In the next section we will use the results of the previous section to justify this conjecture. In this section we will try to make it more clear and prepare the transition to further results of Kottwitz. We will freely use terms and results from him.
 
 We start with a simple example that allows us to check the signs. Let \( G = \mathbb G_m \), so that \( X_*(G) = \mathbb{Z} \), and let \( h \) be the homomorphism \( (z, w) \to zw \), so that \( \mu = 1 \in \mathbb{Z} \). The corresponding variety \( \text{Sh} \) is of dimension zero, and its set of points is
 \begin{align}\label{5.h}
 \text{Sh}(\overline{\mathbb{Q}}) \cong  \mathbb{Q}^\times \backslash I_f / K^p \cdot K_p, \quad  I_f = \mathbb A_f^\times .
 \end{align}
 We only consider the case where \( K_p = \mathbb{Z}_p^\times \). In Deligne's definition, which we use, two inverses appear (see [De1], §2.2.3), namely the inverse of the reciprocity isomorphism in class field theory, which itself is the inverse of the conventional one. Therefore, the Frobenius element acts on \( \text{Sh}(\overline{\mathbb{Q}}) \) as \( I_f \ni a  \mapsto p^a \), where \( p \) is considered as an element of \( \mathbb{Q}_p^\times \).
 
 On the other hand, according to \eqref{5.a}, \( \phi = \psi_\mu \) represents the only admissible class. If we choose \( x \in X^p \), we can use the map \( t \mapsto xt \) to identify \( X^p \) with \( I_f^p\). Furthermore, \( I_\phi = \mathbb{Q}^\times \). Let \( \mathcal O \) be the ring of integers of \(\mathfrak k \). Then \( \mathscr X =\mathfrak  k^\times / \mathcal O^\times = \mathbb{Q}_p^\times / \mathbb{Z}_p^\times \).
 
 The homomorphism \( \psi_\mu \circ \zeta_p \) is \( \xi_{-\mu} \)-equivalent. Since \( G \) splits over \( \mathbb{Q} \), \( \xi_{-\mu} \) is defined on \( \mathscr D_{\mathbb{Q}_p}^{\mathbb{Q}_p} \) as 
 \[
 z \mapsto z^{-1}.
 \]
 \marginpar{170}Because we take the inverse of the conventional reciprocity isomorphism, \( p^{-1} \) corresponds to the Frobenius element, and \( p^{-1} \) is mapped to \( p\rtimes 1 \in G(\mathbb{Q}_p) \rtimes \text{Gal}(\mathbb{Q}_p^{\text{un}} / \mathbb{Q}_p) \). Consequently, \( F \) sends the element \( a \in \mathfrak k^\times \pmod {\mathcal O^\times}\) to \( pa \). 
 According to the sign that appears in [T], Eq. 1.2(1), we have: 
 \[
 \text{inv}(a, pa) = -\omega(p^{-1}) = 1 = \mu.
 \]
 Thus, \( X_p = \mathscr  X = \mathbb{Q}_p^\times / K_p \), and 
 \begin{align}\label{5.i}
I_\phi \backslash X_p \times X^p / K^p \cong \mathbb{Q}^\times \backslash I_f / K^p \cdot K_p
 \end{align} 
 The mappings in \eqref{5.h} and \eqref{5.i} establish an identification between \( I_\phi \backslash X_p \times X^p / K^p \) and \( \text{Sh}(\overline{\mathbb{Q}}) \), under which \( F = \Phi \) corresponds to the action of the Frobenius element. We emphasize that this identification is by no means
 canonical.
 
 It is evident that \( \text{Sh} \) has a model over \( \mathbb{Z}_p \) with good reduction, and in this case, the conjecture is trivially valid. In general, if \( G \) is a torus, the conjecture can be easily verified.
 
 Then \( \mathscr  X \) is \( T(\mathfrak k) / T(\mathcal O) \). If \( T \) splits over the unramified extension \( L \) of \( \mathbb{Q}_p \), let \( k = [L : \mathbb{Q}_p] \) and let \( \sigma \) be a Frobenius element in \( \text{Gal}(L / \mathbb{Q}_p) \). The fundamental class is given as follows: 
 \begin{align*}
 a_{i,j} = 1 , \text{ for }  0 \leq i, j \leq  k , i + j < k , \\    a_{i,j} = p^{-1} ,  \text { for }  0 \leq i, j \leq  k , i + j \geq k. 
 \end{align*}
 The element \( 1 \times \sigma \in \mathscr D^L_L \) corresponds to the preimage in \( W_{L/\mathbb{Q}_p} \) of a Frobenius element. Thus, \( F = p^\mu \times \sigma \).
 
 If \( \nu \) is defined as in [T], §1.2, then:
 \[
 \text{inv}(x, Fx) - \mu = \text{inv}(x, \sigma(x) + \nu(p^\mu) ) - \mu = \text{inv}(x, \sigma(x)),
 \]
 from which it easily follows that:
 \[
 X_p \cong T(\mathbb{Q}_p) / T(\mathbb{Z}_p).
 \]
 Let \( r \) be the smallest power of \( \sigma \) that fixes \( \mu \). The field \( E \) is unramified over \( \mathbb{Q}_p \), and \( r = [E : \mathbb{Q}_p] \). Consequently, 
 \[
 \Phi = p^{-\nu_2} \times \sigma^r,
 \]
 where \( \sigma \) now denotes the Frobenius automorphism of \( \mathfrak k / \mathbb{Q}_p \), and 
 \[
 -\nu_2 = \mu + \sigma\mu + \dots + \sigma^{r-1}\mu.
 \]
 Thus, \( \Phi \) acts on \( X_p \) as translation by \( p^{-\nu_2} \in T(\mathbb{Q}_p) \), and \( p^{-\nu_2} \) is precisely the image of \( p \in E^\times \) under the homomorphism \( \text{NR}(\mu) : E^\times \to T \), as defined in [De1], §2.2.3.
  
  These remarks allow  us to identify  \( T(\mathbb{Q}) \backslash T(\mathbb{A}_f) / K^p \cdot K_p \) with \( I_\phi \backslash X_p \times X^p / K^p \), so that \( \Phi \) corresponds to the Frobenius substitution at \( \mathfrak p\), and thus to test the conjecture.
  
  Let
  \marginpar{171} \( T \) be a Cartan subgroup of \( G \), defined over \( \mathbb{Q} \). Assume \( T \) is fundamental, meaning the group \( T_{\text{ad}}(\mathbb{R}) \) is compact, where \( T_{\text{ad}} \) is the image of \( T \) in the adjoint group \( G_{\text{ad}} \). Then \( h \) can be chosen to factor through \( T \), thereby defining a  cocharacter \( \mu \) of \( T \):
  \[
  h(z, w) = z^\mu w^{\bar \mu}.
  \]
  Let \( \psi_{T, \mu} \) denote the composition
  \[\mathscr Q \xrightarrow{\psi_\mu}   \mathscr  G_T \subset \mathscr  G = \mathscr G_G. 
  \]
  \begin{lem}\label{5.2}
  	$\psi_{T,\mu}$ is admissible. 
  \end{lem}

\marginpar{\textit{Proof of Lem.~\ref{5.2}.}}
The conditions \eqref{5.a} and \eqref{5.c} are clear, as is \eqref{5.b}, since the fundamental class at infinity is defined by \( a_{\iota, \iota} = -1 \), and consequently, \( \xi_\infty \) is, by definition, equivalent to \( \psi_\mu \circ \xi_\infty \).  

Let  
\[
\nu_p = \nu_2 = - \sum_{\sigma \in \text{Gal}(L_{v_2}/\mathbb{Q}_p)} \sigma \mu.
\]
Let \( P \) be the parabolic subgroup defined over \( \mathbb{Q}_p \) that contains \( T \), with roots defined by \( \langle \nu_p, \alpha \rangle \leq 0 \). The group \( P \) contains the Levi factor \( J \), whose roots are defined by \( \langle \nu_p, \alpha \rangle = 0 \). The field \( \mathbb{Q}_p^{\text{un}} \) splits \( J \).  

We may assume that the homomorphism defining \( X_p \) is given by  
\( \xi_p'  = \mathrm{ad}\,h \circ \xi_p
\)   
with \( h \in T(\overline{\mathbb{Q}}_p) \). Then \( J_\phi \subset  J(\overline{\mathbb{Q}}_p) \) resp.~ \( J_\phi^\prime  \subset J(\mathfrak k) \). Indeed, \( J \) is the centralizer of the image of the kernel of \(\mathscr  D \), so that \( \xi_p \) resp.~\( \xi_p^\prime \) define a twisted form \( \mathscr J_\phi \) resp.~\( 
\mathscr J_\phi^\prime \) of \( J \). The groups \( J_\phi \) and \( J_\phi^\prime \) are the groups of \( \mathbb{Q}_p \)-rational points on these twisted forms.  

The Bruhat-Tits buildings \( \mathscr B(G, \mathfrak k) \), \(\mathscr  B(J, \mathfrak k) \), as well as \(\mathscr  B(G, \mathbb{Q}_p) \), \(\mathscr  B(J, \mathbb{Q}_p) \), are described in [T], §2.1. Since \( J \) contains a maximal split torus of \( G \), this characterization yields a commutative diagram of embeddings (see also [T], §2.6):  
\[ \xymatrix{ 
	\mathscr B(J, \mathbb{Q}_p) \ar@{^(->}[r]  \ar@{^(->}[d] & \mathscr  B(G, \mathbb{Q}_p) \ar@{^(->}[d] \\
	\mathscr 	B(J, \mathfrak k) \ar@{^(->}[r]    &\mathscr B(G, \mathfrak k). }
\]
We also need an elliptic Cartan subgroup \( T^\prime \) of \( J \), defined over \( \mathbb{Q}_p \), that splits over \( \mathfrak k \), and such that the maximal compact subgroup of \( T^\prime(\mathbb{Q}_p) \) is contained in \( K_p \). In case (b), the second condition is automatically satisfied, and for the existence of \( T^\prime \), we can refer to [Kn], p. 271. In case (a), we also refer to [T], §3.4, and [SGAD]. 
Then \( x_1^0, \ldots, x_r^0 \) lie in the apartment of \( T^\prime \) (see [T], §3.6).

Let\marginpar{172} \( \mu' \) be a  cocharacter of \( T' \), conjugate to \( \mu \) under \( J \). We will first show that if \( F' \) is defined using \( \xi_{-\mu'} \) instead of \( \xi_p \) or the equivalent \( \xi_{-\mu} \), but otherwise defined as \( F \), then the following holds:  
\begin{align}\label{5.j}
 \text{inv}(x_{\sigma(i)}^0, F'x_i^0) = \{\mu'\} = \{\mu\}.
\end{align} 
The group \( T' \) splits over an unramified extension. Consequently, as we have already calculated in the abelian case,  we have \( 
F' = p^{\mu'} \rtimes \sigma,
\) and   
\[
F'x_i^0 = \sigma(x_i^0) - \mu'.
\]  
Thus \eqref{5.j} would be proven.

Our goal, however, was to show that \( X_p \) is non-empty, and for this, it suffices to show that \( \xi_{ - \mu} \) and \( \xi _{- \mu'} \), as homomorphisms from \(\mathscr  D \) to \( \mathscr G_J \), are equivalent when the image of \( T' \) in the adjoint group of \( J \) is anisotropic, as we may assume. This follows immediately from the results of Kottwitz [K4] (see the Note at the end of this section). We denote the set of equivalence classes of these homomorphisms by \( C(J) \). Since the image \( T' \) in \( J_{\text{ad}} \) is anisotropic, we know that \[
\nu'_p = -\sum_{\sigma \in \text{Gal}(L_{v _2}/ \mathbb{Q}_p)} \sigma \mu' = \nu_p
\] is central. Here, it is assumed that \( L \) is sufficiently large such that \( L_{v_2} \) splits the torus \( T' \). Consequently, the classes of \( \xi _{- \mu} \) and \( \xi_{ - \mu'} \) lie in \( C(J)_b \cong B(J)_b \) (see the Note) and are labeled, according to Kottwitz, by elements of \(X^* (Z(\widehat{J})^{\Gamma(p)}) 
\), where \( \Gamma(p) = \text{Gal}(\mathbb{Q}_p^{\text{un}} / \mathbb{Q}_p) \) and \( Z(\widehat{J}) \) denotes the center of the connected component of the \( L \)-group. By definition, \( \xi _{- \mu} \) corresponds to the character \( z \mapsto z^\mu \), and \( \xi _{- \mu'} \) corresponds to the character \( z \mapsto z^{\mu'} \), and these characters are equal on \( Z(\widehat{J}) \). More precisely, in the bijection defined by Kottwitz, a sign is involved. We choose it so that \( \xi _{- \mu} \) corresponds to the character \( z \mapsto z^\mu \). \marginpar{\textit{QED Lem.~\ref{5.2}.}}\\

We can now explain why we assumed \( G_{\text{der}} \) to be simply connected. Let \( G \) be arbitrary. Then we can find a \( z \)-extension \( G' \) in the sense of Kottwitz [K1] and a factorization
\[
\xymatrix{
\mathbb 	S \ar[r]^-{h'} \ar[rd]_-{h} & G' \ar[d]^-{\alpha} \\  & G  }
\] We can also assume ([MS2], §3.4) that the reflex field \( E(G', h') \)  is the field \( E = E(G, h) \). Let \( A \) be the kernel of \( \alpha \).

Let\marginpar{173} \( K' \) be an open compact subgroup of \( G'(\mathbb A_f) \) such that \( \alpha(K') \subset K \). The map  
\[
G'(\QQ) \backslash G'(\mathbb A) / K' \to G(\QQ) \backslash G(\mathbb A) / K
\]
is the restriction on the set of complex-valued points of a morphism from \( \text{Sh}_{K'}(G', h') \) to \( \text{Sh}_K(G, h) \) defined over \( E \). This map is surjective because \( G'(\QQ_v) \to G(\QQ_v) \) is surjective for every place \( v \) of \( \mathbb{Q} \). It follows from [T], §3.9.1, that \( \alpha(K') \) is a subgroup of finite index in \( K \), and we may even assume that \( \alpha(K') \) is a normal subgroup of \( K \).  

Then, the finite group 
\[
B = A(\QQ) \backslash \alpha^{-1}(K) / K'
\]
acts on the right on \( \text{Sh}_{K'}(G', h') \), and the quotient is \( \text{Sh}_K(G, h) \), since \( G'(\QQ) \to G(\QQ ) \) is surjective.  

Ultimately, we are interested in the cohomology groups of certain sheaves on \( \text{Sh}_K(G, h) \) (see [L1]), which are associated with a representation \( \xi \) of \( G \). These can be obtained as the fixed points of \( B \) in the cohomology of the pullback of the sheaves, defined via \( \xi \circ \alpha = \xi' \), on \( \text{Sh}_{K'}(G', h') \). Consequently, we may, without loss of generality, work with \( G' \) instead of \( G \).  

Nevertheless, we can introduce the concept of an admissible homomorphism \( \phi: \mathscr  Q \to \mathscr G = \mathscr G_G \). Later, we will demonstrate with an example that the conjecture, as stated above, cannot hold simultaneously for both \( G \) and \( G' \) in certain cases. For this  we need a theorem, which we will prove first and which is important anyway for the transition to the trace formula.

\begin{thm}\label{5.3}
Assume the Hasse principle holds for \( G_{\mathrm{sc}} \). If \( G \) is quasi-split over \( \mathbb{Q}_p \) and \( K_p \) is hyperspecial, then every admissible \( \phi \) is equivalent to some \( \psi_{T,\mu} \).  
\end{thm}
 
We start the proof with a lemma.  

\begin{lem} \label{5.4}
 Assume the Hasse principle holds for \( G_{\mathrm{sc}} \). Let \( G \) be quasi-split over \( \mathbb{Q}_p \), and let \( \phi \) be admissible. Then \( \phi \) is equivalent to some \( \phi' \) such that \( \phi'(\delta_n) \) lies in \( G(\mathbb{Q}) \) for sufficiently large \( n \).
\end{lem}

\marginpar{\textit{Proof of Lem.~\ref{5.4}.}}
We treat \( \phi \) as a map \( \phi: \mathscr Q(L, m) \to \mathscr G \). Let \( \gamma_n = \phi(\delta_n) \), and let \( g _{\rho} \rtimes \rho \) denote the image of \( q _{\rho}\). Then  
\[
\gamma_n = \phi(\delta_n) = \phi(q_{\rho} \delta_n q_{\rho}^{-1}) = g_\rho \rho (\gamma_n) g_{\rho}^{-1}.
\]  
Hence the conjugacy class of \( \gamma_n \) is rational.  

Let \( G^* \) be the quasi-split form of \( G \), and let \( \psi: G \to G^* \) be an isomorphism over \( \overline{\mathbb{Q}} \) such that \( \psi^{-1} \sigma(\psi) \) is inner for every \( \sigma \in \text{Gal}(\overline{\mathbb{Q}}/\mathbb{Q}) \). Define \( \gamma_n^* = \psi(\gamma_n) \). Its conjugacy class is also rational. Replacing \( n \) by a sufficiently large multiple if necessary, we can ensure that the group \( C_{\gamma_n^*} \), which appears in Theorem 4.7 of [K1], is trivial. This theorem allows us to replace \( \phi \) with \( \mathrm{ad}\,g \circ \phi \), making \( \gamma_n^* \) rational.  

\begin{lem}\label{5.5}The
	\marginpar{174} Zariski closure of \( \{ \delta^{k}_n \mid k \in \mathbb{Z} \} \) in \( Q(L, m) \) is \( Q(L, m) \).
\end{lem} 

\marginpar{\textit{Proof of Lem.~\ref{5.5}.}}
This closure is an algebraic subgroup of \( Q(L, m) \). If it were not \( Q(L, m) \) itself, there would exist a nontrivial character \( \chi_\pi \) of \( Q(L, m) \) such that  
\[
1 = \chi_\pi(\delta_n) = \pi^{\frac{n}{m}}.
\]  
This implies that \( \pi \) is a root of unity, and \( \chi_\pi = 1 \),  a contradiction.  
\marginpar{\textit{QED Lem.~\ref{5.5}.}}\\

For simplicity, we use \( \psi \) to identify \( G \) and \( G^* \). Then  the action of the Galois group on \( G^* \) is defined by  
\[
\sigma^*(g) = b_\sigma \sigma(g) b_\sigma^{-1},
\]  
where \( b_\sigma \in G(\overline{\mathbb{Q}}) \), and \( b_\rho \rho(b_\sigma) b_{\rho \sigma}^{-1}  \in Z(\overline{\mathbb{Q}}) \), with \( Z \) denoting the center of \( G \). Let \( I^* \) be the centralizer of \( \gamma_n^* \) in \( G^* \), which is also the centralizer of \( \phi(Q(L, m)) \).

We set  
\[
g_\rho = h_\rho b_\rho. 
\]  
Then  
\[
\gamma_n^* = \gamma_n = g_\rho \rho( \gamma_n) g_{\rho}^{-1} = h_\rho b_\rho \rho(\gamma_n) b_{\rho}^{-1} h_{\rho}^{-1} = h_\rho \rho^*(\gamma_n^*) h_{\rho}^{-1} = h_\rho \gamma_n^* h_{\rho}^{-1},
\]  
and \( h_\rho \in I^* \).

We now show that \( h_\rho \rho^*(h_\sigma) h_{\rho \sigma}^{-1}  \) lies in the center of \( I^* \), so that \( \{h_\rho\} \) defines a form \( I' \) of \( I^* \) on which the Galois group acts according to the equations  
\[
\sigma'(h) = h_{\sigma} \sigma^*(h) h^{-1}_\sigma.
\]  
Specifically,  
\[
h_{\rho} \rho^*(h_\sigma) h_{\rho \sigma}^{-1}  = h_\rho b_\rho  \rho( h_\sigma) b_{\rho}^{-1} h_{\rho \sigma}^{-1}   = g_\rho \rho(g_{\sigma}) \rho ( b_{\sigma}^{-1})  b_{\rho}^{-1} h_{\rho \sigma}^{-1}.
\]  
Since  
\[
z = b_{\rho \sigma} \rho(b_\sigma^{-1}) b_{\rho}^{-1} \in Z(\overline{\mathbb{Q}}),
\]  
the last term in this equation becomes  
\[ g_\rho \rho(g_{\sigma}) b_{\rho \sigma}^{-1} h_{\rho \sigma}^{-1} z  =  g_\rho \rho(g_{\sigma}) g_{\rho \sigma}^{-1}   z  = \phi(g_{\rho,\sigma}) z ,
\]  
and hence it lies in the center of \( I^* \).  

According to [K1], Lemma 3.3, we can assume that \( I^* \) is quasi-split. Since \( I' \) contains a Cartan subgroup defined over \( \mathbb{Q} \), there exists a Cartan subgroup \( T^* \) of \( I^* \) defined over \( \mathbb{Q} \) and an element \( h \in I^* \) such that \( h h_\rho \rho^*(h^{-1}) \in T^* \) for every \( \rho \). We replace \( \phi \) by \( \mathrm{ad}\,h \circ \phi \) and assume that \( h_\rho \in T^* \).

We\marginpar{175} first show that \( T^* \), which is also a Cartan subgroup of \( G^* \), transfers to  the group \( G \) (in the sense of [L3]). This means that there exists \( g \in G(\overline{\mathbb{Q}}) \) such that  
\[
\sigma^*(t) = g^{-1} \big ( \sigma(g t g^{-1} ) \big )  g, \quad t \in T^*.
\]  
Then \( g T^* g^{-1} \) is a Cartan subgroup \( T \) of \( G \) defined over \( \mathbb{Q} \). Replacing \( \phi \) by \( \mathrm{ad}\,g \circ \phi \), \( \gamma_n \) becomes \( g \gamma_n g^{-1} \), which lies in \( T(\mathbb{Q}) \) because \( \gamma_n \) lies in \( T^*(\mathbb{Q}) \). Thus  Lemma \ref{5.4} would be proven.

To show that \( T^* \) transfers to \( G \), we need a sequence of considerations from [L3].  

\begin{lem}\label{5.6}
 Assume the Hasse principle holds for \( G_{\mathrm{sc}} \). If the Cartan subgroup \( T^* \) transfers locally everywhere to \( G \), and if \( T^*_{\mathrm{ad}} \) is anisotropic at at least one place of \( \mathbb{Q} \), then \( T^* \) transfers  globally to \( G \).  
\end{lem} 

\marginpar{\textit{Proof of Lem.~\ref{5.6}.}}
In [L3] (see Lemma 7.15), an obstruction is introduced, whose vanishing is necessary and sufficient for \( T^* \) to transfer to  \( G \). This obstruction lies in  
$$
H^{-1}(\text{Gal}(L/\mathbb{Q}), X_*(T^*_{\text{sc}}))$$ modulo $$ \bigoplus_v \xi_v\left(H^{-1}(\text{Gal}(L_v/\mathbb{Q}_v), X_*(T^*_{\text{sc}}))\right), $$
where \( \xi_v \) denotes the natural map  $$ H^{-1}(\text{Gal}(L_v/\mathbb{Q}_v), X_*(T^*_{\text{sc}})) \to  H^{-1}(\text{Gal}(L/\mathbb{Q}), X_*(T^*_{\text{sc}})) . $$   But if \( T^*_{\text{ad}} \) is anisotropic at a place \( v \), then \( \xi_v \) is clearly surjective, and the obstruction is automatically zero.  \marginpar{\textit{QED Lem.~\ref{5.6}.}}\\

We apply Lemma \ref{5.6} to our \( T^* \). It is clear that \( T^* \) transfers over \( \mathbb{Q}_p \), because \( G \) is quasi-split over \( \mathbb{Q}_p \) by assumption. At infinity, one obtains the twist \( I' \) by restricting \( \xi_\infty \) to \( I^* \). Then \( I' \) is contained in the group \( G' \) from Lemma \ref{5.1}. Since \( T^* \) is a Cartan subgroup of \( I' \), it follows from that lemma that \( T^*_{\text{ad}} \) is anisotropic over \( \mathbb{R} \). Thus  \( T^* \) transfers to  \( G \) (at infinity).  

By assumption, \( \sigma \mapsto e_\sigma(l) g_\sigma \rtimes \sigma = a_\sigma \rtimes \sigma, ~\sigma \in \Gal (L_l / \QQ_L)\) is equivalent to the canonical neutralization of \( \mathscr G \). Hence, \( \sigma_1: g \mapsto a_\sigma \sigma(g) a_\sigma^{-1} \) defines a form of \( G \) that is isomorphic to \( G \). However, \( e_\sigma(l) \) lies in \( \phi(Q(L, m)) \), which is contained in the center of \( I^* \). For \( t \in T^* \), it therefore holds that  
\[
\sigma_1(t) = \sigma^*(t),
\]  
and \( T^* \) transfers to \( G \) over \( \mathbb{Q}_l \).  
\marginpar{\textit{QED Lem.~\ref{5.4}.}}\\

\marginpar{\textit{Proof of Thm.~\ref{5.3}.}}
We now return to Theorem \ref{5.3}, continuing to assume that \( \gamma_n = \phi(\delta_n) \) is rational for sufficiently large \( n \). The centralizer of \( \gamma_n \) is \( \mathscr I \). Replacing \( n \) by an appropriate multiple, we may assume that \( \mathscr I \) is connected. The homomorphism \( \phi \) defines a twist \( \mathscr I_\phi \) of \( \mathscr  I \), with \( \mathscr  I_\phi(\mathbb{Q}) = I_\phi \).

We\marginpar{176} have \( \phi(q_\rho) = g_\rho \rtimes \rho \), where \( g_\tau \) now lies in \( \mathscr I \) because \( \gamma_n \) is rational. We need a torus \( T \)  in \(\mathscr  I \) defined over \( \mathbb{Q} \), and an \( a \in I(\overline{\mathbb{Q}}) \) such that \( a g_ \rho \rho(a)^{-1} \in T \) for every \( \rho \). Then we can replace \( \phi \) with \( \mathrm{ad}\,a \circ \phi \) and assume that  
\begin{align}\label{5.k}
g_\rho \rho(t) g_\rho^{-1} = \rho(t),
\end{align} 
for every \( \rho \) and \( t \in T(\overline{\mathbb{Q}}) \). Thus, \( T \) becomes also a \( \mathbb{Q} \)-torus in \( \mathscr I_\phi \).  

The group \( \mathscr I \) satisfies the Hasse principle with \( G \).\footnote{The authors probably mean that if one assumes that Hasse principle holds for $G_{\mathrm{sc}}$, then it also holds for $\mathscr I_{\mathrm{sc}}$. The Hasse principle for all semi-simple simply connected groups have been proved unconditionally.  } Hence, we can apply the following lemma to \( \mathscr I \) to obtain \( T \). Although the lemma must be well-known, we have not found a reference in the literature.

\begin{lem}\label{5.7}
Assume the Hasse principle holds for \( G \),\footnote{Here $G$ should be replaced by $G_{\mathrm{sc}}$.  } and let \( \psi_1: G \to G_1 \) be an isomorphism defined over \( \overline{\mathbb{Q}} \) such that \( \psi_1^{-1} \sigma (\psi_1) \) is inner. Then there exist Cartan subgroups \( T \) and \( T_1 \) of \( G \) and \( G_1 \), both defined over \( \mathbb{Q} \), and an element \( g \in G_1(\overline{\mathbb{Q}}) \) such that \( \mathrm{ad}\,g \circ \psi_1 \) maps \( T \) to \( T_1 \) and is an isomorphism between \( T \) and \( T_1 \) defined over \( \mathbb{Q} \).
\end{lem} 

\marginpar{\textit{Proof of Lem.~\ref{5.7}.}}
The group \( G_1 \) is an inner form of \( G^* \), and the lemma asserts the existence of a Cartan subgroup \( T^* \) that transfers to  both  \( G \) and \( G_1 \). This is a consequence of Lemma \ref{5.6} and the following three lemmas, which we apply to \( G \) and \( G_1 \).
\begin{lem}[{[Sh], Corollary 2.9}] \label{5.8}
If \( T^* \) is fundamental in \( G^* \) over \( \mathbb{R} \), then it transfers locally at infinity to \( G \).
\end{lem} 
\begin{lem}\label{5.9}
 Let \( v \) be a finite place of \( \mathbb{Q} \), and let \( T^* \) be a Cartan subgroup of \( G^* \) over \( \mathbb{Q}_v \) with \( T^*_{\mathrm{ad}} \) anisotropic. Then \( T ^* \) transfers to  \( G \).
\end{lem} 

\marginpar{\textit{Proof of Lem.~\ref{5.9}.}}
Let \( N \) be the kernel of \( G^*_{\text{sc}} \to G^*_{\text{ad}} \). The lemma follows as a well-known consequence of the existence of a bijection \( 
H^1(\mathbb{Q}_p, G^*_{\text{ad}}) \to H^2(\mathbb{Q}_p, N),
\) 
(see [Kn], Theorem 2) and the exact sequence:
\[
H^1(\mathbb{Q}_p, T^*_{\text{ad}}) \to H^2(\mathbb{Q}_p, N) \to H^2(\mathbb{Q}_p, T^*_{\text{sc}}) = 0.
\]
For the existence of such \( T^* \), refer to [Kn], p. 271.
\marginpar{\textit{QED Lem.~\ref{5.9}.}}

\begin{lem}\label{5.10}  Let \( S \) be a finite set of   places of \( \mathbb{Q} \), and  for each \( v \in S \) let \( T^*_v \) be a Cartan subgroup of \( G^* \) over \( \mathbb{Q}_v \). Then there exists a Cartan subgroup \( T^* \) of \( G^* \) defined over \( \mathbb{Q} \) such that for each \( v \in S \), \( T^* \) is conjugate to \( T^*_v \) over \( \mathbb{Q}_v \). 
\end{lem}

\marginpar{\textit{Proof of Lem.~\ref{5.10}.}}
It follows easily from the Bruhat decomposition that the variety \( G^*_{\text{sc}} \) is rational. Thus, the weak approximation theorem holds for \( G^*_{\text{sc}} \), and we immediately conclude the lemma by approximating a product of regular elements \( \prod_v t_v \) from \( T_v(\mathbb{Q}_v) \).
\marginpar{\textit{QED Lem.~\ref{5.10}, \ref{5.7}.}} \\

The torus \( T \), for which \eqref{5.k} holds, is also a Cartan subgroup of \( G \) defined over \( \mathbb{Q} \). From Lemma 5.1 and assumption \eqref{5.a}, it follows that \( T_{\text{ad}} \) is anisotropic over \( \mathbb{R} \). Therefore, we can choose \( h \) such that it can be factored through \( T \):
\[
h(z, w) = z^{\mu_h} w^{\bar \mu_ h}.
\]

\begin{lem}\label{5.11}
There\marginpar{177} exists  \( \mu \) in the orbit of \( \mu_h \) under the Weyl group such that on the kernel \( \phi \) and \( \psi_{T,\mu} \) are equal, and such that
\[
\phi(q_\sigma) = b_\sigma \psi_{T,\mu}(q_\sigma),
\]
where \( \{b_\sigma\} \) is a 1-cocycle with values in \( T \).
\end{lem} 

\marginpar{\textit{Proof of Lem.~\ref{5.11}.}}
We emphasize that \( \psi_{T,\mu} \) may not be admissible, as \( \mu \) may not be \( \mu_h \). The restriction of \( \phi \) to \( Q(L,m) \) maps \( Q(L,m) \) into \( T \) and is defined over \( \mathbb{Q} \). We can assume, by replacing \( Q(L,m) \) with \( Q(L,n) \), that \( n = m \). Let \( \nu_p = \phi(\nu_2) \), which is  a cocharacter of \( T \). It suffices to show that there exists a \( \mu \) in the orbit of \( \mu_h \) such that
\begin{align}\label{5.l}
\text{Nm}_{L_p/\QQ_p}(\mu) = -\nu_p,
\end{align} 
since
\begin{align*}
\Big|\prod_{\sigma \in \text{Gal}(L_p/\mathbb{Q}_p)} \sigma \lambda(\gamma_m)\Big|_p & = \Big|\prod_{\sigma \in \text{Gal}(L_p/\mathbb{Q}_p)} \sigma \phi^*(\lambda)(\delta_m)\Big|_p \\ &  = \Big|\prod_{\sigma \in \text{Gal}(L_p/\mathbb{Q}_p)} \sigma \pi_ \lambda\Big|_p \\ & = q^{\langle \nu_2, \chi _{\pi_ \lambda} \rangle}.
\end{align*} 
Therefore, by the definition of \( \psi_{T,\mu} \), it is defined on \( Q(L,m) \) by \( \delta_m \mapsto  \gamma_m \).

We will conclude the existence of \( \mu \) from the fact that \( X_p \) is non-empty. We need to replace the homomorphism \( \xi_p = \phi \circ \zeta_p \) with \( \xi'_p = \mathrm{ad}\,u \circ \xi_p \) to define \( F \). We can assume that \( u \in T \).

Let \( L \) be a finite unramified extension of \( \mathbb{Q}_p \) such that \( \xi'_p \) is defined on \( \mathscr  D_L^L \), and let
\[
\xi'_p(d_\sigma) = g_\sigma \rtimes \sigma,
\]
where \( \sigma \) is the Frobenius element in  \( \text{Gal}(L/\mathbb{Q}_p) \). To the homomorphism \( \xi'_p \), we associate a parabolic subgroup \( P \) and a Levi factor \( J \). Let \( M \) be the centralizer of a maximal split subtorus of \( T \). The group \( M \) is a Levi factor of a parabolic subgroup \( Q \) of \( G \) defined over \( \mathbb{Q}_p \) and contained in \( P \). The group \( M \) is quasi-split.

We have already seen that we can embed the building \( \mathscr B(M,\mathfrak k) \) into \( \mathscr B(G,\mathfrak k) \) in such a way that the point \( x^0 \) used to define $K_p$ is contained in \( \mathscr B(M, \mathfrak k) \). We can even assume that \( x^0 \) lies in the apartment of a Cartan subgroup \( T' \) of \( M \) which is  defined over \( \mathbb{Q}_p \) and split over \(\mathfrak  k \). A point \( x \in X_p \) can be written as \( x = h x^0 \) with \( h = nm  \in Q(\mathfrak k) \), where \( m \in M(\mathfrak k) \) and \( n \) lies in the unipotent radical of \( Q \). Then
\begin{align}\label{5.m}
\text{inv}(x, F_x) = \text{inv}(hx^0, g_\sigma \sigma(h) x^0) = \text{inv}(x^0, h^{-1} g_\sigma \sigma(h) x^0),
\end{align} 
since \( \sigma(x^0) = x^0 \).

Let\marginpar{178} \( \{ \mu'\} \) be the orbit in \( X_*(T') \) corresponding to \( \mu' \). Define \( t_{\mu'} \) by \[ \chi(t_{\mu'}) = p^{\langle \chi, \mu '\rangle} , \quad \chi \in X^*(T')  ,\]  and let \( f_{\mu'} \) be the characteristic function of the double class \( K_0 t_{\mu'} K_0 \), where \( K_0 \) is the stabilizer of \( x^0 \) in \( G(\mathfrak k) \). From \eqref{5.m}, it follows that
\[
f_{\mu'}(m^{-1} g_\sigma \sigma(m) n') = 1,
\]
where \[ n' = \sigma(m)^{-1} g^{-1}_\sigma n^{-1} g_\sigma \sigma(n) \sigma (m) \] is in the unipotent radical of \( Q \).

There exists   \( \mu'' \in X_*(T') \) such that
\[
m^{-1} g_\sigma \sigma(m) \in (K_0 \cap M) t_{\mu''} (K_0 \cap M).
\]
From Lemmas 2.3.3 and 2.3.9 of [K3], it follows easily that \( \mu'' \) lies in the orbit of \( \mu' \), because \( \mu' \) is a ``minuscule weight''. 

If \( \lambda \) is the homomorphism \( M(\mathfrak k ) \to X^*(Z(\widehat{M})^{\Gamma})   \) introduced in Lemma 3.3 of [K3],  then
\( 
\lambda(m^{-1} g_\sigma \sigma(m)) 
\)
is the restriction of \( \mu'' \) to \( Z(\widehat{M}^\Gamma) \).

The definition of [K3] is actually introduced only for finite extensions of \( \mathbb{Q}_p \), but it can clearly be extended to \( M(\mathfrak k) \).

Equation \eqref{5.l} holds if and only if
\begin{align}\label{5.n}
[L_p : \mathbb{Q}_p] \langle \mu, \chi \rangle = -\langle \nu_p, \chi \rangle
\end{align} 
for every character \( \chi \) of \( T \) defined over \( \mathbb{Q}_p \). Such a  character can be viewed as a  character of \( M \), and if \( \mu \) is a cocharacter of \( T \) conjugate to \( \mu'' \) in \( M \), then
\[
\langle \mu, \chi \rangle = \langle \mu'', \chi \rangle.
\]
Furthermore, \( \mu \) lies in the orbit of \( \mu_h \) under the Weyl group of \( G \).

By the definition of \( \lambda \), we have
\[
|\chi(m^{-1} g_\sigma \sigma(m))| = p^{-\langle \mu'', \chi \rangle}.
\]
Since \( \chi \) is defined over \( \mathbb{Q}_p \), we have 
\[ \chi(m^{-1} g_\sigma \sigma(m)) ^l  = \chi(m^{-1} g_\sigma \sigma(g_\sigma) \dots \sigma^{l-1}(g_\sigma) \sigma^l(m)), 
\]
where \( l = [L_p : \mathbb{Q}_p] \), and 
\[
p^{-l \langle \mu'', \chi \rangle} = |\chi(m^{-1} g_\sigma \sigma(g_\sigma) \dots \sigma^{l-1}(g_\sigma) \sigma^l(m))|.
\]
\marginpar{179}By the definition, we also know that
\[
g_\sigma \sigma(g_\sigma) \dots \sigma^{l-1}(g_\sigma) = \xi'_p(p^{-1}) = p^{-\nu_p},
\]
and  
\[
|\chi(p^{-\nu_p})| = p^{\langle \nu_p, \chi \rangle}.
\]
Since
\[
|\chi(m)| = |\chi(\sigma^l(m))|,
\]
we
\marginpar{\textit{QED Lem.~\ref{5.11}.}} have proven equation \eqref{5.n}.\\

Theorem \ref{5.3} is not yet proven. As the next approximation, we prove the following lemma.
 \begin{lem}\label{5.12}
 	There exist  a torus \( T' \)  over  \( \mathbb{Q} \) (inside $G$) with \( T'_{\mathrm{ad}} \) anisotropic over \( \mathbb{R} \), and an $h'$ factoring through  \( T' \), such that \( \phi \) is equivalent to \( \phi' \), where \( \phi' = \psi_{T', \mu_{h'}} \) on the kernel of \( \mathscr Q\), and

\[
\phi'(q_\sigma) = b_\sigma \psi_{T', \mu_{h'}}(q_\sigma),
\]
where \( \{ b_\sigma \} \) is a 1-cocycle with values in \( T' \).
\end{lem}

\marginpar{\textit{Proof of Lem.~\ref{5.12}.}}

Let \( \mu = \omega \cdot \mu_h \) with \( \omega \) in the Weyl group. Then we have a 1-cocycle \( \alpha^\infty : \text{Gal}(L_\infty / \QQ_\infty)  \to T_{\text{sc}}(\mathbb{C}) \) defined by $1 \mapsto 1, \iota \mapsto (-1)^{\omega \mu - \mu }$, assuming that \( L_\infty = \mathbb{C} \). There exists a \( w \) in the normalizer of \( T_{\text{sc}} \) in \( G_{\text{sc}}(\mathbb{C}) \) that represents \( \omega \), and for this, we have:
\[
\alpha^\infty_\sigma = w \sigma(w^{-1}), \quad  \sigma \in \text{Gal}(L_\infty / \QQ_\infty), 
\]  (see [Sh], Prop. 4.2).

According to Lemma 7.16 of [L3], there exists a global cocycle \( \alpha \) with values in \( T_{\text{sc}} \) that is equivalent to \( \alpha^\infty \) at infinity. We can even assume that by appropriately modifying \( \alpha^\infty \) and \( w \), \( \alpha^\infty \) is the restriction of \( \alpha \) to \( \text{Gal}(L_\infty / \QQ_\infty) \). The image of \( \alpha^\infty \) in \( H^1(\QQ_\infty, G_{\text{sc}}) \) is trivial. By the Hasse principle, we conclude that \( \alpha \) defines the trivial class in \( H^1(\QQ, G_{\text{sc}}) \). Let
\[
\alpha_\sigma = u^{-1} \sigma(u), \quad u \in G_{\text{sc}}(\overline{\mathbb{Q}}).
\]
The torus \( T'  = uT u^{-1} \) and the homomorphism \( t \mapsto t' = utu^{-1} \) are both defined over \( \mathbb{Q} \). We replace \( \phi \) by \( \phi' = \mathrm{ad}\,u \circ \phi \), and thus \( \gamma_n \) by \( \gamma_n' = u \gamma_n u^{-1} \). Let \( \mu'\) be the  cocharacter of \( T' \)  corresponding to $\mu$. The homomorphism \( \psi_{T', \mu'} \) is defined, and \( \psi_{T', \mu'} = \phi'\) on the kernel. It follows that 
\[
\phi'(g_\sigma) = b'_\sigma \psi_{T', \mu'}(q_\sigma), \quad b'_\sigma \in T'.
\]
The goal is to show that \( \mu' = \mu_{h'} \), where \( h' \) is still to be defined.

We\marginpar{180} have 
\[
u^{-1} \sigma(u) = w \sigma(w^{-1}), \quad \sigma \in \text{Gal}(L_\infty / \QQ_\infty).
\] Hence \( uw \in G_{\text{sc}}(\mathbb{R}) \). Let \( h' = \mathrm{ad}\,{uw} \circ h \). Then 
\[
\mu_{h'} = \mathrm{ad}\,u \circ \mu = \mu'.
\]
\marginpar{\textit{QED Lem.~\ref{5.12}.}} \\ 

Up to this point, we have not used the assumption \( G_{\text{der}} = G_{\text{sc}} \). All lemmas hold without this assumption. It is only needed in the final step. From the previous lemma, we replace \( \phi, T, \mu \) by \( \phi', T', \mu'= \mu_{h'} \).

The cocycle \( \{ b'_\sigma \} \) takes values in \( T \) and therefore in \( G \). Since \( \phi \) is admissible, the image of \( \{ b_\sigma \} \) in \( G_{\text{ab}} \) represents the trivial class. We can therefore assume that \( b_\sigma \in T_{\text{sc}}(\overline{\mathbb{Q}}) \).

The next lemma is a remark from [De1].

\begin{lem}\label{5.13}
If \( G_{\mathrm{sc}} = G_{\mathrm{der}} \) and \( G_{\mathrm{sc}} \) satisfies the Hasse principle, then the Hasse principle holds for the image of \( H^1(\QQ, G_{\mathrm{sc}}) \) in \( H^1(\QQ, G) \).
\end{lem} 
\marginpar{\textit{Proof of Lem.~\ref{5.13}.}}
We have a diagram with exact sequences:
$$\xymatrix{ H^0(\QQ, G_{\text{ab}}) \ar[r] \ar[d] &  H^1(\QQ, G_{\text{sc}}) \ar[r] \ar[d] &  H^1(\QQ, G)  \ar[d] \\ H^0(\RR, G_{\text{ab}}) \ar[r] &  H^1(\RR, G_{\text{sc}}) \ar[r] &  H^1(\RR, G).  } $$
Moreover, the image of \( H^0(\QQ, G_{\text{ab}}) \) in \( H^0(\RR, G_{\text{ab}}) \) is dense, and a 1-cocycle with values in a small neighborhood of the identity in \( G_{\text{sc}}(\mathbb{C}) \) is a coboundary. Thus the lemma follows.
\marginpar{\textit{QED Lem.~\ref{5.13}.}}
\begin{lem}\label{5.14}
 Let \( G' \) be a connected reductive group over \( \mathbb{R} \) with \( G'_{\mathrm{ad}}(\mathbb{R}) \) compact, and let \( T' \) be a Cartan subgroup of \( G' \). Then \( H^1(\RR, T') \to H^1(\RR, G') \) is injective.
\end{lem} 
\marginpar{\textit{Proof of Lem.~\ref{5.14}.}}
Let \( \{ t_\sigma \} \) be a cocycle with values in \( T' \), which is a coboundary in \( G' \):
\[
t_\sigma = g^{-1} \sigma(g).
\]
Then  \( g T'  g^{-1} \) is also a Cartan subgroup over \( \mathbb{R} \). Since all Cartan subgroups over \( \mathbb{R} \) are conjugate in \( G(\mathbb{R}) \), we can assume that \( T' = g T' g^{-1} \). Then  \( g \) defines an element of the Weyl group, which can also be defined by an element in \( G(\mathbb{R}) \). Therefore  \( \{ t_\sigma \} \) is a coboundary in \( T'(\mathbb{C}) \).
\marginpar{\textit{QED Lem.~\ref{5.14}.}}\\

Theorem \ref{5.3} has\marginpar{181} not been proven yet, but we are at the stage where \( \phi \) can be written as a product of the cocycle \( \{ b_\sigma \} \) and \( \psi_{T, \mu} \). In particular,
\begin{align*}
\phi(q_\sigma) = g_\sigma \rtimes  \sigma, \\ \psi_{T,\mu}(q_\sigma) = h_{\sigma} \rtimes \sigma, \\ g_\sigma = h_\sigma b_\sigma, 
\end{align*} with \( h_\sigma, b_\sigma \) in \( T \).
 
By assumption \eqref{5.b}, for \( \sigma \in \text{Gal}(L_\infty / \mathbb{Q}_\infty) \), we have 
\[
g_\sigma \rtimes  \sigma = u(h_\sigma \rtimes \sigma) u^{-1}, 
\]
with $u \in G(L_\infty)$. 
Hence the restriction of \( \{b_\sigma\} \) to \( \text{Gal}(L_\infty / \mathbb{Q}_\infty) \) is trivial in the twisted group \( G' \), and therefore already trivial in \( T \). From Lemma \ref{5.13}, we conclude that there exists a \( v \in G(L) \) such that
\[
b_\sigma = v\sigma(v^{-1}), \quad \sigma \in \text{Gal}(L/\mathbb{Q}).
\]
Then,
\begin{align}\label{5.o}
v^{-1}(g_\sigma \rtimes \sigma) v = v^{-1}(h_\sigma b_\sigma \rtimes \sigma) v = v^{-1} h_\sigma v \rtimes \sigma.
\end{align} 
But the isomorphism \( t \mapsto v^{-1} t v \) is a \( \mathbb{Q} \)-isomorphism between the two \( \mathbb{Q} \)-tori \( T \) and \( T' = v^{-1} T v \), and if \( \mu' \) is the image of \( \mu \), it follows from \eqref{5.o} that 
\[
\text{ad } v^{-1} \circ \phi = \psi_{T',\mu'}.
\]
Thus, Theorem \ref{5.3} is proven. \marginpar{\textit{QED Thm.~\ref{5.3}.}}\\

Kottwitz's unpublished results show, at least in the case of an anisotropic group \( G \), how to calculate the local zeta function based on a good enumeration of points on the reduced variety using partially proven and partially unproven results from harmonic analysis. The  desired enumeration, to be given in Theorem \ref{5.21} and Theorem \ref{5.25}, arises from the  description conjectured above, but not without further effort. The purpose of the remaining part of this section is to present the necessary considerations for this. We begin with the definition of an invariant introduced by Kottwitz, which we will need and which we will call the Kottwitz invariant (or shorter, the $K$-invariant).

The Shimura variety corresponds to a group \( G \) and a class \( X_\infty \). It is still assumed that \( G_{\text{der}} \) is simply connected. The purpose of [K3] and other works is to determine the alternating sum of the traces on suitable cohomology groups of a product of a power of the Frobenius at a  place \( \mathfrak p \) of \( E \) with a Hecke operator. We do not want to repeat the considerations from [K3] and we do not want to anticipate further considerations from Kottwitz. We only note, for clarity, that when it involves the \( m \)-th power of the Frobenius, the natural number \( n \) that we define now is equal to \( mr \), where \( r = [E_{\mathfrak p} : \mathbb{Q}_p] \). Let \( L_n \subset \overline{\mathbb{Q}}_p \) be the unramified extension of \( \mathbb{Q}_p \) of degree \( n \).

Consider \marginpar{182} a semi-simple element $\varepsilon$ in \( G(\mathbb{Q}) \), an element \( \gamma = (\gamma(l))_{l\neq p} \) with \( \gamma(l) \in G(\mathbb{Q}_l) \), and an element \( \delta \in G(L_n) \). The \( K \)-invariant is associated with the triple \( (\gamma, \delta; \varepsilon) \). Let \( G(\varepsilon) \) denote the centralizer of \( \varepsilon \) in \( G \). It is connected because \( G_{\mathrm{der}} \) is simply connected. Kottwitz denotes the connected components of the \( L \)-groups of \( G \) and \( G(\varepsilon) \) by \( \widehat{G} \) and \( \widehat{G}(\varepsilon) \), and their centers by \( Z(\widehat{G}) \) and \( Z(\widehat{G}(\varepsilon)) \), respectively. The global Galois group \( \text{Gal}(\overline{\mathbb{Q}}/\mathbb{Q}) \) is denoted by \( \Gamma \), and the local Galois group at a place \( v \) by \( \Gamma(v) \). We adopt these notations and simplify \( Z(\widehat{G}) \) and \( Z(\widehat{G}(\varepsilon)) \) to \( Z \) and \( Z(\varepsilon) \), respectively. Kottwitz also considers the group \( (Z(\varepsilon)/Z)^{\Gamma} \) of \( \Gamma \)-invariant elements in the quotient \( Z(\varepsilon)/Z \), as well as the group \( \pi_0((Z(\varepsilon)/Z)^\Gamma) \) of its connected components and its dual group \( X \). We note that \( Z(\varepsilon)/Z = Z(\widehat{H}) \) and \( \widehat{G}(\varepsilon)/Z = \hat H \), where \( H = G(\varepsilon) \cap G_\text{der} \). For each place \( v \), let \( Y_v \) be the group of those characters of \( \pi_0((Z(\varepsilon)/Z)^{\Gamma(v)}) \) that are trivial on \( \pi_0(Z(\varepsilon)^{\Gamma(v)}) \), and \( X_v \) its image in \( X \). The invariant \( \kappa(\gamma, \delta; \varepsilon) = \kappa(\gamma, \delta) \) associated with the triple is in the group \( X / \prod_v X_v \). This is the group \( \mathfrak K(I_0/F)^D \), which appears in formula (6.5.1) of [K5].

To define \( \kappa(\gamma, \delta) \), we need to impose some conditions on \( (\gamma, \delta; \varepsilon) \). First, we require:
\begin{enumerate}
	\item [(i)] \( \varepsilon \) is elliptic in \( G(\mathbb{R}) \). 
	\item [(ii)] For every \( l \neq p \), \( \gamma(l) \) and \( \varepsilon \) are stably conjugate. Since \( G_\text{der} \) is simply connected, this means that \( \varepsilon \) and \( \gamma(l) \) are conjugate in \( G(\overline{\mathbb{Q}}_l) \). For almost all \( l \), \( \varepsilon \) and \( \gamma(l) \) should be conjugate in \( G(\mathbb{Q}_l) \).
	\item [(iii)] If \( \sigma \) denotes the Frobenius element in \( \text{Gal}(\mathbb{Q}_p^\text{un}/\mathbb{Q}_p) \) and 
	\[
	N_{L_n/\mathbb{Q}_p} \delta = \delta \sigma(\delta) \cdots \sigma^{n-1}(\delta),
	\]
	then \( N_{L_n/\mathbb{Q}_p} \delta \) should be conjugate to \( \varepsilon \) in \( G(\overline{\mathbb{Q}}_p) \).
\end{enumerate} 

Every \( h \in X_\infty \) can be written as 
\[
h(z, w) = z^{\mu_ h }w ^{\bar{\mu}_h} = z^\mu w^{\bar{\mu}}.
\]

Let \(\mathscr M_{\CC} \) be the conjugacy class of homomorphisms from \( \mathbb G_m \) to \( G \) containing \( \mu \). It depends only on \( X_\infty \) and not on \( h \), and it also defines \( \mathscr M_{\overline{\mathbb{Q}}} \) and \( \mathscr M_{\overline {\mathbb{Q}}_p} \), as \( \overline{\mathbb{Q}} \) is embedded in both \( \mathbb{C} \) and \( \overline{\mathbb{Q}}_p \).

All elements of \(\mathscr  M_{\overline{\mathbb{Q}}_p} \) (or \( \mathscr M_{\CC} \)) define the same homomorphism \( \mu_2 \) in \[ \text{Hom}(\mathbb G_m, G_\text{ab}) = X^*(Z). \]
On the other hand, \( \delta \) clearly defines an element of the set \( B(G_\text{ab}) \) introduced in [K4] and, consequently, by the theorem proven there, an element of \( X^*(Z^{\Gamma(p)}) \). It is important to emphasize that \( Z = \hat{G}_\text{ab} \). We impose the following condition on \( \delta \).

\begin{enumerate}
	\item[(\(  *(\delta)  \))]\label{*delta} The element\marginpar{183}  in \( X^*(Z^{\Gamma(p)}) \) defined by \( \delta \) is the restriction of \( \mu_2 \).
\end{enumerate} 

The second condition is a requirement on \( \varepsilon \). Let \( M = M(\varepsilon) \) denote the centralizer in \( G \) of the maximally split torus over \( \mathbb{Q}_p \) in the center of \( G(\varepsilon) \). If \( \mu : \mathbb G_m \to M \), then \( \mu \) can be factored through at least one Cartan subgroup \( T \), and therefore defines an element of \( X_*(T) \), whose conjugacy class is well-defined. For any Cartan subgroup \( \hat{T} \) of \( \hat{M} \), \( \mu \) thus defines an orbit in \( X^*(\hat{T}) \) under the Weyl group of \( M \), and finally a well-defined element \( \mu_1 \in X^*(Z(\hat{M})) \) that does not depend on \( T \) or \( \hat{T} \). Since \( G \) splits over an unramified extension of \( \mathbb{Q}_p \), so does \( M \), and the homomorphism \( \lambda = \lambda_M : M(\mathbb{Q}_p) \to X^*(Z(\hat{M}))^{\Gamma(p)} \) defined in [K3], §3, is available.
  
The second condition is as follows:  
\begin{enumerate}
	\item [(\( \ast(\varepsilon) \))] \label{*epsilon}There exists a \( \mu : \mathbb G_m \to M \) defined over \( L_n \), which, through the embedding \( M \subset G \), defines an element in \( \mathscr  M_{\overline{\mathbb{Q}}_p} \), and such that the equation  
	\[
	\lambda(\varepsilon) = N_{L_n/\mathbb{Q}_p} \mu_1
	\]
	holds.  
\end{enumerate}

The invariant \( \kappa(\gamma, \delta; \varepsilon) \) is defined as a character on the preimage \( U \) of  
\( 
(Z(\varepsilon)/Z)^\Gamma
\)  
in \( Z(\varepsilon) \), specifically as a product: 
\begin{align}\label{5.p}
\kappa(\gamma, \delta; \varepsilon) = \prod_v \beta(v).
\end{align}  
Here, \( \beta(v) \) is the restriction to \( U \) of a character \( \beta'(v) \) on the preimage \( U_v \) of  $
(Z(\varepsilon)/Z)^{\Gamma(v)}.$ 
Almost all \( \beta'(v) \) are trivial, and the product is trivial on \( Z \).  

For \( v = l \neq p \), \( \varepsilon \) and \( \gamma_l \) are stably conjugate. Consequently, there exists a \( c \in G(\overline{\mathbb{Q}}_l) \) such that  
\[
\gamma_l = c \varepsilon c^{-1},
\]  
and \( \{b_\sigma\} = \{c^{-1} \sigma(c)\} \) is a 1-cocycle of \( \Gamma(v) \) with values in \( G(\varepsilon) \). According to [K2], Proposition 6.4, this cocycle defines an element of  $
\pi_0(Z(\varepsilon)^{\Gamma(v)})^D.$ Since the image of the cocycle in \( G_\text{ad} \) is trivial, this element is trivial on \( \pi_0(\hat{G}_\text{ab}^{\Gamma(v)}) = \pi_0(Z^{\Gamma(v)}) \). Therefore, it defines a character of  $
\pi_0(Z(\varepsilon)^{\Gamma(v)}) / \pi_0(Z^{\Gamma(v)}),$ 
which we can extend to \( \pi_0((Z(\varepsilon)/Z)^{\Gamma(v)}) \) by noting that the connected component of  $
(Z(\varepsilon)/Z)^{\Gamma(v)}$ 
is the image of the connected component of \( Z(\varepsilon)^{\Gamma(v)} \).   
By lifting the character thus obtained from \( U_v \), we get \( \beta'(v) \). It is trivial on \( Z \).

 At the place \( p \), we have:  
 \[
 N_{L_n / \mathbb{Q}_p} \delta = c \varepsilon c^{-1},
 \]
 where \( c \) is an element of \( G \) with coefficients in \( \overline{\mathbb{Q}}_p \). By Steinberg's theorem, we can assume that \( c \in G(\mathbb{Q}_p^\text{un}) \). We have 
 \[
 \delta^{-1} c \varepsilon c^{-1} \delta = \sigma(\delta) \cdots \sigma^n(\delta) = \sigma(c \varepsilon c^{-1}) = \sigma(c) \varepsilon \sigma(c^{-1}).
 \]
 Thus, \( b = c^{-1} \delta \sigma(c) \) commutes with \( \varepsilon \) and defines (using the notation of [K4]) an element of  $
 B(G(\varepsilon)), $
 which depends only on \( \varepsilon \) and \( \delta \).

 \begin{lem}\label{5.15}
\( b \) is basic in the sense of [K4].   \marginpar{184}
 \end{lem} 

\marginpar{\textit{Proof of Lem.~\ref{5.15}.}}
 Let \( n' \) be divisible by \( n \), so that \( n' = kn \). Then  
 \[
 b \sigma(b) \cdots \sigma^{n-1}(b) = \varepsilon c^{-1} \sigma^n(c),
 \]
 and so, first, \( c^{-1} \sigma^n(c) \in G(\varepsilon) \), and second, 
 \[
 b \sigma(b) \cdots \sigma^{n'-1}(b) = \varepsilon^k c^{-1} \sigma^{n'}(c).
 \]
 From the definitions ([K4], §4), it must be shown that for sufficiently large \( n' \), there exists a homomorphism \( \nu' \) from \( \mathbb G_m \) to the center of \( G(\varepsilon) \) such that   
 \[
 \varepsilon^k c^{-1} \sigma^{n'}(c) = \nu'(p) d^{-1} \sigma^{n'}(d),
 \]
 with \( d \in G(\mathfrak k, \varepsilon) \).\footnote{Here $G(\mathfrak k, \varepsilon) = G(\varepsilon)(\mathfrak k)$.  } 
 Let \( T \) be a Cartan subgroup of \( G(\varepsilon) \) over \( \mathbb{Q}_p \). Let 
 \[
 |\lambda(\varepsilon)|_p^k = p^{-\nu'(\lambda)}, \quad \lambda \in X^*(T).
 \]
 If \( k \) is sufficiently large, then \( \nu'(\lambda) \in \mathbb{Z} \) for every \( \lambda \in X^*(T) \), and consequently \( \nu' \in X_*(T) \). It is clear that \( \nu' \) is invariant under the Galois group. Therefore, \( \nu' : x \mapsto x^{\nu'} \) is a homomorphism from \( \mathbb G_m \) to \emph{the center} of \( G(\varepsilon) \) defined over \( \mathbb{Q}_p \).   For sufficiently large \( n' \), \( c^{-1} \sigma^{n'}(c) \) is contained in a given neighborhood of $1$. Since the eigenvalues of \( \nu'(p)^{-1} \varepsilon^k \) are units, we can replace \( k \) with a large multiple \( ak \), and \( \nu' \) with \( a \nu' \), and assume that
 \[
 \varepsilon^k c^{-1} \sigma^{n'}(c) = \nu'(p) \cdot e,
 \]
 where \( e \) lies in a given neighborhood \( V \) of 1 in \( G(\mathfrak k, \varepsilon) \). But as implicitly noted in the proof of [K3], Lemma 1.4.9, there exists a \( V \) such that for any natural number \( n' \) we have $
 V \subset \{x \sigma^{n'}(x)^{-1} \mid x \in G(\mathfrak k, \varepsilon)\}. $  \marginpar{\textit{QED Lem.~\ref{5.15}.}} \\
 
 Let \( \alpha(p) \) denote the element of \( X^*(Z(\varepsilon)^{\Gamma(p)}) \) corresponding to \( b \) in the sense of [K4]. Due to condition \( \ast(\delta) \), the restriction of \( \alpha(p) \) to \( Z^{\Gamma(p)} \) must equal the restriction of \( \mu_2 \). Consequently, there exists a character \( \beta'(p) \) of \( U_p \), which equals \( \alpha(p) \) on \( Z(\varepsilon)^{\Gamma(p)} \) and \( \mu_2 \) on \( Z \). Moreover, \( \beta'(p) \) is uniquely determined up to a character of $ \pi_0((Z(\varepsilon)/Z)^{\Gamma(p)}),
 $ which is trivial on $
 \pi_0(Z(\varepsilon)^{\Gamma(p)}).$
 
 To define \( \beta'(\infty) \), we choose an \( h \in X_\infty \) that factors through \( G(\varepsilon) \), which is possible because \( \varepsilon \) is elliptic. Let \( h(z, w) = z^\mu w^{\overline{\mu}} \).  
 
\begin{lem}\label{5.16}
	\begin{enumerate}
		\item [(a)] The restriction \( \alpha(\infty) \) of \( -\mu \) to \( Z(\varepsilon)^{\Gamma(\infty)} \) does not depend on the choice of \( h \).  
		\item[(b)] On \( Z^{\Gamma(\infty)} \), \( \alpha(\infty) \) equals \( -\mu_2 \). 
	\end{enumerate}
\end{lem}  

\marginpar{\textit{Proof of Lem.~\ref{5.16}.}}
The lemma, whose second statement is already clear, allows us to define $\beta'(\infty)$ in the same way we defined $\beta'(p)$, except that it should equal $-\mu_2$ on $Z$. The first statement is to be understood as follows: if \( \hat{T} \) is a Cartan subgroup of \( \hat{G}(\varepsilon) \), then $\mu$  defines an orbit under the Weyl group \( \Omega(\hat{T}, \hat{G}(\varepsilon)) \) in \( X^*(\hat T) \).  All elements in the orbit have the same restriction on \( Z(\varepsilon) \) and  hence on  \( Z(\varepsilon)^{\Gamma(\infty)} \). This handles the restriction.  \marginpar{\textit{QED Lem.~\ref{5.16}.}}\\

Let\marginpar{185} \( h_1 \) and \( h_2 \) be two elements of \( X_\infty \) that factor through \( G(\varepsilon) \). We want to prove that the corresponding \( \mu_1 \) and \( \mu_2 \) have the same restriction to \( Z(\varepsilon)^{\Gamma(
\infty)} \). This is clear if \( h_1 \) and \( h_2 \) are conjugate under \( G(\varepsilon, \mathbb{R}) \). We can therefore assume that they also factor through a common elliptic Cartan subgroup \( T \subset G(\varepsilon) \). Let \( h_2 = h _1^g = \mathrm{ad}\,{g^{-1}} \circ h_1, g \in G(\mathbb{R}) \). Then \( K_2 = K_1 ^g \), where \( K_i \) is the centralizer of \( h_i \) in \( G(\mathbb{R}) \). Consequently, \( T \) and \( T^g \) are Cartan subgroups of \( K_2 \).  We can thus replace \( g \) with \( gk \), \( k \in K_2 \), and assume \( T = T^g \). The element \( g \) represents an element of the Weyl group, which we associate with a class in \( H^1(\mathbb{R}, T) \) in the usual manner ([Sh]). This class is defined as  
\[ 
\iota \mapsto (-1)^{g^{-1} \mu_1 - \mu_1} = (-1)^{\mu_2 - \mu_1}
\]  
([MS2], Prop. 4.2). On the other hand, it is trivial. Hence \( \mu_2 - \mu_1 \) defines the trivial class in \( H^1(\mathbb{R}, X_*(T)) \), from which it immediately follows that \( \mu_2 - \mu_1 \) is trivial on \( \hat{T}^{\Gamma(\infty)} \). In particular, it is trivial on \( Z(\varepsilon)^{\Gamma(\infty)} \), because \( Z(\varepsilon) \subset \hat T \) and the actions of \( \Gamma(\infty) \) on both groups are compatible.  

It remains to show that the product   $\prod \beta(v)$ 
on the identity component \( [(Z(\varepsilon)/Z)^\Gamma]^0 \) is trivial, i.e., that   $
\beta(\infty) \cdot \beta(p) $ 
is trivial on \( [Z(\varepsilon)^\Gamma]^0 \). We consider the image of \( \alpha(p) \) under the composition of maps  
\[
X^*(Z(\varepsilon)^{\Gamma(p)}) \to X^*(Z(\hat{M})^{\Gamma(p)}) \to X^*(Z(\hat{M})^{\Gamma(p)}) \otimes \mathbb{Q}   
\]  
$$  X^*(Z(\hat{M})^{\Gamma(p)}) \otimes \mathbb{Q}    =  X^*(Z(\varepsilon) )^{\Gamma(p)} \otimes \mathbb{Q}  \to X^*(Z(\varepsilon) )^{\Gamma} \otimes \mathbb{Q}  .   $$

The image of \( \alpha(p) \) in \( X^*(Z(\hat{M})^{\Gamma(p)}) \otimes \mathbb{Q} \) is equal to  
\[
\frac{1}{n}\lambda_M(\text{Nm}_{L_n/\mathbb{Q}_p}{b}) |_{Z(\hat{M})^{\Gamma(p)}} = \frac{1}{n}\lambda_M(\varepsilon)|_{Z(\hat{M})^{\Gamma(p)}}.
\]  
On the other hand, the image of \( \mu \in X^*(Z(\varepsilon)^{\Gamma(\infty)}) \) in \( X^*(Z(\varepsilon))^\Gamma \otimes \mathbb{Q} \) equals the image of \( \mu_1 \in X^*(Z(\hat{M})) \), that is,  
$ 
\frac{1}{n}\text{Nm}_{L_n/\mathbb{Q}_p}\mu_1. $ 
The claim follows from condition \( *(\varepsilon) \).  

The invariant \( \kappa(\gamma,\delta) \) is now defined via the product \eqref{5.p}. Let \( \varepsilon \) and \( \varepsilon' \) be stably conjugate elements of \( G(\mathbb{Q}) \). Conditions (i), (ii), (iii) also hold for \( (\gamma, \delta, \varepsilon') \), as well as condition \( *(\delta) \).  

\begin{lem}\label{5.17}
Let \( \varepsilon \in G(\mathbb{Q}_p) \) satisfy condition \( *(\varepsilon) \), and let \( \varepsilon ' \in G(\mathbb{Q}_p) \) be stably conjugate to \( \varepsilon \). Then \( \varepsilon' \) satisfies condition \( *(\varepsilon') \). 
\end{lem} 

\marginpar{\textit{Proof of Lem.~\ref{5.17}.}}
 Let \( \varepsilon' =  \varepsilon^ g \), \( g \in G(\overline{\mathbb{Q}}_p ) \). Then \( M' = M^g \), and the induced map \( M_{\text{ab}} \to M'_{\text{ab}} \) is defined over \( \mathbb{Q}_p \). This equation also allows us to identify \( Z(\hat{M}) \) with \( Z(\hat{M}') \). Then \( \mu \) and \( \mathrm{ad}\,{g^{-1}} \circ 
\mu \) correspond to the same orbit in \( X_*(T) \), so \( \mu_1 = \mu'_1 \).  
Since \( \mu \) is defined over \( L_n \), the conjugacy class of \( \mathrm{ad}\,{g^{-1}} \circ \mu \) is defined over \( L_n \). From [K3], Lemma 1.1.3, it follows that there exists \( \mu ': \mathbb G_m \to M' \) defined over \( L_n \), which is conjugate to \( \mathrm{ad}\,{g^{-1}} \circ \mu \) in \( M'(\overline{\mathbb{Q}}_p) \). Since \( G \), and hence \( M' \), is quasi-split over \( \mathbb{Q}_p \), if we can show that \( M_{\text{der}} \) and \( M'_{\text{der}} \) are simply connected then we immediately get \( \lambda_M(\varepsilon) = \lambda_{M'}(\varepsilon') \). The lemma will then be proven.

 We  \marginpar{186} consider \( M \) alone and show, for a later purpose, more than is currently necessary. \( M \) is a Levi factor of a quasi-split group \( G \) defined over \( \mathbb{Q}_p \), with \( G_{\text{der}} \) simply connected. Since  
$ M_{\text{der}} \subset M \cap G_{\text{der}},$   
 we temporarily assume that \( G = G_{\text{der}} \). Let \( T \subset M \) be a Cartan subgroup of \( G \) defined over \( \mathbb{Q}_p \) which is contained in a Borel subgroup \( B \) defined over \( \mathbb{Q}_p \). The simple roots \( \alpha_1, \ldots, \alpha_s \) of \( T \) in \( G \) with respect to \( B \) can be divided into two sets: the roots \( \alpha_1, \ldots, \alpha_r \) of \( T \) in \( M \), and the others \( \alpha_{r+1}, \ldots, \alpha_s \).  
 As a basis for \( X^*(T) \), we choose \( \lambda_1, \ldots, \lambda_s \), where \( \lambda_i(\hat{\alpha}_j) = \delta_{ij} \), and \( \hat{\alpha}_j \) is the coroot corresponding to \( \alpha_j \). Since the coroots with respect to \( G \) are also coroots with respect to \( M \), \( M_{\text{der}} \) is simply connected.  
 Furthermore, \( M \) is the semidirect product of \( M_{\text{der}} \) and the kernel \( R \subset T \) of \( \{ \lambda_1, \ldots, \lambda_r \} \). The group \( R \) is defined over \( \mathbb{Q}_p \), and \( \{ \lambda_{r+1}, \ldots, \lambda_s \} \) forms a basis for \( X^*(R) \). Since \( \lambda_{r+1}, \ldots, \lambda_s \) are permuted by the Galois group, it follows that \( H^1(\mathbb{Q}_p, R) = \{ 1 \} \) and \( H^1(\mathbb{Q}_p, M) = \{ 1 \} \). In the general case, only \( H^1(\mathbb{Q}_p, M \cap G_{\text{der}}) = \{ 1 \} \) holds.  
 \marginpar{\textit{QED Lem.~\ref{5.17}.}}
 
 \begin{lem}\label{5.18}
 \( \kappa(\gamma, \delta; \varepsilon') = \kappa(\gamma, \delta; \varepsilon) \).  
 \end{lem} 
 
 To make sense of this lemma, we must identify the groups in which the invariants lie. Let \( \varepsilon' = g^{-1} \varepsilon g \), with \( g \in G_{\text{der}}(\overline{\mathbb{Q}}) \). Then \( u \mapsto u' = g^{-1} u g \) is an isomorphism from \( G(\varepsilon) \) to \( G(\varepsilon') \), which allows us to treat \( G(\varepsilon') \) as an inner twist of \( G(\varepsilon) \), since for every \( \sigma \in \text{Gal}(\overline{\mathbb{Q}} / \mathbb{Q}) \), the element \( a_\sigma = g \sigma (g^{-1}) \) lies in \( G(\varepsilon) \).  
 We can therefore identify \( \hat{G}(\varepsilon') \) with \( \hat{G}(\varepsilon) \) as well as \( Z(\varepsilon') \) with \( Z(\varepsilon) \), and consequently \( X' /\prod_v X_v' \) with \(X /\prod_v X_v \). We identify \( G(\varepsilon') \) with \( G(\varepsilon) \) via the map \( u \mapsto u' \). This results in a new action of the Galois group on \( G(\varepsilon) \), defined by the equations:  
 \[ g^{-1} \sigma'(u) g = \sigma(g^{-1} u g). \]  
 Thus we have  \( \sigma'(u) = a_\sigma \sigma(u) a_\sigma^{-1} \).  
 
 To prove Lemma \ref{5.18}, we require two additional lemmas, which we prove not only for \( G(\varepsilon) \) but also for any (reductive) group \( G \). Let \( \{a_\sigma\} \) be a cocycle of \( \text{Gal}(\overline{\mathbb{Q}} / \mathbb{Q}) \) with values in \( G \), and let \( G' \) be the corresponding inner twist. It is easy to verify that the map \( \{v_\sigma\} \mapsto \{v'_\sigma = v_\sigma a_\sigma^{-1}\} \) defines a bijection  $ \xi: H^1(\mathbb{Q}, G) \to H^1(\mathbb{Q}, G'), $ 
 as well as local bijections  
 \[ \xi(v): H^1(\mathbb{Q}_v, G) \to H^1(\mathbb{Q}_v, G'). \]  
 As above, we can identify \( Z' = Z(\hat G') \) with \( Z \), and consequently \( \pi_0(Z^{\prime \Gamma(v)})^D \) with \( \pi_0(Z^{\Gamma(v)})^D \). The cocycle \( \{a_\sigma\} \) defines a class \( \alpha \in H^1(\mathbb{Q}, G) \) with localizations \( \alpha(v) \). Let \( \delta(v) \) be the element of \( \pi_0(Z^{\Gamma(v)})^D \) corresponding to \( \alpha(v) \) (see [K2], §6.4, for a finite place, and [K5], §1.2, for the general case), and \( \eta(v) \) the isomorphism arising from multiplication by \( \delta(v)^{-1} \):  
 \[ \pi_0(Z^{\Gamma(v)})^D \to \pi_0(Z^{\Gamma(v)})^D = \pi_0(Z^{
 	\prime \Gamma(v)})^D. \]

\begin{lem}\label{5.19}
 The diagram   \marginpar{187}
 \[  \xymatrix{   
 	H^1(\mathbb{Q}_v, G)  \ar[r] ^ {\xi(v)} \ar[d] & H^1(\mathbb{Q}_v, G') \ar[d] \\  
 	\pi_0(Z^{\Gamma(v)})^D \ar[r]_{\eta(v)} & \pi_0(Z^{\prime \Gamma(v)})^D   }
 \]  
 is commutative.  
\end{lem}

\marginpar{\textit{Proof of Lem.~\ref{5.19}.}}
 We require a central extension  
 \[ 1 \to Z \to H \to G \to 1, \]  
 with \( H_{\text{der}} \) simply connected, for which \( H^1(\mathbb{Q}_v, H) \to H^1(\mathbb{Q}_v, G) \) is surjective. For a non-archimedean place, such an extension is constructed in [K2], §6.4. For an archimedean place, any \( z \)-extension (see [K1], §1), for which \( Z \) is a torus induced from \( \mathbb{C} \), has the desired property.  It suffices to prove the lemma for \( H \) instead of \( G \). Then, according to the definitions, we can replace \( H \) by \( H_{\text{ab}} \), and for \( H_{\text{ab}} \) the claim is clear.
  \marginpar{\textit{QED Lem.~\ref{5.19}.}}\\
  
 The second lemma concerns an element \( a \in G(\mathfrak k) \) such that for a sufficiently large \( n \) the product  
 $ a \sigma(a) \cdots \sigma^{n-1}(a) $   
 is central in \( G \), where \( \sigma \) is the Frobenius element. Then \( a \) defines a twist \( G' \) of \( G \), for which  
 \[ G'(\QQ_p) = \{ g \in G(\mathfrak k) \mid a \sigma(g) a^{-1} = g \}. \]  
 We have a bijection \( \xi_p \) from \( B(G) \) to \( B(G') \) defined  by  
$ b \mapsto b' = ba^{-1}. $ 
  
 We have  
 \[ b' \sigma' (b') \cdots \sigma^{\prime n-1}(b') = b \sigma(b) \cdots \sigma^{n-1}(b) \sigma^{n-1}(a)^{-1} \cdots \sigma(a)^{-1} a^{-1}. \]  
 It then follows easily from [K4], Prop.~4.3.3, that \( \xi_p \) is also a bijection between \( B(G)_b \) and \( B(G')_b \). The element \( \delta(p) \) associated with \( a \) lies in \( X^*(Z^{\Gamma(p)}) \). By multiplying with \( \delta(p)^{-1} \), we obtain an isomorphism  
 \[ \eta(p):  X^*(Z^{\Gamma(p)}) \to X^*(Z^{\Gamma(p)}) = X^*(Z^{\prime \Gamma(p)}). \]

 \begin{lem}\label{5.20}
 	The diagram  
 \[  
\xymatrix{ 
 	B(G)_b \ar[r]^{\xi(p)}  \ar[d] & B(G')_b \ar[d] \\    
 	X^*(Z^{\Gamma(p)})\ar[r]_{\eta(p)} & X^*(Z^{\prime \Gamma(p)})   }
 \]  
 is commutative.   
\end{lem}

 \marginpar{\textit{Proof of Lem.~\ref{5.20}.}}
 This \marginpar{188} lemma is proved in exactly the same way as the previous one. One chooses, as in [K2], §6.4, a central extension  
 \[ 1 \to Z \to H \to G \to 1, \]  
 with \( H_{\text{der}} = H_{\text{sc}} \), and uses the definition from [K4].
  \marginpar{\textit{QED Lem.~\ref{5.20}.}}\\

     \marginpar{\textit{Proof of Lem.~\ref{5.18}.}}
 To prove Lemma \ref{5.18}, we apply the two lemmas to \( G(\varepsilon) \cap G_{\text{der}} \). According to Prop.~2.6 from [K5], the lemma would be proven if we could show that when we replace \( \varepsilon \) with \( \varepsilon' \), \( \beta'(v) \) can be replaced by \( \delta^{-1}(v) \beta'(v) \).
 Firstly, it should be noted that we had chosen \( g \in G_{\text{der}}(\overline{\mathbb{Q}}) \) so that the cocycle \( \{a_{\rho}\} \) lies in \( G(\varepsilon) \cap G_{\text{der}} \), and secondly, it should be remembered that \( Z(\varepsilon)/Z = Z(\hat{H}) \), where \( H = G(\varepsilon) \cap G_{\text{der}} \). 
 
 When \( v = l \neq p \), we have  
 \[ \gamma = c\varepsilon c^{-1} = cg\varepsilon' g^{-1} c^{-1} \]  
 and  
 \[ g^{-1}c^{-1}\rho(c)\rho (g) = g^{-1}(b_{\rho}a_{\rho}^{-1})g . \]  
 Therefore, the class \( \alpha(v) \) of \( \{b_{\rho}\} \) is replaced by \( \xi(v)(\alpha(v)) \). We can thus replace \( \beta'(v) \) by \( \delta^{-1}(v)\beta'(v) \).
 
 At the place \( p \), we have 
 \[ \text{Nm}_{L_n/\mathbb{Q}_p}(\delta) = cg\varepsilon' g^{-1}c^{-1}. \]  
 Unfortunately, \( g \) may not lie in \( G_{\text{der}}(\mathbb{Q}_p^{\text{un}}) \). However, it can be written as a product \( g = k g_1 \) with \( g_1 \in G_{\text{der}}(\mathbb{Q}_p^{\text{un}}) \) and \( k \in G_{\text{der}}(\overline{\mathbb{Q}}_p) \cap G(\varepsilon) \). 
 
 Thus we have  
 \[ g \rho(g)^{-1} = k g_1\rho(g_1)^{-1}
 \rho(k)^{-1}, \]  
 so \( g \) and \( g_1 \) give the same class in \( H^1(\mathbb{Q}_p, G(\varepsilon) \cap G_{\text{der}}) \). The identification of \( \hat{G}(\varepsilon') \) and \( \hat{G}(\varepsilon) \) defined by \( g \) is also the one defined by \( g_1 \). Therefore, we can replace \( g \) with \( g_1 \) if we want to show that \( \beta'(v) \) can be replaced by \( \delta^{-1}(v) \beta'(v) \).
 
 We have  
 \[ \text{Nm}_{L_n/\mathbb{Q}_p}(\delta) = cg_1\varepsilon' g_1^{-1} c^{-1}. \]  
 Thus, we replace \( b \) with  
 \[ b' = g_1^{-1} \left( b {\sigma}(g_1) g_1^{-1} \right) g_1, \]  
 so that the desired replacement follows from Lemma \ref{5.19} and the compatibility ([K4]) of the two bijections \( B(G)_b \to X^*(Z^{\Gamma(p)}) \) and \( H^1(\mathbb{Q}_p, G) \to \pi_0(Z^{\Gamma(p)})^D \).

 For \marginpar{189} the infinite case, the claim that \( \beta'(\infty) \) can be replaced by \( \delta^{-1}(\infty)\beta'(\infty) \) is local and transitive. It holds when \( \varepsilon \) and \( \varepsilon' \) are conjugated within \( G_{\text{der}}(\RR) \). Hence, we can assume that there is a Cartan subgroup \( T \) of \( G(\varepsilon) \) defined over \( \RR \) such that \( T^g = T \). We can further assume that \( h : \mathbb S  \to T \). 
 When we identify the two groups \( Z(\varepsilon) \) and \( Z(\varepsilon') \) via \( g \), then \( \alpha'(\infty) \) is the restriction of \( -g\mu \) to \( Z(\varepsilon)^{\Gamma(\infty)} \), i.e., the product of \( \alpha(\infty) \) with the inverse of the restriction of \( g\mu - \mu \). According to [K5], Th.~1.2, and [MS], Prop.~4.2, this restriction is precisely \( \delta(\infty) \).
\marginpar{\textit{QED Lem.~\ref{5.18}.}}\\

 We now return to our presumed description of points on a reduced Shimura variety. The conditions of Theorem \ref{5.3} are assumed to hold. Let \( m \) be a given natural number. We consider pairs \( (\phi, \varepsilon) \) with the following properties:
 \begin{enumerate}
 	\item [(i)] \( \phi : \mathscr{Q} \to \mathscr G_G \) is admissible.  
 	\item [(ii)] $\varepsilon \in I_\phi$. 
 	\item[(iii)] There exist \( \gamma = (\gamma(l)) \in G(\mathbb A_f^p) \), \( y \in X^p \), and \( x \) in the \( G(\mathfrak k) \)-orbit of the point \( x^0 \), such that  
 	\[
 	\varepsilon y = y \gamma \quad \text{and} \quad \varepsilon x = \Phi^ m x.
 	\]
\end{enumerate}

 Such pairs are called \emph{admissible}. It is not required that \( x \) lies in \( X_p \). If \( y \) is replaced by \( y h \), where \( h \in G(\mathbb A_f^p) \), then \( \gamma \) changes to \( h^{-1} \gamma h \). Therefore, the conjugacy class of \( \gamma \) is uniquely determined by the pair \( (\phi, \varepsilon) \). Since \( \mathscr I_{\phi} \) is a twisted form of a subgroup of \( G \), it is meaningful both globally and locally to state that \( \varepsilon \) is stably conjugate to an element of \( G \). In particular, it is stably conjugate to each \( \gamma(l) \). It also follows from Lemma \ref{5.23} below that \( \varepsilon \) is stably conjugate to an element \( \varepsilon_0 \in G(\mathbb{Q}) \). Then, \( \varepsilon_0 \) is stably conjugate to each \( \gamma(l) \). It follows easily from the definition of \( X_p \) and Greenberg's Lemma [G] that \( \varepsilon_0 \) and \( \gamma(l) \) are conjugate almost everywhere.
 
 To define \( X_p \), we have replaced \( \xi_p \) with \( \xi_p' \) and introduced \( F = \xi(w) = b \times \sigma \), where \( b \in G(\mathfrak k) \). If \( \xi_p' = \mathrm{ad}\,u \circ \xi_p \), we must replace \( \varepsilon \) with \( \varepsilon' = u \varepsilon u^{-1} \) in order to obtain its action on \( X_p \) and also for the condition \( \varepsilon x = \Phi^m x \) on \( \mathscr B(G, \mathfrak  k) \). Therefore, the condition should rather be \( \varepsilon' x = \Phi ^m x \).
 According to Lemma 1.4.9 from [K3], condition (iii) implies the existence of \( c \in G(\mathfrak k) \) and \( \delta \in G(L_n) \) such that  
$
 c b \sigma c^{-1} = \delta \sigma $ and $\text{Nm}_{L_n/\mathbb{Q}_p} \delta = c \varepsilon' c^{-1} = c u \varepsilon u^{-1} c^{-1} = c u g \varepsilon_0 g^{-1} u^{-1} c^{-1}, $  if $ \varepsilon = g \varepsilon_0 g^{-1}.
$ 
 Thus, the triple \( \gamma, \delta, \varepsilon_0 \) satisfies conditions (ii) and (iii) (where \( \varepsilon \) is replaced by \( \varepsilon_0 \)), which were used to define the \( K \)-invariant. That condition (i) is also satisfied follows immediately from \eqref{5.b}. The condition \( *(\delta) \) follows from \eqref{5.a}, because \( b \) and \( \delta \) define the same class in \( B(G_{\mathrm{ab}}) \). The condition \( *(\varepsilon) \) is also satisfied, as follows from Theorem \ref{5.21}. We can therefore introduce the invariant \( \kappa(\gamma, \delta; \varepsilon_0) \). We immediately emphasize that the \( b \) that appears in the definition of the invariants is the same as the one that appears in \( F = b \times \sigma \).
 
 Theorem \ref{5.21} below describes those \( \varepsilon_0 \) for which there exists an admissible pair \( (\phi, \varepsilon) \) with \( \varepsilon \) stably conjugate to \( \varepsilon_0 \). Two admissible pairs \( (\phi, \varepsilon) \) and \( (\phi', \varepsilon') \) are called \emph{equivalent}, if there exists \( g \in G(\overline{\mathbb{Q}}) \) such that  
$
 \phi' = \mathrm{ad}\,g \circ \phi,   \varepsilon' = \mathrm{ad}\,g (\varepsilon). $
 Theorem \ref{5.25} counts the classes of \( (\phi, \varepsilon) \) with \( \varepsilon \) stably conjugate to a given \( \varepsilon_0 \). For the significance of this enumeration, we refer to [K3] and other unpublished works by Kottwitz.
 
 \begin{thm}\label{5.21}
 There\marginpar{190} exists an admissible pair \( (\phi, \varepsilon) \) with \( \varepsilon \) stably conjugate to \( \varepsilon_0 \) if and only if \( \varepsilon_0 \) is elliptic at infinity and satisfies the condition \( *(\varepsilon) \) with \( n = r m \).
 \end{thm}

 \marginpar{\textit{Proof of Thm.~\ref{5.21}.}}
 The fact that \( \varepsilon_0 \) must be elliptic at infinity follows from conditions (i) and (ii) for admissibility. To prove the necessity of the second condition, we need some preparation that will also be useful in the proof of Theorem \ref{5.25}.
 
 The map \( \phi \) is called \emph{well-positioned}\footnote{Translated from the German term ``\textit{g\"unstig gelegen''.}  } if \( \gamma_n = \phi(\delta_n) \) lies in $G(\QQ)$ for large \( n \). According to Lemma \ref{5.4}, we can assume that \( \phi \) is favorable. Then \( I_{\phi} = \mathscr I_{\phi}(\mathbb{Q}) \)  where \( \mathscr I_{\phi} \) is a twisted form of the centralizer \( \mathscr I \) of \( \gamma_n \) (for sufficiently large \( n \)). If \( \phi(q_\rho) = b_\rho \times \rho  \), then \( b_\rho \) lies in \( \mathscr I \), and \( \{b_\rho\} \) defines the twisting. The pair \( (\phi, \varepsilon) \) is called \textit{well-positioned} if \( \phi \) is favorable and \( \varepsilon \in I_{\phi} \subset \mathscr I_{\phi} (\overline \QQ) = \mathscr I(\overline \QQ)  \subset  G(\overline{\mathbb{Q}}) \) also lies in \( G(\mathbb{Q}) \). Then \( b_\rho \) lies in \( G(\varepsilon) \), which is necessarily connected because \( G_{\text{der}} \) is simply connected.
 
 \begin{lem}\label{5.22}
 Every pair \( (\phi, \varepsilon) \), with \( \varepsilon \in I_{\phi} \), is equivalent to a well-positioned pair \( (\phi', \varepsilon') \).
 \end{lem} 
 
 It is useful to prove a stronger lemma. Let \( T \) be an elliptic Cartan subgroup defined over \( \mathbb{Q} \) and let \( h \in X_{\infty} \) factor through \( T \). Let \( \mu = \mu_h \). We say that \( (\phi, \varepsilon) \) is \emph{nested} in \( (T, h) \) if \( \phi = \psi_{T, \mu} \) and \( \varepsilon \in T(\mathbb{Q}) \).
 
 \begin{lem}\label{5.23}For every pair \( (\phi, \varepsilon) \), with \( \varepsilon \in I_{\phi} \), we can find an equivalent pair \( (\phi', \varepsilon') \) and \( T' \), \( h' \) such that \( (\phi', \varepsilon') \) is nested in \( (T', h') \).
 \end{lem} 
 
 \marginpar{\textit{Proof of Lem.~\ref{5.23}.}}
 It is clear that Lemma \ref{5.22} follows from Lemma \ref{5.23}. The proof of Lemma \ref{5.23} partially repeats the proof of Theorem \ref{5.3}. We can assume that \( \phi \) is well-positioned. Let \( \varepsilon_1 \) be an element of \( I_{\phi} \) whose centralizers in \( \mathscr I_{\phi}, \mathscr I, \) or \( G \) are Cartan subgroups that contain \( \varepsilon \). These centralizers are then all the same and are denoted by \( T \). As a subgroup of \( \mathscr I_{\phi} \), \( T \) is defined over \( \mathbb{Q} \), but it may not be so as a subgroup of \( G \). The conjugacy class of \( \varepsilon_1 \) in \( G(\overline{\mathbb{Q}}) \) is rational because \( \mathscr I \) is an inner twisting of \( \mathscr I_{\phi} \). We had fixed \( \psi : G \to G^* \) up to inner automorphism. According to Steinberg's theorem (see [K1]), we can assume that \( \varepsilon_1^* = \psi(\varepsilon_1) \) is rational. Let \( T^* = \psi(T) \) be the centralizer of \( \varepsilon_1^* \). We use Lemma \ref{5.6} to prove that \( T^* \) transfers to \( G \). We show \( T^* \)  transfers locally and that it is elliptic at infinity. 
 
 That \( T^* \) transfers to \( G \) at $p$ is clear because \( G \) is quasi-split over \( \mathbb{Q}_p \). At a place \( l \neq p \), we have  \( b_\rho = c_\rho d_\rho , \rho \in \Gal(\overline \QQ_{l}/ \QQ_l)\), where \( c_\rho \) is in the image of \( Q \) and thus in the center of \( \mathscr  I \), and \( d_\rho = u \rho (u^{-1}),  u \in G(\overline{\mathbb{Q}}_l) \). Then \( d_\rho \rho (\varepsilon_1) d_{\rho}^{-1} = \varepsilon_1 \) and \( u^{-1} \varepsilon_1 u = \rho (u^{-1} \varepsilon_1 u) \). Thus, the centralizer \( T_l \) of \( u^{-1} \varepsilon_1 u \) is defined over \( \mathbb{Q}_l \), and \( \psi \circ \mathrm{ad}\,u (T_l) = T^* \). At infinity, \( b_\rho = c_\rho d_\rho \) with \( c_\rho \) in the image of \( Q \) and $$ d_\iota \times \iota = u \xi_\infty (w(\iota)) u^{-1}, $$ where $\xi_\infty$ is defined at the beginning of the section. Therefore, \( g \mapsto u^{-1} g u \) is an embedding over \( \mathbb{R} \) of \( \mathscr I_{\phi} \)  into the group \( G' \) defined by \( \xi_{\infty} \). In particular, \( u^{-1} \varepsilon_1 u \) lies in \( G'(\mathbb{R}) \). On the other hand, if \( \xi \) is defined by \( h \in X_{\infty} \) and \( h \) factors through \( T_1 \), then \( T_1(\mathbb{R}) \) lies in \( G'(\mathbb{R}) \). It follows that \( \varepsilon_1 = v^{-1} \varepsilon_2 v \) with \( \varepsilon_2 \in T_1(\mathbb{R}) \). Since \( \varepsilon_1^* = \psi \circ \mathrm{ad}\,v (\varepsilon_2) \), \( T^* \) transfers over \( \mathbb{R} \) and in particular it is elliptic as \( T_1 \).

 Thus,\marginpar{191} there exists a \( g \in G(\overline{\mathbb{Q}}) \) such that \( \mathrm{ad}\,g (\psi^{-1}(T^*)) \) is defined over \( \mathbb{Q} \), and such that the homomorphism \( t^* \mapsto t = \mathrm{ad}\,g (\psi^{-1}(t^*)) \in  \mathrm{ad}\,g(T) \) is also defined over \( \mathbb{Q} \). Consequently, \( \mathrm{ad}\,g(\varepsilon_1) \in G(\mathbb{Q}) \). Since \( \varepsilon \) and \( \gamma_n \) lie in \( T(\mathbb{Q}) \), \( \mathrm{ad}\,g(\varepsilon) \) and \( \mathrm{ad}\,g(\gamma_n) \) are contained in \( \mathrm{ad}\,g(T) \). We now replace \( \varepsilon \), \( \phi \), and \( T \) with \( \mathrm{ad}\,g(\varepsilon) \), \( \mathrm{ad}\,g \circ \phi \), and \( \mathrm{ad}\,g(T) \). We are now in the same situation as we were during the proof of Theorem \ref{5.3}, before we proved Lemma \ref{5.11}, except that we are now given an additional datum, namely \( \varepsilon \in T(\mathbb{Q}) \). If we apply Lemmas \ref{5.11} and \ref{5.12}, as well as the final consideration in the proof of Theorem \ref{5.3}, to adjust \( \phi \) and \( T \) step by step, we can carry \( \varepsilon \) along, ultimately obtaining the desired \( \phi', \varepsilon', T', h' \).  \marginpar{\textit{QED Lem.~\ref{5.23}.}}\\
 
 We now come to the necessity of condition \( *(\varepsilon) \). We can assume that there exist \( T \) and \( h \) such that \( (\phi, \varepsilon) \) is nested in \( (T, h) \). Furthermore, we can assume that \( \varepsilon_0 = \varepsilon \). We want to show that under these circumstances, \( M \subset \mathscr J \). This also holds more generally when \( (\phi, \varepsilon) \) is well-positioned and \( \varepsilon_0 = \varepsilon \). The general claim follows immediately from the special case.
 
 As in the proof of Lemma \ref{5.2}, we can assume that \( \xi'_p = \xi_{ - \mu' }\), where \( \mu' \) is a cocharacter of a \( \mathfrak k \)-split Cartan subgroup \( T' \) of \( \mathscr J \). Then \( F = p^{\mu'} \times \sigma \), as in the example preceding  Lemma \ref{5.2}. Let \( \xi _{- \mu'} = \mathrm{ad}\,v(\xi_p) \), \( \varepsilon' = \mathrm{ad}\,v(\varepsilon) \), where \( v \in \mathscr J(\overline{\mathbb{Q}}_p) \). Then we have 
 \[
 \Phi^{tm} = F ^{tn} = p^{-t' \nu_2} \times \sigma ^{tn}, 
 \]
 where for sufficiently large \( s \), 
 \[
 \nu_2 = -\mu' \cdot \sigma \mu' \cdots \sigma^{s-1} \mu', \quad tn = t's.
 \]
The element \( \varepsilon' \) commutes with \( \Phi^{t m} \). However, when \( t \) is large enough such that \( T' \) splits over \( L_{tn} \), it follows easily from the Bruhat decomposition that \( \varepsilon' \) can only commute with \( \Phi^{t m} \) if it also commutes with \( \sigma ^{tn} \) and \( p^{-t \nu_2}\).
 
 From the equation \( \varepsilon' x = \Phi^m x \), it generally follows that
 \[
 (\varepsilon')^t x = \Phi^{tm} x, \quad t \in \mathbb{Z}.
 \]
 Let \( t \) be large enough such that \( \sigma^t(x) = x \). Then \( k' = p^{t' \nu_2} \cdot (\varepsilon')^t \) has a fixed point, and its eigenvalues all have absolute value 1. 
 
 The decomposition
 \[
 (\varepsilon')^t = p^{-t' \nu_2 } k'
 \]
 is a kind of polar decomposition of \( (\varepsilon')^t \) and is uniquely determined. From the uniqueness, it follows that \( p^{-t '  \nu_2 } \) lies in the center of \( G(\varepsilon) \). Since \( \mathscr J \) is the centralizer of \( p^{-t' \nu_2} \) and \( p^{-t' \nu_2} \)  obviously generates a one-dimensional splitting torus, it follows that \( M \) is contained in \( \mathscr J \).

  From\marginpar{192} [K3], §3, it follows easily that  
  \[
  \lambda_M(\varepsilon) = \frac{1}{t} \lambda_M(\varepsilon^t) = \frac{1}{t} \lambda_M(p^{-t' \nu_2}) = -\frac{t'}{t}\nu_2, 
  \]
  and the restriction of \( -\frac{t'}{t}\nu_2 \)   to \( Z(\hat{M}) \subset \hat{T} \)  is equal to \( \text{Nm}_{L_n/\mathbb{Q}_p} \mu_1'\). Although \( \mu' \) itself may not be defined over \( L_n \), by the definition of \( E \) and Lemma 1.1.3 from [K3], there exists a \( \mu: \mathbb G_m \to M \) conjugate to \( \mu' \) and defined over \( L_n \). It is clear that \( \mu_1 = \mu_1' \).

  We now assume that \( \varepsilon_0 \) satisfies the two conditions of Theorem \ref{5.21} and prove the existence of an admissible pair \( (\phi, \varepsilon) \) with \( \varepsilon \) stably conjugate to \( \varepsilon_0 \). We can find a Cartan subgroup \( T \) of \( G(\varepsilon_0) \) defined over \( \mathbb{Q} \) and a cocharacter \( \mu \) of \( T \) such that \( T \) is elliptic at infinity and at \( p \), and such that \( \mu \), though not necessarily the homomorphism from \( \mathbb G_m \) to \( M \) required by assumption (b), is at least conjugate to it in \( M \). It might be the case that \( \psi_{T, \mu} \) is not admissible. In that case, we use the proof of Lemma \ref{5.12} to find a \( u \in G(\overline{\mathbb{Q}}) \) such that \( T' = u T u^{-1} \) and \( t \mapsto t' = u t u^{-1} \) are both defined over \( \mathbb{Q} \), while \( \mu' = \mathrm{ad}\,u \circ \mu = \mu_h ,  h \in X_\infty \). If necessary, we replace \( \varepsilon_0 \) by \( \varepsilon'_0 = \mathrm{ad}\,u(\varepsilon_0) \) and \( T, \mu \) by \( T', \mu' \), so that \( \psi_{T, \mu} \) becomes admissible. 
  Then \( X^p \) contains a point \( y \) from \( \prod_{l\neq p} T(\overline{\mathbb{Q}}_l) \), which necessarily commutes with \( \varepsilon_0 \), and so \( \varepsilon_0 y = y \varepsilon_0 \). Consequently, the pair \( (\phi, \varepsilon) = (\psi_{T, \mu}, \varepsilon_0) \) satisfies the conditions (i), (ii), and the part of (iii) about the existence of the point \( y \).
  
  To verify the second part of (iii), we can assume that \( \xi'_p \), which defines \( X_p \) and \( \Phi \), is equal to \( \mathrm{ad}\,u \circ \xi_p = \xi _{ - \mu' } \), where \( u \in M(\overline{\mathbb{Q}}_p) \), and where \( \xi _{ - \mu' } \) corresponds to an elliptic Cartan subgroup \( T' \) of \( M \) split over \(\mathfrak k \) and a cocharacter \( \mu' \) conjugate to \( \mu \). Let \( \varepsilon' = \mathrm{ad}\,u(\varepsilon) \). Then
  \[
  \Phi^m = p^\nu \times \sigma^n
  \]
  with 
  \[
  \nu = \sum_{j=0}^{n-1} \sigma^j(\mu'). 
  \]
  By assumption \( *(\varepsilon) \), we have
  \[
  \lambda_M(\varepsilon') = \lambda_M(\varepsilon) = \lambda_M(p^\nu).
  \]
  Since the cocharacter of \( M_{\mathrm{ab}} \) defined by \( \mu' \)  is defined over \( \mathbb{Q}_p \), we conclude that the image of \( (\varepsilon')^{-1} p^\nu \) lies in the maximal compact subgroup of \( M_{\mathrm{ab}}(\mathfrak k) \). Therefore, there exists  \( c \in T'(\mathfrak k) \) such that \( c^{-1} (\varepsilon')^{-1} \Phi^m c \in M_{\mathrm{sc}}(\mathfrak k) \times \sigma^n \).
  
  To prove the existence of \( x \), we must (see [K3], Lemma 1.4.9) show the existence of a \( u \in G(\mathfrak k) \) such that \( u^{-1} (\varepsilon')^{-1} \Phi^m u = \sigma^n \). Instead, we seek something stronger: a \( d \in M_{\mathrm{sc}}(\mathfrak k) \) such that \( d^{-1} c^{-1} (\varepsilon')^{-1} \Phi^m c d = \sigma^n \).

   It\marginpar{193} is sufficient to show that \( c^{-1} \varepsilon'^{-1} \Phi^m c \sigma^{-n} \in M_{\mathrm{sc}}(\mathfrak k) \) is basic, since Prop.~5.4 of [K4] states that \( B(M_{\mathrm{sc}})_b\) is trivial over any finite extension of \( \mathbb{Q}_p \), in particular over \( L_n \). According to condition (4.3.3) of [K4], \( c^{-1} \varepsilon'^{-1} \Phi^m c \sigma^{-n} \) is basic if for sufficiently large \( t \),
   \begin{align}\label{5.q}
  c^{-1} (\varepsilon'^{-1} \Phi^m)^t c = e \sigma^{tn} e^{-1}.
   \end{align} 
   For a large \( r \), we have \( \Phi^{mr} = p^{\nu'} \times \sigma^{rn} \) with
   \[
   \nu' = \sum_{j=0}^{rn-1} \sigma^j(\mu),
   \]
   and \( \varepsilon' \) commutes with \( p^{\nu'} \) and \( \sigma^{rn} \). Furthermore, \( \nu' \) is central, and \( (\varepsilon')^r \) has a polar decomposition \( (\varepsilon')^r = p^{\nu '} k \).
   Therefore,
   \[
   (\varepsilon'^{-1} \Phi^m)^{sr} = k^s \times \sigma^{srn},
   \]
   and 
   \[
   c^{-1} (\varepsilon'^{-1} \Phi^m)^{sr} c = c^{-1} k^{-s} \sigma^{srn}(c) \times \sigma^{srn} = (c^{-1} k^{-s} c)  ( c^{-1} \sigma^{srn}(c) ) \times \sigma ^{srn}.
   \]
   For large \( s \) and \( r \), \( k^{-s} \) and \( c^{-1} \sigma^{srn}(c) \) are  inside an arbitrarily given neighborhood of $1$. Thus, the existence of a solution to \eqref{5.q} is proven  and thus Theorem \ref{5.21}.
   \marginpar{\textit{QED Thm.~\ref{5.21}.}}\\
   
   In Theorem \ref{5.25} below, an invariant \( i(\varepsilon_0) \) appears, namely the order of the kernel of   
   \[
   H^1(\QQ, G(\varepsilon_0)) \to H^1(\QQ, G_{\mathrm{ab}}) \times \prod_v H^1(\QQ_v, G(\varepsilon_0)).
   \]   
   If \( \varepsilon_0 \) is replaced by a stably conjugate \( \varepsilon'_0 \), \( G(\varepsilon_0) \) is replaced by \( G(\varepsilon'_0) \), a twisted form of \( G(\varepsilon_0) \). To ensure that \( i(\varepsilon_0) \) is well-defined, we prove the following simple lemma.
   
   \begin{lem}\label{5.24}
Let \( G \) be a reductive group over \( \mathbb{Q} \) without an \( E_8 \)-factor, \( A \) a torus over \( \mathbb{Q} \), and \( G \to A \) a surjective homomorphism defined over \( \mathbb{Q} \). Let \( G' \) be an inner twist of \( G \). Then the orders of the kernels of the two homomorphisms
\[
H^1(\QQ, G) \to H^1(\QQ, A) \times \prod_v H^1(\QQ_v, G)
\]
and
\[
H^1(\QQ, G') \to H^1(\QQ, A) \times \prod_v H^1(\QQ_v, G')
\]
are equal.
   \end{lem} 
   
   \marginpar{\textit{Proof of Lem.~\ref{5.24}.}}
   As in the proof of Lemma 4.3.2 from [K2], we can replace \( G \) by a \( z \)-extension and assume that \( G_{\mathrm{der}}\) is simply connected. By Lemma 4.3.1 of the same work, we can then replace \( G \) by \( G_{\mathrm{ab}} \), and the statement is clear for \( G_{\mathrm{ab}} \).
    \marginpar{\textit{QED Lem.~\ref{5.24}.}}\\

    We\marginpar{194} have already noted that the conjugacy class of \( \gamma_l \), which corresponds to an admissible pair, is well-defined. The twisted conjugacy class of \( \delta \in G(L_n) \) is also well-defined. For example, if \( c b \sigma c^{-1} = \delta \sigma \) and \( u c b \sigma c^{-1} u^{-1} = \delta' \sigma \) while \( \text{Nm}_{L_n/\mathbb{Q}_p} \delta = c \varepsilon' c^{-1} \) and \( \text{Nm}_{L_n/\mathbb{Q}_p} \delta' = u c \varepsilon' c^{-1} u^{-1} \), then \( \delta' = u \delta \sigma(u)^{-1} \) and \( u c \varepsilon' c^{-1} \sigma^n(u)^{-1} = u \text{Nm}_{L_n/\mathbb{Q}_p} \delta \sigma^n(u)^{-1} = u c \varepsilon' c^{-1} u^{-1} \). Therefore, \( u \) lies in \( G(L_n) \).
   
   \begin{thm}\label{5.25}
   Let \( \varepsilon_0 \in G(\mathbb{Q}) \) be an element satisfying the conditions of Theorem \ref{5.21}, and suppose for each \( l \) a conjugacy class \( \{ \gamma_l \} \) in \( G(\mathbb{Q}_l) \) is given, along with a twisted conjugacy class \( \{ \delta \} \) in \( G(L_n) \). Assume that \( \gamma_l \) and \( \varepsilon_0 \) are stably conjugate for every \( l \), and conjugate for almost every \( l \). Also, assume that \( \mathrm{Nm}_{L_n/\mathbb{Q}_p} \delta \) and \( \varepsilon_0 \) are stably conjugate. Finally, assume that \( \delta \) satisfies the condition \( \ast(\delta) \). Then, \( \gamma = (\gamma_l) \) and \( \delta \) correspond to an admissible pair \( (\phi, \varepsilon) \) with \( \varepsilon \) stably conjugate to \( \varepsilon_0 \) if and only if \( \kappa(\gamma, \delta; \varepsilon_0) = 1 \). In this case, \( (\gamma_l) \) and \( \delta \) correspond to exactly \( i(\varepsilon_0) \) non-equivalent pairs.
\end{thm} 
    
    \marginpar{\textit{Proof of Thm.~\ref{5.25}.}}
    We first prove: If \( (\phi, \varepsilon) \) exists, then the \( K \)-invariant \( \kappa(\gamma, \delta; \varepsilon_0) = 1 \). We can assume that there is a pair \( (T, \mu) \) such that \( (\phi, \varepsilon) \) is nested in \( (T, h) \). We can also assume that \( \varepsilon_0 = \varepsilon \). Then \( \gamma_l \) and \( \varepsilon \) are conjugate for every \( l \). Therefore, we can choose \( \beta(v) \), which appears in \eqref{5.p}, to be equal to 1.
    
    At the place \( p \), the element $b$ defined by \( \delta \) and \( \varepsilon \) is the same as the one that appears in \( F = b \times \sigma \). Since \( \phi = \psi_{T, \mu} \), the corresponding element in \( X^*(Z(\varepsilon) ^{\Gamma(p)} ) \) is nothing other than the restriction of \( \mu \) to \( Z(\varepsilon)^{ \Gamma(p)} \), as already mentioned in the proof of Theorem \ref{5.21}. Accordingly, we choose \( \beta'(p) \) as the restriction of \( \mu \) to \( U_p \). On the other hand, \( \mu = \mu h \) with some \( h \in X_\infty \), which factors through \( T \subset  G(\varepsilon) \). Therefore, we can choose \( \beta'(\infty) \) as the restriction of \( -\mu \) to \( U_\infty \). It immediately follows that
    \[
    \kappa(\gamma, \delta; \varepsilon) = \beta(\infty) \beta(p) = 1.
    \]
    
    Each class of admissible pairs \( (\phi, \varepsilon) \) with \( \varepsilon \) stably conjugate to \( \varepsilon_0 \) contains a pair \( (\phi, \varepsilon_0) \). We now show that the restriction of \( \phi \) to the kernel of \( \mathscr Q \) is determined by \( \varepsilon_0 \) along, and that its image lies in the center of \( G(\varepsilon_0) \). Let \( \phi \) and \( \phi' \) be two admissible homomorphisms from \( \mathscr Q \) to \( \mathscr G_G \) with \( \varepsilon_0 \in I_{\phi} \cap G(\mathbb{Q}) \), \( \varepsilon_0 \in I_{\phi'} \cap G(\mathbb{Q}) \). Let \( \gamma_k = \phi(\delta_k) \), \( \gamma'_k = \phi'(\delta_k) \). We must show that for sufficiently large \( k \), \( \gamma_k \) and \( \gamma'_k \) are central in \( G(\varepsilon_0) \) and equal. In any case, \( \gamma_k \) and \( \gamma'_k \) lie in \( G(\varepsilon_0) \). It is also clear that \( \gamma_k \equiv \gamma'_k \pmod{G_{\text{der}}} \).
    Consequently, we can replace \( G \) by \( G_{\mathrm{ad}} \) and \( \phi \) and \( \phi' \) by the corresponding compositions. Thus, for the moment, let \( G \) be an adjoint group.

     Let \( T \) be a Cartan subgroup of \( G(\varepsilon_0) \) defined over \( \mathbb{Q} \), which is elliptic at infinity and contains \( \gamma_k \). Since \( \varepsilon_0 \) lies in the center of \( G(\varepsilon_0) \), the element $\nu'$ defined by the equations
     \[
     |\lambda(\varepsilon_0)|_p = p^{\nu'(\lambda)}, \quad \lambda \in X^*(T)
     \]
     is central in $X_*(T) \otimes \QQ$, i.e., invariant under the Weyl group of \( G(\varepsilon_0) \). There exists a natural number \( t \) and an element \( \eta = p^{-t \nu'} \in T(\mathbb{Q}) \), which is a unit outside \( p \), such that
     \[
     |\lambda(\eta)|_p = p^{t \nu' (\lambda)}, \quad \lambda \in X^*(T).
     \]
     The\marginpar{195} element \( \eta \) can be chosen to be central in \( G(\varepsilon_0) \). Let \( s = \frac{k}{tn}\). We show that for sufficiently large \( k \), \( \gamma_k = \eta^s \). It remains to prove that \( \eta^{-s} \gamma_k \) is a unit at every finite place, since \( \gamma_k \) is necessarily elliptic at infinity, and we can replace \( k \) by a multiple. The claim is clear outside \( p \).
     
     At the place \( p \), consider the homomorphism \( \xi_p = \phi \circ \zeta_p \) and the corresponding \( \nu \). On the one hand, \( \nu \) is defined by
     \[
     |\lambda(\gamma_k)|_p = q^{- \langle \lambda, \nu \rangle}, \quad q = p^k,  \lambda \in X^*(T),
     \]
     (see the Note). On the other hand, since \( (\phi, \varepsilon_0) \) is admissible, we have
     \[
     |\lambda(\varepsilon_0)|_p = p^{-n \langle \lambda, \nu \rangle}
     \]
     and
     \[
     |\lambda(\eta)|_p = p^{-tn \langle \lambda, \nu \rangle}.
     \]
     Thus, the claim is also proved at \( p \), and the equation \( \gamma_k = \gamma'_k \) follows.
     
     We fix an admissible and well-positioned pair \( (\phi_0, \varepsilon_0) \). If \( (\phi, \varepsilon_0) \) is also admissible, then \( \phi = \phi_0 \) on the kernel. Let \( G'(\varepsilon_0) \) be the twisting of \( G(\varepsilon_0) \) defined by \( \phi_0 \). We consider a homomorphism \( \phi \) that agrees with \( \phi_0 \) on the kernel. Then  we have $
     \phi(q_\rho) = a_\rho \phi_0(q_\rho),
 $    where \( a = \{a_\rho\} \) is a cocycle of \( \text{Gal}(\overline{\mathbb{Q}}/\mathbb{Q}) \) with values in \( G'(\varepsilon_0) \). We write $    \phi = a \cdot \phi_0.$  
 
 \begin{lem}\label{5.26}
\( (\phi, \varepsilon_0) \) is admissible if and only if the images of \( a \) in \( H^1(\mathbb{Q}, G_{\mathrm{ab}}) \) and in \( H^1(\mathbb{R}, G'(\varepsilon_0)) \) are both trivial.
 \end{lem} 
     
     We assume this lemma for the moment and conclude the proof of Theorem \ref{5.25}. Let \( (\gamma_0, \delta_0) \) and \( (\gamma, \delta) \) be the pairs corresponding to admissible pairs \( (\phi_0, \varepsilon_0) \) and \( (\phi, \varepsilon) \), respectively. To define \( \kappa(\gamma_0, \delta_0; \varepsilon_0) \) and \( \kappa(\gamma, \delta; \varepsilon) \), we have introduced characters \( \beta_0'(v) \) and \( \beta'(v) \) for each place \( v \). The difference \( \beta'(v)(\beta_0'(v))^{-1} \) is a character on \( (Z(\varepsilon_0)/Z ) ^{\Gamma(v)} \), although only its restriction to \( Z(\varepsilon_0) ^{\Gamma(v) }\) is uniquely determined. On the other hand, as shown in [K2], Prop.~6.4, for each $v$ a character \( \alpha(v) \) of
     \[
     Z(\hat{G}'(\varepsilon_0)) ^{\Gamma(v)} = Z(\hat{G}(\varepsilon_0)) ^{ \Gamma(v) } = Z(\varepsilon_0) ^{\Gamma(v)}
     \] is assigned to the class \( a \).
        
        \begin{lem}\label{5.27}
         \( \alpha(v) \) is the restriction of \( \beta'(v) (\beta_0'(v))^{-1} \) to \( Z(\varepsilon_0)^{\Gamma(v)} \). 
     \end{lem}
         
         The proof is deferred. From this, it follows that \( \gamma \) and \( \gamma_0 \) are conjugate everywhere (except at infinity and \( p \)) if and only if \( a \) is locally trivial outside of \( \infty \) and \( p \). 
       If \( \delta' = u \delta \sigma(u)^{-1}, u \in G(L_n) \) and \( \text{Nm}_{L_n/\mathbb{Q}_p} \delta = c \varepsilon_0 c^{-1} \), then \( \text{Nm}_{L_n/\mathbb{Q}_p} \delta' = u c \varepsilon_0 c^{-1} u^{-1} \) and \( b' = c^{-1} u^{-1} \delta' u c = b \), or at least \( b' = v b \sigma(v)^{-1}, v \in G(\varepsilon_0) \), since \( c \) is only determined modulo \( G(\varepsilon_0) \). Conversely, if \( b' = v b \sigma(v)^{-1} \), we can choose \( \delta'= \delta \). 
       Thus, \( \delta \) and \( \delta_0 \) are twisted conjugates in \( G(L_n) \) if and only if \( a \) is trivial at \( p \). The last statement of Theorem \ref{5.25} follows from the two previous lemmas.
       Let \( (\gamma, \delta) \) be such that \( \kappa(\gamma, \delta; \varepsilon_0) = 1 \), and let \( a(v) \) be the corresponding adelic cohomology class with a trivial component at infinity.\marginpar{196} From Lemma \ref{5.27}, it easily follows that we can lift \( (a(v)) \) into a cohomology class \( (\tilde{a}(v)) \) with values in \( G'(\varepsilon_0) \cap G_{\mathrm{der}} \), which lies in the kernel of the natural map
       \[
      \bigoplus_v H^1(\mathbb{Q}_v, G'(\varepsilon_0) \cap G_{\mathrm{der}}) \to \big(\pi_0(Z(\varepsilon_0)/Z)^\Gamma\big)^D.
       \]
       From [K5], Prop.~2.6, we conclude that \( (\tilde{a}(v)) \) is the localization of a global cohomology class. Its image \( a \in H^1(\mathbb{Q}, G'(\varepsilon_0)) \), defined according to Lemma \ref{5.26}, defines an admissible pair \( (\phi, \varepsilon_0) \). This completes the proof of Theorem \ref{5.25}.     \marginpar{\textit{QED Thm.~\ref{5.25}.}}\\
       
              \marginpar{\textit{Proof of Lem.~\ref{5.26}, \ref{5.27}.}}
       The proof of Lemmas \ref{5.26} and \ref{5.27} remains. The necessity of the first condition of Lemma \ref{5.26} is clear. The necessity of the second condition follows from the injectivity of $
       H^1(\mathbb{R}, G'(\varepsilon_0)) \to H^1(\mathbb{R}, G'), $ 
       which itself is a consequence of Lemma \ref{5.14} and the next lemma.
       
       \begin{lem} Assume the 
conditions of Lemma \ref{5.14}. Then \( H^1(\mathbb{R}, T') \to H^1(\mathbb{R}, G') \) is surjective.
       \end{lem} 
       
       \marginpar{\textit{Proof of Lem.~5.28.}}
       The lemma can be immediately reduced to the case of an adjoint group, for which it asserts that every inner twisting of \( G' \) contains an elliptic Cartan subgroup. This is Lemma \ref{5.8}. 
       \marginpar{\textit{QED Lem.~5.28.}}\\
       
       Now suppose \( a \) satisfies the conditions of Lemma \ref{5.26}. We will simultaneously show that \( (\phi, \varepsilon_0) \) is admissible and that Lemma \ref{5.27} holds. Conditions \eqref{5.a} and \eqref{5.b} are certainly satisfied. To check the other conditions, we can assume that \( a_\rho \in G(\varepsilon_0) \cap G_{\mathrm{der}} \) for each \( \rho \).
       
       Let
       \[
       \phi_0 \circ \zeta_l = \text{ad}\,y \circ \xi_l, \quad \varepsilon_0 y = y \gamma^0_ l, \quad \gamma^0_l \in G(\mathbb{Q}_l).
       \]
       Then \( \{y^{-1} a_\rho y \mid  \rho \in \text{Gal}(\overline{\mathbb{Q}}_l/\mathbb{Q}_l)\} \) is a cocycle with values in \( G_{\mathrm{der}} \), and hence trivial:  \( y^{-1} a_\rho y = c^{-1} \rho (c) \). Let \( y' = y c^{-1} \). Then we have
       \[
       \phi \circ \zeta_l = \text{ad} \, y' \circ \xi_l,
       \]
       so condition \eqref{5.c} is satisfied. We also have   $$ \varepsilon_0 y' = \varepsilon_0 y c^{-1} = y \gamma^0_l c^{-1} = y' c \gamma^0_l c^{-1} $$ 
       and  
        $$ \rho (c \gamma^0_l c^{-1}) =  \rho (c y^{-1} \varepsilon_0 y c^{-1}) = c y^{-1} \varepsilon_0 y c = c \gamma_l^0 c^{-1}, $$   
       because \(a_\rho y \rho(y^{-1}) \) commutes with \(\varepsilon_0\). Thus, \(\gamma_l = c \gamma^0 _l c^{-1},\) and \(\beta_0'(l)\) (resp.~$\beta'(l)$) is defined by the cocycle \(\{ y \rho(y)^{-1} \}\) (resp.~\(\{ y' \rho(y')^{-1} \} = \{ a _\rho y \rho (y)^{-1} \}\) ). The group \(G '( \varepsilon_0 ) \) is defined by the cocycle \(\{ y \rho (y)^{-1} \}\), and we obtain \(G(\varepsilon_0)\) from \(G'( \varepsilon_0 ) \) via the cocycle \(\{ \rho (y) y^{-1} \}\). We apply Lemma \ref{5.18} to \(G '(\varepsilon_0) , G(\varepsilon_0)\) instead of \(G, G'\) to conclude Lemma \ref{5.27} at the place \(l\).

       The\marginpar{197} remarks preceding Theorem \ref{5.21} can be applied to \(\mathscr J\). Consequently, \(H^1(\QQ_p, \mathscr  J \cap G_{\text{der}}) = \{ 1 \}\) and \(H^1(\QQ_p, \mathscr J_{\phi_0}   \cap G_{\text{der}}) = \{ 1 \}\). It is clear that \(\phi_0\) and \(\phi\) define the same \(\mathscr J\). Since \(a\) defines the trivial class in \(\mathscr J_{ \phi_0}\), \(\phi_0\) and \(\phi\) define the same element in \(B(\mathscr J)_b\). Therefore, \(\phi\) is an admissible  homomorphism as \(\phi_0\) is. Except for the existence of \(x\) in condition (iii), it is also clear that the pair \((\phi, \varepsilon_0)\) is admissible. The existence of \(x\) is verified as in the proof of Theorem \ref{5.21}. Thus, Lemma \ref{5.26} is proven.

       \(\phi_0\) and \(\phi\) are mappings from \(\mathscr Q\) to \(\mathscr G_{G(\varepsilon_0)} \subset \mathscr G_G\). To verify the claim of Lemma \ref{5.27} at the place \(p\), we can replace \(\xi_p = \phi \circ \zeta_p\) by \(\xi'_p = \mathrm{ad} \, u \circ \xi_p\) with \(u \in G(\varepsilon_0)\). We can further assume that the localization of \(a\) is given by a cocycle from \( \text{Gal}(\QQ^{\mathrm{un}}_p /\QQ_p) \). Let \(\sigma\) be the Frobenius element. Then, \(b = a_\sigma \cdot b_\sigma\). Thus the claim at \(p\) follows immediately from Lemma \ref{5.19}, applied to \(G(\varepsilon_0), G '(\varepsilon_0)\), and \(a^{-1}_\sigma\).
              \marginpar{\textit{QED Lem.~\ref{5.26}, \ref{5.27}.}}\\

 \textbf{Note.} Let \(\mathscr D^n_n = \mathscr D_{L_n} ^{ L_n}\), and let \( \mathscr D^n\) (resp.~\(\mathscr D^n_{\mathfrak k}\)) be the group obtained by pushing out via \(L_n \subset \QQ^{\mathrm{un}}_p\) (resp.~\(L_n  \subset \mathfrak k\)) and pulling back via \(\text{Gal}(\QQ^{\mathrm{un}}_p / \QQ_p) \to \text{Gal}(L_n / \QQ_p)\). We define \(\mathscr D_{n}^n\) by the fundamental cocycle \begin{align*}
a_{\sigma^i, \sigma^j} = 1, \quad 0 \leq i, j < n, i + j < n \\
a_{\sigma^i, \sigma^j} = p^{-1}, \quad 0 \leq i, j < n, i + j \geq n.
 \end{align*} If \(\sigma\) is the Frobenius element in \(\text{Gal}(\QQ^{\mathrm{un}}_p / \QQ_p)\), the \(n\)-th power of \(d_\sigma\) is equal to \(p^{-1}\) in \(\mathscr D^n\) and \(\mathscr D^n_{\mathfrak k}\). Therefore, the homomorphisms \(\mathscr D^{mn} \to \mathscr D^n\) (resp.~\(\mathscr D^{mn}_{\mathfrak k} \to \mathscr D^n_{\mathfrak k} \)), which map \(x \mapsto x^m\) for \(x \in (\QQ^{\mathrm{un}}_p)^\times\) (resp.~\(\mathfrak k^\times\)), and \(d_\sigma \mapsto d_\sigma\), form a compatible family. In \(D^n_{\mathfrak k}\), each element \(x d_\sigma\) with \(x \in \mathfrak k^\times, |x| = 1\) is conjugate to \(d_\sigma\).
       
       Any homomorphism \(\xi_p : \mathscr D \to \mathscr G_G\) is equivalent, by Steinberg’s theorem (see [Se], p.~3), to a homomorphism \(\xi'_p\) that can be induced from $$ \xi'_p : \mathscr D_n \to G(\QQ^{\mathrm{un}}_p) \rtimes \text{Gal}(\QQ^{\mathrm{un}}_p / \QQ_p). $$ Let us assume that \(\xi_p\) itself can be defined in this way. 
       
       The element \(\xi(w)\) that appears in the definition of \(X_p\) is nothing other than \(\xi_p(d_\sigma)\). In general, \(\xi_p(d_\sigma) = b \times \sigma\), and \(b\) defines a class in \(B(G)\), which depends only on the class of \(\xi_p\). If we choose \(n\) sufficiently large, the \(n\nu\) appearing in [K4], equation (4.3.3), is the restriction of \(\zeta^{-1}_p\) on \((\QQ^{\mathrm{un}}_p)^\times \subset \mathscr  D^n\), which is a homomorphism of algebraic groups. We obtain in this way a map from \(C(G)\), the set of classes in \(\text{Hom}(\mathscr D, \mathscr G_G)\), to \(B(G)\). Let \(C(G)_b\) be the preimage of \(B(G)_b\). We do not know whether \(C(G) \to B(G)\) is bijective. However, it follows easily from [K4] that \(C(G)_b \to B(G)_b\) is bijective.
       
       For tori, one has the functor\footnote{Here the two appearances of ``functor'' should be understood as ``natural transformation''.  } \(C(T) \to B(T)\), as well as the functor \(\mu \mapsto \xi_{-\mu}\) from \(X_*(T)\) to \(C(T)\). Since \(\zeta_{-\mu}(d_\sigma) = p^\mu \times  \sigma\), the composition maps \(1 \in X_*(\mathbb G_m)\) to \(b = p\). It therefore follows from Lemma 2.2 of [K4] that the composition is an isomorphism. Consequently, \(C(T) \to B(T)\) is surjective. It is injective because the number of classes in \(C(T)\) (resp.~\(B(T)\)) for a given \(\nu\) is equal to \(|H^1(\QQ_p, T)|\). For a general group, injectivity is also proved in this way. Surjectivity follows from [K4], Prop.~5.3.

   \section{The conjecture as a consequence of the Standard Conjectures in some cases}
   
   \marginpar{198}
    
   In this section, we want to show that, for certain Shimura varieties, the conjecture in §5 follows from the identification of the pseudo-motivic Galois group and the motivic Galois group, which was carried out in §4 under the assumption of the standard conjectures, the Tate conjecture, and the Hodge conjecture. In doing so, we closely follow Zink’s work [Z2], so the reader is expected to have at least a superficial familiarity with this work.   
   We begin with the definition of the relevant Shimura varieties and recall some group-theoretic facts, for whose proofs we refer to [De3] and [Z2].  
   
   Let \( D \) be a simple, finite-dimensional \(\mathbb{Q}\)-algebra of degree \( d^2 \) over its center \( L \), equipped with a positive involution \( x \mapsto x^* \). Let \( L_0 \) be the field of invariants of the involution in \( L \). Let \( (V, \psi) \) be a symplectic \( D \)-module, i.e., a \( D \)-module equipped with a non-degenerate alternating \(\mathbb{Q}\)-bilinear form such that   
   \[
   \psi(xu,v) = \psi(u,x^*v), \quad x \in D, \quad u,v \in V.
   \] 
   Let \( \dim_L V = d \cdot m \) and \( [L_0 : \mathbb{Q}] = n \). We assume that \( (D,*) \) is of one of the following two types:  
   
   (A) The involution is of the second kind, i.e., it induces a nontrivial automorphism of \( L \). In this case, \( L_0 \) is a totally real number field.  
   
   (C) The involution is of the first kind, i.e., it induces the trivial automorphism on \( L \). If \( K \supset L \) is a splitting field of \( D \), then we find an isomorphism \( D \otimes_L K \cong M_d(K) \) and a \( K \)-linear extension of the involution. In matrix notation, we assume \( x^* = c \, ^t x c^{-1} \) with \( ^t c = c \). In this case, \( L \) is a totally real number field.  
   
   Let \( G \) be the group of symplectic \( L_0 \)-similarities, regarded as an algebraic group over \(\mathbb{Q}\):   
   \[
   G = \{ g \in \GL_D(V) \mid \psi(gu, gv) = \psi(c(g) u, v), \quad c(g) \in L_0^\times \}.
   \]
   This is a connected reductive group over \(\mathbb{Q}\), whose derived group is simply connected ([De3], 5.8). (In case (A), \( G _\der\) is a product of special unitary groups, and in case (C), a product of symplectic groups). The groups \( G \) and \( G_{\text{der}} \) satisfy the Hasse principle ([De3], 5.11). Furthermore, there exists a homomorphism   
   \[
   h: \mathbb S \to G_{\mathbb{R}},
   \] 
   which, up to conjugation, is characterized by the fact that the Hodge structure induced on \( V \) is of type \( (+1,0) + (0,+1) \), and that \( \psi(u, h(i)v) \) is symmetric and positive definite ([Z2], 3.1). The pair \( (G, h) \) defines a Shimura variety, where it should be noted that the weight homomorphism   $
   \mu + \bar{\mu}: \mathbb{G}_m \to G $ 
   is defined over \(\mathbb{Q}\) and that \( X_*(G_{\text{ab}}) \) satisfies the Serre condition \eqref{3.e}. The \( D \)-module \( V \) is uniquely determined up to isomorphism by   
   \[
   t(x) = \text{Tr}(x \mid  V_h^{1,0}), \quad x \in D,
   \] 
   (see [De3], 5.10). Let\marginpar{199}  
   \[
   E = E(G, h) = \mathbb{Q}(t(x) \mid x \in D).
   \] 
   This is the Shimura field. (In case (C), \( E = \mathbb{Q} \).)  Let \( K \subset G(\mathbb A_f) \) be an open compact subgroup. The Shimura variety \( \text{Sh}_K(G,h) \) has a canonical model over \( E \) and is a \emph{coarse} moduli scheme for the moduli problem on \( (\text{Sch}/E) \), which associates to an \( E \)-scheme \( S \) the following data:  
    \begin{lr} \label{6.a}
    	\begin{enumerate}
    		\item \textit{An abelian \( D \)-scheme up to isogeny \( (A, \iota: D \to \mathrm{End}\, A) \) such that   
    		\[
    		\mathrm{Tr}(x \mid  \mathrm{Lie}^* A) = t(x), \quad x \in D.
    		\] 
    		Here, \( \mathrm{Lie}^* A \) denotes the cotangent space of \( A \).  }
    		\item \textit{An \( L_0 \)-homogeneous polarization \( \Lambda \), which induces the involution \( * \) on \( D \).  }
    		\item \textit{A class of \( D \)-module isomorphisms of sheaves for the étale topology   
    		\[
    		\bar{\eta}: H^1(A, \mathbb A_f) \isom V \otimes \mathbb A_f \mod K,
    		\] 
    		which preserves the bilinear forms on both sides up to a constant in \( L_0 \otimes \mathbb  A_f \).  } 
    	\end{enumerate}  
    \end{lr}

   Here, \( H^1(A, \mathbb A_f) = \prod H^1(A, \mathbb{Q}_l) \) denotes the \( l \)-adic cohomology with coefficients in \( \mathbb A_f \).  
   
   We now fix our prime number \( p \). Let \( O_D \) be an order in \( D \) and \( V_{\mathbb{Z}} \) an \( O_D \)-invariant lattice in \( V \), such that 
   \begin{lrnumber}\label{6.1}
   	\begin{enumerate}
   		\item [a)] \( p \) is unramified in \( D \),
   		\item [b)] \( D \otimes \mathbb{Q}_p \) is a product of matrix algebras, and \( O_D \otimes \mathbb{Z}_p \) is a maximal order,  
   		\item [c)] \( \psi: V_{\mathbb{Z}_p} \times V_{\mathbb{Z}_p} \to \mathbb{Z}_p \) is a perfect pairing.   
   	\end{enumerate} 
   \end{lrnumber}

   Then we have \(* (O_D \otimes \mathbb{Z}_p) = O_D \otimes \mathbb{Z}_p\). Let  
   \[ K_p = G(\mathbb{Q}_p) \cap \text{End}_{O_D} V_{\mathbb{Z}_p}. \]  
   Then \( G \) is quasi-split over \(\mathbb{Q}_p\) and split over an unramified extension, and  $K_p \subset G(\QQ_p)$ is
   the group of integer points, i.e.([T], 3.8) a hyperspecial subgroup. Let \( K = K^p \cdot K_p \).   
   To define a model of \( \Sh_K \) over \( \text{Spec}(\mathcal O_E \otimes \mathbb{Z}_{(p)}) \), we reformulate the moduli problem \eqref{6.a}. A point of this moduli problem with values in an \( \mathcal O_E \otimes \mathbb{Z}_{(p)} \)-scheme \( S \) consists of:  
   \begin{lr}\label{6.b}
   	\begin{enumerate}
   		\item \textit{An abelian \( O_D \)-scheme \( A \), up to prime-to-\( p \) isogeny, such that  
   			\[ \mathrm{Tr}(x \mid \mathrm{Lie}^* A) = t(x), \quad x \in O_D. \]    }
   			\item \marginpar{200}\textit{An \( L_0 \)-homogeneous polarization \( \Lambda \) that induces the given involution \( * \) on \( D \), such that there exists \( \lambda \in \Lambda \) whose degree is prime to \( p \).    }
   			\item \textit{A class of \( O_D \otimes \mathbb A_f^p \)-module isomorphisms of sheaves for the étale topology   
   				\[ \bar{\eta}: H^1(A, \mathbb A_f^p) \isom V \otimes \mathbb A_f^p \mod K^p, \]  
   				preserving the bilinear forms on both sides up to a constant in \( L_0 \otimes \mathbb A_f^p \).    }
   	\end{enumerate}
   \end{lr} 
   We refer to [Z2], 3.5, for a proof of the fact that the moduli problems \eqref{6.a} and \eqref{6.b} coincide over \( \text{Spec}\, E \). As in §5, let \( \mathcal O_\fkp \) denote the ring of integers in \( E_\fkp \).  
   
   \begin{thm}\label{6.2}
   	The moduli problem \eqref{6.b} admits a coarse moduli scheme over \( \mathrm{Spec}\,(\mathcal O_\fkp) \). If \( K^p \) is sufficiently small, then the moduli scheme is smooth, and if \( L_0 = \mathbb{Q} \), it is also a fine moduli scheme.  
   \end{thm} 

   We outline a proof, which follows the proof of the corresponding Theorem 1.7 from [Z2]. The \( O_D \)-lattice \( \eta^{-1} (V_{\mathbb{Z}} \otimes \widehat{\mathbb{Z}}^p) \subset H^1(A, \mathbb A_f^p) \) is independent of the choice of \( \eta \in \bar{\eta} \) and defines an abelian \( O_D \)-scheme \( B \) in the isogeny class of \( A \), equipped with an \( O_D \)-module isomorphism  
  $ \bar{\eta}: H^1(B, \hat{\mathbb{Z}}^p) \cong  V_{\mathbb{Z}} \otimes \widehat {\mathbb{Z}}^p \mod K^p. $   
   This shows that the moduli problem \eqref{6.b} is equivalent to the following moduli problem.  
   
   \begin{lr}\label{6.c}
\begin{enumerate}
	\item \textit{An abelian \( O_D \)-scheme \( B \), such that  
		\[ \mathrm{Tr}(x \mid  \mathrm {Lie}^* B) = t(x), \quad x \in O_D. \]  
	}
\item \textit{As in \eqref{6.b}, (ii).  }
\item \textit{A class of \( O_D \)-module isomorphisms  
	\[ \bar{\eta}: H^1(B, \hat{\mathbb{Z}}^p) \isom  V_{\mathbb{Z}} \otimes \hat{\mathbb{Z}}^p \mod K^p, \]  
	preserving the bilinear forms on both sides up to a constant in \( L_0 \otimes \mathbb  A_f^p \).  
}
\end{enumerate}
   \end{lr}
    
   By using the methods of Mumford [Mu2] (or those of Artin [A], if one is satisfied with the construction of an algebraic space), we obtain the existence of a fine moduli scheme if in (ii) we require the presence of an effective involution-preserving polarization and if \( K^p \) is taken to be sufficiently small.  
   If \( (B, \Lambda, \bar{\eta}) \) is a \( T \)-valued point of the moduli problem \eqref{6.c} and \( \lambda \in \Lambda \) is prime to \( p \) and \( \eta \in \bar{\eta} \), then, if \( E^{\lambda} \) denotes the associated Riemann form,  
   \[ \psi \big(\eta(x), \eta (y)\big) = E^{\lambda}(t x, y), \quad x, y \in H^1(B, \hat{\mathbb{Z}}^p) \]  
   for some \( t \in L_0^{\times} (\mathbb A_f^p) \).  
   The residue class \( \tau(B, c, \lambda, \bar{\eta}) \) of \( t \) in the \( \text{Gal}(\overline{\mathbb{Q}}_p / E_\fkp) \)-module $ \big ( L_0^{\times} (\mathbb A_f^p) / c(K^p) \big) (+1) $  is independent of the choice of \( \eta \in \bar{\eta} \) and can be regarded as a section of this sheaf over \( T \).   
   Two polarizations \( \lambda \) and \( \lambda' \) differ by a totally positive element of \( L_0 \) that is a unit at \( p \). Thus, we obtain a section\marginpar{201} of $ L_0^+ \backslash L_0^{\times} (\mathbb A_f) / c(K) \,(+1) $ 
   that depends only on \( (B, \Lambda, \bar{\eta}) \).   
   Let \( \tau_1, \dots, \tau_r \) be sections of \( L_0^{\times} (\mathbb A_f^p) / c(K^p) \, (+1) \) over a finite unramified extension \( R \) of \( \mathcal O_\fkp \) that form a system of representatives. We consider the following moduli problem on \( (\text{Sch}/R) \):  
   \begin{lr}\label{6.d}
   	\begin{enumerate}
   		\item \textit{As in \eqref{6.c} (i).}
   		\item \textit{An involution-preserving polarization \( \lambda \) that is principal at \( p \).}
   		\item \textit{A class of \( O_D \)-module isomorphisms  
   			\[ \bar{\eta}: H^1(B, \hat{\mathbb{Z}}^p) \isom V_{\mathbb{Z}} \otimes \hat{\mathbb{Z}}^p \mod K^p, \]  
   			such that \( \tau(B, c, \lambda, \bar{\eta}) = \tau_i \) for some \( i = 1, \dots, r \).  }
   	\end{enumerate}
   \end{lr} 

    If we choose the $\tau_i$ such that they are representatives of \emph{integral} ideals, then $\lambda$ is an effective polarization whose degree is determined by $\tau_i$. Consequently, according to Mumford, the moduli problem \eqref{6.d} for sufficiently small $K^p$ is representable by a quasi-projective $R$-scheme \( \tilde{\mathscr M}_K \). The group \( L_0^+ \cap c(K) \) acts on \( \tilde{\mathscr M}_K \) (viewed as an $\mathcal O_\fkp$-scheme) by changing the polarization. It is easy to verify that this action factors over the finite quotient \( L_0^+ \cap c(K)/c(L^\times \cap K) \), and that the quotient is a coarse moduli scheme for the modular problem \eqref{6.c} over $\mathrm{Spec}\, \mathcal O_\fkp$. If \( L_0 = \mathbb{Q} \), then \( L_0^+ \cap c(K) \) is trivial. We show that in general, for sufficiently small $K^p$, the action of \( L_0^+ \cap c(K) / c(L^\times \cap K) \) on \( \tilde{\mathscr M}_K \) is free. Let \( \zeta \in L_0^+ \cap c(K) \), and $$ \alpha : (B, \iota, \lambda, \bar{\eta}) \isom  (B, \iota, \zeta \lambda, \bar{\eta}) $$ be an isomorphism. We must show that \( \zeta \in c(L^\times \cap K) \). Let \( K' = K^{\prime p} \cdot K_p \) be a sufficiently small open compact subgroup. We consider only \( K \subset K' \), and denote the equivalence class of \( \bar{\eta} \mod K' \) by \( \bar{\eta}' \). Then, there is at most one isomorphism $$ \alpha' : (B, \iota, \lambda, \bar{\eta}') \isom  (B, \iota, \zeta \lambda, \bar{\eta}') .$$
    Let \( K \) be chosen small enough such that \( L_0^+ \cap c(K) \subset (L_0^\times \cap K')^2 \). Then, \( \zeta = \zeta_1^2 \), and \( \zeta_1 \) defines an isomorphism between \( (B, \iota, \lambda, \bar{\eta}') \) and \( (B, \iota, \zeta \lambda, \bar{\eta}') \). But then \( \zeta_1 \) coincides with \( \alpha \), and therefore \( \zeta_1 \in K \), implying \( \zeta = \zeta_1^2 \in c(L^\times \cap K) \).
    
    To convince ourselves of the smoothness of \( \tilde{\mathscr M}_K \), we need to show that there are no deformation obstructions. We use the theory of Grothendieck and Messing [Me]. Let \( T \subset T' \) be a nil-immersion with divided powers of affine $R$-schemes, on which \( p \) is nilpotent. Let \( (B, \iota, \lambda, \bar{\eta}) \) be a $T$-valued point of the moduli problem \eqref{6.d}. We consider the value of the Dieudonné crystal at \( T' \), $$ \mathbb D(B)_{(T,T')}, $$ which is equipped with an action of \( O_D \) and an alternating \emph{perfect} bilinear form \( E^\lambda \) with values in \( \mathcal O_{T'} \), satisfying the identity
    \[
    E^\lambda(xu, v) = E^\lambda(u, x^* v).
    \]
    Its\marginpar{202} restriction to \( T \) is equipped with the Hodge filtration, which is totally isotropic with respect to \( E^\lambda \) and for which
    \[
    \text{Tr}(x \mid \text{Fil}) = t(x), \quad x \in O_D.
    \]
    It must be shown that this filtration can be lifted to a totally isotropic filtration of \(\mathbb D_{(T, T')} \). The decomposition \(   O_L \otimes \mathbb{Z}_p =\prod   O_{L_v} \) provides an orthogonal decomposition  \(  (\mathbb D_{(T, T')}, E^\lambda) = \bigoplus  (D_v', E_v') \).
    
    In case (C), \( L_0 = L \). We choose an isomorphism \( O_D \otimes _{ O_L}  O_{L_v} \cong M_d( O_{L_v}) \). Then, the involution \( * \) is induced by a perfect \emph{symmetric} pairing \( \beta \) on \( O_{L_v} ^d \), cf. [Z2], 2.9. We introduce the bilinear form \( \tilde E'_v \) with values in \( O_{L_v} \otimes \mathcal O_{T'} \):
    \[
    \text{Tr}_{L_v/\mathbb{Q}_p} f \cdot \tilde E'_v(u, w) = E_v'(u, f \cdot w), \quad f \in O_{L_v}.
    \]
    Then the \( M_d(O_{L_v}) \otimes \mathcal O_{T'} \)-module \( (D'_v, \tilde E'_v) \) decomposes as \( D_v' = O_{L_v}^d \otimes V' \), \( \tilde E'_v =   \beta \otimes \Phi' \), where \( \Phi '\) is an \emph{alternating} perfect bilinear form on the \(\mathcal O_{T'} \)-module \( V '\). The corresponding result holds for \( \mathbb D_{(T, T)} \). Here \( \text{Fil}_v \subset D_v \) is generated by a totally isotropic part (with respect to \( \Phi \)) of a basis of \( V \). Since this basis can be lifted to such a basis of \( V '\), the assertion follows.
    
    The case (A) is handled analogously (cf.~[Z2], 3.4).

     With this, our sketch of the proof of Theorem \ref{6.2} is complete. We have obtained a smooth model over $\mathcal O_\fkp$, whose points over \( \FF = \bar{\kappa} \) we want to analyze, assuming the validity of the results of \S 4. A \emph{special point} of the moduli problem \eqref{6.b} is defined by the following data ([Z2], 4.2): a point \( (A, \iota, \Lambda, \bar{\eta}) \) of the moduli problem, a CM-algebra \( R \) (= product of CM fields) with \( \dim _{\mathbb{Q}} R = 2\dim A/d \), and an involution-preserving embedding $$ \vartheta_A : R \to \text{End}_ D^0 A . $$

    Let \( (A, \iota, \Lambda, \bar{\eta}, R, \vartheta) \) be a special point over \( \mathbb C \). Then \( D \otimes_L R \) is a matrix algebra over \( R \) ([Z2], 2.8), and we obtain a decomposition \( A \cong B^d \), where \( B \) is an abelian variety with complex multiplication \( (R, \Phi) \). Let
    \[
    \mu : \mathbb{C}^\times \to (R \otimes \mathbb{C})^\times = \prod_{\varphi \in \Phi}\mathbb{C}^\times \times  \prod_ { \varphi \notin \Phi } \mathbb{C}^\times, \quad z \mapsto \prod \varphi(z) \times \prod 1.
    \]
    Let \( T = \{ r \in R^\times \mid r \cdot \bar{r} \in L_0 \} \), where \( r \mapsto \bar{r} \) denotes the Rosati involution on \( R \). Then \( T \) is a maximal torus defined over \( \mathbb{Q} \) in \( G \), and the above homomorphism factors through \( T \) and is of the form \( \mu = \mu_h \) for some \( h : \mathbb S \to G_\RR \) in the conjugacy class we fixed. Thus, \( (A, \iota, \Lambda, \bar{\eta}, R, \vartheta) \) defines a special point of \(\Sh_K(\CC) \) in the usual sense ([De1], 2.2.4). The module \( X_*(^KS) \) contains a distinguished element \( \mu = \mu_S \), which generates it. Since \( X_*(G_{\text{ab}}) \) satisfies the Serre condition \eqref{3.e}, so does \( X_*(T) \), and for sufficiently large fields \( K \), there exists a homomorphism \( \psi : {} ^K S \to T \) defined over \( \mathbb{Q} \) that sends \( \mu_S \) to \( \mu \). Let \( G^0 = \{ g \in G \mid c(g) \in \mathbb{Q}^\times \} \). Then \( \mu \) is a cocharacter of \( T^0 = T \cap G^0 \) and \( \psi : {}^KS \to T^0 \). We denote also by $\psi$ the corresponding homomorphism\footnote{The second map should simply be the one induced by the inclusion $T^0 \subset G^0$, and the symbol $\psi_{T^0,\mu}$ should be deleted.   } 
    $$ \mathscr G_S \To \mathscr G_{T^0} \xrightarrow{\psi_{T^0,\mu}} \mathscr G_{G^0} \subset \mathscr G_G. $$ 
    The\marginpar{203} group \( \mathscr G_S \) is the motivic Galois group of the Tannakian category generated by abelian varieties of CM type. Thus, \( \psi \) defines a motive with additional structure, namely the homogeneous component of degree 1 of \( (A, \iota, \lambda, R, \vartheta) \) with \( \lambda \in \Lambda \). We introduce the concept of \emph{isogeny} between two points \( (A_1, \iota_1, \Lambda_1, \bar{\eta}_1) \) and \( (A_2, \iota_2, \Lambda_2, \bar{\eta}_2) \) of the moduli problem \eqref{6.b}: an isogeny \( \alpha : A_1 \to A_2 \) that is compatible with \( \iota_1 \) and \( \Lambda_1 \) and \( \iota_2 \) and \( \Lambda_2 \). The trivializations of the Tate modules do  not enter here. An isogeny prime to \( p \) is nothing other than an isomorphism
    \[
    \alpha : (A_1, \iota_1, \Lambda_1) \isom  (A_2, \iota_2, \Lambda_2).
    \]
    In terms of the moduli problem \eqref{6.c}, an isogeny between abelian varieties \( \beta : B_1 \to B_2 \) is one that respects the actions of \( O_D \) and the polarization classes. One speaks of an isogeny between special points if there is an isomorphism \( R_1 \cong R_2 \) such that the isogeny additionally transfers the action \( \vartheta_1 \) into the action \( \vartheta_2 \).
    
    \begin{thm}\label{6.3}
The set of isogeny classes of points of \( \Sh(\bar{\kappa}) \) is in a one-to-one correspondence with the equivalence classes of admissible homomorphisms \( \phi : \mathscr P \to \mathscr G_G \).
    \end{thm} 
    
    We assume, however, that \( \mathscr P \cong \mathscr M\) in $\mathscr G_S$.  According to the definition of $\mathscr M$, the homogeneous motive $M^1(A)$ of degree 1 associated with an abelian variety corresponds to a homomorphism $\phi : \mathscr M \to \mathscr G_{V'}$.   Since $$ \text{Hom}(M^1(A), M^1(A')) \cong  \text{Hom}_\mathbb{Q}(A, A') , $$ \( A \) is defined as an abelian D-variety by an embedding \( \iota ': D \to \text{End}\, V' \) that is compatible with \( \phi \). Similarly, the dual variety \( A^* \) is defined by the contragredient homomorphism \( \phi^* \) on \( V^{\prime *} \). For \( T \) the Tate motive, we have \( V^{\prime *} = \text{Hom}(V', T) \). A polarization corresponds to a homomorphism \( \psi : V ' \to V^{\prime *} \) that is compatible with \( \phi \) and \( \phi^* \), i.e., a bilinear form \( \psi' \) on \( V' \) for which the equation
    \[
    \psi'(\phi(g)u, \phi(g)v) = \tau(g) \cdot \psi'(u, v)
    \]
    holds, where \( \tau \) denotes the action corresponding to $T$ on \( \overline {\mathbb{Q}} \). The isomorphism class of \( (A, \iota, \lambda) \) is clearly determined by \( (V', \iota', \psi') \). Due to the existence of a normal form, we can assume that \( (V', \iota', \psi') = (V, \iota, \psi) \). Then, the image of \(\mathscr  M \) lies in \( \mathscr G_{G_0} \subset \mathscr  G_V \). Two homomorphisms \( \phi_1, \phi_2 : \mathscr M \to \mathscr G_{G_0} \) are equivalent under \( G(\overline {\mathbb{Q}}) \) if and only if \( (A_1, \iota_1, \lambda_1) \cong  (A_2, \iota_2, c\lambda_2), c \in L_0^\times \). We have  \( G^0_{\text{ab}} = \mathbb{Q}^\times \subset G_{\text{ab}} = L_0^\times \). Since the cocharacter \( \mu_{\text{ab}} \) factors through \( G^0_{\text{ab}} \), every admissible \( \phi : \mathscr {Q} \to \mathscr G_G \) factors through \(\mathscr G_{G_0} \). Since \( X_*(G_{\text{ab}}) \) satisfies the Serre condition \eqref{3.e}, \( \phi \) also factors through \(\mathscr  P = \mathscr M \).

     To complete the proof, it suffices to show that if \( (A, \iota, \Lambda) \) arises from \( (A, \iota, \Lambda, \bar{\eta}) \), then the corresponding \( \phi \) is admissible, and vice versa. This is not difficult. However, we prefer a different approach that directly associates an admissible \( \phi \) to each \( (A, \iota, \Lambda, \bar{\eta}) \) and vice versa, without assuming any unproven conjectures, though we do not at present prove the bijectivity of this correspondence. Furthermore, it is not straightforward to apply this if \( G \) is not quasi-split over \( \mathbb{Q}_p \) (see the second example in \S 7).
     
     Since\marginpar{204} \( A \) is already defined over a finite field, it is an abelian variety of CM type, and we can consider the point \( (A, \iota, \Lambda, \bar{\eta}) \) as a special point \( (A, \iota, \Lambda, \bar{\eta}, R, \vartheta) \). According to Zink's lifting theorem [Z2], §4.4, there exists a mixed characteristic discrete valuation ring \(\mathcal O \) with residue field \( \bar{\kappa} \), and a special point \( (\tilde{A}, \tilde{\iota}, \tilde{\Lambda}, \tilde{\bar{\eta}}, R, \tilde{\vartheta}) \), whose reduction is isogenous to the original point. Let \( (T, \mu) \) be the corresponding torus with cocharacter in \( G \). The homomorphism \( \psi_{T,\mu} : \mathscr  P \to \mathscr G_S \to \mathscr G_T \subset \mathscr  G_G \) is admissible. Since \( \phi \cong \psi_{T,\mu} \), it follows that \( \phi \) is admissible. Conversely, if \( \phi \) is admissible, it is equivalent to \( \psi_{T,\mu} \). It defines \( \psi_{T,\mu} \) and \( (\tilde{A}, \tilde{\iota}, \tilde{\Lambda}, \tilde{\bar{\eta}}, R, \tilde{\vartheta}) \), which by reduction yields \( (A, \iota, \Lambda, \bar{\eta}) \).
     
     In the following we fix an isogeny class \( \mathscr J = (A, \iota, \Lambda)_\mathbb{Q} \) and the corresponding admissible homomorphism \( \phi \) up to equivalence. We denote by \( \mathscr J_K \) the set of points of \( \Sh_K(\FF) \) in \( \mathscr J \). Let \( r = [E_\fkp : \mathbb{Q}_p] \), and \( q = p^r \) be the number of elements in \( \kappa \). If \( (A, \iota, \Lambda, \bar{\eta}) \) is a point of \( \Sh_K(\FF) \), then the inverse image \( (A^{(q)}, \iota^{(q)}, \Lambda^{(q)}, \bar{\eta}^{(q)}) \) under the \( r \)-th power of the Frobenius is again a point of \( \Sh_K(\FF) \). The existence of the relative Frobenius 
     \[
  \Fr^r : (A, \iota, \Lambda, \bar{\eta}) \to (A^{(q)}, \iota^{(q)}, \Lambda^{(q)}, \bar{\eta}^{(q)})
     \]
     shows that we obtain an action of \( \text{Gal}(\FF/\kappa) \) on \( \mathscr J_K \), which is also obviously compatible with the change of \( K^p \).
     
     \begin{thm}
There are bijections
\[
\mathscr J_K\cong X_\phi(K),
\]
such that the action of \( \Phi \) on the right hand set (see \eqref{5.e}) corresponds to the action of the Frobenius in \( \text{Gal}(\FF/\kappa) \) on the left hand set, and they are compatible with the change of \( K^p \).
     \end{thm} 
     
    In the following we fix a point \( (A_0, \iota_0, \Lambda_0, \bar{\eta}_0) \) in \( \mathscr J_K \), which is the reduction modulo \( p \) of a special point \( (\tilde A_0, \tilde {\iota}_0, \tilde {\Lambda}_0, \tilde{\bar \eta}_0) \) and defines, as in the proof of \ref{6.2}, the torus \( (T, \mu) \).
    
     An isogeny decomposes into a chain of a prime-to-$p$ isogeny and a \( p \)-power isogeny. If an element \( \eta_0 \in \bar{\eta}_0 \) is fixed, an element \( g \in G(\mathbb A_f^p) \) defines a prime-to-$p$ isogeny:
     \[
     (A_0, \iota_0, \Lambda_0, \bar{\eta}_0) \to (A_0, \iota_0, \Lambda_0,\overline{ g^{-1} \circ \eta_0} ).
     \]
     Two elements differing by an element of \( K^p \) define the same isogeny. Conversely, any prime-to-$p$ isogeny is induced by a well-defined element of \( G(\mathbb A_f^p)/K^p \).

     An isogeny \( \alpha : (A_0, \iota_0, \Lambda_0) \to (A, \iota, \Lambda) \), whose degree is divisible by \( p \), defines \( (A, \iota, \Lambda, \bar{\eta}) \), where \( \bar{\eta} \) is the composition
     \[
     H^1(A, \mathbb A_f^p) \overset{\alpha}{\isom} H^1(A_0, \mathbb A_f^p) \xrightarrow{\eta_0} V \otimes \mathbb A_f^p.
     \]
     We now consider the \( p \)-divisible group \( X_0 \) of \( A_0 \) and its \emph{contravariant} Dieudonné module \( M_0 \) ([Dem]). The cotangent space of \( X_0 \) canonically identifies with \( M_0 / FM_0 \) ([Fo], III.4.3). Then \( M_0 \) is equipped with\marginpar{205} an action of \( O_D \otimes \mathbb{Z}_p \) and a class of perfect alternating bilinear forms \( \psi \), which differ from one another by units in \( O_{L_0} \otimes O_{\mathfrak k} \). Here \( O_\fkk \) denotes the ring of Witt vectors of \( \FF \).
     We extend the action of \( D \otimes \mathbb{Q}_p \) and the class of alternating bilinear forms to the rational Dieudonné module \( M_{0\QQ} \). Let \( Y_p \) denote the set of sublattices \( M \subset M_{0\QQ} \) that satisfy the following conditions:
     \begin{enumerate}
     	\item \( M \) is invariant under the action of \( O_D \),
     	\item There exists an element \( c(M) \in (L_0 \otimes O_\fkk)^\times \) such that the dual space with respect to one of the forms \( \psi \) is \( M^* = c(M) \cdot M \),
     	\item \( M \) is \( V \)- and \( F \)-invariant, and it satisfies
     	\[
     	\text{Tr}(x \mid M/FM) = t(x)  
     	\]
     	(equality of elements of \( \FF \)). 
     \end{enumerate}
     \begin{lem}
     	\label{6.5}
The \( V \)-invariance required in (iii) follows from the condition (ii) and the \( F \)-invariance. 
     \end{lem}\begin{proof}
We have 
     \[
     \psi(Fu, Fv) = p \cdot \psi(u, v) ^\sigma,
     \]
     where \( \sigma : O_\fkk \to O_\fkk \) denotes the Frobenius on the Witt vectors. From \( FV = p \), it follows that
     \[
     \psi(u, V v) = \psi(Fu, v) ^{ \sigma^{-1}}.
     \]
     Thus, we have
     \[
     (FM)^* = V^{-1} M^* = V^{-1} \cdot c(M) M = c(M) ^{\sigma^{-1}} \cdot V^{-1} M.
     \]
     From the inclusion \( FM \subset M \), it follows that \( V M \subset M \).    \end{proof}

 	Let \(\alpha: (A_0, \iota_0, \Lambda_0, \bar{\eta}_0) \rightarrow (A, \iota, \Lambda, \bar{\eta})\) be an isogeny. The Dieudonn\'e module of \(A\) becomes an element of \(Y_p\) via \(\alpha\). By decomposing \(\alpha\) into its prime-to-\(p\) component and its \(p\)-primary component, we obtain an element of \((G(\mathbb{A}^p_f)/K^p )\times Y_p\). Its image in
 	\[
 	Y(K) \underset{\text{Def}}{= }   \text{Aut}(A_0, \iota_0, \Lambda_0) \backslash \big( G(\mathbb{A}^p_f)/K^p \times Y_p   \big)
 	\]
 	depends only on \((A, \iota, \Lambda, \bar{\eta}) \in \mathscr J_K\) and not on the choice of the isogeny. Indeed, if \(\alpha\) is a composition,
 	\[
 	(A_0, \iota_0, \Lambda_0, \bar{\eta}_0) \xrightarrow{\alpha^p} (A', \iota', \Lambda', \bar{\eta}') \xrightarrow{\alpha_p} (A, \iota, \Lambda, \bar{\eta}),
 	\]
 	with \(\eta = \eta' \circ \alpha_p^*\), we obtain \(g\) by writing \(\eta' = g^{-1} \circ \eta_0 \circ \alpha^{p*}\). Then we regard \(\alpha_p\) as an isogeny from \(A_0\) to \(A\) to embed \(M(A)\) into \(M(A_0)\). Via \(\eta_0\), we map \(\text{Aut}(A_0, \iota_0, \Lambda_0)\) into \(G(\mathbb{A}^p_f)\), and if we replace \(\alpha\) by \(\alpha \circ \gamma\), \(\gamma \in \text{Aut}(A_0, \iota_0, \Lambda_0)\),\marginpar{206} then \(g\) is replaced by \(\gamma^{-1}g\). Moreover, \(M(A)\) is replaced by \(\gamma^{-1}M(A)\), where it is important to note that \(\gamma^{-1}M(A)\) depends only on the \(p\)-primary component of \(\gamma\).
 	
 	We first show that the mapping just defined is a bijection from \(\mathscr J_K\) to \(Y(K)\). The theory of Dieudonn\'e modules immediately shows that every element of \(G(\mathbb{A}^p_f)/K^p \times Y_p\) defines an isogeny \(\alpha\), so surjectivity is clear. Conversely, if \(\alpha_1\) and \(\alpha_2\) are two isogenies with the same image in \(Y(K)\), we obtain two elements in \(G(\mathbb{A}^p_f)/K^p \times Y_p\) that differ by an element \(\alpha \in \text{Aut}(A_0, \iota_0, \Lambda_0)\). Let \(\gamma = \frac{1}{n} \cdot \gamma'\), where \(\gamma'\) is an endomorphism of the abelian variety \(A_0\) (and \(n\) is a power of \(p\)). If we replace \(\alpha_1\) by \(\alpha_1 \circ \gamma'\) and \(\alpha_2\) by \(\alpha_2 \circ n\), we may assume that the elements of \(G(\mathbb{A}^p_f)/K^p \times Y_p\) defined by \(\alpha_1\) and \(\alpha_2\) are identical. Let
 	\[
 	\varphi_1 = \alpha_2 \alpha_1^{-1}, \quad \text{and} \quad \varphi_2 = \alpha_1^{-1} \alpha_2.
 	\]
 	If we can show that \(\varphi_1\) and \(\varphi_2\) are morphisms in the category of abelian varieties up to prime-to-\(p\) isogeny and not just isogenies, then they define mutually inverse \emph{isomorphisms} between \((A_1, \iota_1, \Lambda_1)\) and \((A_2, \iota_2, \Lambda_2)\), which are also compatible with the level structures. Let (with a new \(n\)) \(\varphi_1 = \frac{1}{n} \cdot \varphi_1'\) with \(\varphi_1' \in \text{Hom}(A_1, A_2)\). On the Dieudonn\'e modules, we obtain an inclusion
 	\[
 	\varphi_1'(M_2) \subset M_1.
 	\]
 	But since, by assumption, the elements of \(Y_p\) defined by \(\alpha_1\) and \(\alpha_2\) are identical, we have \(\varphi_1'(M_2) \subset n \cdot M_1\), and thus \(\varphi_1'\) factors through the quotient of \(A_1\) by its \(n\)-torsion points, so \(\varphi_1\) is a morphism. Similarly, one shows that \(\varphi_2\) is a morphism. In the constructed bijection, the action of \(\Phi\) on \(\mathscr J_K\) corresponds to the operator on \(Y(K)\) that acts on the \(Y_p\)-component and replaces the lattice \(M\) by the lattice \(F^rM\). This follows immediately from the fact that the relative Frobenius \(\text{Fr}: A \rightarrow A^{(p)}\) defines a map on the Dieudonn\'e modules, \(M^{(p)} \rightarrow M\), which identifies the lattice \(M^{(p)}\) with \(FM\) (see [Dem], p.~63).
 	
 	Next, we want to compare the sets \(Y(K)\) and \(X_\phi(K)\). After choosing a \(D\)-linear $L_0$-symplectic similarity \(M_{0\mathbb{Q}} \cong V \otimes \fkk\), the group \(G(\fkk) \rtimes \text{Gal}(\overline{\mathbb{Q}}_p/\mathbb{Q}_p)\) acts on \(M_0\). (Here, \(V\) is again a \(D\)-module.) The set of lattices \(M \subset M_{0\mathbb{Q}}\) satisfying conditions (i) and (ii) before \ref{6.5} forms a homogeneous space under \(G(\fkk)\) and can be identified with \(\mathscr X = G(\fkk) \cdot x^0\) (existence of a normal form). Due to the nature of the identification of \(\mathscr P\) with \(\mathscr M\), \(\xi_p = \phi \circ \zeta_p\) determines the rational Dieudonn\'e module \((M_{0\mathbb{Q}}, F)\) up to isomorphism. This means that there exists an element \(h_1 \in G(\fkk)\) such that
 	\[
 	F = \text{ad}\,{h_1}(F_\phi),
 	\]
 	where \(F_\phi = \xi'(d_\sigma)\) is the ``Frobenius'' constructed before \eqref{5.d}. Here, \(d_\sigma \in W_{L_n/\mathbb{Q}_p}\) denotes a lift of the Frobenius element in \(\text{Gal}(L_n/\mathbb{Q}_p)\), and \(\xi'\) denotes the homomorphism from the Weil group\marginpar{207} \(W_{L_n/\mathbb{Q}_p}\) to \(G(\fkk) \rtimes \text{Gal}(\fkk/\mathbb{Q}_p)\) used in the construction of \(\xi_p' = \text{ad}\, h (\xi_p): \mathscr D_{L_n}^{L_n} \rightarrow \mathscr G_G\). In the above equation, the right hand side is of course to be interpreted as the operator on   \(M_{0\mathbb{Q}}\) defined by the element there.  We show that the condition
 	\[
 	\text{inv}(x^1, x^2) = \mu,
 	\]
 	where \(x^i\) corresponds to the lattice \(M^i\), is equivalent to the inclusion
 	\[
 	pM^1 \subset M^2 \subset M^1,
 	\]
 	such that \(\text{Tr}(x\mid M^1/M^2) = t(x)\), \(x \in \mathcal{O}_D\). 
 	
 	Clearly, it suffices to find two lattices \(M^1, M^2\) satisfying the above condition, and such that for the corresponding points \(x^i \in \mathscr  X\), the invariant is equal to \(\mu\). Let \(T'\) be a Cartan subgroup of \(G\) defined over \(\mathbb{Q}_p\) that splits over \(\fkk\), and let \(\mu' \in X_*(T')\) be a cocharacter conjugate to \(\mu\). We choose in the apartment corresponding to \(T'\) the special point \(x^1\) corresponding to \(T'(\mathbb{Z}_p)\) ([T], §3.6.1), which corresponds to a lattice \(M^1\). Then \(T'(\mathcal{O}_\fkk)\) acts on \(M^1\) and \(T'(\FF)\) acts on \(M^1/pM^1\), and the resulting representation is the reduction modulo \(p\) of the representation of \(T'\) on \(V_{\mathbb{Q}_p}\). We set
 	\[
 	M^2 = p \cdot M^1 + \sum_{\substack{\langle\mu', \chi \rangle = 0}} \FF_\chi \subset M^1,
 	\]
 	where \(\FF_\chi\) is the character space corresponding to \(\chi\). Since \(\mu'\) is conjugate to \(\mu\), the characters appearing in \(M^1/M^2\) are precisely those characters \(\chi\) that appear in \(V_\CC\) and for which \(\langle \mu, \chi \rangle = 1\). Thus,
 	\[
 	\mathrm{Tr}(x \mid M^1/M^2) = \text{Tr}(x\mid V^{1,0}_h) \pmod{p}
 	\]
 	for \(x \in O_D\). It is also clear that \(x^2 = p^{\mu'} x^1\), so 
 	\[
 	\text{inv}(x^1, x^2) = -\text{inv}(x^2, x^1) = \mu
 	\]
 	(cf. example before \ref{5.2}). We obtain that for \(M \in Y(K)\),
 	\[
 	\text{inv}(M, FM) = \mu.
 	\]

 	It follows that \(h_1: \mathscr X \rightarrow \mathscr X\) establishes a bijection between \(X_p\) and \(Y_p\), which transforms the operator \(\Phi\) into the action of Frobenius, and that \(g \mapsto h_1 \cdot g \cdot h_1^{-1}\) defines an isomorphism between \(J_\phi\) and \(\text{Aut}(M_{0\mathbb{Q}}, \iota_0, \Lambda_0)\). From the definition of the motivic gerbe, it follows that we can identify \(I_\phi\) with \(\text{Aut}(A_0, \iota_0, \Lambda_0)\), as well as the action of \(I_\phi\) on \(X^p/K^p \times X_p\) with the action of \(\text{Aut}(A_0, \iota_0, \Lambda_0)\) on \(G(\mathbb{A}^p_f)/K^p \times Y_p\), provided we choose the base point of \(X^p/K^p\) determined by the \(l\)-adic cohomology of \(A_0\) and identify \(X_p\) with \(Y_p\) and \(J_\phi\) with \(\text{Aut}(M_{0\mathbb{Q}}, \iota_0, \Lambda_0)\) in the manner described above. We obtain the desired bijection between \(X_\phi(K)\) and \(Y(K)\).
 	 
   \section{Two examples}
    The\marginpar{208} first example will be a \( z \)-extension in the sense of Kottwitz [K1],  
    \[ 1 \to A \to G' \to G \to 1, \]
    such that the conjecture of §5 cannot hold simultaneously for both \( G \) and \( G' \). In this example, the derived groups of \( G' \) and $G$ cannot both be simply connected. Below we will construct a group \( G'' \) over \( \mathbb{Q} \) and a map \( h'': \mathbb S \to G''_{\RR} \) with the following properties:
    \begin{lr}\label{7.a}
 \( (G'', h'') \) defines a Shimura variety.
    \end{lr}
    \begin{lr}\label{7.b}
Let \( Z'' \) be the center of \( G'' \). Then
\[
H^1(\mathbb{Q}, Z'') = 0 \quad \text{and} \quad H^1(\mathbb{Q}_v, Z'') = 0
\]
for each prime \( v \) of \( \mathbb{Q} \).
    \end{lr}
     \begin{lr}\label{7.c} Let \( Z''_{\text{der}} \) be the center of \( G''_{\der} \). Then the map
    \[
    H^2(\mathbb{Q}, Z''_{\text{der}}) \to \prod_v H^2(\mathbb{Q}_v, Z''_{\text{der}})
    \]
    is injective.
\end{lr}
\begin{lr}\label{7.d}
There exists a finite algebraic group \( U \) over \( \mathbb{Q} \) and an embedding \( u'': U \to Z''_{\text{der}} \) such that the map
\[
H^2(\mathbb{Q}, U) \to \prod_v H^2(\mathbb{Q}_v, U)
\]
is  not  injective.
\end{lr} 
    
    We embed \( U \) in an induced torus \( A \) $$u:U \to A$$ and set
    \[ G' = A \times G'' / (u \times u'')(U), \quad G = G'' / u''(U).
    \]
    Then \( G' \) is a \( z \)-extension of \( G \) by  \( A \) as  above, and $$ G'_\der = G''_\der, \quad G_\der = G''_\der / u''(U). $$
     Let \( \bar G \) be an inner twist of \( G \), and \( \bar{G}'\)  and $\bar G''$ the corresponding twistings of \( G'\) and $G''$.
    
    Next\marginpar{209}, we define a cocycle \( g_\sigma '\) with values in \( \bar G'\), whose image in \( \bar G \) is locally trivial everywhere and globally trivial in \( G_{\text{ab}} \), but whose image in \( G'_{\text{ab}} \) is not trivial. Since \( A \) is an induced torus, we obtain a commutative diagram with exact rows:
    \[ \xymatrix{
    0 \ar[r] &  H^1(\mathbb{Q}, A/U) \ar[r] \ar[d] &  H^2(\mathbb{Q}, U) \ar[r] \ar[d] &   H^2(\mathbb{Q}, A)   \ar@{^(->}[d]   \\ 
    0 \ar[r]  & \prod _v H^1(\mathbb{Q}_v, A/U) \ar[r]  & \prod _v  H^2(\mathbb{Q}_v, U) \ar[r]  & \prod _v  H^2(\mathbb{Q}_v, A)   .  }
    \]
    
    The map in the last column is injective. Consequently, the class \( \{ a_{\rho, \sigma} \} \in H^2(\mathbb{Q}, U) \), which is locally trivial everywhere and whose existence we required in \eqref{7.d}, is the image of a class in \( H^1(\mathbb{Q}, A/U) \). Let \( \sigma \mapsto a'_\sigma \in A \) be a lift of a representative of this class. The coboundary of this cochain is a 2-cocycle with values in \( Z''_{\text{der}} \), which is locally trivial everywhere and thus trivial, according to \eqref{7.c}. Let \( \{ z_\sigma \} \) be a cochain whose coboundary is the inverse of the coboundary of \( \{ a_\sigma' \} \). Then,
    \[
    \sigma \mapsto g'_\sigma = a'_\sigma \cdot z_\sigma
    \]
    is a cocycle with values in \( A \cdot Z''_{\text{der}}/U \subset \bar G' \). Its image \( g_\sigma \) in \( \bar G = \bar G'' /U \) is the image of \( z_\sigma \). Since \( G_{\text{ab}} = \bar G''  /\bar G''_{\text{der}} \), its class in \( G_{\text{ab}} \) is obviously trivial. Locally,
    \[
    z_\sigma = v_\sigma \cdot u_\sigma,
    \]
    where \( v_\sigma \) is a cocycle with values in \( Z_{\text{der}}'' \) and \( u_\sigma \in U \). Consequently, \( g_\sigma \) is also the image of \( v_\sigma \) and is thus trivial by \eqref{7.b}. Therefore, the class \( \{ g_\sigma \} \) is locally trivial everywhere in \( \bar {G} \).
    
    We define \( h' \) as   \( 1 \times h'' \), and  define $h$ as the composition of \( h' \) and the projection from \( G' \) to \( G \). Then, \( (G', h') \) and \( (G, h) \) are Shimura varieties. The natural map \( \mathscr G_{G'} \to \mathscr G_ G \) assigns to each homomorphism \( \phi': \mathscr {Q} \to \mathscr G_{G'} \) a homomorphism \( \phi: \mathscr {Q} \to  \mathscr G_G \).

\begin{lem}\label{7.1}
	Two homomorphisms \( \phi_1' \) and \( \phi_2' \), which on the kernel induce the same homomorphism to \( G_{\text{ab}}' \), and give rise to equivalent homomorphisms $\phi_1$ and $\phi_2$, are equivalent.  
\end{lem}

The homomorphisms \( \phi_1 \) and \( \phi_2 \) are equal. The following sequence is exact:
\[
1 \to U \to G' \to G_{\text{ab}}' \times G.
\]
From the assumptions, it follows easily that \( \phi_1 \) and \( \phi_2 \) induce identical homomorphisms on the kernel. The statement then follows from \( H^1(\mathbb{Q}, A) = 0 \).

We\marginpar{210} have already seen (Remark before \ref{5.3}) that \( \text{Sh}_{K'}(G', h') \) is a surjective covering space of \( \text{Sh}_K(G, h) \) with the fiber
\[
A(\mathbb{Q}) \backslash \alpha^{-1}(K)/K'.
\]
Next, we want to show that if \( \phi ' \) and \( \phi \) are admissible, where \( \phi\) is induced by \( \phi' \), the set $$ Y_{\phi'} = I_{\phi'} \backslash X_p' \times X^{\prime p } / K^{\prime p} $$ is a surjective covering of the corresponding set \( X_\phi \) with the same fiber. Thus, the conjecture of §5 can only hold for \( G' \) and $G$ simultaneously, at least in a compatible way, if every admissible homomorphism \( \phi \) comes from an admissible homomorphism \( \phi' \). We will later find an admissible \( \phi' \) such that \( \phi'\) is rational on the kernel and \( I = G' \). Then, \( \mathscr I_{\phi'} = \bar G' \) and \( \mathscr I_\phi = \bar G  \) are inner twistings of \( G' \) and \( G \). Let \( \{ g_\sigma ' \} \) be the cocycle constructed above and \( \{ g_\sigma \} \) its image in \( \bar {G} \). Then $$  g \cdot \phi : \mathscr {Q} \to \mathscr G_ G $$ is admissible. In fact, because the image of \( \{ g_\sigma \} \) in \( G_{\text{ab}} \) is trivial, the validity of condition \eqref{5.a} is preserved under the transition from \( \phi \) to \( g \cdot \phi \), and since \( \{ g_\sigma \} \) is locally trivial everywhere, conditions \eqref{5.b},  \eqref{5.c}, \eqref{5.d} are also preserved. The homomorphism \( g' \cdot \phi' \) gives rise to \( g \cdot \phi \) and is not admissible because the image of \( g_\sigma \) in \( G'_{\text{ab}} \) is not trivial. If there were an admissible homomorphism \( \tilde{\phi}' \) that gives \( \phi \), it would induce the same homomorphism on the kernel as \( g' \cdot \phi' \), and thus, by Lemma \ref{7.1}, would coincide with \( g' \cdot \phi' \), which is impossible.

To prove the above statement, we specify that we take \( G' \) and $G$ to be  quasi-split over \( \mathbb{Q}_p \), and that \( K_p' \) and \( K_p \) are chosen to be hyperspecial. The torus \( A \) splits over an unramified extension. We may assume that the restriction of \( \phi' \) to the kernel is defined over \( \mathbb{Q} \). Then, \( I_{\phi'} = \mathscr I_{\phi'}(\mathbb{Q}) \) and \( I_\phi = \mathscr I_\phi(\mathbb{Q}) \), and the sequence
\[
1 \to A \to \mathscr I_{\phi'} \to \mathscr I_{\phi} \to 1
\]
is exact. Hence, $$ I_{\phi'} \to I_\phi$$ is surjective. Similarly, one concludes that \( X^{\prime p} \to X^p \) is surjective. Since \( G'(\fkk) \to G(\fkk) \) is also surjective, \( \mathscr X' \to \mathscr  X \) is surjective. Here, $$ \mathscr X' \subset \mathscr B (G', \fkk) = \mathscr  B(A, \fkk) \times \mathscr B(G'', \fkk) . $$
Let \( x' = (g'y^0, g'x^0) \in \mathscr X ' \) with image in \( X_p \):
\[
\text{inv}((gx^0), F(gx^0) )= \mu'' = \mu.
\]
Then\marginpar{211} we have
\begin{align*}
	 \text{inv}\big (  (g'y^0, g'x^0), F(g'y^0, g'x^0) \big ) & = \text{inv}\big( (g'y^0, g'x^0), (\sigma(g')y^0, Fg' x^0) \big ) \\ & = 0 + \mu '',
\end{align*} 
if \( g'y^0 \) is rational over \( \mathbb{Q}_p \). Since we can always find such a \( g' \), \( X_p' \to X_p \) is surjective.
Now consider the map
\[
X_p' \times X^{\prime p} \to X_p \times X^p.
\]
If \( x'_p \times x^{\prime  p} \) and \(\bar  x'_p \times \bar x^{\prime  p} \) have the same image in \( Y_\phi \), then for the images in \( X_p \times X^p \),
\[ \bar x_p \times \bar x^p= \gamma(x_p \times x^p) k^p, \quad \gamma \in I_\phi, \quad k^p  \in K^p.
\]
Since \( I_{\phi'} \to I_\phi \) is surjective, we can, after changing \( \bar x_p'  \times \bar x^{\prime p } \) by \( \gamma ' \), assume that \( \gamma = 1 \). Then,
\begin{align*}
	\bar x'_p &= a \cdot x'_p \quad a \in A(\QQ_p) , \\
	\bar x^{\prime p} & = x^{\prime p} g' \quad g' \in \alpha^{-1} (K^p ) .
\end{align*} 
Thus, \( Y_{\phi'} \to Y_\phi \) is a surjective covering with fiber
\[
A(\mathbb{Q}) \backslash \alpha^{-1}(K^p) \cdot A(\mathbb{Q}_p) / K^{\prime p}\cdot V,
\]
where \( V \) is the stabilizer of \( y^0 \) in \( A(\mathbb{Q}_p) \). The statement follows because \( K_p' \to K_p \) is surjective, and $$ \alpha^{-1}(K_p)/K_p' \cong A(\QQ_p)/V.    $$

It remains to construct an example. Let \( L \) be a totally imaginary quadratic extension of a totally real number field \( L_0 \). Let \( J \) be a 4-dimensional Hermitian form with coefficients in  \( L \), and let \( H_J \) be the corresponding group of unitary similarities:
\[
H_J(\mathbb{Q}) = \{ A \in \text{GL}(4, L) \mid A^* J A = \lambda \cdot J \}.
\]
The center of \( H_J \) is \( L^\times \). The group \( G'' \) will be a product of the form
\[
G'' =\prod_{ (L^i, L_0^i, J^i)} H_{J^i},
\]
and \( h'' \) is chosen as usual so that \eqref{7.a} holds. Property \eqref{7.b} is obvious. To check \eqref{7.c}, it suffices to consider \( G'' = H_J \). Then  \( Z_{\text{der}} ''\) is the induced module
\[
Z_{\text{der}}'' = \text{Ind}^{\text{Gal}(\overline{\mathbb{Q}}/\mathbb{Q})}_{\text{Gal}(\overline{\mathbb{Q}}/L_0)} \mathbb{Z}/4.
\]
It\marginpar{212} must be shown that
\[
H^2(L_0, \mathbb{Z}/4) \to \prod_v H^2(L_{0v}, \mathbb{Z}/4)
\]
is injective, or equivalently, by the Poitou-Tate theorem ([Se1], p.~II-49), that
\[
H^1(L_0, \mu_4) \to \prod_v H^1(L_{0v}, \mu_4)
\]
is injective. This follows from the Grunewald-Wang theorem ([AT], Chap.~10, Th.~1), because \( -1 \in L_0 \) is not a square.

Suppose we can find a totally real Galois extension \( L_0 \) of \( \mathbb{Q} \) and a Galois module \( U \) over \( \text{Gal}(L_0/\mathbb{Q}) \) such that \( 4U = 0 \), and such that the localization map in \eqref{7.d} is not injective. Then we can clearly find \( G '' \) as above such that \( U \) can be embedded into \( Z_{\text{der}}'' \). Instead of \( U \), we construct $$ V = \text{Hom}(U, \mu_4) . $$ 

 We require that \( 4V = 0 \), that \( V \) is a $\Gal(L/\QQ)$-module for a Galois CM-extension \( L \) of \( \mathbb{Q} \), for which complex conjugation \( \iota \) acts as \( -1 \), and that
\[
H^1(\mathbb{Q}, V) \to \prod_v H^1(\mathbb{Q}_v, V)
\]
is not injective. We imitate the procedure of Serre ([Se1], p.~III-39). Let
\begin{align*} 
L & = \mathbb{Q}(\sqrt{-1}, \sqrt{73}, \sqrt{89}) = \mathbb{Q}(\sqrt{-1}) \cdot \mathbb{Q}(\sqrt{73}, \sqrt{89}) \\ &  = K \cdot L_0,
\end{align*} 
with \( K = \mathbb{Q}(\sqrt{-1}) \). Then,
\begin{align*} 
\text{Gal}(L/\mathbb{Q}) & = \text{Gal}(L/L_0) \times \text{Gal}(L/K) \\ & \cong \mathbb{Z}/2 \times (\mathbb{Z}/2 \oplus \mathbb{Z}/2).
\end{align*} 
It is easy to verify that all decomposition groups are either trivial or cyclic of order 2. This is clear at the archimedean places and the places outside 2, 73, and 89, because at those places, the extension \( L \) is unramified. In \( \mathbb{Q}_2 \), 73 and 89 are squares because \( 73 \equiv 89 \equiv 1 \mod 8 \). Accordingly,
\[ \left(
\frac{-1}{73} \right) = 1, \quad \left(\frac{89}{73} \right) = \left( \frac{16}{73} \right) = 1 
\]
and 
\[ \left(
\frac{-1}{89} \right) = 1, \quad \left(\frac{73}{89} \right) = \left( \frac{-16}{89} \right) = 1.
\] 
Let\marginpar{213}
\[
M = \text{Ind}^{\text{Gal}(L/\mathbb{Q})}_{\text{Gal}(L/L_0)} \mu_4.
\]
Then \( M \) is the set of \( \mu_4 \)-valued functions on \( \text{Gal}(L/K) \), and \( \iota \) acts as \( -1 \). Let \( \text{Gal}(L/K) = \{1, \rho_1, \rho_2,\rho_3\} \).

Let
\[
V_0 = \{ (\varepsilon_\rho) \mid \varepsilon _\rho = \pm 1, \, \prod_{\rho \in \Gal(L/K)} \varepsilon_\rho  = 1 \}.
\] 
and 
$$ V= V_0 \cup V_0 \cdot v$$
 with
 $$v = (\alpha, -\alpha, -\alpha, -\alpha ), \quad \alpha \in \mu_4, \alpha \neq \pm 1. $$ 
 Then \( V \) is a $\Gal(L/\QQ)$-module. We have a commutative diagram with exact rows:
 \[\xymatrix{
 	H^0(\mathbb{Q}, M) \ar[r] \ar[d] & H^0(\mathbb{Q}, M/V) \ar[r] \ar[d] & H^1(\mathbb{Q}, V) \ar[r] \ar[d] & H^1(\mathbb{Q}, M) \ar@{^(->}[d] \\
 	\prod_v H^0(\mathbb{Q}_v, M) \ar[r] & \prod_v H^0(\mathbb{Q}_v, M/V) \ar[r] & \prod_v H^1(\mathbb{Q}_v, V) \ar[r] & \prod_v H^1(\mathbb{Q}_v, M).
 }
 \]
 Here, according to Grunewald’s theorem, the right vertical arrow is injective because \( -1 \in L_0 \) is not a square. To show that the localization map for \( V \) is not injective, it suffices to find an element \( y \) in \( M/V \) such that for each \( \sigma \in \text{Gal}(L/\mathbb{Q}) \) it is  the residue class of a \( \sigma \)-invariant element \( y_\sigma \in M \), but it is not the residue class of a Gal(\( L/\mathbb{Q} \))-invariant element (in $M$). The residue class \( y \) of  \( (1,-1,-1,-1) \) is such an element. In fact, it is not the image of a Gal(\( L/\mathbb{Q} \))-invariant element because only \( \pm(1,1,1,1) \) are invariant.
 \begin{enumerate}
 	\item If \( \sigma = \iota \), then we choose \( y_\sigma = (1,-1,-1,-1) \). 
 	\item If \( \sigma \in \text{Gal}(L/K) \), then we choose \( y_\sigma = (\alpha, \alpha, \alpha, \alpha) \).  
 	\item If \( \sigma = \iota \cdot \rho, \rho \in \text{Gal}(L/K) \), then we choose \( y = (z_1, z_\rho, z_\tau, z_{\tau\rho}) \) with $$z_1 = \alpha, \quad  z_\rho = -\alpha, \quad z_\tau = \alpha, \quad  z_{\tau \rho}  = -\alpha . $$
 \end{enumerate} 
 Finally, we need a \( \phi' \) such that \( \mathscr I_{\phi'}\) is an inner form of \( G' \). We define \( \phi' \) as \( \psi_{T', \mu'} \) with an appropriate choice of \( (T', \mu') \). Instead of \( T', \mu' \), it is sufficient to find \( (T'', \mu'') \) such that the corresponding element \( \gamma_n'' \in G''(\mathbb{Q}) \) is central, and for this, we can assume that \( G'' \) is of the form \( H_J \). We then define \( T '' \) using an extension \( K \)  of $L$ of degree 4. It suffices to choose \( K \) such that the primes lying above \( p \) in \( K \) are inert\footnote{It seems that the authors use the term `inert'' (träge) to describe any situation where $\mathfrak p \mathcal O_K = \mathfrak P^e$ for a prime $\mathfrak p$ of $L$ and $\mathfrak P$ of $K$.  } over $L$, because then the image of \( \gamma_n'' \) in \( G''_{\ad} \) will be elliptic at \( p \), and since it is elliptic at infinity, it is of finite order. We still have the free choices $K$ and \( p \), which should however be unramified in \( L \). We\marginpar{214} select \( K \) of the form \( K = L \cdot K_1 \), where \( K_1 \) is an extension of degree 4 over \( \mathbb{Q} \), such that \( p \) is totally ramified in \( K_1 \). Namely, if \( p \equiv 1 \mod 4 \), then there is a subgroup of index 4 in \( \mathbb{Z}_p^\times \), which defines a totally ramified abelian extension of degree 4 over \( \mathbb{Q}_p \), and one could choose a totally real global number field of degree 4 that behaves this way in \( p \). We define \( J \) by the equation $$ J(x, y) = \text{Tr}_{K/L} x \bar y . $$\\
 
 We now want to give an example of a group where there are admissible homomorphisms \( \phi \) that are not equivalent to a \( \psi_{T, \mu} \). Let \( E \) be a totally real extension of degree \( n \) of \( \mathbb{Q} \) where \( p \) remains prime. Let \( D \) be a quaternion algebra over \( E \), which is unramified at all infinite places and ramified at \( p \). The group \( G \) is the multiplicative group \( D^\times \), viewed as an algebraic group over \( \mathbb{Q} \). After choosing an isomorphism \( G_\RR \cong \GL(2)_{\mathbb{R}} ^n \), we define \( h : \mathbb S \to G_\RR \) component-wise as usual:
 \[
 h_i : a + b\sqrt{-1} \in \mathbb{C}^\times \mapsto \begin{pmatrix} a & b \\ -b & a \end{pmatrix} \in \GL(2,\mathbb{R}).
 \] 
 Let \( K \) be a purely imaginary quadratic extension of \( E \) that is embedded in \( D \) at all places outside \( p \), while \( p \) splits as \( p = \fkp_1 \cdot \fkp_2 \). Let \( q = p^m \). As in §3, for sufficiently large \( m \), one constructs a Weil number \( \gamma \in K \) for which   \( \text{ord}_{\fkp_i} \gamma / \text{ord}_{\fkp_i} q = v_i / 2n \), where \( v_1 \) and \( v_2 \) are two given positive integers with 
 \begin{equation} \label{7.e}
  v_1 + v_2 = 2n.
  \end{equation}
 
 If \( E \neq \mathbb{Q} \), we can choose \( v_i \) to be odd (for later purposes) and different from each other. Then \( K = E(\gamma) \). Let \( T = K^\times \) as an algebraic torus. First, we embed \( T(\overline{\mathbb{Q}}) \) into \( G(\overline{\mathbb{Q}}) \), and then we construct an admissible homomorphism \( \phi : \mathscr P \to \mathscr G_G \), where the group \( K_p \subset G(\mathbb{Q}_p) \) appearing in condition \eqref{5.d}  is the unique maximal compact subgroup, such that \( \phi(\delta) = \gamma \), i.e., \( \phi(\delta_{km}) = \gamma^k \). This gives the sought example, as clearly \( \gamma \in I_\phi = K^\times \), which cannot be embedded in \( G(\mathbb{Q}) \).
 
  For the construction of \( \phi \), we need some cohomological considerations, which we will introduce first. We fix a quadratic extension \( K' \) of \( E \), in which \( p \) remains prime, and which is embedded in \( D \). Let \( L = K' \cdot K \). The following diagram describes the situation and defines the elements of the Galois group:
  
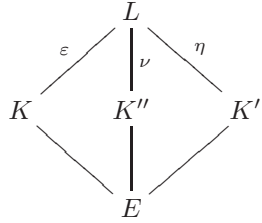
\[
  \xymatrix{
  	& L \ar@{-}[ld]_{\varepsilon} \ar@{-}[rd]^{\eta} \ar@{-}[d]^{\nu} & \\
  	K \ar@{-}[rd] & K'' \ar@{-}[d] & K' \ar@{-}[ld] \\
  	& E &
  }
  \]
  We\marginpar{215} have \( \nu = \varepsilon \cdot \eta = \eta \cdot \varepsilon \). We first consider \( D^\times \) and \( K^\times \) as algebraic groups over \( E \), and denote them by \( G^1 \) and \( T^1 \). We embed \( T^1 \) in \( \GL(2)_K \) such that \( \gamma \) goes to $\begin{pmatrix}
  	s_1 \\ & s_2
  \end{pmatrix}, ~ s_1\in K, s_2 = \eta(s_1)$. The action of the Galois group \( \text{Gal}(L/E) \) on \( \GL(2,L) \) is:
  \begin{align*}
  	\eta_* \begin{pmatrix}
  		s_1 \\ & s_2
  	\end{pmatrix} = \begin{pmatrix}
  	s_2 \\ & s_1
  \end{pmatrix}, \\
\varepsilon_* \begin{pmatrix}
	s_1 \\ & s_2
\end{pmatrix} = \begin{pmatrix}
	s_1 \\ & s_2
\end{pmatrix}.
  \end{align*}
  Over \( K' \), \( G^1 \) is isomorphic to \( \GL(2) \). As the twisting cocycle of \( \text{Gal}(K'/E) \), we choose the image in the adjoint group of the following cocycle in \( \GL(2) \):
  \[
  1 \mapsto 1, \quad \bar{\varepsilon} \mapsto j = \begin{pmatrix} 0 & 1 \\ \lambda & 0 \end{pmatrix}.
  \]
  Here, \( \lambda \in E^\times \) is chosen to be totally negative with \( \text{ord}_p \lambda = 1 \), so that \( \lambda \) is not a norm of an element of \( K'_p \). The twisted action of \( \text{Gal}(K'/E) \) is 
  \[
  \bar {\varepsilon} \begin{pmatrix} a & b \\ c & d \end{pmatrix} = j  \begin{pmatrix}
  	\varepsilon(a) & \varepsilon(b) \\ \varepsilon(c) & \varepsilon(d)
  \end{pmatrix} j^{-1}.
  \]
  For the twisted action of \( \text{Gal}(L/E) \) on \( \gamma \), we get
  \begin{equation}\label{7.f}
  	\varepsilon \begin{pmatrix}
  		s_1 \\ & s_2
  	\end{pmatrix} = \begin{pmatrix}
  	s_1 \\ & s_2
  \end{pmatrix}, \quad \eta \begin{pmatrix}
  s_1 \\ & s_2
\end{pmatrix} = \begin{pmatrix}
s_2 \\ & s_1
\end{pmatrix}. 
  \end{equation} 
  The central 2-cocycle corresponding to the above cochain is  
  \[
  a_{1,1} = a_{1, \bar \varepsilon} = a_{\bar \varepsilon, 1} = 1, \quad a_{\bar \varepsilon, \bar \varepsilon} = \lambda.
  \]
  We consider the following cochain of \( \text{Gal}(L/E) \) with values in \( G^1(L) \): 
  \[
  g_\varepsilon = j = g_\eta, \quad g_1 = g_\nu  = 1.
  \]  
  Its coboundary is: 
  \[
  g_{1,*} = g_{*,1} = 1, \quad g_{\nu,*} = g_{*,\nu } = 1, \quad g_{\varepsilon, \eta} = g_{\eta, \varepsilon} = \lambda.
  \]
  By the above formulas \eqref{7.f}, we have
  \begin{align} \label{7.g}
g_{\sigma} \sigma(\gamma)g^{-1}_{\sigma} = \gamma, \quad \sigma \in \text{Gal}(L/E).
  \end{align}   So on the left side of this equation is the usual action of  \( \text{Gal}(L/E) \) on \( T^1(L) \).
   
   This\marginpar{216} formula also holds when \( g_\sigma \) is replaced by \( t_\sigma \cdot g \), where \( t_\sigma \in T^1 \). Therefore, we consider \( \{ g_{\rho,\sigma} \} \) as a 2-cocycle with values in \( T^1 \). Its class in 
   \[
   H^2(L/E, T^1) = H^2(L/K, \mathbb G_m) = \ker \big( \text{Br}(K) \to \text{Br}(L) \big)
   \] 
   is clearly the image of the twisting cocycle of \( G^1 \) under the map from \( \text{Br}(E) \) to \( \text{Br}(K) \). Since the field \( K \) can be embedded locally everywhere outside \( p \) into \( G^1 \), the class of \( \{ g_{\rho,\sigma} \} \) is trivial outside \( p \), while in $p$ it is the image of the nontrivial element of \( \text{Br}(E_p)_2 \): \begin{align*}
   	 \text{Br}(E_p)_2 & \To \text{Br}(K_{\fkp_1})_2 \oplus \text{Br}(K_{\fkp_2})_2
   	 \\ \ZZ/2 & \xrightarrow{\text{diagonal}} \mathbb{Z}/2 \oplus \mathbb{Z}/2.
   \end{align*} 
   Now, to define the homomorphism \( \phi \), let \( \mathfrak A \) be a left representative system of \( \text{Gal}(\overline{\mathbb{Q}}/\mathbb{Q}) \) modulo \( \text{Gal}(\overline{\mathbb{Q}}/E) \) that contains 1. Then
   \[
   \mathscr G_G = \prod_{\tau \in \mathfrak A} \GL(2, \overline \QQ) \rtimes \Gal(\overline \QQ/\QQ), 
  \]  where \[  \sigma( \prod_\tau A_\tau) = \prod_\tau \sigma_\tau (A_{\tau'}), \qquad \tau \sigma = \sigma_\tau \cdot \tau'; \quad \tau, \tau' \in \mathfrak A, \quad  \sigma_\tau \in \Gal(\overline \QQ/E).  
   \]
   On the kernel \( P \), \( \phi \) is defined by:
   \[
   \phi(\delta_k) = (\gamma^k_\tau), \quad \delta_k \in P(L,m) 
   \] 
   with \( \gamma_\tau = \gamma \) for all \( \tau \in \mathfrak A \). Then, \( \phi|_P \) is a homomorphism of tori defined over \( \mathbb{Q} \), because \( \gamma \in T(\mathbb{Q}) \). Thus, \( \phi(p_{\rho,\sigma}) \) is a 2-cocycle of \( \text{Gal}(\overline{\mathbb{Q}}/\mathbb{Q}) \) with values in \( T \).
   
   On the generators \( p_\rho \) of \(\mathscr  P \), we define \( \phi \) tentatively by 
   \[
   \phi(p_\rho) = \prod_\tau g_\tau (\rho) \times \rho  \] with $g_\tau(\rho) = g_{\rho\tau}.$ By \eqref{7.g}, we have 
   \begin{align*}  
  \phi(p_\rho) \cdot \phi(\delta) \cdot \phi(p_\rho)^{-1} & = \prod_\tau  g_\tau(\rho) \cdot \rho_\tau (\gamma) \cdot  g_\tau(\rho)^{-1}   \\ & = \prod_\tau \gamma_\tau =   \phi(\delta).
 \end{align*} 
   This identity remains true even if \( g_\tau(\rho) \) is replaced by \( t_\tau(\rho) \cdot g_\tau(\rho) \) with \( t_\tau \in T^1 \), and we reserve the right to make this change.
   
   For \( \phi \) to be a homomorphism, we must also have:    
   \[
   g_\tau(\rho) \cdot \rho_\tau(  g_{\tau'}(\sigma) ) = \phi(p_{\rho , \sigma}) \cdot   g_\tau(\rho \sigma).
   \]
   Since\marginpar{217} \( (\rho\sigma)_\tau = \rho_\tau \sigma_\tau \), we have 
   \[
   g_\tau(\rho) \cdot \rho_\tau (g_{\tau'}(\sigma)) \cdot g_\tau(\rho \sigma)^{-1} = g_{\rho_\tau} \cdot \rho_\tau(g_{\sigma_{\tau'}}) \cdot g^{-1}_{\rho_{\tau}\cdot  \sigma_{\tau'}}. 
   \]
   We must now compare this 2-cocycle of \( \text{Gal}(\overline{\mathbb{Q}}/F) \) with values in \( T^1 \), which we already identified earlier, with the 2-cocycle \( \phi(p_{\rho , \sigma}) \) of \( \text{Gal}(\overline{\mathbb{Q}}/\mathbb{Q}) \) with values in \( T \). According to Shapiro's Lemma, \( \phi(p_{\rho , \sigma}) \) is cohomologous to an induced cocycle with values in \( T^1 \), namely the cocycle \( \phi^1(p_{\rho, \sigma}) \) restricted to \( \text{Gal}(\overline{\mathbb{Q}}/E) \), where \( \phi^1: P \to T^1 \) maps \( \delta \) to \( \gamma \). 
   By construction of \(\mathscr  P \), this cocycle is trivial outside \( p \) and at the archimedean places. Since \( K \) is totally imaginary, the Brauer group at infinity is trivial, so the cocycle is also trivial at the infinite places. To determine the class of the cocycle at the place \( p \), we can replace \( \phi \) by \( \xi_p = \phi \circ \zeta_p \) and then project onto the first component.
   Let \( L' \) be a Galois extension of \( \mathbb{Q} \) that contains \( L \). We can assume that \( L ' \) is unramified at \( p \). Let \(\mathfrak  p \) be an extension of \( \fkp_1 \) to \( L ' \), and let \( [L'_\fkp : \mathbb{Q}_p] = r \). Then \( \xi_p \) is defined on \( \mathscr D_{L_r}^{L_r} \).
   If we introduce the fundamental cocycle \( a_{\sigma^i, \sigma^j} \) in the usual way (where $\sigma$ is the Frobenius element), we have 
   $
   a_{\sigma^i, \sigma^j} = 1,$ $ 0 \leq i,j < r, i+j < r;$ $a_{\sigma^i, \sigma^j} = p^{-1},$ $ 0 \leq i,j < r, i+j \geq r,
   $ and  
   \[
   \xi_p^1(p^{-1}) = \nu_2(p^{-1}),
   \] 
   where \( \nu_2 \) is defined by 
   \[
   \left | \prod_{\sigma \in \Gal(L'_\fkp/\QQ_p)} \sigma \lambda(\gamma) \right |_p = q^{\nu_2(\lambda)}, \quad \lambda \in X^*(T^1).
   \]

    By the definition of \( \gamma \), the left hand side is equal to 
    \[
    q^{- \frac{(\lambda_1 v _1 + \lambda_2 v _2) r} {2n} } ,  
    \]
    where \( \lambda_1 \) and \( \lambda_2 \) have an obvious meaning. The cocycle over \( E_p \) can therefore be defined on \( \text{Gal}(K_p' / E_p) \), namely by
    \[
    \varepsilon \mapsto p^{-(\lambda_1 v _1 + \lambda_2 v _2)}.
    \] 
    Since \( v_1 \) and \( v_2 \) are odd, we get the image in \( H^2(K_p' / E_p, T^1) \) of the nontrivial cocycle in \( H^2(K_p' / E_p, \mathbb G_m) \). 
    
    Thus, we see that the two cocycles being compared are locally equivalent everywhere and hence   equivalent. After correcting \( g_\tau(\rho) \) with suitable \( t_\tau(\rho) \), the desired homomorphism \( \phi \) is defined. It is important to note that the equivalence class of \( \phi \) is independent of the choice of correction factors \( (t_\tau) \), because \( H^1(\mathbb{Q}, T) = 0 \).
    
   We still need to prove that \( \phi \) is admissible. It is immediately checked that \( \phi_{\ab} \) and \( \psi_{G_{\ab},\mu _{\ab}} \) agree on the kernel. (This follows from \eqref{7.e}.) Since \( H^1(\mathbb{Q}, \mathbb G_m) = 0 \), they indeed agree. Outside of \( p \) and \( \infty \), \( \mathscr  P \) is trivial, so \( \phi \circ \zeta_l \) defines a cocycle with values in \( G \). Since \( H^1(\mathbb{Q}_l, G) = 0 \), \( \phi \circ \zeta_l \) is  equivalent to the canonical neutralization of \( \mathscr G_G \).
    
    The\marginpar{218} composition \( \phi \circ \zeta_\infty \) is, on the kernel \( \CC^\times \) of \(\mathscr  W \), equal to the diagonal homomorphism \( \CC^i \to G(\CC) = \prod \GL(2, \mathbb{C}) \). The corresponding cohomology class in 
    \[
    H^1(\mathbb{R}, G_{\text{ad}}) \cong  H^2(\mathbb{R}, R_{E / \QQ} \mathbb G_m) \cong  (\mathbb{Z}/2)^n
    \]
    was determined earlier. Since \( \lambda \) was taken to be totally negative, it is the diagonal element \( (-1, \dots, -1) \) of the group on the right hand side. The homomorphism \( \xi_\infty \) determines the same cohomology class. Since the center of \( G \) has trivial \( H^1 \), it follows that \( \xi_\infty \) and \( \phi \circ \zeta_\infty \) are  equivalent.
    
    The condition at \( p \) is more complicated to verify. The building of \( G \) is a product 
    \[
    \mathscr B(G, \fkk) = \mathscr B(G_{\text{der}}, \fkk) \times X_*(G_{\text{ab}}) \otimes \mathbb{R},
    \] 
    and it is the same as the building of \(R_{E_p/\QQ_p} 
     \GL(2) \), albeit with the twisted action. On \( X_*(G_{\text{ab}}) \otimes \mathbb{R} \), this action agrees with the usual one, while on the building \( \mathscr B(R_{E_p/\QQ_p} \SL(2), \fkk) \), the twisted action of \( \sigma \) is given by 
    \[
    \sigma(x^0, \dots, x^{n-1}) = (x^1, \dots, x^{n-1}, j \cdot \sigma^n_*(x^0)).
    \] 
    Here  the \( x^i \) are in the building of \( \SL(2) \). Let 
    \[
    x_0 = (x_0^{0}, \dots, x^{n-1}_0 ) \in \mathscr B(G, \fkk)
    \] 
    be the vertex corresponding to the maximal compact subgroup of the integral elements of \( \GL(2, E) \). Then it is clear that \( \sigma^{2n}(x_0) = x_0 \), and if we denote the translates of \( x_0 \) under the Galois group by \( x_i = \sigma^i(x_0) \), the \( x_i \) for \( i \in \mathbb{Z}/2n \) form the vertices of a polysimplex that is invariant under \( \text{Gal}(\fkk/\mathbb{Q}_p) \). We consider its factor in \( \mathscr B(G_{\text{der}}, \fkk) \). It is easy to see that its center of mass is the unique fixed point of the Galois action on the building of \( G_{\text{der}} \). By the Fixed Point Theorem of Bruhat and Tits, the building  \( \mathscr B(G_{\text{der}}, \mathbb{Q}_p) \) consists of this single point. Since we had chosen \( K_p \) to be the well-defined maximal compact subgroup of \( G(\mathbb{Q}_p) \), i.e., \( K_p \cap G_{\text{der}}(\mathbb{Q}_p) = G_{\text{der}}(\mathbb{Q}_p) \), and because the stabilizer of a chamber in a simply connected group fixes all points in the chamber, we have 
    \[
    K_p = \{ g \in G(\mathbb{Q}_p) \mid g \cdot x_i = x_i, \, i \in \mathbb{Z}/2n \}.
    \] 
    Thus, we have represented \( K_p \) in the form used in \S 5 for the formulation of the admissibility conditions. Next, we calculate the operator \( F \). Let \( v_1 = 2a + 1 \) and \( v_2 = 2b +1 \), so \( a + b = n - 1 \). We choose the fundamental class for the unramified extension \( L_{2n} = L_{\fkp_1} \) of degree \( 2n \) as in the example before \ref{5.2}. We send the element \( d_\sigma = 1 \times \sigma \in \mathscr D_{L_{2n}}^{L_{2n}} \) to
    \[
    F = \left( \begin{pmatrix}
    	p \\ & 1 
    \end{pmatrix} , \dots, \begin{pmatrix}
    1 \\ & p 
\end{pmatrix} , \dots  , \begin{pmatrix}
1 \\ & p 
\end{pmatrix}  , j \right) \times \sigma.
    \]
    Here\marginpar{219} the matrix $\begin{pmatrix}
    	p \\ & 1 
    \end{pmatrix} $ occurs $a$ times and $\begin{pmatrix}
    1 \\ & p
\end{pmatrix} $ occurs  $b$ times.    On the kernel, we define the homomorphism $$ \nu : \lambda \mapsto \lambda_1 v_1 + \lambda_2 v_2 . $$ 
     To verify that a homomorphism \( \tilde{\xi}: \mathscr D^{L_{2n}} \to \mathscr G_G \) is defined in this way, it suffices to show that \( F^{2n} = \nu(p^{-1}) \) in \( G(\mathbb{Q}_p) = G^1(E_p) \). By using the fact that both \( \sigma^n \) and the element \( j \) act by swapping the entries of a diagonal matrix, 
     \[
     \sigma^n \begin{pmatrix} s_1 \\ & s_2 \end{pmatrix} = \begin{pmatrix} s_2 \\  & s_1 \end{pmatrix} = j \begin{pmatrix} s_1 \\  & s_2 \end{pmatrix} j^{-1},
     \]
      we obtain
$   F^n = \begin{pmatrix}
	p^a \\ & p^b
\end{pmatrix}  j \times \sigma^n  \in G^1(E_p) \times \sigma^n. $
     Hence we have 
     \[F^{2n} = \begin{pmatrix}
     	p^{2a+1} \\ & p^{2b+1}
     \end{pmatrix} = \begin{pmatrix}
     p^{v_1} \\ & p^{v_2}
 \end{pmatrix}
 = \nu(p^{-1}).      \] 
     The localized homomorphism \( \phi \circ \zeta_p \) agrees with \( \tilde{\xi} \) on the kernel. Therefore, they differ by a 1-cocycle with values in \( T \), and since \( H^1(\mathbb{Q}_p, T) = 0 \), they are   equivalent. This implies that the image of \( d_\sigma \) is the operator \( F \) defined in \S 5. Clearly  we have 
    \begin{align*}
    	\inv (\sigma x_i , Fx_i) &  = \inv \left( x_{i+1}, \left( \begin{pmatrix}
     	p \\ & 1
     \end{pmatrix} , \dots, \begin{pmatrix}
     1 \\ & p 
 \end{pmatrix} , j \right) \cdot x_{i+1} \right)  \\ & = (\mu,\dots, \mu). 
\end{align*} 
     Hence \( X_p \) contains the point \( (x_0, \dots, x_{2n-1}) \) and is non-empty. Therefore, \( \phi \) also satisfies the admissibility condition \eqref{5.d}.

\end{document}